\documentclass[10pt]{amsart}
\usepackage{amsmath}
\usepackage{amsxtra}
\usepackage{amscd}
\usepackage{amsthm}
\usepackage{amsfonts}
\usepackage{amssymb}
\usepackage{eucal}
\usepackage[all]{xy}
\usepackage{graphicx}
\usepackage{pifont}

\newtheorem{prop}[subsubsection]{Proposition}

\newtheorem{cor}[subsubsection]{Corollary}
\newtheorem{lem}[subsubsection]{Lemma}

\newtheorem{defn}[subsubsection]{Definition}

\newtheorem{thm}[subsubsection]{Theorem}

\theoremstyle{remark}
\newtheorem{rem}[subsubsection]{Remark}

\newcommand{\lemref}[1]{Lemma~\ref{#1}}
\newcommand{\thmref}[1]{Theorem~\ref{#1}}
\newcommand{\secref}[1]{Sect.~\ref{#1}}
\newcommand{\corref}[1]{Corollary~\ref{#1}}
\newcommand{\propref}[1]{Proposition~\ref{#1}}

\numberwithin{equation}{section}

\newcommand{\nc}{\newcommand}

\nc{\ssec}{\subsection}
\nc{\sssec}{\subsubsection}

\nc{\renc}{\renewcommand}
\nc{\on}{\operatorname}
\nc\ol{\overline}
\nc\wt{\widetilde}
\nc{\Loc}{\on{Loc}}
\nc{\Bun}{\on{Bun}}
\nc{\BQ}{{\mathbb{Q}}}
\nc{\BA}{{\mathbb{A}}}
\nc{\BC}{{\mathbb{C}}}
\nc{\BH}{{\mathbb{H}}}
\nc{\BG}{{\mathbb{G}}}
\nc{\BK}{{\mathbb{K}}}
\nc{\BN}{{\mathbb{N}}}
\nc{\BD}{{\mathbb{D}}}
\nc{\BV}{{\mathbb{V}}}
\nc{\BL}{{\mathbb{L}}}
\nc{\BP}{{\mathbb{P}}}
\nc{\CA}{{\mathcal{A}}}
\nc{\CC}{{\mathcal{C}}}
\nc{\CO}{{\mathcal{O}}}
\nc{\CP}{{\mathcal{P}}}
\nc{\CR}{{\mathcal{R}}}
\nc{\CW}{{\mathcal{W}}}
\nc{\CU}{{\mathcal{U}}}
\nc{\CK}{{\mathcal{K}}}
\nc{\CM}{{\mathcal{M}}}
\nc{\CL}{{\mathcal{L}}}
\nc{\CF}{{\mathcal{F}}}
\nc{\CE}{{\mathcal{E}}}
\nc{\CY}{{\mathcal{Y}}}
\nc{\CZ}{{\mathcal{Z}}}

\nc{\D}{{\mathcal{D}}}
\nc{\fg}{{\mathfrak{g}}}
\nc{\fD}{{\mathfrak{D}}}
\nc{\fh}{{\mathfrak{h}}}
\nc{\fc}{{\mathfrak{c}}}
\nc{\fn}{{\mathfrak{n}}}
\nc{\sM}{{\mathsf M}}
\nc{\ppart}{(\!(t)\!)}
\nc{\hg}{{\widehat\fg}}
\nc{\sF}{{\mathsf F}}
\nc{\sG}{{\mathsf G}}
\nc{\sT}{{\mathsf T}}
\nc{\sh}{{\mathsf h}}
\nc{\si}{{\mathsf i}}
\nc{\sj}{{\mathsf j}}
\renc{\sp}{{\mathsf p}}

\nc{\oS}{\overset{\circ}S{}}
\nc{\oT}{\overset{\circ}T{}}
\nc{\of}{\overset{\circ}f{}}

\nc{\bC}{{\mathbf{C}}}
\nc{\bE}{{\mathbf{E}}}
\nc{\bZ}{{\mathbf{Z}}}
\nc{\bD}{{\mathbf{D}}}
\nc{\bO}{{\mathbf{O}}}
\nc{\bU}{{\mathbf{U}}}
\nc{\bc}{{\mathbf{c}}}
\nc{\bd}{{\mathbf{d}}}
\nc{\be}{{\mathbf{e}}}
\nc{\bi}{{\mathbf{i}}}
\nc{\bM}{{\mathbf{M}}}
\nc{\bA}{{\mathbf{A}}}
\nc{\bI}{{\mathbf{I}}}
\nc{\bJ}{{\mathbf{J}}}
\nc{\bK}{{\mathbf{K}}}
\nc{\bDelta}{{\mathbf {\Delta}}}

\nc{\fW}{{\mathfrak{W}}}
\nc{\reg}{{\text{\rm reg}}}
\nc{\nilp}{{\text{\rm nilp}}}
\nc{\cG}{{\check{G}}}
\nc{\cB}{{\check{B}}}
\nc{\cg}{{\check{\fg}}}
\nc{\cb}{{\check{\fb}}}
\nc{\cn}{{\check{\fn}}}
\nc{\mer}{{\on{mer}}}
\nc{\Const}{\mathsf{Const}}
\nc{\Whit}{\on{Whit}}
\nc{\KL}{\on{KL}}
\nc{\FS}{\on{FS}}
\nc{\LocSys}{\on{LocSys}}
\nc{\QCoh}{\on{QCoh}}
\nc{\Coh}{\on{Coh}}
\nc{\IndCoh}{{\on{IndCoh}}}
\nc{\Cat}{\on{Cat}}
\nc{\Op}{\on{Op}}
\nc{\Gr}{\on{Gr}}
\nc{\Fl}{\on{Fl}}
\nc{\Rep}{\on{Rep}}
\renc{\mod}{{\on{-mod}}}
\nc{\Conn}{\on{Conn}}
\nc{\unit}{{\mathbf{1}}}

\nc{\Hom}{\on{Hom}}
\nc{\Maps}{\on{Maps}}
\nc{\CMaps}{{\mathcal Maps}}
\nc{\End}{\on{End}}
\nc{\Vect}{\on{Vect}}
\nc{\Av}{\on{Av}}
\nc{\Ind}{\on{Ind}}
\nc{\Spec}{\on{Spec}}
\nc{\KG}{K\backslash G}
\nc{\comult}{{co\text{-}mult}}
\nc{\counit}{{co\text{-}unit}}
\nc{\uHom}{{\underline{\Hom}}}
\nc{\dgSch}{\on{DGSch}}
\nc{\Sch}{\on{Sch}}
\nc{\affdgSch}{\on{DGSch}^{\on{aff}}}
\nc{\affSch}{\on{Sch^{\on{aff}}}}
\nc{\Groupoids}{\on{Grpd}}
\nc{\inftygroup}{\infty\on{-Grpd}}
\nc{\inftyCat}{\infty\on{-Cat}}
\nc{\StinftyCat}{\on{DGCat}}
\nc{\MoninftyCat}{\infty\on{-Cat}^{Mon}}
\nc{\SymMoninftyCat}{\infty\on{-Cat}^{\on{SymMon}}}
\nc{\PicCat}{\infty\on{-Grpd}^{\on{Pic}}}
\nc{\SymMonStinftyCat}{\on{DGCat}^{\on{SymMon}}}
\nc{\SymMonModStinftyCat}{\on{DGCat}^{\on{SymMon+Mod}}}
\nc{\SymMonModopStinftyCat}{\on{DGCat}^{\on{SymMon}^{\on{op}}+\on{Mod}}}
\nc{\MonStinftyCat}{\on{DGCat}^{\on{Mon}}}
\nc{\inftystack}{\on{Stk}}
\nc{\inftystackalg}{Stk^{1\text{-}alg}}
\nc{\inftyprestack}{\on{PreStk}}
\nc{\inftydgnearstack}{\on{NearStk}}
\nc{\inftydgstack}{\on{Stk}}
\nc{\inftydgstackalg}{DGStk^{1\text{-}alg}}
\nc{\inftydgprestack}{\on{PreStk}}

\nc{\mmod}{{\on{-}{\mathbf{mod}}}}
\nc{\wh}{\widehat}
\nc{\sotimes}{\overset{!}\otimes}

\nc{\starr}{\text{\dh}}

\begin{document}

\title[Ind-coherent sheaves]{Ind-coherent sheaves}

\author{Dennis Gaitsgory}

\dedicatory{To the memory of I.~M.~Gelfand}

\date{\today}

\begin{abstract}
We develop the theory of ind-coherent sheaves on schemes and stacks. The category of
ind-coherent sheaves is closely related, but inequivalent, to the category of quasi-coherent
sheaves, and the difference becomes crucial for the formulation of the categorical Geometric
Langlands Correspondence.
\end{abstract}

\maketitle

\tableofcontents

\section*{Introduction}

\ssec{Why ind-coherent sheaves?}

This paper grew out of a series of digressions in an attempt to write down the formulation
of the categorical Geometric Langlands Conjecture.

\medskip

Let us recall that the categorical Geometric Langlands Conjecture is supposed to say that
the following two categories are equivalent. One category is the (derived) category of D-modules
on the stack $\Bun_G$ classifying principal $G$-bundles on a smooth projective curve $X$.
The other category is the (derived) category of quasi-coherent sheaves on the stack $\LocSys_{\cG}$
classifying $\cG$-local systems on $X$, where $\cG$ is the Langlands dual of $G$.

\medskip

However, when $G$ is not a torus, the equivalence between
$$\on{D-mod}(\Bun_G) \text{ and } \QCoh(\LocSys_{\cG})$$
does not hold, but it is believed it will hold once we slightly modify
the categories $\on{D-mod}(\Bun_G)$ and $\QCoh(\LocSys_{\cG})$.

\sssec{}

So, the question is what ``modify slightly" means. However, prior to that, one
should ask what kind of categories we want to consider. 

\medskip 

Experience shows that when working with triangulated categories, 
\footnote{In the main body of the paper, instead of working with 
triangulated categories we will work with DG categories and functors.} the following framework
is convenient: we want to consider categories that are cocomplete (i.e., admit arbitrary
direct sums), and that are generated by a set of compact objects. The datum of such a category is 
equivalent (up to passing to the Karoubian completion) to the datum of the corresponding 
subcategory of compact objects. So, let us stipulate that this is the framework in which we want
to work. 
As functors between two such triangulated categories we will take those triangulated 
functors that commute with direct sums; we call such functors continuous.

\sssec{Categories of D-modules}

Let $X$ be a scheme of finite type (over a field of characteristic zero). We assign to it the (derived) category
$\on{D-mod}(X)$, by which we mean the unbounded derived category with quasi-coherent cohomologies. 
This category is compactly generated; the corresponding subcategory $\on{D-mod}(X)^c$ of
compact objects consists
of those complexes that are cohomologically bounded and coherent (i.e., have cohomologies in finitely many degrees,
and each cohomology is locally finitely generated). 

\medskip

The same goes through when $X$ is no longer a scheme, but a quasi-compact algebraic stack, under a mild 
additional hypothesis, see \cite{DrGa1}. Note, however, that in this case, compact objects of the category 
$\on{D-mod}(X)^c$ are less easy to describe, see \cite[Sect. 8.1]{DrGa1}.

\medskip

However, the stack $\Bun_G(X)$ 
is not quasi-compact, so the compact generation of the category $\on{D-mod}(\Bun_G)$, although true, is not obvious. 
The proof of the compact generation of $\on{D-mod}(\Bun_G)$ is the subject of the paper \cite{DrGa2}.

\medskip

In any case, in addition to the subcategory $\on{D-mod}(\Bun_G)^c$
of actual compact objects in $\on{D-mod}(\Bun_G)$, there are several other small subcategories
of $\on{D-mod}(\Bun_G)$, all of which would have been the same, had $\Bun_G$ been a quasi-compact scheme.
Any of these categories can be used to ``redefine" the derived category on $\Bun_G$. 

\medskip

Namely, if $\on{D-mod}(\Bun_G)^c{}'$ is such a category, one can consider its ind-completion
$$\on{D-mod}(\Bun_G)':=\on{Ind}(\on{D-mod}(\Bun_G)^c{}').$$Replacing the initial $\on{D-mod}(\Bun_G)$ by 
$\on{D-mod}(\Bun_G)'$ is what we mean by "slightly modifying" our category. 

\sssec{Categories of quasi-coherent sheaves}

Let $X$ be again a scheme of finite type, or more generally, a DG scheme
almost of finite type (see \secref{sss:DG schemes}), over a field of characteristic zero.
It is well-known after \cite{TT} that the category $\QCoh(X)$ is compactly
generated, where the subcategory $\QCoh(X)^c$ is the category 
$\QCoh(X)^{\on{perf}}$ of perfect objects. 

\medskip

In this case also, there is another natural candidate to replace the category
$\QCoh(X)^{\on{perf}}$, namely, the category $\Coh(X)$ which consists of bounded
complexes with coherent cohomologies. The ind-completion $\QCoh(X)':=\Ind(\Coh(X))$ is the
category that we denote $\IndCoh(X)$, and which is the main object of study in this 
paper. 

\medskip

There may be also other possibilities for a natural choice of a small subcategory
in $\QCoh(X)$. Specifically, if $X$ is a locally complete intersection, one can attach a 
certain subcategory
$$\QCoh(X)^{\on{perf}}\subset \QCoh(X)^{\on{perf}}_{{\mathcal N}}\subset \Coh(S)$$
to any Zariski-closed conical subset ${\mathcal N}$ of $\Spec_X(\on{Sym}(H^1(T_X)))$,
where $T_X$ denotes the tangent complex of $X$, see \cite{AG} where this theory
is developed. 

\sssec{}

An important feature of the examples considered above is that the resulting modified categories
$\on{D-mod}(\Bun_G)'$ and $\QCoh(X)'$ all carry a t-structure, such that
their eventually coconnective subcategories (i.e., $\on{D-mod}(\Bun_G)'{}^+$
and $\QCoh(X)'{}^+$, respectively) are in fact equivalent to the corresponding 
old categories (i.e., $\on{D-mod}(\Bun_G)^+$ and $\QCoh(X)^+$, respectively). 

\medskip

I.e., the
difference between the new categories and the corresponding old ones is that the former
are not left-complete in their t-structures, i.e., Postnikov towers do not necessarily converge
(see \secref{ss:QCos as left compl} where the notion of left-completeness is reviewed).

\medskip

However, this difference is non-negligible from the point of view of Geometric Langlands:
the conjectural equivalence between the modified categories
$$\on{D-mod}(\Bun_G)' \simeq \QCoh(\LocSys_{\cG})'$$
has an unbounded cohomological amplitude, so an object which is trivial with respect to the
t-structure on one side (i.e., has all of its cohomologies equal to $0$), may be 
non-trivial with respect to the t-structure on the other side. An example is provided by
the ``constant sheaf" D-module $$k_{\Bun_G}\in \on{D-mod}(\Bun_G)^+=\on{D-mod}(\Bun_G)'{}^+.$$

\ssec{So, why ind-coherent sheaves?}

Let us however return to the question why one should study specifically the category $\IndCoh(X)$.
In addition to the desire to understand the possible candidates for the spectral side of Geometric 
Langlands, there are three separate reasons of general nature. 

\sssec{}

Reason number one is that when $X$ is not a scheme, but an ind-scheme, the category $\IndCoh(X)$
is usually much more manageable than $\QCoh(X)$. This is explained in \cite[Sect. 2]{GR1}. 

\sssec{}

Reason number two has to do with the category of D-modules. For any scheme or stack $X$, the category
$\on{D-mod}(X)$ possesses two forgetful functors:
$${\mathbf{oblv}}^l:\on{D-mod}(X)\to \QCoh(X) \text{ and } {\mathbf{oblv}}^r:\on{D-mod}(X)\to \IndCoh(X),$$
that realize $\on{D-mod}(X)$ as ``left" D-modules and ``right" D-modules, respectively.

\medskip

However, the ``right" realization is much better behaved, see \cite{GR2} for details. 

\medskip

For example,
it has the nice feature of being compatible with the natural
t-structures on both sides: an object
$\CF\in  \on{D-mod}(X)$ is coconnective (i.e., cohomologically $\geq 0$) if and only if 
${\mathbf{oblv}}^r(\CF)$ is. \footnote{When $X$ is smooth, the ``left" realization is also compatible
with the t-structures up to a shift by $\dim(X)$. However, when $X$ is not smooth, the functor 
${\mathbf{oblv}}^l$ is not well-behaved with respect to the t-structures, which is the reason for
the popular misconception that ``left D-modules do not make sense for singular schemes." 
More precisely, the derived category of left D-modules makes sense always, but the corresponding
abelian category is indeed difficult to see from the ``left" realization.}

\medskip

So, if we want to study the category of D-modules in its incarnation as ``right" D-modules,
we need to study the category $\IndCoh(X)$. 

\sssec{}

Reason number three is even more fundamental. Recall that the Grothendieck-Serre duality
constructs a functor $f^!:\QCoh(Y)^+\to \QCoh(X)^+$ for a proper morphism
$f:X\to Y$, which is the right adjoint to the functor $f_*:\QCoh(X)\to \QCoh(Y)$.

\medskip

However, if we stipulate that we want to stay within the framework of cocomplete categories
and functors that commute with direct sums, a natural way to formulate this adjunction is
for the categories $\IndCoh(X)$ and $\IndCoh(Y)$. Indeed, the right adjoint to $f_*$,
considered as a functor $\QCoh(X)\to \QCoh(Y)$ will not be continuous, as $f_*$
does not necessarily send compact objects of $\QCoh(X)$, i.e., $\QCoh(X)^{\on{perf}}$,
to $\QCoh(Y)^{\on{perf}}$. But it does send $\Coh(X)$ to $\Coh(Y)$.

\medskip

We should remark that developing the formalism of the !-pullback and the base change
formulas that it satisfies, necessitates passing from the category or ordinary schemes
to that of DG schemes. Indeed, if we start with with a diagram of ordinary schemes
$$
\CD
& & X \\
& & @VV{f}V \\
Y' @>{g_Y}>> Y,
\endCD
$$
their Cartesian product $X':=X\underset{Y}\times Y'$ must be considered as a DG scheme,
or else the base change formula
\begin{equation} \label{e:base change 0}
g_Y^!\circ f_*\simeq f'_*\circ g_X^!
\end{equation}
would fail, where $f':X'\to Y'$ and $g_X:X'\to X$ denotes the resulting base changed maps.

\sssec{What is done in this paper}

We shall presently proceed to review the actual contents of this paper. 

\medskip

We should say right away that there is no
particular main theorem toward which other results are directed. Rather, we are trying to provide
a systematic exposition of the theory that we will be able to refer to in subsequent papers
on foundadtions of derived algebraic geometry needed for 
Geometric Langlands. 

\medskip

The present paper naturally splits into three parts: 

\medskip

Part I consists of Sections \ref{s:ind-coherent}-\ref{s:properties}, which deal
with properties of $\IndCoh$ on an individual DG scheme, and functorialities for an individual
morphism between DG schemes. The techniques used in this part essentially amount to 
homological algebra, and in this sense are rather elementary. 

\medskip

Part II consists of 
Sections \ref{s:!}-\ref{s:Serre} where we establish and exploit the functoriality properties of the assignment $$S\mapsto \IndCoh(S)$$
on the category of DG schemes as a whole. This part substantially relies on the theory
of $\infty$-categories.  

\medskip

Part III consists of Sections \ref{s:stacks}-\ref{s:Artin}, where we extend $\IndCoh$ to stacks and prestacks.  

\ssec{Contents of the paper: the ``elementary" part}

\sssec{}

In \secref{s:ind-coherent} we introduce the category $\IndCoh(S)$ where $S$ is a Noetherian
DG scheme. We note that when $S$ is a classical scheme (i.e., structure sheaf 
$\CO_X$ satisfies $H^i(\CO_X)=0$ for $i<0$), we recover the category that was introduced by
Krause in \cite{Kr}. 

\medskip

The main result of \secref{s:ind-coherent} is \propref{p:equiv on D plus}. It says that the eventually
coconnective part of $\IndCoh(X)$, denoted $\IndCoh(X)^+$, maps isomorphically to
$\QCoh(X)^+$. This is a key to transferring many of the functorialities of $\QCoh$ to
$\IndCoh$; in particular, the construction of direct images for $\IndCoh$ relies on this
equivalence. 

\sssec{}

Section \ref{s:non-Noeth} is a digression that will not be used elsewhere in the present paper.
Here we try to spell out the definition of $\IndCoh(S)$ for a scheme $S$ which is not
necessarily Noetherian. 

\medskip

The reason we do this is that we envisage the need of $\IndCoh$ when studying objects
such as the loop group $G\ppart$. In practice, $G\ppart$ can be presented as an inverse
limit of ind-schemes of finite type, i.e., in a sense we can treat $\IndCoh(G\ppart)$ be referring
back to the Noetherian case. However, the theory would be more elegant if we could give
an independent definition, which is what is done here. 

\sssec{}

In \secref{s:functorialities} we  introduce the three basic functors between the categories
$\IndCoh(S_1)$ and $\IndCoh(S_2)$ for a map $f:S_1\to S_2$ of Noetherian DG schemes.

\medskip

The first is direct image, denoted 
$$f^{\IndCoh}_*:\IndCoh(S_1)\to \IndCoh(S_2),$$ which is ``transported" from the usual
direct image $f_*:\QCoh(S_1)\to \QCoh(S_2)$. 

\medskip

The second is
$$f^{\IndCoh,*}:\IndCoh(S_2)\to \IndCoh(S_1),$$
which is also transported from the usual inverse image $f^*:\QCoh(S_2)\to \QCoh(S_1)$.
However, unlike the quasi-coherent setting, the functor $f^{\IndCoh,*}$ is only defined
for morphisms that are \emph{eventually coconnective}, i.e., those for which the usual
$f^*$ sends $\QCoh(S_2)^+$ to $\QCoh(S_1)^+$.

\medskip

The third functor is a specialty of $\IndCoh$: this is the !-pullback functor $f^!$
for a proper map $f:S_1\to S_2$, and which is defined as the right adjoint of
$f^{\IndCoh}_*$.

\sssec{}

In \secref{s:properties} we discuss some of the basic properties of $\IndCoh$,
which essentially say that there are no unpleasant surprises vis-\`a-vis the familiar
properties of $\QCoh$. E.g., Zariski descent, localization with respect to an
open embedding, etc., all hold like they do for $\QCoh$. 

\medskip

A pleasant feature of $\IndCoh$  (which does not hold for $\QCoh$)
is the convergence property: for a DG scheme $S$, the category
$\IndCoh(S)$ is the limit of the categories $\IndCoh(\tau^{\leq n}(S))$, where 
$\tau^{\leq n}(S)$ is the $n$-coconnective truncation of $S$. 

\ssec{Interlude: $\infty$-categories}

A significant portion of the work in the present paper goes into enhancing the functoriality
of $\IndCoh$ for an individual morphism $f$ to a functor from the 
category of $\dgSch_{\on{Noeth}}$ of Noetherian DG schemes
to the category of DG categories.

\medskip

We need to work with DG categories rather than with triangulated categories as
the former allow for such operation as taking limits. The latter is necessary in order
to extend $\IndCoh$ to stacks and to formulate descent.

\medskip

Since we are working with DG schemes, $\dgSch_{\on{Noeth}}$ form an $(\infty,1)$-category
rather than an ordinary category. Thus, homotopy theory naturally enters the picture.
For now we regard the category of DG categories, denoted $\StinftyCat$,
also as an $(\infty,1)$-category (rather than as a $(\infty,2)$-category).

\sssec{}

Throughout the paper we formulate statements of $\infty$-categorical nature in a way
which is independent of a particular model for the theory of $\infty$-categories (although
the model that we have in mind is that of quasi-categories as treated in \cite{Lu0}).

\medskip

However, formulating things at the level of $\infty$-categories is not such an easy thing to do.
Any of the models for $\infty$-categories is a very convenient computational/proof-generating
machine. However, given a particular algebro-geometric problem, such as the assignment 
$$S\mapsto \IndCoh(S),$$
it is rather difficult to feed it into this machine as an input. 

\medskip

E.g., in the model of quasi-categories, an $\infty$-category is a simplicial set, and a functor 
is a map of simplicial sets. It appears as a rather awkward endeavor to organize DG schemes
and ind-coherent sheaves on them into a particular simplicial set...
\footnote{It should be remarked that when one works with a specific manageable diagram of DG categories 
arising from algebraic geometry (e.g., the category of D-modules on the loop group viewed
as a monoidal category), it is sometimes possible to work in a specific model, and to carry 
out the constructions ``at the chain level", as is done, e.g., in \cite{FrenGa}. But in 
slightly more general situations, this approach does not have much promise.}

\sssec{}

So, the question is: how, for example, do me make $\IndCoh$ into a functor $$\dgSch_{\on{Noeth}}\to \StinftyCat,$$
e.g., with respect to the push-forward operation, denoted $f^{\IndCoh}_*$?  The procedure is multi-step and can 
be described as follows:

\medskip

\noindent {\it Step 0:} We start with the assignment $$A\mapsto A\mod:(\on{DG-rings})^{\on{op}}\to \StinftyCat,$$
which is elementary enough that it can be carried out in any given model. We view it as 
a functor
$$\QCoh_{\affdgSch}:\affdgSch\to \StinftyCat,$$
with respect to $f_*$. 

\medskip

\noindent {\it Step 1:} Passing to left adjoints, we obtain a functor
$$\QCoh^*_{\affdgSch}:(\affdgSch)^{\on{op}}\to \StinftyCat,$$
with respect to $f^*$. 

\medskip

\noindent {\it Step 2:} We apply the right Kan extension along the natural functor
$$\affdgSch\to \inftydgprestack$$
(here $\inftydgprestack$ is the $(\infty,1)$-category
category of derived $\infty$-prestacks, see \cite{Stacks}, Sect. 1.1.1.)
to obtain a functor 
$$\QCoh^*_{\inftydgprestack}:(\inftydgprestack)^{\on{op}}\to \StinftyCat.$$

\medskip

\noindent {\it Step 3:} We restrict along the natural functor $\dgSch\to \inftydgprestack$
(see \cite{Stacks}, Sect. 3.1) to obtain a functor
$$\QCoh^*_{\dgSch}:(\dgSch)^{\on{op}}\to \StinftyCat,$$
with respect to $f^*$. 

\medskip

\noindent {\it Step 4:} Passing to right adjoints, we obtain a functor
$$\QCoh_{\dgSch}:\dgSch\to \StinftyCat,$$
with respect to $f_*$.

\medskip

\noindent {\it Step 5:} Restricting to $\dgSch_{\on{Noeth}}\subset \dgSch$, and taking the eventually
coconnective parts, we obtain a functor
$$\QCoh^+_{\dgSch_{\on{Noeth}}}:\dgSch_{\on{Noeth}}\to \StinftyCat.$$

\medskip

\noindent {\it Step 6:} Finally, using \propref{p:equiv on D plus}, in \propref{p:upgrading IndCoh to functor} we use 
the latter functor do construct the desired functor
$$\IndCoh_{\dgSch_{\on{Noeth}}}:\dgSch_{\on{Noeth}}\to \StinftyCat.$$

\sssec{}

In subsequent sections of the paper we will extend the functor $\IndCoh_{\dgSch_{\on{Noeth}}}$,
most ambitiously, to a functor out of a certain $(\infty,1)$-category of correspondences between prestacks.

\medskip

However, the procedure will always be of the same kind: we will never be able to do it ``by hand" by
saying where the objects and $1$-morphisms go, and specifying associativity up to coherent homotopy. 
Rather, we will iterate many times the operations of restriction, and left and right Kan extension,
and passage to the adjoint functor.
\footnote{One observes that higher algebra, i.e., algebra done in $\infty$-categories, loses one of
the key features of usual algebra, namely, of objects being rather concrete (such as a module over
an algebra is a concrete set with concrete pieces of structure).} 

\sssec{A disclaimer}

In this paper we did not set it as our goal to give complete proofs of statements
of $\infty$-categorical nature. Most of these statements have to do with 
\emph{categories of correspondences}, introduced in \secref{s:!}. 

\medskip

The corresponding proofs are largely combinatorial in nature, and would 
increase the length of the paper by a large factor. The proofs of most of such 
statements will be supplied in the forthcoming book \cite{GR3}. In this sense,
the emphasis of the present paper is to make sure that homological algebra
works (maps between objects in a given DG category are isomorphisms
when they are expected to be such), while the emphasis of \cite{GR3}
is to show that the assignments of the form
$$\bi\in \bI \mapsto \bC_i\in \inftyCat,$$ where $\bI$ is some $\infty$-category,
are indeed functors $\bI\to \inftyCat$.

\medskip

That said, we do take care to single out every statement that requires a 
non-standard manipulation at the level of $\infty$-categories. I.e., we avoid assuming
that higher homotopies can be automatically arranged in every given problem at hand. 

\ssec{Contents: the rest of the paper}

\sssec{}

In \secref{s:!} our goal is to define the set-up for the functor of !-pullback for arbitrary
morphisms (i.e., not necessarily proper, but still of finite type), such that the base change
formula \eqref{e:base change 0} holds. 

\medskip

As was explained to us by J.~Lurie, a proper formulation of the existence of !-pullback
together with the base change property is encoded by enlarging the category $\dgSch$:
we leave the same objects, but $1$-morphisms from $S_1$ to $S_2$ 
are now correspondences
\begin{equation} \label{e:corr 0}
\CD
S_{1,2} @>{g}>> S_1 \\
@V{f}VV \\
S_2.
\endCD
\end{equation}
and compositions of morphisms are given by forming Cartesian products. I.e., when we want to
compose a morphism $S_1\to S_2$ as above with a morphism $S_2\to S_3$ given by
$$
\CD
S_{2,3} @>{g'}>> S_2 \\
@V{f'}VV \\
S_3,
\endCD
$$
the composition is given by the diagram
$$
\CD
S_{1,2}\underset{S_2}\times S_{2,3}  @>>> S_1 \\
@VVV  \\
S_3.
\endCD
$$
For a morphism \eqref{e:corr 0}, the corresponding functor $\IndCoh(S_1)\to \IndCoh(S_2)$
is given by $$f_*^{\IndCoh}\circ g^!.$$

\medskip

Suppose now that we have managed to define $\IndCoh$ as a functor out of the category of 
correspondences. However, there are still multiple compatibilities that are supposed to hold
that express how the isomorphisms \eqref{e:base change 0} are compatible with the adjunction
$(f^{\IndCoh}_*,f^!)$ when $f$ is a proper morphism. 

\medskip

Another idea of J.~Lurie's is that this structure is most naturally encoded by further expanding 
our category of correspondences, by making it into a $2$-category, where we allow $2$-morphisms to 
be morphisms between correspondences
that are not necessarily isomorphisms, but rather proper maps. 

\sssec{}

In \secref{s:construction of !} we indicate the main steps in the construction of $\IndCoh$ as a functor out
of the category of correspondences. The procedure mimics the classical construction
of the !-pullback functor, but is done in the $\infty$-categorical language. Full details of this
construction will be supplied in \cite{GR3}.

\sssec{}

In \secref{s:Gorenstein} we study the behavior of the !-pullback functor under eventually
coconnective, Gorenstein and smooth morphisms.

\sssec{}

The goal of \secref{s:descent} is to prove that the category $\IndCoh$ satisfies
faithfully flat descent with respect to the !-pullback functor. The argument we give
was explained to us by J.~Lurie.

\sssec{}

In \secref{s:Serre} we discuss the self-duality property of the category $\IndCoh(S)$ for
a DG scheme $S$ almost of finite type over the ground field. The self-duality boils
down to the classical Serre duality anti-equivalence of the category $\Coh(S)$, and
is automatic from the formalism of $\IndCoh$ as a functor on the category of correspondences.

\sssec{}

In \secref{s:stacks} we take the theory as far as it goes when $\CY$ is an arbitrary prestack.

\sssec{}

In \secref{s:Artin} we specialize to the case of Artin stacks. The main feature of $\IndCoh(\CY)$
for Artin stacks is that this category can be recovered from looking at just affine schemes 
equipped with a \emph{smooth} map to $\CY$. This allows us to introduce a t-structure on 
$\IndCoh(\CY)$, and establish a number of properties that make the case of Artin stacks
close to that of schemes. 

\ssec{Conventions, terminology and notation}

\sssec{Ground field}

Throughout this paper we will be working with schemes and DG schemes defined over a ground field
$k$ of characteristic $0$. 

\sssec{$\infty$-categories}

By an $\infty$-category we shall always mean an $(\infty,1)$-category. By a slight abuse
of language we will sometimes talk about ``categories" when we actually mean 
$\infty$-categories.
 \footnote{Throughout the paper we will ignore set-theoretic
issues. The act of ignoring can be replaced by assuming that the $\infty$-categories
and functors involved are $\kappa$-accessible for some cardinal $\kappa$ in the sense of 
\cite{Lu0}, Sect. 5.4.2.}

\medskip

As was mentioned above, our usage of $\infty$-categories is not tied to any particular model,
however, the basic reference for us is Lurie's book \cite{Lu0}.

\medskip

For an $\infty$-category $\bC$ and $\bc_1,\bc_2\in \bC$, we let
$$\Maps_\bC(\bc_1,\bc_2)$$
denote the corresponding $\infty$-groupoid of maps. We also denote
$$\Hom_\bC(\bc_1,\bc_2)=\pi_0(\Maps_\bC(\bc_1,\bc_2)).$$

\medskip

For a functor $\Phi:\bC_1\to \bC_2$ between $\infty$-categories, and a functor 
$F:\bC_1\to \bD$, where $\bD$ is another $\infty$-category that contains limits
(resp., colimits), we let
$$\on{RKE}_\Phi(F):\bC_2\to \bD \text{ and } \on{LKE}_\Phi(F):\bC_2\to \bD$$
the right (resp., left) Kan extension of $F$ along $\Phi$.

\sssec{Subcategories}  \label{sss:subcategories}

Let $\bC$ and $\bC'$ be an $\infty$-category, and $\phi:\bC'\to \bC$ be a functor. 

\medskip

We shall say that $\phi$ is \emph{0-fully faithful}, or just \emph{fully faithful}
if for any $\bc'_1,\bc'_2\in \bC'$, the map
\begin{equation} \label{e:map on Homs}
\on{Maps}_{\bC'}(\bc'_1,\bc'_2)\to \on{Maps}_{\bC}(\phi(\bc'_1),\phi(\bc'_2))
\end{equation}
is an isomorphism (=homotopy equivalence)
of $\infty$-groupoids. In this case we shall say that $\phi$ makes $\bC'$ into a \emph{0-full}
(or just \emph{full}) subcategory of $\bC'$. 

\medskip

Below are two weaker notions:

\medskip

We shall say that $\phi$ is \emph{1-fully faithful}, or just \emph{faithful}, if for any $\bc'_1,\bc'_2\in \bC'$, 
the map \eqref{e:map on Homs} is a fully faithful map of $\infty$-groupoids. Equivalently, the map 
\eqref{e:map on Homs} induces an injection on $\pi_0$ and a bijection on the homotopy groups $\pi_i$, $i\geq 1$ 
on each connected component of the space $\on{Maps}_{\bC'}(\bc'_1,\bc'_2)$.

\medskip

I.e., $2$- and higher morphisms between $1$-morphisms in $\bC'$ are the same in $\bC'$
and $\bC$, up to homotopy. 

\medskip

We shall say that $\phi$ is \emph{faithful and groupoid-full} if it is faithful, and 
for any $\bc'_1,\bc'_2\in \bC'$, the map
\eqref{e:map on Homs} is surjective on those connected components of 
$\on{Maps}_{\bC}(\phi(\bc'_1),\phi(\bc'_2))$ that correspond to isomorphisms.
In this case we shall say that $\phi$ is an equivalence onto a \emph{1-full}
subactegory of $\bC$. 

\sssec{DG categories}   \label{sss:DG categories}

Our conventions on DG categories follow those of \cite{DG}. For the purposes of this 
paper one can replace DG categories by (the equivalent) $(\infty,1)$-category of
stable $\infty$-categories tensored over $\Vect$, where the latter is the DG category
of complexes of $k$-vector spaces. 

\medskip

By $\StinftyCat_{\on{non-cocomplete}}$ we shall denote the category of all DG-categories. 

\medskip

By $\StinftyCat$ we will denote the full subcategory of $\StinftyCat_{\on{non-cocomplete}}$ that consists of 
cocomplete DG categories (i.e., DG categories closed under direct sums, which is equivalent to
being closed under all colimits). 

\medskip

By $\StinftyCat_{\on{cont}}$ we shall denote the 1-full subcategory of $\StinftyCat$
where we restrict $1$-morphisms to be continuous (i.e., commute with directs sums,
or equivalently with all colimits). 

\medskip

Let $\bC$ be a DG category and $\bc_1,\bc_2\in \bC$ two objects. We shall denote by
$$\CMaps_\bC(\bc_1,\bc_2)$$
the corresponding object of $\Vect$. The $\infty$-groupoid $\Maps_\bC(\bc_1,\bc_2)$
is obtained from the truncation $\tau^{\leq 0}(\CMaps_\bC(\bc_1,\bc_2))$ by the
Dold-Kan functor
$$\Vect^{\leq 0}\to \inftygroup.$$








\sssec{t-structures}   \label{sss:t structures}

Given $\bC\in \StinftyCat_{\on{non-cocomplete}}$ equipped with a t-structure, we shall denote by
$\bC^+$ the corresponding full subcategory of $\bC$ that consists of \emph{eventually coconnective}
objects, i.e.,
$$\bC^+=\underset{n}\cup\, \bC^{\geq -n}.$$

Similarly, we denote
$$\bC^-=\underset{n}\cup\, \bC^{\leq n},\quad \bC^b=\bC^+\cap \bC^-,\quad 
\bC^\heartsuit=\bC^{\leq 0}\cap \bC^{\geq 0}.$$

\sssec{(Pre)stacks and DG schemes}  \label{sss:DG schemes}

Our conventions regarding (pre)stacks and DG schemes follow \cite{Stacks}. 

\medskip

We 
we shall abuse the terminology slightly and use the expression ``classical scheme" for a DG 
scheme which is $0$-coconnective, see \cite{Stacks}, Sect. 3.2.1.  For $S\in \Sch$, we shall 
use the same symbol $S$ to denote the corresponding derived scheme. Correspondingly, 
for a derived scheme $S$ we shall use the notation $^{cl}\!S$, rather
than $\tau^{cl}(S)$, to denote its $0$-coconnective truncation. 

\sssec{Quasi-coherent sheaves}    \label{sss:intr qc}

Conventions regarding the category of quasi-coherent sheaves
on (pre)stacks follow those of \cite{QCoh}. 

\medskip

In particular, we shall denote by $\QCoh^*_{\affdgSch}$ the corresponding functor 
$$(\affdgSch)^{\on{op}}\to \StinftyCat_{\on{cont}},$$
by $\QCoh^*_{\inftydgprestack}$ its right Kan extension along the tautological functor
$$(\affdgSch)^{\on{op}}\hookrightarrow (\inftydgprestack)^{\on{op}},$$
and by $\QCoh^*_{\bC}$ the restriction of the latter to various subcategories $\bC$ of
$\inftydgprestack$, such as $\bC=\dgSch,\inftydgstack, \inftydgstack_{\on{Artin}}$, etc. 

\medskip

The superscript ``$*$" stands for the fact that for a morphism $f:\CY_1\to \CY_2$,
the corresponding functor $\QCoh(\CY_2)\to \QCoh(\CY_1)$ is $f^*$.

\medskip

We remark (see \cite[Corollary 1.3.12]{QCoh}) that for a classical scheme $S$, the category
$\QCoh(S)$ is the same whether we undertand $S$ as a classical or DG scheme.

\sssec{Noetherian DG schemes}  \label{sss:loc Noeth}

We shall say that an affine DG scheme $S=\Spec(A)$ is Noetherian if 

\begin{itemize}

\item $H^0(A)$ is a Noetherian ring.

\item Each $H^i(A)$ is a finitely generated as a module over $H^0(A)$.

\end{itemize} 

We shall say that a DG scheme is locally Noetherian if it admits a Zariski cover by
Noetherian affine DG schemes. It is easy to see that this is equivalent to
requiring that any open affine subscheme is Noetherian. We shall say that a DG
scheme is Noetherian if it is locally Noetherian and quasi-compact; we shall denote
the full subcategory of $\dgSch$ spanned by Noetherian DG schemes by
$\dgSch_{\on{Noeth}}$.

\sssec{DG schemes almost of finite type} \label{sss:intr aft}

Replacing the condition on $H^0(A)$ of being Noetherian by that of being of finite type
over $k$, we obtain the categories of DG schemes \emph{locally almost of finite type} 
and \emph{almost of finite type} over $k$, denoted $\dgSch_{\on{laft}}$ and
$\dgSch_{\on{aft}}$, respectively. 

\ssec{Acknowledgements}

The author's debt to J.~Lurie is obvious. In particular, it should be emphasized that
the following two crucial points in this paper: the idea to use the category of 
correspondences as the source of the $\IndCoh$ functor, and the assertion that 
$\IndCoh$ satisfies faithfully flat descent, are both due to him.

\medskip

The author would like to thank V.~Drinfeld and N.~Rozenblyum 
for their collaboration on the project of writing down the foundations of 
ind-coherent sheaves for the purposes of Geometric Langlands, 
and for sharing their ideas related to the subject of this paper. 

\medskip

The author would also like to thank D.~Arinkin, J.~Barlev and S.~Raskin for many 
helpful discussions.

\medskip

Finally, the author would like to emphasize that part of the material in this paper has been
independently developed by T.~Preygel.

\medskip

The author was supported by NSF grant DMS-1063470.

\bigskip

\centerline{\bf Part I. Elementary properties.}

\bigskip

\section{Ind-coherent sheaves}   \label{s:ind-coherent}

Let $S$ be a DG scheme. In this section we will be assuming that $S$ is Noetherian, see \secref{sss:loc Noeth},
i.e., quasi-compact and locally Noetherian. 

\ssec{The set-up}   \label{ss:defn}

\sssec{}

Consider the category $\QCoh(S)$. It carries a natural t-structure, whose heart $\QCoh(S)^\heartsuit$
is canonically isomorphic to $\QCoh({}^{cl}\!S)^\heartsuit$, where $^{cl}\!S$ is the underlying classical 
scheme.

\begin{defn}  \label{d:IndCoh}
Let $\Coh(S)\subset \QCoh(S)$ be the full subcategory that 
consists of objects of bounded cohomological amplitude and {\it coherent}
cohomologies.
\end{defn}

The assumption that $^{cl}\!S$ is Noetherian, implies that the subcategory 
$\Coh(S)$ is stable under the operation of taking cones, so it is a DG subcategory
of $\QCoh(S)$.

\medskip

However, $\Coh(S)$ is, of course, \emph{not} cocomplete. 

\medskip

\begin{defn}
The DG category $\IndCoh(S)$ is defined as the ind-completion of $\Coh(S)$. 
\end{defn}

\begin{rem}
When $S$ is a classical scheme, the category $\IndCoh(S)$ was first introduced by
Krause in \cite{Kr}. Assertion (2) of Theorem 1.1 of {\it loc.cit.} can be stated as 
$\IndCoh(S)$ being equivalent to the DG category of \emph{injective complexes}
on $S$. Many of the results of this section are simple generalizations of Krause's
results to the case of a DG scheme.
\end{rem}

\sssec{}

By construction, we have a canonical 1-morphism in $\StinftyCat_{\on{cont}}$:
$$\Psi_S:\IndCoh(S)\to \QCoh(S),$$
obtained by ind-extending the tautological embedding $\Coh(X)\hookrightarrow \QCoh(X)$. 

\begin{lem} \label{l:smoothness}
Assume that $S$ is a regular classical scheme. Then $\Psi_S$ is an equivalence.
\end{lem}

\medskip

\begin{rem}
As we shall see in \secref{ss:inverse implic}, the converse is also true. 
\end{rem}

\begin{proof}

Recall that by \cite{BFN}, Prop. 3.19, extending the arguments of \cite{Ne}, the category $\QCoh(S)$
is compactly generated by its subcategory $\QCoh(S)^c=\QCoh(S)^{\on{perf}}$ consisting of perfect complexes. 

\medskip

Suppose that $S$ is a regular classical scheme. Then the subcategories
$$\Coh(S)\subset \QCoh(S)\supset \QCoh(S)^{\on{perf}}$$
coincide, and the assertion is manifest.

\end{proof}

\ssec{The t-structure}

\sssec{}

The category $\Coh(S)$ carries a natural t-structure. Hence, $\IndCoh(S)$ acquires a canonical
t-structure, characterized by the properties that

\begin{enumerate}

\item It is compatible with filtered colimits (i.e., the truncation functors commute with filtered colimits), and 

\item The tautological embedding $\Coh(S)\to \IndCoh(S)$ is t-exact.

\end{enumerate}

\medskip

We have:

\begin{lem}
The functor $\Psi_S$  is t-exact.
\end{lem}

\begin{proof}
Follows from the fact that the t-structure on $\QCoh(S)$ is compatible with filtered colimits.
\end{proof}

\sssec{}

The next proposition, although simple, is crucial for the rest of the paper. It says that the eventually
coconnective subcategories of $\IndCoh(S)$ and $\QCoh(S)$ are equivalent:

\begin{prop}  \label{p:equiv on D plus}
For every $n$, the induced functor
$$\Psi_S:\IndCoh(S)^{\geq n}\to \QCoh(S)^{\geq n}$$
is an equivalence.
\end{prop}

\begin{proof}

With no restricion of generality we can assume that $n=0$. 
We first prove fully-faithfulness. We will show that the map
\begin{equation} \label{Psi ff}
\Maps_{\IndCoh(S)}(\CF,\CF')\to \Maps_{\QCoh(S)}(\Psi_S(\CF),\Psi_S(\CF'))
\end{equation}
is an isomorphism for $\CF'\in \IndCoh(S)^{\geq 0}$ and \emph{any}
$\CF\in \IndCoh(S)$.

\medskip

By definition, it suffuces to consider the case $\CF\in \Coh(S)$.

\medskip

Every $\CF'\in \IndCoh(S)^{\geq 0}$
can be written as a filtered colimit
$$\underset{i}{colim}\, \CF'_i,$$
where $\CF'_i\in \Coh(S)^{\geq 0}$.

\medskip

By the definition of $\IndCoh(S)$, the left-hand side of \eqref{Psi ff} is the colimit of 
$$\Maps_{\IndCoh(S)}(\CF,\CF'_i),$$ where each term, again by definition, is isomorphic
to 
$$\Maps_{\Coh(S)}(\CF,\CF'_i)\simeq \Maps_{\QCoh(S)}(\CF,\CF'_i).$$

\medskip

So, it suffices to show that for $\CF\in \Coh(S)$, the functor
$\Maps_{\QCoh(S)}(\CF,-)$ commutes with filtered colimits \emph{taken} within
$\QCoh(S)^{\geq 0}$. 

\medskip

The assumptions on $S$ imply that for any $\CF\in \Coh(S)$ there exists
an object $\CF_0\in \QCoh(S)^c$  (i.e., $\CF_0$ is perfect, see \cite{QCoh}, Sect. 4.1)
together with a map $\CF_0\to \CF$ such that
$$\on{Cone}(\CF_0\to \CF)\in \QCoh(S)^{<0}.$$

\medskip

The functor $\Maps_{\QCoh(S)}(\CF_0,-)$ does commute with filtered colimits since $\CF_0$
is compact, and the induced map
$$\Maps_{\QCoh(S)}(\CF_0,\CF')\to \Maps_{\QCoh(S)}(\CF,\CF')$$
is an isomorphism for any $\CF'\in \QCoh(S)^{\geq 0}$. This implies the needed result.

\bigskip

It remains to show that $\Psi_S:\IndCoh(S)^{\geq 0}\to \QCoh(S)^{\geq 0}$ is essentially surjective. 
However, this follows from the fact that any object $\CF\in \QCoh(S)^{\geq 0}$ can be written
as a filtered colimit $\underset{i}{colim}\, \CF_i$ with $\CF_i\in \Coh(S)^{\geq 0}$.

\end{proof} 

\begin{cor} \label{c:equiv on D plus}
An object of $\IndCoh(S)$ is connective if and only if its image in $\QCoh(S)$ under $\Psi_S$ is.
\end{cor}

The proof is immediate from \propref{p:equiv on D plus}.

\begin{cor}  \label{c:Karoubian}
The subcategory $\Coh(S)\subset \IndCoh(S)$ equals $\IndCoh(S)^c$, i.e., is the category
of \emph{all} compact objects of $\IndCoh(S)$.
\end{cor}

\begin{proof}

A priori, every compact object of $\IndCoh(S)$ is a direct summand of an object of
$\Coh(S)$. Hence, all compact objects of $\IndCoh(S)$ are eventually coconnective, i.e.,
belong to $\IndCoh(S)^{\geq -n}$ for some $n$. Now, the assertion follows from 
\propref{p:equiv on D plus}, as a direct summand of every coherent object in $\QCoh(S)$
is itself coherent.

\end{proof}

\sssec{}
Let $\IndCoh(S)_{nil}\subset \IndCoh(S)$ be the full subcategory of infinitely connective, i.e.,
t-nil objects:
$$\IndCoh(S)_{nil}:=\underset{n\in \BN}\cap\, \IndCoh(S)^{\leq -n}.$$ 

Since the t-structure on $\QCoh(S)$ is separated (i.e., if an object has zero cohomologies with respect to the
t-structure, then it is zero), we have:
$$\IndCoh(S)_{nil}=\on{ker}(\Psi_S).$$

At the level of homotopy categories, we obtain that $\Psi_S$ factors as
\begin{equation} \label{mod nil}
\on{Ho}(\IndCoh(S))/\on{Ho}(\IndCoh(S)_{nil})\to \on{Ho}(\QCoh(S)).
\end{equation}

However, in general, the functor in \eqref{mod nil} is not an equivalence.

\ssec{$\QCoh$ as the left completion of $\IndCoh$}  \label{ss:QCos as left compl}

\sssec{}

Recall that a DG category $\bC$ equipped with a t-structure is said to be \emph{left-complete} in its
t-structure if the canonical functor
\begin{equation} \label{e:functor to left compl}
\bC\to \underset{n\in \BN^{\on{op}}}{lim}\, \bC^{\geq -n},
\end{equation}
is an equivalence, where for $n_1\geq n_2$, the functor
$$\bC^{\geq -n_1}\to \bC^{\geq -n_2}$$ is $\tau^{\geq -n_2}$. The functor in \eqref{e:functor to left compl}
is given by the family
$$n\mapsto \tau^{\geq -n}.$$

\sssec{}

Any DG category equipped with a t-structure admits a left completion. This is a DG category $\bC'$
equipped with a t-structure in which it is left-complete, such that it receives a t-exact functor
$\bC\to \bC'$, and which is universal with respect to these properties.

\medskip

Explicitly, the left completion of $\bC$ is given by the limit $\underset{n\in \BN^{\on{op}}}{lim}\, \bC^{\geq -n}$.

\sssec{}

We now claim:

\begin{prop}  \label{p:QCoh as left compl}
For $S\in \dgSch_{\on{Noeth}}$, the functor $\Psi_S:\IndCoh(S)\to \QCoh(S)$
identifies $\QCoh(S)$ with the left completion of $\IndCoh(S)$ in its t-structure.
\end{prop}

\begin{proof}

First, it easy to see that the category $\QCoh(S)$ is left-complete in its t-structure. (Proof: the assertion
reduces to the case when $S$ is affine. In the latter case, the category $\QCoh(S)$ admits a conservative
limit-preserving t-exact functor to a left-complete category, namely, $\Gamma(S,-):\QCoh(S)\to \Vect$.)

\medskip

Now, the assertion of the proposition follows from \propref{p:equiv on D plus}.

\end{proof}

\ssec{The action of $\QCoh(S)$ on $\IndCoh(S)$}  \label{ss:action}

The category $\QCoh(S)$ has a natural (symmetric) monoidal structure. We claim that 
$\IndCoh(S)$ is naturally a module over $\QCoh(S)$.

\medskip

To define an action of $\QCoh(S)$ on $\IndCoh(S)$ it suffices to define the action of the
(non-cocomplete) monoidal DG category $\QCoh(S)^{\on{perf}}$ on the (non-cocomplete) DG category
$\Coh(S)$. 

\medskip

For the latter, it suffices to notice that the action
of $\QCoh(S)^{\on{perf}}$ on $\QCoh(S)$ preserves the non-cocomplete subcategory $\Coh(S)$. 

\medskip

By construction, the action functor
$$\QCoh(S)\otimes \IndCoh(S)\to \IndCoh(S)$$
sends compact objects to compact ones. 

\medskip

We will use the notation
$$\CE\in \QCoh(S),\CF\in \IndCoh(S)\mapsto \CE\otimes \CF\in \IndCoh(S).$$

\sssec{}

From the construction, we obtain:

\begin{lem}  \label{l:Psi compat with action}
The functor $\Psi_S$ has a natural structure of 1-morphism between
$\QCoh(S)$-module categories.
\end{lem}

At the level of individual objects, the assertion of \lemref{l:Psi compat with action} says that
for $\CE\in \QCoh(S)$ and $\CF\in \IndCoh(S)$, we have a canonical isomorphism
\begin{equation} \label{e:Psi compat with action}
\Psi_S(\CE\otimes \CF)\simeq \CE\otimes \Psi_S(\CF).
\end{equation}

\ssec{Eventually coconnective case}  \label{ss:eventually coconnective}

Assume now that $S$ is eventually coconnective, see \cite{Stacks}, Sect. 3.2.6
where the terminology is introduced.

\medskip

We remind that by definition, this means
that $S$ is covered by affines $\Spec(A)$ with $H^{-i}(A)=0$ for all $i$ large enough.
Equivalently, $S$ is eventually coconnective if and only if $\CO_S$ belongs to
$\Coh(S)$. 

\medskip

For example, any classical scheme is $0$-coconnective, and hence eventually
coconnective when viewed as a DG scheme. 

\medskip

\begin{rem} In \secref{ss:convergence} we will show that in a certain sense
the study of $\IndCoh$ reduces to the eventually coconnective case. 
\end{rem}

\sssec{}

Under the above circumstances we claim:

\begin{prop} \label{p:Xi}
The functor $\Psi_S$ admits a left adjoint $\Xi_S$. Moreover, $\Xi_S$ is fully
faithful, i.e., the functor $\Psi$ realizes $\QCoh(S)$ is a co-localization 
of $\IndCoh(S)$ with respect to $\IndCoh(S)_{nil}$.
\end{prop}

Shortly, we shall see that
a left adjoint to $\Psi$ exists \emph{if and only if} $S$ 
is eventually coconnective.

\begin{proof}

To prove the proposition, we have to show that the left adjoint $\Xi_S$ is well-defined and fully faithful on 
$\QCoh(S)^{\on{perf}}=\QCoh(S)^c$.

\medskip

However, this is clear: for $S$ eventually coconnective, we have a natural fully faithful inclusion
$$\QCoh(S)^{\on{perf}}\hookrightarrow \Coh(S).$$

\end{proof}

\sssec{}

By adjunction, from \lemref{l:Psi compat with action}, and using \cite[Corollary 6.2.4]{DG},
we obtain:

\begin{cor}  \label{c:upgrading Xi}
The functor $\Xi_S$ has a natural structure of $1$-morphism between
$\QCoh(S)$-module categories.
\end{cor}

At the level of individual objects, the assertion of \corref{c:upgrading Xi} says that for
$\CE_1,\CE_2\in \QCoh(S)$, we have a canonical isomorphism
\begin{equation} \label{e:upgrading Xi}
\Xi_S(\CE_1\otimes \CE_2)\simeq \CE_1\otimes \Xi_S(\CE_2),
\end{equation}
where in the right-hand side $\otimes$ denotes the action of \secref{ss:action}.

\medskip

In particular, for $\CE\in \QCoh(S)$, we have
$$\Xi_S(\CE)\simeq \CE\otimes \Xi_S(\CO_S).$$

\sssec{}

We note the following consequence of \propref{p:Xi}:

\begin{cor}  \label{c:coloc}
If $S$ is eventually coconnective, the functor of triangulated categories
\eqref{mod nil} is an equivalence. 
\end{cor}

\medskip

The following observation is useful:

\begin{lem}  \label{l:Xi in Coh}
Let $\CF$ be an object of $\QCoh(S)$ such that $\Xi(S)\in \Coh(S)\subset\IndCoh(S)$.
Then $\CF\in \QCoh(S)^{\on{perf}}$.
\end{lem}

\begin{proof}
Since $\Xi_S$ is a fully faithful functor that commutes with filtered colimits, 
an object of $\QCoh(S)$ is compact if its image under $\Xi_S$ is compact. 
\end{proof}

\ssec{Some converse implications}  \label{ss:inverse implic}

\sssec{}

We are now going to prove:

\begin{prop}  \label{p:Xi and event coconn}
Assume that the functor $\Psi_S:\IndCoh(S)\to \QCoh(S)$ admits a left adjoint. Then
$S$ is eventually coconnective. 
\end{prop}

\begin{proof}

Let $\Xi_S:\QCoh(S)\to \IndCoh(S)$ denote the left adjoint in question. Since $\Psi_S$
commutes with colimits, the functor $\Xi_S$ sends compact objects to compact ones.
In particular, by \corref{c:Karoubian}, $\Xi_S(\CO_S)$ belongs to $\Coh(S)$, and in
particular, to $\IndCoh(S)^{\geq -n}$ for some $n$. Hence, by \propref{p:equiv on D plus},
the map $\CO_S\to \Psi_S(\Xi_S(\CO_S))$ induces an isomorphism on Homs to
any eventually coconnective object of $\QCoh(S)$. Since the t-structure on $\QCoh(S)$
is separated, we obtain that $\CO_S\to \Psi_S(\Xi_S(\CO_S))$ is an isomorphism. In 
particular,
$\CO_S\in \Coh(S)$.

\end{proof}

\sssec{}

Let us now prove the converse to \lemref{l:smoothness}:

\begin{prop}
Let $S$ be a DG scheme such that $\Psi_S$ is an equivalence. Then $S$ is a regular
classical scheme. 
\end{prop}

\begin{proof}

By \propref{p:Xi and event coconn} we obtain that $S$ is eventually coconnective. 
Since $\Psi_S$ is an equivalence, it induces an equivalence between the corresponding
categories of compact objects. Hence, we obtain that 
\begin{equation} \label{perf and coh}
\Coh(S)=\QCoh(S)^{\on{perf}}.
\end{equation}
as subcategories of $\QCoh(S)$.

\medskip

The question of being classical and regular, and the equality \eqref{perf and coh}, are local.
So, we can assume that $S$ is affine, $S=\Spec(A)$. 

\medskip

The inclusion $\subset$ in \eqref{perf and coh}
implies that $H^0(A)$ admits
a finite resolution by projective $A$-modules. However, it is easy to see that this
is only possible when all $H^i(A)$ with $i\neq 0$ vanish. Hence $A$ is classical. 

\medskip

In the latter case, the inclusion $\subset$ in \eqref{perf and coh} means
that every $A$ module is of finite projective dimension. Serre's theorem implies
that $A$ is regular.

\end{proof}

\section{$\IndCoh$ in the non-Noetherian case} \label{s:non-Noeth}

This section will not be used in the rest of the paper.  We will indicate the
definition of the category $\IndCoh(S)$ in the case when $S$ is not necessarily locally Noetherian.
We shall first treat the case when $S=\Spec(A)$ is affine, and then indicate how to treat
the case of a general scheme if some additional condition is satisfied. 

\ssec{The coherent case}

Following J.~Lurie, give the following definitions: 

\begin{defn}
A classical ring $A_0$ is said to be \emph{coherent} if the category of
finitely presented $A_0$ modules is abelian, i.e., is stable under 
taking kernels and cokernels. 
\end{defn}

\begin{defn} A DG ring $A$ is said to be coherent if:
\begin{enumerate}

\item The classical ring $H^0(A)$ is coherent. 

\item Each $H^i(A)$ is finitely presented as a $H^0(A)$-module.

\end{enumerate}
\end{defn}

\sssec{}

We note that the entire discussion in the preceeding section goes through when 
the assumption that $A$ be Noetherian is replaced by that of it being coherent. 

\medskip

In what follows, we shall show how to define $\IndCoh(S)$ without the 
coherence assumption either. 

\ssec{Coherent sheaves in the non-Noetherian setting}

\sssec{}

Informally, we say that an object $\QCoh(S)$ is coherent if it is cohomologically
bounded and can be approximated 
by a perfect object to any level of its Postnikov tower. 

\medskip

A formal definition is as follows: 

\medskip

\begin{defn}
A cohomologically bounded object $\CF\in \QCoh(S)$ is said to be coherent if for any $n$ there exists
an object $\CF_n\in \QCoh(S)^{\on{perf}}$ together with a
morphism $\CF_n\to \CF$ such that $\on{Cone}(\CF_n\to \CF)\in \QCoh(S)^{\leq -n}$.
\end{defn}

We have:

\begin{lem}  \label{l:crit coh}
For a cohomologically bounded object $\CF\in \QCoh(S)$ the following conditions are
equivalent: 

\begin{enumerate}

\item  $\CF$ is coherent.

\item For any $n$, the functor on $\QCoh(S)^\heartsuit$ given by $\CF'\mapsto \Hom(\CF,\CF'[n])$ 
commutes with filtered colimits.
\end{enumerate}

\end{lem}

\sssec{}

Let $\Coh(S)\subset \QCoh(S)$ denote the full subcategory spanned by coherent objects. 
It follows from \lemref{l:crit coh} that $\Coh(S)$ is stable under taking cones, so it is a (non-cocomplete) 
DG subcategory of $\QCoh(S)$.

\begin{defn}
We define the DG category $\IndCoh(S)$ to be the ind-completion of $\Coh(S)$.
\end{defn}

\medskip

By construction, we have a tautological functor $\Psi_S:\IndCoh(S)\to \QCoh(S)$. 

\sssec{}

We shall now define a t-structure on $\IndCoh(S)$. We let $\IndCoh(S)^{\leq 0}$ be
generated under colimits by objects from $\Coh(S)\cap \QCoh(S)^{\leq 0}$. Since
all these objects are compact in $\IndCoh(S)$, the resulting t-structure on 
$\IndCoh(S)$ is compatible with filtered colimits. 

\begin{rem}
The main difference between the present situation and one when $S$ is coherent
is that now the subcategory $\Coh(S)\subset \IndCoh(S)$ is not necessarily preserved
by the truncation functors. In fact, it is easy to show that the latter condition is equivalent 
to $S$ being coherent.
\end{rem}

By construction, the functor $\Psi_S$ is right t-exact.

\ssec{Eventual coherence}

\sssec{}

Note that unless some extra conditions on $S$ are imposed, it is not clear that the
category $\Coh(S)$ contains any objects besides $0$. Hence, in general, we cannot
expect that an analog of \propref{p:equiv on D plus} should hold. We are now 
going to introduce a condition on $S$ which would guarantee that an appropriate
analog of \propref{p:equiv on D plus} does hold:

\medskip

\begin{defn}  \label{defn:event coherent}
We shall say that $S$ is eventually coherent if there exists an integer $N$, such that for 
all $n\geq N$ the truncation $\tau^{\geq -n}(\CO_S)$ is coherent. 
\end{defn}

We are going to prove:

\begin{prop} \label{p:equiv on D plus non-Noeth}
Assume that $S$ is eventually coherent. Then:

\smallskip

\noindent{\em(a)} The functor $\Psi_S$ is t-exact.

\smallskip

\noindent{\em(b)} For any $m$, the resulting functor $\Psi_S:\IndCoh(S)^{\geq m}\to \QCoh(S)^{\geq m}$
is an equivalence.

\end{prop}

\begin{proof}

Let $\CF\simeq \underset{i,\IndCoh(S)}{colim}\, \CF_i$ be an object of $\IndCoh(S)^{> 0}$, where
$\CF_i\in \Coh(S)$. To prove that $\Psi_S$ is left t-exact, we need to show that 
$\underset{i,\QCoh(S)}{colim}\, \CF_i$ belongs to $\QCoh(S)^{> 0}$. Since $S$ is affine, the latter 
condition is equivalent to 
$$\Maps(\CO_S,\underset{i,\QCoh(S)}{colim}\, \CF_i)=0,$$
and since $\CO_S$ is a compact object of $\QCoh(S)$, the left-hand side of the latter expression is
\begin{equation} \label{e:colim to vanish 1}
\underset{i}{colim} \, \Maps_{\QCoh(S)}(\CO_S,\CF_i).
\end{equation}

Since each $\CF_i$ is cohomologically bounded, the map 
$$\Maps_{\QCoh(S)}(\tau^{\geq -n}(\CO_S),\CF_i)\to \Maps_{\QCoh(S)}(\CO_S,\CF_i)$$
is an isomorphism for $n\gg 0$. Hence, the map
$$\underset{n}{colim}\, \Maps_{\QCoh(S)}(\tau^{\geq -n}(\CO_S),\CF_i)\to \Maps_{\QCoh(S)}(\CO_S,\CF_i)$$
is an isomorphism. Therefore,
\begin{multline} \label{e:colim to vanish 2}
\underset{i}{colim} \, \Maps_{\QCoh(S)}(\CO_S,\CF_i)\simeq \underset{i}{colim}\,\underset{n}{colim}\, 
\Maps_{\QCoh(S)}(\tau^{\geq -n}(\CO_S),\CF_i)\simeq \\
\simeq \underset{n}{colim}\,\underset{i}{colim}\,
\Maps_{\QCoh(S)}(\tau^{\geq -n}(\CO_S),\CF_i).
\end{multline} 

However, by assumption, for $n\geq N$, $\tau^{\geq -n}(\CO_S)\in \Coh(S)$, and for such $n$
$$\Maps_{\QCoh(S)}(\tau^{\geq -n}(\CO_S),\CF_i)\simeq \Maps_{\IndCoh(S)}(\tau^{\geq -n}(\CO_S),\CF_i),$$
and hence
\begin{multline*}
\underset{i}{colim}\,\Maps_{\QCoh(S)}(\tau^{\geq -n}(\CO_S),\CF_i)\simeq 
\underset{i}{colim}\,\Maps_{\IndCoh(S)}(\tau^{\geq -n}(\CO_S),\CF_i)\simeq  \\
\simeq \Maps_{\IndCoh(S)}(\tau^{\geq -n}(\CO_S),\CF),
\end{multline*}
which vanishes by the definition of the t-structure on $\IndCoh(S)$. Hence, the expression in \eqref{e:colim to vanish 2}
vanishes, as required. This proves point (a).

\medskip

To prove point (b), it is enough to consider the case $m=0$. We first prove fully-faithfulness. 
As in the proof of \propref{p:equiv on D plus},
it is sufficient to show that for $\CF\in \Coh(S)$ and $\CF'\in \IndCoh(S)^{\geq 0}$, the map
$$\Maps_{\IndCoh(S)}(\CF,\CF')\to \Maps_{\QCoh(S)}(\Psi_S(\CF),\Psi_S(\CF'))$$ is an
isomorphism. 

\medskip

Let $i\mapsto \CF'_i$ be a diagram in $\Coh(S)$,
such that $\CF':=\underset{i,\IndCoh(S)}{colim}\, \CF'_i$
belongs to $\IndCoh(S)^{\geq 0}$. We need to show that the map
\begin{equation} \label{e:colim to comp 1}
\underset{i}{colim} \, \Maps_{\QCoh(S)}(\CF,\CF'_i)\to 
\Maps_{\QCoh(S)}(\CF,\underset{i,\QCoh(S)}{colim} \, \CF'_i)
\end{equation}
is an isomorphism. 

\medskip

Let $\CF_0\in \QCoh(S)^{\on{perf}}$ be an object equipped
with a map $\CF_0\to \CF$ such that $$\on{Cone}(\CF_0\to \CF)\in \QCoh(S)^{<0}.$$
Since $\CF$ is cohomologically bounded, for $n\gg 0$, the above morphism factors
through a morphism 
$$\tau^{\geq -n}(\CO_S)\otimes \CF_0=:\CF_0^n\to \CF;$$
moreover, by the eventually coherent assumption on $S$, we can take $n$ to be large enough 
so that $\CF^n_0\in \Coh(S)$.

\medskip

Consider the commutative diagram:
\begin{equation} \label{diag:colimits}
\CD
\underset{i}{colim} \, \Maps_{\QCoh(S)}(\CF_0,\CF'_i)  @>>>  \Maps_{\QCoh(S)}(\CF_0,\underset{i,\QCoh(S)}{colim} \, \CF'_i) \\
@AAA    @AAA   \\
\underset{i,n}{colim}\, \Maps_{\QCoh(S)}(\CF_0^n,\CF'_i)   @>>> 
\underset{n}{colim}  \, \Maps_{\QCoh(S)}(\CF_0^n,\underset{i,\QCoh(S)}{colim} \,\CF'_i) \\
@AAA   @AAA  \\
\underset{i}{colim} \, \Maps_{\QCoh(S)}(\CF,\CF'_i) @>>>  \Maps_{\QCoh(S)}(\CF,\underset{i,\QCoh(S)}{colim} \, \CF'_i).
\endCD
\end{equation}

We need to show that in the bottom horizontal arrow in this diagram is an isomorphism. We will do so
by showing that all other arrows are isomorphisms.

\medskip

For all $n\gg 0$ we have a commutative diagram
$$
\CD
\underset{i}{colim}\, \Maps_{\QCoh(S)}(\CF_0^n,\CF'_i)  @>{\sim}>> 
\Maps_{\IndCoh(S)}(\CF_0^n,\CF')  \\
@AAA    @AAA   \\
\underset{i}{colim} \, \Maps_{\QCoh(S)}(\CF,\CF'_i) @>{\sim}>> \Maps_{\IndCoh(S)}(\CF,\CF'),
\endCD
$$
in which the right vertical arrow is an isomorphism since $\on{Cone}(\CF_0^n\to \CF)\in \IndCoh(S)^{<0}$
and $\CF'\in \IndCoh(S)^{\geq 0}$. Hence, the left vertical arrow is an isomorphism as well. This implies that the lower
left vertical arrow in \eqref{diag:colimits} is an isomorphism. 

\medskip

The upper left vertical arrow in \eqref{diag:colimits} is an isomorphism by the same argument as in the proof of point (a)
above. The top horizontal arrow is an isomorphism since $\CF_0$ is compact in $\QCoh(S)$. Finally, both
right vertical arrows are isomorphisms since $\underset{i,\QCoh(S)}{colim} \, \CF'_i\in \QCoh(S)^{\geq 0}$,
as was established in point (a). 

\medskip

Finally, let us show that the functor
$$\Psi_S:\IndCoh(S)^{\geq 0}\to \QCoh(S)^{\geq 0}$$
is essentially surjective. By fully faithfulness, it suffices to show that the essential
image of $\Psi_S(\IndCoh(S)^{\geq 0})$ generates $\QCoh(S)^{\geq 0}$ under filtered
colimits.

\medskip

Since $\QCoh(S)$ is generated under filtered colimits by objects of the form $\CF_0\in \QCoh(S)^{\on{perf}}$,
the category $\QCoh(S)^{\geq 0}$ is generated under filtered colimits by objects of the form 
$\tau^{\geq 0}(\CF_0)$ for $\CF_0\in \QCoh(S)^{\on{perf}}$. Hence, it suffices to show that such objects
are in the essential image of $\Psi_S$. However,
$$\tau^{\geq 0}(\CF_0)\simeq \tau^{\geq 0}(\Psi_S(\CF_0^n))\simeq \Psi_S(\tau^{\geq 0}(\CF_0^n))$$
for $n\gg 0$ and $\CF_0^n=\tau^{\geq -n}(\CO_S)\otimes \CF_0\in \Coh(S)$. 

\end{proof}

\sssec{The eventually coconnective case}

Assume now that $S$ is eventually coconnective; in particular, 
it is automatically eventually coherent.

\medskip

In this case $\QCoh(S)^{\on{perf}}$ is contained in $\Coh(S)$. As in \secref{ss:eventually coconnective},
this gives rise to a functor 
$$\Xi_S:\QCoh(S)\to \IndCoh(S),$$
left adjoint to $\Psi_S$, obtained by ind-extension of the tautological functor
$$\QCoh(S)^{\on{perf}}\hookrightarrow \Coh(S).$$

\medskip

As in \propref{p:Xi}, it is immediate that the unit of adjunction
$$\on{Id}\to \Psi_S\circ \Xi_S$$
is an isomorphism. I.e., $\Xi_S$ is fully faithful, and $\Psi_S$
realizes $\QCoh(S)$ as a colocalization of $\IndCoh(S)$
with respect to the subcategory $\IndCoh(S)_{nil}$.

\begin{rem}
Here is a quick way to see that $\Psi_S$ is left t-exact in the eventually coconnective case: indeed, it is
the right adjoint of $\Xi_S$, and the latter is right t-exact by construction.
\end{rem}

\ssec{The non-affine case}

In this subsection we will indicate how to extend the definition of $\IndCoh(S)$ to the case when 
$S$ is not necessarily affine.

\sssec{}

First, we observe that if $S_1\hookrightarrow S_2$ is an open embedding of affine
DG schemes and $S_2$ is eventually coherent, then so is $S_1$. This observation gives
rise to the following definition:

\begin{defn}
We say that a scheme $S$ is locally eventually coherent if for it admits an open
cover by eventually coherent affine DG schemes.
\end{defn}

\sssec{}

Assume that $S$ is locally eventually coherent. We define the category $\IndCoh(S)$
as 
$$\underset{U\to S}{lim}\, \IndCoh(U),$$
where the limit is taken over the category of DG schemes $U$ that are disjoint unions
of eventually coherent affine DG schemes, equipped with an open embedding into $S$.

\medskip

\propref{p:equiv on D plus non-Noeth} allows us to prove an analog of \propref{p:Zariski descent},
which implies that for $S$ affine, we recover the original definition of $\IndCoh(S)$.

\section{Basic Functorialities}  \label{s:functorialities}

In this section all DG schemes will be assumed Noetherian. 

\ssec{Direct images}  \label{direct image}

Let $f:S_1\to S_2$ be a morphism of DG schemes.

\begin{prop}  \label{p:direct image}
There exists a unique continuous functor
$$f^{\IndCoh}_*:\IndCoh(S_1)\to \IndCoh(S_2),$$
which is left t-exact with respect to the t-structures, and which
makes the diagram
$$
\CD
\IndCoh(S_1) @>{f^{\IndCoh}_*}>> \IndCoh(S_2) \\
@V{\Psi_{S_1}}VV    @VV{\Psi_{S_2}}V   \\
\QCoh(S_1) @>{f_*}>> \QCoh(S_2)
\endCD
$$
commute.
\end{prop}

\begin{proof}

By definition, we need to construct a functor
$$\Coh(S_1)\to \IndCoh(S_2)^+,$$ such
that the diagram
$$
\CD
\Coh(S_1) @>>> \IndCoh(S_2)^+ \\
@V{\Psi_{S_1}}VV    @VV{\Psi_{S_2}}V   \\
\QCoh(S_1) @>{f_*}>> \QCoh(S_2)^+
\endCD
$$
commutes. Since the right vertical arrow is an equivalence (by \propref{p:equiv on D plus}),
the functor in question is uniquely determined to be the composition
$$\Coh(S_1) \hookrightarrow \QCoh(S_1)^+ \overset{f_*}\longrightarrow \QCoh(S_2)^+.$$

\end{proof}

\sssec{}

Recall that according to \secref{ss:action}, we can regard $\IndCoh(S_i)$ as a module category over the
monoidal category $\QCoh(S_i)$. 
In particular, we can view both $\IndCoh(S_1)$ and $\IndCoh(S_2)$
as module categories over $\QCoh(S_2)$ via the monoidal functor
$$f^*:\QCoh(S_2)\to \QCoh(S_1).$$

\medskip

We now claim:

\begin{prop}    \label{p:dir image as modules}
The functor $f^{\IndCoh}_*:\IndCoh(S_1)\to \IndCoh(S_2)$ has a unique
structure of $1$-morphism of $\QCoh(S_2)$-module categories that
makes the diagram
$$
\CD
\IndCoh(S_1)  @>{\Psi_{S_1}}>>  \QCoh(S_1)  \\
@V{f^\IndCoh_*}VV    @VV{f_*}V   \\
\IndCoh(S_2)  @>{\Psi_{S_2}}>>  \QCoh(S_2)  
\endCD
$$
commute.
\end{prop}

At the level of individual objects, the assertion of \propref{p:dir image as modules} says that
for $\CE_2\in \QCoh(S_2)$ and $\CF_1\in \IndCoh(S_1)$ we have a canonical isomorphism

\begin{equation} \label{e:dir image as modules}
\CE_2\otimes f^\IndCoh_*(\CF_1)\simeq f^\IndCoh_*(f^*(\CE_2)\otimes\CF_1),
\end{equation}
where $\otimes$ is understood in the sense of the action of \secref{ss:action}.

\begin{proof}

It is enough to show that the functor
$$f^\IndCoh_*|_{\Coh(S_1)}:\Coh(S_1)\to \IndCoh(S_2)$$
has a unique structure of $1$-morphism of module categories
over $\QCoh(S_2)^{\on{perf}}$, which makes the diagram
$$
\CD
\Coh(S_1)  @>{\Psi_{S_1}}>>  \QCoh(S_1)  \\
@V{f^\IndCoh_*}VV    @VV{f_*}V   \\
\IndCoh(S_2)  @>{\Psi_{S_2}}>>  \QCoh(S_2)  
\endCD
$$
commute. 

\medskip

However, by construction, the latter structure equals the one induced by the functors
$$f_*|_{\Coh(S_1)}:\Coh(S_1)\to \QCoh(S_2)^+$$
and
$$\QCoh(S_2)^+\overset{\Psi_{S_2}}\simeq \IndCoh(S_2)^+\hookrightarrow \IndCoh(S_2).$$

\end{proof}

\ssec{Upgrading to a functor}  \label{ss:upgrading direct image to a functor}

\sssec{}

We now claim that the assigment
$$S\rightsquigarrow \IndCoh(S),\quad f\rightsquigarrow f^\IndCoh_*$$
upgrades to a functor
\begin{equation} \label{e:IndCoh functor}
\dgSch_{\on{Noeth}}\to \StinftyCat_{\on{cont}},
\end{equation}
to be denoted $\IndCoh_{\dgSch_{\on{Noeth}}}$.

\sssec{}

First, we recall the functor
$$\QCoh^*_{\dgSch}:(\dgSch)^{\on{op}}\to \StinftyCat_{\on{cont}},$$
see \secref{sss:intr qc}.

\medskip

Passing to right adjoints, we obtain a functor 
$$\QCoh_{\dgSch}:\dgSch\to \StinftyCat_{\on{cont}}.$$
Restricting to $\dgSch_{\on{Noeth}}\subset \dgSch$ we obtain a functor
$$\QCoh_{\dgSch_{\on{Noeth}}}:\dgSch_{\on{Noeth}}\to \StinftyCat_{\on{cont}}.$$

\sssec{}

Now,  we claim:

\begin{prop} \label{p:upgrading IndCoh to functor}
There exists a uniquely defined functor 
$$\IndCoh_{\dgSch_{\on{Noeth}}}:\dgSch_{\on{Noeth}}\to \StinftyCat_{\on{cont}},$$
equipped with a natural transformation
$$\Psi_{\dgSch_{\on{Noeth}}}:\IndCoh_{\dgSch_{\on{Noeth}}}\to \QCoh_{\dgSch_{\on{Noeth}}},$$
which at the level of objects and $1$-morphisms is given by the assignment
$$X\rightsquigarrow \IndCoh(X),\quad f\rightsquigarrow f^\IndCoh_*.$$
\end{prop}

The rest of this subsection is devoted to the proof of this proposition. 

\sssec{}  \label{sss:Stinfty +}

Consider the following $(\infty,1)$-categories:
$$\StinftyCat^{+_{\on{cont}}} \text{ and } \StinftyCat^{t}_{\on{cont}}:$$

\medskip

The category $\StinftyCat^{+_{\on{cont}}}$ consists if of non-cocomplete DG categories $\bC$, 
endowed with a t-structure, such that $\bC=\bC^+$. We also require that $\bC^{\geq 0}$ contain 
filtered colimits and that the embedding $\bC^{\geq 0}\hookrightarrow \bC$ commute with filtered colimits. 
As $1$-morphisms we take those exact functors $F:\bC_1\to \bC_2$ that are \emph{left t-exact up to a finite shift}, 
and such that $F|_{\bC_1^{\geq 0}}$ commutes with filtered colimits. The higher categorical structure is
uniquely determined by the requirement that the forgetful functor
$$\StinftyCat^{+_{\on{cont}}}\to \StinftyCat_{\on{non-cocomplete}}$$
be 1-fully faithful. 

\medskip

The category $\StinftyCat^{t}_{\on{cont}}$ consists of cocomplete DG categories $\bC$, 
endowed with a t-structure, such that $\bC^{\geq 0}$ is closed under filtered colimits, 
and such that $\bC$ is compactly generated by objects from $\bC^+$. As $1$-morphisms we allow those
exact functors $F:\bC_1\to \bC_2$ that are continuous and \emph{left t-exact up to a finite shift}. The higher 
categorical structure is uniquely determined by the requirement that the forgetful functor
$$\StinftyCat^{t}_{\on{cont}}\to \StinftyCat_{\on{cont}}$$
be 1-fully faithful. 

\medskip

We have a naturally defined functor
\begin{equation} \label{e:taking +}
\StinftyCat^{t}_{\on{cont}}\to \StinftyCat^{+_{\on{cont}}},\quad \bC\mapsto \bC^+.
\end{equation}

\begin{lem}  \label{l:taking +}
The functor \eqref{e:taking +} is 1-fully faihful.
\end{lem}

\sssec{}

We will use the following  general assertion. Let $T:\bD'\to \bD$ be a 1-fully faithful
functor between $(\infty,1)$-categories. Let $\bI$ be another $(\infty,1)$-category,
and let
\begin{equation} \label{e:assignment}
(\bi\in \bI)\rightsquigarrow (F'(\bi)\in \bD'),
\end{equation}
be an assignment, such that the assignment
$$\bi\mapsto T\circ F'(\bi)$$
has been extended to a functor $F:\bI\to \bD$. 

\begin{lem} \label{l:factoring functor}
Suppose that for every $\alpha\in \Maps_{\bI}(\bi_1,\bi_2)$, the point $F(\alpha)\in \Maps_{\bD}(F(\bi_1),F(\bi_2))$
lies in the connected component corresponding to the image of
$$\Maps_{\bD'}(F'(\bi_1),F'(\bi_2))\to \Maps_{\bD}(F(\bi_1),F(\bi_2)).$$
Then there exists a unique
extension of \eqref{e:assignment} to a functor $F':\bI\to \bD'$ equipped with an isomorphism $T\circ F'\simeq F$.
\end{lem}

Let now $F'_1$ and $F'_2$ be two assignments as in \eqref{e:assignment}, satisfying the assumption
of \lemref{l:factoring functor}. Let us be given an assignment
\begin{equation} \label{e:assignment bis}
\bi\rightsquigarrow \psi'_\bi\in \Maps_{\bD'}(F'_1(\bi),F'_2(\bi)).
\end{equation}

\begin{lem}  \label{l:factoring functor bis}
Suppose that the assignment 
$$\bi \rightsquigarrow T(\psi'_\bi)\in \Maps_{\bD}(F_1(\bi),F_2(\bi))$$
has been extended to a natural transformation $\psi:F_1\to F_2$. Then there exists a unique extension of 
\eqref{e:assignment bis} to a natural transformation $\psi:F'_1\to F'_2$ equipped with an isomorphism
$T\circ \psi'\simeq \psi$.
\end{lem}

\sssec{}

We are now ready to prove \propref{p:upgrading IndCoh to functor}:

\medskip

\noindent{\it Step 1.} We start with the functor
$$\QCoh_{\dgSch_{\on{Noeth}}}:\dgSch_{\on{Noeth}}\to \StinftyCat_{\on{cont}},$$
and consider
$$\bI=\dgSch_{\on{Noeth}},\,\, \bD=\StinftyCat_{\on{cont}},\,\, \bD':=\StinftyCat^{t}_{\on{cont}},\,\, F=\QCoh_{\dgSch_{\on{Noeth}}},$$
and the assignment 
$$(X\in \dgSch_{\on{Noeth}}) \rightsquigarrow (\QCoh(X)\in \StinftyCat^t_{\on{cont}}).$$

Apping \lemref{l:factoring functor}, we obtain a functor
\begin{equation} \label{e:IndCoh cont + 1}
\QCoh^t_{\dgSch_{\on{Noeth}}}:\dgSch_{\on{Noeth}}\to \StinftyCat^t_{\on{cont}}.
\end{equation}

\medskip

\noindent{\it Step 2.} Note that \propref{p:direct image} defines a functor
$$\IndCoh^t_{\dgSch_{\on{Noeth}}}:\dgSch_{\on{Noeth}}\to \StinftyCat^t_{\on{cont}},$$
and the natural transformation 
$$\Psi^t_{\dgSch_{\on{Noeth}}}:\IndCoh^t_{\dgSch_{\on{Noeth}}}\to \QCoh^t_{\dgSch_{\on{Noeth}}}$$
at the level of objects and $1$-morphisms.

\medskip

Since the functor $\StinftyCat^t_{\on{cont}}\to\StinftyCat_{\on{cont}}$ is 1-fully faithful,
by Lemmas \ref{l:factoring functor} and \ref{l:factoring functor bis}, 
the existence and uniqueness of the pair $(\IndCoh_{\dgSch_{\on{Noeth}}},\Psi_{\dgSch_{\on{Noeth}}})$ with
a fixed behavior on objects and $1$-morphisms, is equivalent to that of
$(\IndCoh^t_{\dgSch_{\on{Noeth}}},\Psi^t_{\dgSch_{\on{Noeth}}})$.
 
\medskip

\noindent{\it Step 3.} By \lemref{l:taking +}, combined with Lemmas \ref{l:factoring functor} and \ref{l:factoring functor bis}, 
we obtain that the existence and uniqueness of the pair
$(\IndCoh^t_{\dgSch_{\on{Noeth}}},\Psi^t_{\dgSch_{\on{Noeth}}})$, with a fixed behavior on objects and $1$-morphisms
is equivalent to the existence and uniqueness of the pair 
$$(\IndCoh^+_{\dgSch_{\on{Noeth}}},\Psi^+_{\dgSch_{\on{Noeth}}}),$$
obtained by composing with the functor \eqref{e:taking +}. 

\medskip

The latter, however, is given by
$$\IndCoh^+_{\dgSch_{\on{Noeth}}}:=\QCoh^+_{\dgSch_{\on{Noeth}}},$$
obtained from \eqref{e:IndCoh cont + 1} by composing with \eqref{e:taking +}, 
and taking
$$\Psi^+_{\dgSch_{\on{Noeth}}}:=\on{Id}.$$

\qed

\ssec{The !-pullback functor for proper maps}  

\sssec{} 

Let $f:S_1\to S_2$ be a map between Noetherian schemes. 

\begin{defn}
We shall say that $f$ is of almost of finite type if the corresponding map of classical
schemes $^{cl}\!S_1\to {}^{cl}\!S_2$ is.
\end{defn}

In particular, for $S_2=\on{Spec}(k)$, the above notion is equivalent to $S_1$
being almost of finite type over $k$, see \secref{sss:intr aft}.

\medskip

We let $(\dgSch_{\on{Noeth}})_{\on{aft}}$ denote the 1-full subcategory of
$\dgSch_{\on{Noeth}}$, where we restrict 1-morphisms to be maps almost of finite type.

\medskip

\begin{defn}  \label{defn:closed embeddings}
A map $f:S_1\to S_2$ between DG schemes is said to be proper/finite/closed embedding
if the corresponding map of classical schemes $^{cl}\!S_1\to {}^{cl}\!S_2$ has this property.
\end{defn}

In particular, any proper map is almost of finite type, as the latter is included in the
definition of properness for morphisms between classical schemes. We let
$$(\dgSch_{\on{Noeth}})_{\on{cl.emb.}}\subset (\dgSch_{\on{Noeth}})_{\on{finite}}\subset 
(\dgSch_{\on{Noeth}})_{\on{proper}} \subset (\dgSch_{\on{Noeth}})_{\on{aft}}$$
denote the corresponding inclusions of 1-full subcategories.

\sssec{}


We have the following basic  feature of proper maps:
\begin{lem}  \label{l:preservation of coh}
If $f:S_1\to S_2$ is proper, then the functor $f_*:\QCoh(S_1)\to \QCoh(S_2)$
sends $\Coh(S_1)$ to $\Coh(S_2)$. 
\end{lem}

\begin{proof}

It is enough to show that $f_*$ sends $\Coh(S_1)^\heartsuit$ to $\Coh(S_2)$.
Let $\imath_i$ denote the canonical maps $^{cl}\!S_i\to S_i$, $i=1,2$. Since
$$(\imath_i)_*:\Coh({}^{cl}\!S_i)^\heartsuit\to \Coh(S_i)^\heartsuit$$
are equivalences, the required assertion follows from the commutative
diagram
$$
\CD
\QCoh({}^{cl}\!S_1)  @>{(\imath_1)_*}>> \QCoh(S_1)  \\
@V{({}^{cl}f)_*}VV   @VV{f_*}V  \\
\QCoh({}^{cl}\!S_2)  @>{(\imath_2)_*}>> \QCoh(S_2),
\endCD
$$
and the well-known fact that direct image preserves coherence for proper maps
between classical Noetherian schemes.

\end{proof}

\begin{cor}
If $f$ is proper, the functor 
$$f_*^{\IndCoh}:\IndCoh(S_1)\to \IndCoh(S_2)$$
sends compact objects to compact ones. 
\end{cor}

\begin{proof}

We need to show that for $\CF\in \Coh(S_1)\subset \IndCoh(S_1)$, the object
$$f^\IndCoh_*(\CF)\in \IndCoh(S_2)$$ belongs to $\Coh(S_2)\subset \IndCoh(S_2)$.
By \propref{p:direct image}, $f^\IndCoh_*(\CF)\in \IndCoh(S_2)^+$. Hence, by
\propref{p:equiv on D plus}, it suffices to show that
$$\Psi_{S_2}(f^\IndCoh_*(\CF))\in \Coh(S_2)\subset \QCoh(S_2).$$
We have
$$\Psi_{S_2}(f^\IndCoh_*(\CF))\simeq f_*(\Psi_{S_1}(\CF)),$$
and the assertion follows from \lemref{l:preservation of coh}.

\end{proof}

\sssec{}

Consider the adjoint functor 
$$f^!:\IndCoh(S_2)\to \IndCoh(S_1)$$
(it exists for general $\infty$-category reasons, see \cite{Lu0}, Corollary 5.5.2.9).

\medskip

Now, the fact that $f_*^{\IndCoh}$ sends compacts to compacts implies that the functor
$f^!$ is continuous. I.e., $f^!$ is a 1-morphism in $\StinftyCat_{\on{cont}}$.

\sssec{}

By adjunction from \propref{p:upgrading IndCoh to functor} we obtain:

\begin{cor} \label{c:upper ! DG funct proper}
The assignment $S\mapsto \IndCoh(S)$ upgrades to a functor
$$\IndCoh^!_{(\dgSch_{\on{Noeth}})_{\on{proper}}}:((\dgSch_{\on{Noeth}})_{\on{proper}})^{\on{op}}\to \StinftyCat_{\on{cont}}.$$
\end{cor}

\sssec{}   \label{sss:upgrading !}

By \propref{p:dir image as modules}, the functor $f_*^{\IndCoh}$ has a natural structure of $1$-morphism 
between $\QCoh(S_2)$-module categories. Hence, from \cite[Corollary 6.2.4]{DG} we obtain:

\begin{cor} \label{c:upgrading !}
The functor $f^!$ has a natural structure of 1-morphism between $\QCoh(S_2)$-module 
categories. 
\end{cor}

At the level of individual objects, the assertion of \corref{c:upgrading !} says that 
for $\CE\in \QCoh(S_2)$ and $\CF\in \IndCoh(S_2)$, we have a
canonical isomorphism:

\begin{equation} \label{e:upgrading !}
f^!(\CE\otimes \CF)\simeq f^*(\CE)\otimes f^!(\CF),
\end{equation}
where $\otimes$ is understood in the sense of the action of \secref{ss:action}.

\ssec{Proper base change}

\sssec{}

Let $f:S_1\to S_2$ be a proper map between Noetherian DG schemes, and let
$g_2:S'_2\to S_2$ be an arbitrary map (according to the conventions of this
section, the DG scheme $S'_2$ is also assumed Noetherian). 

\medskip

Since $f$ is almost of 
finite type, the Cartesian product $S'_1:=S'_2\underset{S_2}\times S_1$ is also Noetherian,
and the resulting map $f':S'_1\to S'_2$ is proper. Let $g_1$ denote the map
$S'_1\to S_1$:
$$
\CD
S'_1  @>{g_1}>>  S_1  \\
@V{f'}VV   @VV{f}V  \\
S'_2 @>{g_2}>>  S_2.
\endCD
$$

The isomorphism of functors
$$f^\IndCoh_*\circ (g_1)_*^{\IndCoh}\simeq (g_2)_*^{\IndCoh}\circ (f')^\IndCoh_*$$
by adjunction gives rise to a natural transformation
\begin{equation} \label{e:proper base change}
(g_1)_*^{\IndCoh}\circ (f')^!\to f^!\circ (g_2)_*^{\IndCoh}
\end{equation}
between the two functors $\IndCoh(S'_2)\rightrightarrows \IndCoh(S_1)$.

\begin{prop}  \label{p:proper base change}
The natural transformation \eqref{e:proper base change} is an isomorphism.
\end{prop}

The rest of this subsection is devoted to the proof of this proposition.

\sssec{}

For a proper morphism $f:S_1\to S_2$, let 
$$f^{\QCoh,!}:\QCoh(S_2)\to \QCoh(S_1)$$
denote the \emph{not necessarily continuous} right adjoint to $f_*:\QCoh(S_1)\to \QCoh(S_2)$. 

\medskip

Since $f_*$ is right t-exact \emph{up to a cohomological shift}, the functor $f^{\QCoh,!}$
is left t-exact \emph{up to a cohomological shift}. Hence, it maps $\QCoh(S_2)^+$ to $\QCoh(S_1)^+$. 

\begin{lem}  \label{l:! and Psi}
The diagram
$$
\CD
\IndCoh(S_1)^+  @>{\Psi_{S_1}}>>  \QCoh(S_1)^+   \\
@A{f^!}AA  @AA{f^{\QCoh,!}}A   \\
\IndCoh(S_2)^+   @>{\Psi_{S_2}}>> \QCoh(S_2)^+
\endCD
$$
obtained by passing to right adjoints along the horizontal arrows in
$$
\CD
\IndCoh(S_1)^+  @>{\Psi_{S_1}}>>  \QCoh(S_1)^+   \\
@V{f^\IndCoh_*}VV    @VV{f_*}V   \\
\IndCoh(S_2)^+   @>{\Psi_{S_2}}>> \QCoh(S_2)^+,
\endCD
$$
commutes.
\end{lem}

\begin{proof}
Follows from the fact that the vertical arrows are equivalences, by \propref{p:equiv on D plus}.
\end{proof}

\begin{rem}
It is not in general true that the diagram
$$
\CD
\IndCoh(S_1)  @>{\Psi_{S_1}}>>  \QCoh(S_1)   \\
@A{f^!}AA  @AA{f^{\QCoh,!}}A   \\
\IndCoh(S_2)   @>{\Psi_{S_2}}>> \QCoh(S_2)
\endCD
$$
obtained by passing to right adjoints along the horizontal arrows in
$$
\CD
\IndCoh(S_1)  @>{\Psi_{S_1}}>>  \QCoh(S_1)   \\
@V{f^\IndCoh_*}VV    @VV{f_*}V   \\
\IndCoh(S_2)   @>{\Psi_{S_2}}>> \QCoh(S_2),
\endCD
$$
commutes.

\medskip

For example, take $S_1=\on{pt}=\Spec(k)$, $S_2=\Spec(k[t]/t^2)$ and
$0\neq \CF\in \IndCoh(S_2)$ be in the kernel of the functor $\Psi_{S_2}$. 
Then $\Psi_X\circ f^!(\CF)\neq 0$. Indeed, $\Psi_{S_1}$ is an equivalence,
and $f^!$ is conservative, see \corref{c:reduced generation}.

\end{rem}

\begin{proof}[Proof of \propref{p:proper base change}]

Since all functors involved are continuous, it is enough to show that the map
$$(g_1)_*^{\IndCoh}\circ (f')^!\to f^!\circ (g_2)_*^{\IndCoh}(\CF)$$
is an isomorphism for $\CF\in \Coh(S_1)$. Hence, it is enough to show
that \eqref{e:proper base change} is an isomorphism when restricted to $\IndCoh(S_1)^+$.

\medskip

By \lemref{l:! and Psi} and \propref{p:equiv on D plus}, this reduces the assertion 
to showing that the natural transformation
\begin{equation} \label{e:proper BC QCoh}
(g_1)_*\circ (f')^{\QCoh,!}\to  f^{\QCoh,!} \circ (g_2)_*
\end{equation}
is an isomorphism for the functors
$$
\CD
\QCoh(S'_1)^+   @>{(g_1)_*}>>   \QCoh(S_1)^+  \\
@A{(f')^{\QCoh,!}}AA    @AA{f^{\QCoh,!}}A   \\
\QCoh(S'_2)^+   @>{(g_2)_*}>>   \QCoh(S_2)^+,
\endCD
$$
where the natural transformation comes from the commutative diagram
$$
\CD
\QCoh(S'_1)^+   @>{(g_1)_*}>>   \QCoh(S_1)^+  \\
@V{f'_*}VV    @VV{f_*}V   \\
\QCoh(S'_2)^+   @>{(g_2)_*}>>   \QCoh(S_2)^+
\endCD
$$
by passing to right adjoint along the vertical arrows. 

\medskip

We consider the commutative diagram
$$
\CD
\QCoh(S'_1)   @>{(g_1)_*}>>   \QCoh(S_1)  \\
@V{f'_*}VV    @VV{f_*}V   \\
\QCoh(S'_2)   @>{(g_2)_*}>>   \QCoh(S_2),
\endCD
$$
and the diagram
$$
\CD
\QCoh(S'_1)   @>{(g_1)_*}>>   \QCoh(S_1)  \\
@A{(f')^{\QCoh,!}}AA    @AA{f^{\QCoh,!}}A   \\
\QCoh(S'_2)   @>{(g_2)_*}>>   \QCoh(S_2),
\endCD
$$
obtained by passing to right adjoints along the vertical arrows.
(Note, however, that the functors involved are no longer
continuous).

\medskip

We claim that the resulting natural transformation
\begin{equation} \label{e:proper BC QCoh 1}
(g_1)_*\circ (f')^{\QCoh,!}\to  f^{\QCoh,!} \circ (g_2)_*
\end{equation}
between functors
$$\QCoh(S'_2)\rightrightarrows \QCoh(S_1)$$
is an isomorphism.

\medskip

Indeed, the map in \eqref{e:proper BC QCoh 1} is obtained by passing to right adjoints
in the natural transformation
\begin{equation} \label{e:proper BC QCoh 2}
(g_2)^*\circ f_*\to f'_*\circ (g_1)^*
\end{equation}
as functors 
$$\QCoh(S_1)\rightrightarrows \QCoh(S'_2)$$
in the commutative diagram
$$
\CD
\QCoh(S'_1)   @<{(g_1)^*}<<   \QCoh(S_1)  \\
@V{f'_*}VV    @VV{f_*}V   \\
\QCoh(S'_2)   @<{(g_2)^*}<<   \QCoh(S_2).
\endCD
$$

Now, \eqref{e:proper BC QCoh 2} is an isomorphism by the usual base change for $\QCoh$. Hence,
\eqref{e:proper BC QCoh 1} is an isomorphism as well.

\end{proof}

\ssec{The $({\IndCoh,*})$-pullback}  \label{ss:* pullback}

\sssec{}

Let $f:S_1\to S_2$ be a morphism between Noetherian DG schemes. 

\begin{defn}  \label{defn:event coconn}
We shall say that $f$ is \emph{eventually coconnective} if the functor
$$f^*:\QCoh(S_2)\to \QCoh(S_1)$$
sends $\Coh(S_2)$ to $\QCoh(S_1)^+$, in which case it automatically sends
it to $\Coh(S_1)$.
\end{defn}

\medskip

It is easy to see that the following conditions are equivalent:

\begin{itemize}

\item $f$ is eventually coconnective;

\item For a closed embedding $S'_2\to S_2$ with $S'_2$ eventually coconnective, the
Cartesian product $S'_2\underset{S_2}\times S_1$ is eventually coconnective. 

\item For a closed embedding $S'_2\to S_2$ with $S'_2$ classical, the
Cartesian product $S'_2\underset{S_2}\times S_1$ is eventually coconnective. 

\end{itemize}

\sssec{}

By the same logic as in \secref{direct image}, we have:

\begin{prop}  \label{p:* pullback}
Suppose that $f$ is eventually coconnective. Then there exists a unique functor
$$f^{\IndCoh,*}:\IndCoh(S_2)\to \IndCoh(S_1),$$
which makes the diagram
$$
\CD
\IndCoh(S_1) @>{\Psi_{S_1}}>>  \QCoh(S_1)   \\
@A{f^{\IndCoh,*}}AA   @AA{f^*}A  \\
\IndCoh(S_2) @>{\Psi_{S_2}}>>  \QCoh(S_2) 
\endCD
$$
commute. 

Furthermore, the functor $f^{\IndCoh,*}$ has a unique structure of 1-morphism of $\QCoh(S_2)$-module
categories, so that the above diagram commutes as a diagram of $\QCoh(S_2)$-modules. 
\end{prop}

We note that the last assertion in \propref{p:* pullback} at the level of individual objects says that for
$\CE\in \QCoh(S_2)$ and $\CF\in \IndCoh(S_2)$, we have a canonical isomorphism
\begin{equation} \label{e:* pullback}
f^{\IndCoh,*}(\CE\otimes \CF)\simeq f^*(\CE)\otimes f^{\IndCoh,*}(\CF),
\end{equation}
where $\otimes$ is understood in the sense of the action of \secref{ss:action}.

\sssec{}

Let $(\dgSch_{\on{Noeth}})_{\on{ev-coconn}}$ denote the 1-full subcategory of $\dgSch_{\on{Noeth}}$
where we allow only eventually coconnective maps as $1$-morphisms. As in \propref{p:upgrading IndCoh to functor}
we obtain: 

\begin{cor} \label{c:upper * DG funct}
The assignment 
$$S\mapsto \IndCoh(S) \text{ and } \Psi_S$$ 
upgrade to a functor
$$\IndCoh^*_{(\dgSch_{\on{Noeth}})_{\on{ev-coconn}}}:(\dgSch_{\on{Noeth}})_{\on{ev-coconn}}^{\on{op}}\to \StinftyCat_{\on{cont}},$$
and a natural transformation
\begin{multline*}
\Psi^*_{(\dgSch_{\on{Noeth}})_{\on{ev-coconn}}}:\IndCoh^*_{(\dgSch_{\on{Noeth}})_{\on{ev-coconn}}}\to 
\QCoh^*_{(\dgSch_{\on{Noeth}})_{\on{ev-coconn}}}:=\\
=\QCoh^*_{\dgSch_{\on{Noeth}}}|_{((\dgSch_{\on{Noeth}})_{\on{ev-coconn}})^{\on{op}}}.
\end{multline*}
\end{cor}

\sssec{The (inverse,direct) image adjunction}

We are now going to show that under the assumptions of \propref{p:* pullback}, the functors
$f^{\IndCoh}_*$ and $f^{\IndCoh,*}$ satisfy the usual adjunction property:

\begin{lem}  \label{l:* pullback adjoint}
Let $f:S_1\to S_2$ be eventually coconnective. Then 
the functor $$f^{\IndCoh,*}:\IndCoh(S_2) \to \IndCoh(S_1)$$ is the left adjoint of 
$f^{\IndCoh}_*:\IndCoh(S_1)\to \IndCoh(S_2) $.
\end{lem}

\begin{proof}

We need to construct a functorial isomorphism
\begin{equation} \label{e:* adj to constr}
\Maps_{\IndCoh(S_1)}(f^{\IndCoh,*}(\CF_2),\CF_1)\simeq \Maps_{\IndCoh(S_2)}(\CF_2,f^\IndCoh_*(\CF_1)),
\end{equation}
for $\CF_1\in \IndCoh(S_1)$ and $\CF_2\in \IndCoh(S_2)$.

\medskip

By the definition of $\IndCoh(S_2)$, it suffices to do this for $\CF_2\in \Coh(S_2)$. Now, since
the functor $f^{\IndCoh,*}$ sends compact objects to compact ones, for $\CF_2\in \Coh(S_2)$,
each side in \eqref{e:* adj to constr} commutes with filtered colimits. Hence, it suffices to
construct the isomorphism \eqref{e:* adj to constr} when $\CF_2\in \Coh(S_2)$ and 
$\CF_1\in \Coh(S_1)$.

\medskip

In this case, 
$$\Maps_{\IndCoh(S_1)}(f^{\IndCoh,*}(\CF_2),\CF_1)\simeq
\Maps_{\Coh(S_1)}(f^*(\CF_2),\CF_1)\simeq \Maps_{\QCoh(S_1)}(f^*(\CF_2),\CF_1),$$
and
$$\Maps_{\IndCoh(S_2)}(\CF_2,f^\IndCoh_*(\CF_1))
=\Maps_{\IndCoh(S_2)^+}(\CF_2,f^\IndCoh_*(\CF_1)),$$
which by \propref{p:equiv on D plus} identifies with
$$\Maps_{\QCoh(S_2)^+}(\CF_2,f_*(\CF_1))\simeq  \Maps_{\QCoh(S_2)}(\CF_2,f_*(\CF_1)).$$

\medskip

Hence, the required isomorphism follows from the isomorphism
$$\Maps_{\QCoh(S_1)}(f^*(\CF_2),\CF_1)\simeq \Maps_{\QCoh(S_2)}(\CF_2,f_*(\CF_1)),$$
which expresses the $(f^*,f_*)$-adjunction for $\QCoh$.

\end{proof}

In fact, a statement converse to \lemref{l:* pullback adjoint} holds:

\begin{prop}  \label{p:* pullback bis}
Let $f:S_1\to S_2$ be a morphism between DG schemes, such that the functor
$f^{\IndCoh}_*:\IndCoh(S_1)\to \IndCoh(S_2)$ admits a left adjoint. Then $f$ is
eventually coconnective.
\end{prop}

\begin{proof}
Suppose $f^{\IndCoh}_*$ admits a left adjoint; let us denote it by $$'\!f^{\IndCoh,*}:\IndCoh(S_2)\to \IndCoh(S_1).$$
Being a left adjoint to a functor that commutes with colimits, $'\!f^{\IndCoh,*}$ sends compact objects
to compacts, i.e., it is the ind-extension of a functor
$$'\!f^{\IndCoh,*}:\Coh(S_2)\to \Coh(S_1).$$

\medskip

To prove the proposition, it suffices to show that
the natural map 
\begin{equation} \label{f tilde}
f^*\circ \Psi_{S_2}\to \Psi_{S_1}\circ {}'\!f^{\IndCoh,*}
\end{equation}
is an isomorphism.

\medskip

Since the t-structure on $\QCoh(S_1)$ is left-complete, it suffices to show
that for $\CF\in \Coh(S_2)$ and any $n$, the induced map
$$\tau^{\geq n}\left(f^*\circ \Psi_{S_2}(\CF)\right)\to
\tau^{\geq n}\left(\Psi_{S_1}\circ {}'\!f^{\IndCoh,*}(\CF)\right)$$
is an isomorphism, i.e., that the induced map
\begin{equation} \label{weird Hom}
\Hom_{\QCoh(S_1)}(\Psi_{S_1}\circ {}'\!f^{\IndCoh,*}(\CF),\CF')\to \Hom_{\QCoh(S_1)}(f^*\circ \Psi_{S_2}(\CF),\CF')
\end{equation}
is an isomorphism for any $\CF'\in \QCoh(S_1)^{\geq n}$. By \propref{p:equiv on D plus}, we can
take $\CF'=\Psi_{S_1}(\CF_1)$ for some $\CF_1\in \IndCoh(S_1)^{\geq n}$.

\medskip

The object $'\!f^{\IndCoh,*}(\CF)$ belongs to $\Coh(S_1)\subset \IndCoh(S_1)^+$. Hence,
by \propref{p:equiv on D plus}, the left-hand side of \eqref{weird Hom} identifies with
$$\Hom_{\IndCoh(S_1)}({}'\!f^{\IndCoh,*}(\CF),\CF_1),$$
which in turn identifies with 
\begin{multline*}
\Hom_{\IndCoh(S_2)}(\CF,f^{\IndCoh}_*(\CF_1))\simeq
\Hom_{\QCoh(S_2)}(\Psi_{S_2}(\CF),\Psi_{S_2}(f^{\IndCoh}_*(\CF_1)))\simeq \\
\simeq \Hom_{\QCoh(S_2)}(\Psi_{S_2}(\CF),f_*(\Psi_{S_1}(\CF_1))),
\end{multline*}
which identifies by adjunction with the right-hand side of 
\eqref{weird Hom}.

\end{proof}

\sssec{}  

Assume again that $S_1$ and $S_2$ are eventually coconnective. In this case, by adjunction from
\propref{p:direct image}, we obtain:

\begin{lem}  \label{l:Xi and pullback}
The diagram
$$
\CD
\IndCoh(S_1) @<{\Xi_{S_1}}<<  \QCoh(S_1)   \\
@A{f^{\IndCoh,*}}AA   @AA{f^*}A  \\
\IndCoh(S_2) @<{\Xi_{S_2}}<<  \QCoh(S_2) 
\endCD
$$
commutes as well.
\end{lem}

\sssec{}   \label{sss:Xi and pullback}

In particular, we obtain that the assignment $S\mapsto \Xi_S$ extends to a natural
transformation
$$\Xi^*_{({}^{<\infty}\!\dgSch_{\on{Noeth}})_{\on{ev-coconn}}}:\QCoh^*_{({}^{<\infty}\!\dgSch_{\on{Noeth}})_{\on{ev-coconn}}}\to
\IndCoh^*_{({}^{<\infty}\!\dgSch_{\on{Noeth}})_{\on{ev-coconn}}},$$
where $$\QCoh^*_{({}^{<\infty}\!\dgSch_{\on{Noeth}})_{\on{ev-coconn}}} \text{ and }
\IndCoh^*_{({}^{<\infty}\!\dgSch_{\on{Noeth}})_{\on{ev-coconn}}}$$
denote he restrictions of the functors 
$$\QCoh^*_{(\dgSch_{\on{Noeth}})_{\on{ev-coconn}}} \text{ and }\IndCoh^*_{(\dgSch_{\on{Noeth}})_{\on{ev-coconn}}},$$
respectively,
to $(({}^{<\infty}\!\dgSch_{\on{Noeth}})_{\on{ev-coconn}})^{\on{op}}\subset ((\dgSch_{\on{Noeth}})_{\on{ev-coconn}})^{\on{op}}$.

\ssec{Morphisms of bounded Tor dimension}

\sssec{}

We shall say that a morphism $f:S_1\to S_2$ between DG schemes is \emph{of bounded Tor dimension}
if the functor 
$$f^*:\QCoh(S_2)\to \QCoh(S_1)$$
is left t-exact \emph{up to a finite shift}, i.e., is of bounded cohomological amplitude. 

\sssec{}  

First, we claim:

\begin{lem} \label{l:finite Tor}
Let $f:S_1\to S_2$ be a morphism almost of finite type. Then following conditions are equivalent:

\smallskip

\noindent{\em(a)} $f$ is eventually coconnective.

\smallskip

\noindent{\em(b)} $f$ is of bounded Tor dimension.

\end{lem}

\begin{proof}

The implication (b) $\Rightarrow$ (a) is tautological. Let us prove that (a) implies (b).

\medskip

The question is local in Zariski topology, so we can assume that $f$ can be factored as a composition
$$S_1\overset{f'}\to S_2\times \BA^n\to S_2,$$
where $f'$ is a closed embedding. It is easy to see that $f$ satisfies condition (a) (resp., (b)) if and only 
if $f'$ does. Hence, we can assume that $f$ is itself a closed embedding.

\medskip

To test that $f$ is of bounded Tor dimension, it is sufficient to test it on objects from $\Coh(S_2)^\heartsuit$.
Since such objects come as direct images from $^{cl}S_2$, by base change, we can assume that $S_2$
is classical. Since $f$ was assumed eventually coconnective, we obtain that $S_1$ is itself eventually
coconnective, i.e., $f_*(\CO_{S_1})\in \Coh(S_2)$.  

\medskip

We need to show that $f_*(\CO_{S_1})$ is of bounded Tor dimension. This will
follow from the following general assertion: 

\begin{lem}  \label{l:pointwise perfectness}
Let $S$ is a Noetherian DG scheme, and let $\CF\in \Coh(S)$ be such that 
for every geometric point $s:\Spec(K)\to S$, the fiber $s^*(\CF)$ lives in finitely many degrees, then
$\CF$ is perfect.
\end{lem}

\end{proof}

\begin{proof}[Proof of \lemref{l:pointwise perfectness}]

We need to show that the functor $\CF\underset{\CO_S}\otimes -:\QCoh(S)\to \QCoh(S)$ is of 
bounded cohomological amplitude.

\medskip

In fact, we will show that if for \emph{some} geometric point $s:\Spec(K)\to S$, the fiber $s^*(\CF)$ lives in 
degrees $[-n,0]$, then on some Zariski neighborhood of $s$, the above functor is of amplitude
$[-n,0]$. 

\medskip

First, it is enough to test the functor $\CF\underset{\CO_S}\otimes -$ on objects from $\QCoh(S)^\heartsuit$.
Since such objects come as direct images under $^{cl}S\to S$, we can replace $S$ by $^{cl}S$. Now,
out assertion becomes a familiar statement from commutative algebra. Let us prove it for completeness:

\medskip

We shall argue by induction on $n$. For $n=-1$ the assertion follows from Nakayama's lemma:
if $s^*(\CF)=0$, then $\CF$ vanishes on a Zariski neighborhood of $s$.

\medskip

To execute the induction step, 
choose a locally free $\CO_S$-module $\CP$ equipped with a map $\CP\to \CF$ which induces an 
isomorphism
$$H^0(s^*(\CP))\to H^0(s^*(\CM)).$$ 
Set $\CF':=\on{Cone}(\CF\to \CP)[-1]$. By construction, $s^*(\CF')$ lives in cohomological degrees
$[-n+1,0]$, as desired.

\end{proof}

\begin{cor} \label{c:finite Tor}
If $f:S_1\to S_2$ is eventually coconnective and almost of finite type, 
the functor $f^*$ sends $\QCoh(S_2)^+$ to $\QCoh(S_1)^+$,
and the functor $f^{\IndCoh,*}$ sends $\IndCoh(S_2)^+$ to $\IndCoh(S_1)^+$.
\end{cor}

\sssec{}

Assume now that the DG schemes $S_1$ and $S_2$ are eventually coconnective. Note that the 
isomorphism
$$\Psi_{S_2}\circ f^{\IndCoh}_*\simeq f_*\circ \Psi_{S_1}$$
induces a natural transformation
\begin{equation} \label{e:Xi nat tr}
\Xi_{S_2}\circ f_*\to f^{\IndCoh}_*\circ \Xi_{S_1}.
\end{equation}

\begin{prop}  \label{p:Xi and *}
Assume that $f$ is of bounded Tor dimension and almost of finite type. Then \eqref{e:Xi nat tr} is an isomorphism.
\end{prop}

\begin{proof}

\medskip

As we shall see in \lemref{l:localization}, the assertion is local in the Zariski topology on both $S_1$
and $S_2$. Therefore, we can assume that $S_1$ and $S_2$ are affine and that $f$ can be factored as
$$S_1\overset{f'}\to S_2\times \BA^n\to S_2,$$
where $f'$ is a closed embedding. 

\medskip

As in the proof of \lemref{l:finite Tor}, the assumption that $f$ is of bounded Tor dimension implies 
the same for $f'$. Hence, it suffices to prove the proposition separately in the following two cases:
(a) $f$ is the projection $S_1=S_2\times \BA^n \to S_2$; and (b) $f$ is a closed embedding.

\medskip

Note that the two sides of \eqref{e:Xi nat tr} become isomorphic after applying the functor $\Psi_{S_2}$. 
Hence, by \propref{p:equiv on D plus}, the assertion of the proposition is equivalent to the fact that for $\CF\in \QCoh(S_1)^{\on{perf}}$,
the object
$$\Xi_{S_2}(f_*(\CF))\in \IndCoh(S_2)$$
belongs to $\IndCoh(S_2)^+$. Since $S_1$ is assumed affine, it suffices to consider the case of $\CF=\CO_{S_1}$. 

\medskip

Thus, in case (a) we need to show that $\Xi_{S_2}(\CO_{S_2}[t_1,...,t_n])$ belongs to $\IndCoh(S_2)^+$. However,
$\CO_{S_2}[t_1,...,t_n]$ is isomorphic to a direct sum of copies of $\CO_{S_2}$, and therefore $\Xi_{S_2}(\CO_{S_2}[t_1,...,t_n])$
is isomorphic to a direct sum of copies of $\Xi_{S_2}(\CO_{S_2})$. Hence, the assertion follows from
the fact that the t-structure on $\IndCoh(S_2)$ is compatible with filtered colimits.

\medskip

In case (b) the object $f_*(\CO_{S_1})$ belongs to $\QCoh(S_2)^{\on{perf}}$,
and the assertion follows by definition.

\end{proof}

\sssec{A base change formula}   

Let $f:S_1\to S_2$ be again a map of bounded Tor dimension and almost of finite type.
Let $g_2:S'_2\to S_2$
be an arbitrary map between Noetherian DG schemes. 

\medskip

The almost finite type assumption implies that the Cartesian product $S'_1:=S'_2\underset{S_2}\times S_1$
is also Noetherian. Moreover, the resulting morphism $f':S'_1\to S'_2$ 
$$
\CD
S'_1  @>{g_1}>>  S_1  \\
@V{f'}VV    @VV{f}V   \\
S'_2  @>{g_2}>>  S_2
\endCD
$$
is also of bounded Tor dimension.

\medskip

\begin{lem}  \label{l:usual base change}
Under the above circumstances, the map 
$$f^{\IndCoh,*}\circ (g_2)^{\IndCoh}_*\to (g_1)^{\IndCoh}_*\circ (f')^{\IndCoh,*}$$
induced by the $(f^{\IndCoh,*},f^{\IndCoh}_*)$ adjunction, is an isomorphism.
\end{lem}

\begin{proof}

Follows from \corref{c:finite Tor} and 
the usual base change formula for $\QCoh$,
by evaluating both functors on $\Coh(S_1)\subset \QCoh(S_1)^+$. 

\end{proof}

\sssec{A projection formula}

Let $f:S_1\to S_2$ be again a map of bounded Tor dimension. For
$\CE_1\in \QCoh(S_1)$ and $\CF_2\in \IndCoh(S_2)$ consider the canonical map
\begin{equation} \label{e:proj formula map}
f_*(\CE_1)\otimes \CF_2\to f^{\IndCoh}_*(\CE_1\otimes f^{\IndCoh,*}(\CF_2))
\end{equation}
that comes by adjunction from
$$f^{\IndCoh,*}(f_*(\CE_1)\otimes \CF_2)\simeq
f^*(f_*(\CE_1))\otimes f^{\IndCoh,*}(\CF_2)\to \CE_1\otimes f^{\IndCoh,*}(\CF_2).$$

Here $\otimes$ denotes the action of $\QCoh(-)$ on $\IndCoh(-)$ given by \secref{ss:action}.

\begin{prop} \label{p:proj form}
The map \eqref{e:proj formula map} is an isomorphism.
\end{prop}

\begin{rem}
Note that there is another version of the projection formula: namely, \eqref{e:dir image as modules}.
It says that for
$\CE_2\in \QCoh(S_2)$ and $\CF_1\in \IndCoh(S_1)$ we have
$$f^\IndCoh_*(f^*(\CE_2)\otimes \CF_1)\simeq \CE_2\otimes f^\IndCoh_*(\CF_1).$$
The formula holds nearly tautologically and expresses the fact that the functor 
$f^\IndCoh_*$ is $\QCoh(S_2)$-linear, see \propref{p:dir image as modules}.
\end{rem}

\begin{proof} 

It is enough to show that the isomorphism holds for $\CE_1\in \QCoh(S_1)^{\on{perf}}$ and
$\CF_2\in \Coh(S_2)\subset \IndCoh(S_2)$.

\medskip

We also note that the map \eqref{e:proj formula map} becomes an isomorphism 
after applying the functor $\Psi_{S_2}$, by the usual projection formula for $\QCoh$. 
For $\CE_1\in \QCoh(S_1)^{\on{perf}}$ and $\CF_2\in \Coh(S_2)$ we have
$$f^\IndCoh_*(\CE_1\otimes f^{\IndCoh,*}(\CF_2))\in \IndCoh(S_2)^+.$$

Hence, by \propref{p:equiv on D plus}, it suffices to show that in this case
$$f_*(\CE_1)\otimes \CF_2\in \IndCoh(S_2)^+.$$

We note that the object $f_*(\CE_1)\in \QCoh(S_1)^b$ is of bounded Tor dimension.
The required fact follows from the next general observation:

\begin{lem} \label{l:tensoring with bounded}
For $S\in \dgSch_{\on{Noeth}}$ and $\CE\in \QCoh(S)^b$, whose Tor dimension 
is bounded on the left by an integer $n$, the functor
$$\CE\otimes -:\IndCoh(S)\to \IndCoh(S)$$
has a cohomological amplitude bounded on the left by $n$.
\end{lem}

\end{proof}

\sssec{Proof of \lemref{l:tensoring with bounded}}

We need to show that the functor $\CE\otimes -$ sends $\IndCoh(S)^{\geq 0}$ to $\IndCoh(S)^{\geq -n}$.
It is sufficient to show that this functor sends $\Coh(S)^{\geq 0}$ to $\IndCoh(S)^{\geq -n}$.
By cohomological devissage, the latter is equivalent to sending $\Coh(S)^\heartsuit$ to $\IndCoh(S)^{\geq -n}$.

\medskip

Let $i$ denote the closed embedding $^{cl}\!S=:S'\to S$. The functor $i^\IndCoh_*$ induces an equivalence
$\Coh(S')^\heartsuit\to \Coh(S)^\heartsuit$. So, it is enough to show that for $\CF'\in \Coh(S')^\heartsuit$, we have
$$\CE\otimes i^\IndCoh_*(\CF')\in \IndCoh(S)^{\geq -n}.$$

\medskip

We have:
$$\CE\otimes i^\IndCoh_*(\CF')\simeq i^\IndCoh_*(i^*(\CE)\otimes \CF').$$
Note that the functor $i^\IndCoh_*$ is t-exact (since $i_*$ is), and $i^*(\CE)$ has Tor dimension bounded by the same
integer $n$.

\medskip

This reduces the assertion of the lemma to the case when $S$ is classical. Further, by \corref{c:t structure local}
(which will be proved independently later), the statement is Zariski local, so we can assume that $S$ is affine.

\medskip

In the latter case, the assumption on $\CE$ implies that it can be represented by a complex of flat
$\CO_S$-modules that lives in the cohomological degrees $\geq -n$. This reduces the assertion further
to the case when $\CE$ is a flat $\CO_S$-module in degree $0$. In this case we claim that the functor
$$\CE\otimes -:\IndCoh(S)\to \IndCoh(S)$$
is t-exact. 

\medskip

The latter follows from Lazard's lemma: such $\CE$ is a filtered colimit of locally free $\CO_S$-modules
$\CE_0$, while for each such $\CE_0$, the functor $\CE_0\otimes -:\IndCoh(S)\to \IndCoh(S)$ is by definition the 
ind-extension of the functor
$$\CE_0\otimes -:\Coh(S)\to \Coh(S),$$ 
and the latter is t-exact.

\qed

\section{Properties of $\IndCoh$ inherited from $\QCoh$}  \label{s:properties}

In this section we retain the assumption that all DG schemes considered are Noetherian. 

\ssec{Localization}

Let $S$ be a DG scheme, and let $j:\oS\hookrightarrow S$ be an open embedding. By 
\lemref{l:* pullback adjoint}, we have a pair of mutually adjoint functors
$$j^{\IndCoh,*}:\IndCoh(S)\rightleftarrows \IndCoh(\oS):j^{\IndCoh}_*.$$

\begin{lem}  \label{l:localization}
The functor $j^{\IndCoh}_*$ is fully faithful. 
\end{lem}

\begin{proof}
To prove that $j^{\IndCoh}_*$ is fully faithful, we have to show that the counit of the adjunction
\begin{equation} \label{e:localization}
j^{\IndCoh,*}\circ j^{\IndCoh}_*\to \on{Id}_{\IndCoh(\oS)}
\end{equation}
is an isomorphism.

Since both sides are continuous functors, it is sufficient to
do so when evaluated on objects of $\Coh(\oS)$, in which case
both sides of \eqref{e:localization} belong to $\IndCoh(\oS)^+$.
Therefore, by \propref{p:equiv on D plus}, it is sufficient to
show that the map
$$\Psi_{\oS}\circ j^{\IndCoh,*}\circ j^{\IndCoh}_*\to \Psi_{\oS}$$
is an isomorphism. However, by Propositions \ref{p:direct image}
and \ref{p:* pullback}, we have a commutative diagram
$$
\CD
\Psi_{\oS}\circ j^{\IndCoh,*}\circ j^{\IndCoh}_*   @>>>  \Psi_{\oS} \\
@V{\sim}VV     @VV{\on{id}}V   \\
j^*\circ j_*\circ \Psi_{\oS}   @>>>  \Psi_{\oS}
\endCD
$$
and the assertion follows from the fact that $j^*\circ j_*\to \on{Id}_{\QCoh(\oS)}$
is an isomorphism. 

\end{proof}

\sssec{}

Let now $i:S'\hookrightarrow S$ be a closed embedding whose image is
complementary to $\oS$. Let
$\IndCoh(S)_{S'}$ denote the full subcategory of $\IndCoh(S)$ equal to
$$\on{ker}\left(j^{\IndCoh,*}:\IndCoh(S)\to \IndCoh(\oS)\right);$$
we denote the tautological functor $\IndCoh(S)_{S'}\hookrightarrow \IndCoh(S)$
by $\wh{i}_*^{\IndCoh}$.

\begin{rem} \label{r:formal completion}
It is shown in \cite[Proposition 7.4.5]{GR2} that the category $\IndCoh(S)_{S'}$ is intrinsic
to the ind-scheme equal to the formal completion of $S$ at $S'$.
\end{rem}

\sssec{}

From \lemref{l:localization} we obtain:
\begin{cor} \label{c:localization}
The functor $\wh{i}_*^{\IndCoh}$ admits a continuous right adjoint (denoted $\wh{i}^!$)
making $\IndCoh(S)_{S'}$ into a colocalization of $\IndCoh(S)$. 
The kernel of $\wh{i}^!$ is the essential image of the functor $j^\IndCoh_*$.
the kernel of $\wh{i}^!$ is the essential image of the functor $j^\IndCoh_*$.
\end{cor}

We can depict the assertion of \corref{c:localization} as a ``short exact sequence of DG categories":
$$\IndCoh(S)_{S'}\overset{\wh{i}_*^{\IndCoh}}{\underset{\wh{i}^!}\rightleftarrows} \IndCoh(S)
\overset{j^{\IndCoh,*}}{\underset{j^\IndCoh_*}\rightleftarrows} \IndCoh(\oS).$$

\sssec{}

It is clear that the essential image of the functor
$$i^{\IndCoh}_*:\IndCoh(S')\to \IndCoh(S)$$
lies in $\IndCoh(S)_{S'}$, and that its right adjoint
$$i^!:\IndCoh(S)\to \IndCoh(S')$$ factors through the colocalization $\wh{i}^!$. 
We will denote the resulting pair of adjoint functors as follows:
$$'i^{\IndCoh}_*:\IndCoh(S')\rightleftarrows \IndCoh(S)_{S'}:{}'i^!.$$

\begin{prop}  \label{p:generation by closed}  
{\em(a)} The functor $'i^!$ is conservative. The essential image of $'i^{\IndCoh}_*$ generates
the category $\IndCoh(S)_{S'}$.

\smallskip

\noindent{\em(b)} The category $\IndCoh(S)_{S'}$ identifies with the ind-completion of
$$\Coh(S)_{S'}:=\on{ker}\left(j^*:\Coh(S)\to \Coh(\oS)\right).$$

\end{prop}

\begin{proof}
First, we observe that the two statements of point (a) are equivalent. We will prove
the second statement.

\medskip

Observe also that the functor $j^{\IndCoh,*}$ is t-exact, so the subcategory 
$$\IndCoh(S)_{S'}\subset \IndCoh(S)$$ is stable under taking truncations. In particular,
it inherits a t-structure.

\medskip

The category $\IndCoh(S)$ is generated by $\IndCoh(S)^+$. The functor
$\wh{i}^!$ is explicitly given by
$$\CF\mapsto \on{Cone}\left(\CF\to j^{\IndCoh}_*\circ j^{\IndCoh,*}(\CF)\right)[-1],$$
from which it is clear that $\IndCoh(S)_{S'}$ is also generated by 
$$\IndCoh(S)^+_{S'}=\IndCoh(S)_{S'}\cap \IndCoh(S)^+.$$

\medskip

Note $\QCoh(S)$ is right-complete in its t-structure (every object is isomorphic to the colimit
to its $\tau^{\leq n}$-truncations).  Hence, by \propref{p:equiv on D plus}, the same is true for $\IndCoh(S)^+$, 
and threrefore also for $\IndCoh(S)^+_{S'}$. From here we obtain that $\IndCoh(S)_{S'}$
is generated by 
$$\IndCoh(S)^b_{S'}=\IndCoh(S)_{S'}\cap \IndCoh(S)^b,$$
and hence, by devissage, by 
$$\IndCoh(S)_{S'}^\heartsuit=\IndCoh(S)_{S'}\cap \IndCoh(S)^\heartsuit.$$

\medskip

Therefore, it is sufficient to show that every object of $\IndCoh(S)_{S'}^\heartsuit$
admits an (increasing) filtration with subquotients belonging to the essential
image of the functor 
$$i^{\IndCoh}_*:\IndCoh(S')^\heartsuit\to \IndCoh(S)^\heartsuit.$$
However, the latter is obvious: by \propref{p:equiv on D plus}, the functor
$\Psi$ gives rise to a commutative diagram
$$
\CD
\IndCoh(S')^\heartsuit   @>{i^{\IndCoh}_*}>>  \IndCoh(S)^\heartsuit  @>{j^{\IndCoh,*}}>>  \IndCoh(\oS)^\heartsuit \\
@V{\Psi_{S'}}VV      @VV{\Psi_S}V   @VV{\Psi_{\oS}}V   \\
\QCoh(S')^\heartsuit  @>{i_*}>>  \QCoh(S)^\heartsuit  @>{j^*}>>  \QCoh(\oS)^\heartsuit
\endCD
$$
where the vertical arrows are equivalences, and the corresponding assertion for
$$\on{ker}\left(j^*:\QCoh(S)^\heartsuit \to \QCoh(\oS)^\heartsuit\right)$$
is manifest.

\medskip

To prove point (b) we note that ind-extending the tautological embedding
$$\Coh(S)_{S'}\hookrightarrow \Coh(S),$$ we obtain a fully faithful functor
$$\Ind(\Coh(S)_{S'})\to \IndCoh(S)_{S'}.$$
Therefore, it remains to show that it is essentially surjective, but this is implied
by point (a).

\end{proof}

\begin{cor}  \label{c:reduced generation}
Let $f:S_1\to S_2$ be a morphism that induces an isomorphism of the 
underlying reduced classical schemes: $({}^{cl}\!S_1)_{red}\simeq ({}^{cl}\!S_2)_{red}$. Then the essential image
of $f^{\IndCoh}_*$ generates the target category, and, equivalently, the
functor $f^!$ is conservative.
\end{cor}

\sssec{}

Let us note that an argument similar to that in the proof of \propref{p:generation by closed}
shows the following:

\medskip

Let $\eta$ be a generic point of $({}^{cl}S)_{red}$, and let $S_\eta$ be the corresponding localized DG scheme.
We have a natural restriction functor
$$\IndCoh(S)\to \IndCoh(S_\eta).$$

\begin{lem} \label{l:away from generic}
The subcategory $\on{ker}(\IndCoh(S)\to \IndCoh(S_\eta))$ is generated by the
union of the essential images of $\IndCoh(S')$ for closed embeddings $S'\to S$ such that
$({}^{cl}S')_{red}\cap \eta=\emptyset$.
\end{lem} 

\ssec{Zariski descent}

Let $f:S'\to S$ be a Zariski cover of a scheme $S$ (i.e., $S'$ is a finite disjoint union
of open subsets in $S$ that cover it). Let $S'{}^\bullet/S$ be its \v{C}ech 
nerve. 

\medskip

Restriction along open embedding makes $\IndCoh(S)\to \IndCoh(S'{}^\bullet/S)$ into
an augmented cosimplicial object in $\StinftyCat_{\on{cont}}$.

\begin{prop}  \label{p:Zariski descent}
Under the above circumstances, the natural map
$$\IndCoh(S)\to |\on{Tot}(\IndCoh(S'{}^\bullet/S))|$$
is an equivalence.
\end{prop}

\begin{proof}
The usual argument reduces the assertion of the proposition to the following. Let $S=U_1\cup U_2$; 
$U_{12}=U_1\cap U_2$. Let 
$$U_1\overset{j_1}\hookrightarrow S,\,\, U_2\overset{j_2}\hookrightarrow S,\,\, U_{12} 
\overset{j_{12}}\hookrightarrow S,\,\, U_{12} 
\overset{j_{12,1}}\hookrightarrow U_1,\,\, U_{12} 
\overset{j_{12,2}}\hookrightarrow U_2$$
denote the corresponding open embeddings. 

\medskip

We need to show that the functor
$$\IndCoh(S)\to \IndCoh(U_1)\underset{\IndCoh(U_{12})}\times  \IndCoh(U_1)$$
that sends $\CF\in \IndCoh(S)$ to the datum of 
$$\{j_1^{\IndCoh,*}(\CF),j_2^{\IndCoh,*}(\CF),
j_{12,1}^{\IndCoh,*}(j_1^{\IndCoh,*}(\CF))\simeq j_{12}^{\IndCoh,*}(\CF)\simeq 
j_{12,2}^{\IndCoh,*}(j_2^{\IndCoh,*}(\CF))\}$$
is an equivalence.

\medskip

We construct a functor
$$\IndCoh(U_1)\underset{\IndCoh(U_{12})}\times  \IndCoh(U_1)\to \IndCoh(S)$$
by sending 
$$\{\CF_1\in \IndCoh(U_1),\CF_2\in \IndCoh(U_2), \CF_{12}\in \IndCoh(U_{12}), 
j_{12,1}^{\IndCoh,*}(\CF_1)\simeq \CF_{12}\simeq j_{12,2}^{\IndCoh,*}(\CF_2)\}$$
to
$$\on{Cone}\biggl(\left((j_1)^\IndCoh_*(\CF_1)\oplus (j_2)^\IndCoh_*(\CF_1)\right)\to 
(j_{12})^\IndCoh_*(\CF_{12})\biggr)[-1],$$
where the maps $(j_i)^\IndCoh_*(\CF_i)\to (j_{12})^\IndCoh_*(\CF_{12})$
are
\begin{multline*}
(j_i)^\IndCoh_*(\CF_i)\to (j_i)^\IndCoh_*\circ (j_{12,i})_*^\IndCoh\circ (j_{12,i})^{\IndCoh,*}(\CF_i)=\\
=(j_{12})_*^\IndCoh\circ (j_{12,i})^{\IndCoh,*}(\CF_i)\simeq (j_{12})_*^\IndCoh(\CF_{12}).
\end{multline*}

\medskip

It is strightforward to see from Lemmas \ref{l:localization} and \ref{l:usual base change}
that the composition
$$\IndCoh(U_1)\underset{\IndCoh(U_{12})}\times  \IndCoh(U_1)\to \IndCoh(S)\to 
\IndCoh(U_1)\underset{\IndCoh(U_{12})}\times  \IndCoh(U_1)$$
is canonically isomorphic to the identity functor. 

\medskip

To prove that the composition
$$ \IndCoh(S)\to 
\IndCoh(U_1)\underset{\IndCoh(U_{12})}\times  \IndCoh(U_1)\to \IndCoh(S)$$
is also isomorphic to the identity functor, it is sufficient to show that for
$\CF\in \IndCoh(S)$, the canonical map from it to 
\begin{multline} \label{e:Cech Zar} 
\on{Cone}\biggl(\left((j_1)^\IndCoh_*\circ (j_1)^{\IndCoh,*}(\CF)\oplus (j_2)^\IndCoh_*\circ (j_2)^{\IndCoh,*}(\CF)\right)\to \\
\to (j_{12})^\IndCoh_*\circ j_{12}^{\IndCoh,*}(\CF)\biggr)[-1]
\end{multline}
is an isomorphism.

\medskip

Since all functors in question are continuous, it is sufficient to do so for $\CF\in \Coh(S)$. In this case, both 
sides of \eqref{e:Cech Zar} belong to $\IndCoh(S)^+$. So, it is enough to prove that the map in question becomes
an isomorphism after applying the functor $\Psi_S$. However, in this case we are dealing with the map
$$\Psi_S(\CF)\to 
\on{Cone}\biggl(\left((j_1)_*\circ (j_1)^*(\Psi_S(\CF))\oplus (j_2)_*\circ (j_2)^*(\Psi_S(\CF))\right)\to 
(j_{12})_*\circ j_{12}^*(\Psi_S(\CF))\biggr)[-1],$$
which is an isomorphism as it expresses Zariski descent for $\QCoh$. 

\end{proof}

\sssec{}

As a consequence of the above proposition we obtain:

\begin{cor}  \label{c:t structure local}
The t-structure on $\IndCoh(S)$ is Zariski-local. I.e., an object is connective/coconnective
if and only if it is such when restricted to a Zariski open cover.
\end{cor}

\begin{proof}
It is clear that the functor $j^{\IndCoh,*}$ for an open embedding is t-exact. So, if $\CF\in \IndCoh(S)$
is connective/coconnective, then so is $f^{\IndCoh,*}(\CF)$.

\medskip

Vice versa, suppose first that $f^{\IndCoh,*}(\CF)\in \IndCoh(S')$ is connective, and we wish to show
that $\CF$ itself is connective. By \corref{c:equiv on D plus}, it is sufficient to show that $\Psi_S(\CF)$
is connective as an object of $\QCoh(S)$. But the latter follows from the fact that $f^*(\Psi_S(\CF))\simeq 
\Psi_{S'}(f^{\IndCoh,*}(\CF))$ is connective. 

\medskip

Now, let $\CF\in \IndCoh(S)$ be such that $f^{\IndCoh,*}(\CF)\in \IndCoh(S')$ is coconnective. We need to show 
that $\CF$ itself is coconnective, i.e., that for every $\CF'\in \IndCoh(S')^{<0}$, we have $\Maps_{\IndCoh(S)}(\CF',\CF)=0$.
However, this follows from \propref{p:Zariski descent} as all the terms in 
$\Maps_{\IndCoh(S'{}^\bullet/S)}(\CF'|_{S'{}^\bullet/S},\CF|_{S'{}^\bullet/S}[i])$ are zero for $i\leq 0$.
\end{proof}

\ssec{The convergence property of $\IndCoh$}  \label{ss:convergence}

\sssec{}

Let us recall the notion of $n$-coconnective DG scheme (see \cite{Stacks}, Sect. 3.2). 
For a DG scheme $S$ and an integer $n$, let $\tau^{\leq n}(S)$ denote its $n$-coconnective
truncation (see \cite{Stacks}, Sects. 1.1.2, 1.1.3). 

\medskip

Let us denote by $i_n$ the corresponding map
$$\tau^{\leq n}(S)\hookrightarrow S,$$
and for $n_1\leq n_2$, by $i_{n_1,n_2}$ the map
$$\tau^{\leq n_1}(S)\hookrightarrow \tau^{\leq n_2}(S).$$

\begin{rem}
It follows from \cite{Stacks}, Lemma 3.1.5 and Sect. 1.2.6, the map
$$\underset{n}{colim}\, \tau^{\leq n}(S)\to S$$
is an isomorphism, where the colimit is taken in the category $\dgSch$.
\end{rem}

\sssec{}

We have an $\BN$-diagram of categories
$$n\mapsto \IndCoh(\tau^{\leq n}(S)),$$
with the functors 
$$\IndCoh(\tau^{\leq n_1}(S))\to \IndCoh(\tau^{\leq n_2}(S))$$
given by $(i_{n_1,n_2})^{\IndCoh}_*$. 

\medskip

These functors admit right adjoints, given by $(i_{n_1,n_2})^!$, which gives
rise to the corresponding $\BN^{\on{op}}$-diagram of categories. According to
\cite{DG}, Lemma. 1.3.3, we have:
\begin{equation} \label{limit and colimit}
\underset{\BN,(i_{n_1,n_2})^{\IndCoh}_*}{colim}\, \IndCoh(\tau^{\leq n}(S))\simeq 
\underset{\BN,(i_{n_1,n_2})^!}{lim}\, \IndCoh(\tau^{\leq n}(S)).
\end{equation}

\medskip

Moreover, the functors $(i_n)^{\IndCoh}_*$ define a functor
\begin{equation} \label{left convergence}
\underset{\BN,(i_{n_1,n_2})^{\IndCoh}_*}{colim}\, \IndCoh(\tau^{\leq n}(S)) \to \IndCoh(S),
\end{equation}
and the functors $(i_n)^!$ define a functor
\begin{equation} \label{right convergence}
\underset{\BN,(i_{n_1,n_2})^!}{lim}\, \IndCoh(\tau^{\leq n}(S))\leftarrow \IndCoh(S),
\end{equation}
which is the right adjoint of the functor \eqref{left convergence} under the identification
\eqref{limit and colimit}.

\begin{prop}  \label{p:convergence}
The functors \eqref{left convergence} and \eqref{right convergence} are mutually
inverse equivalences.
\end{prop}

\begin{rem}
\propref{p:convergence} can be viewed as an expression of the fact that the study of
the category $\IndCoh$ reduces to the case of eventually coconnective DG schemes.
\end{rem}

\begin{proof}

The essential image of the functor \eqref{left convergence} generates the target
category by \corref{c:reduced generation}. Since it also maps compact objects to 
compact ones, it is sufficient to show that it is fully faithful when restricted to
compact objects.

\medskip

In other words, we have to show that for $\CF',\CF''\in \Coh(\tau^{\leq m}(S))$
for some $m$, the map
\begin{equation} \label{colim of Hom}
\underset{n\geq m}{colim}\, \Hom\left((i_{n,m})_*(\CF'),(i_{n,m})_*(\CF'')\right)\to
\Hom\left((i_{n})_*(\CF'),(i_{n})_*(\CF'')\right)
\end{equation}
is an isomorphism. 

\medskip

We claim that the colimit \eqref{colim of Hom} stabilizes at some finite $n$, and the stable
value maps isomorphically to the right-hand side of \eqref{colim of Hom}. This follows from
\propref{p:Zariski descent} and the next general observation:

\medskip

Let $A$ be a connective DG ring, and let $M'$ and $M''$ be two $A$-modules
belonging to $(A\mod)^b$. In particular, we can view $M'$ and $M''$
as $\tau^{\geq -n}(A)$-modules for all $n$ large enough. Suppose that $M'\in (A\mod)^{\leq m'}$
and that $M''\in (A\mod)^{\geq -m''}$. 

\begin{lem}
Under the above circumstances, the map 
$$\Hom_{\tau^{\geq -n}\!(A)\mod}(M',M'')\to \Hom_{A\mod}(M',M'')$$
is an isomorphism whenever $n>m'+m''$.
\end{lem}

\end{proof}


\ssec{Tensoring up}  









\sssec{}  

Let $f:S_1\to S_2$ be a map of bounded Tor dimension. By \propref{p:* pullback}, $f^{\IndCoh,*}$ induces
a functor
\begin{equation} \label{e:induce IndCoh}
\QCoh(S_1)\underset{\QCoh(S_2)}\otimes \IndCoh(S_2)\to \IndCoh(S_1).
\end{equation}

\begin{prop}    \label{p:tensoring up IndCoh}
The functor in \eqref{e:induce IndCoh} is fully faithful.
\end{prop}

\begin{proof}

We note that the left-hand
side in \eqref{e:induce IndCoh} is compactly generated by objects of the form
$$\CE_1\otimes \CF_2\in \QCoh(S_1)\underset{\QCoh(S_2)}\otimes \IndCoh(S_2),$$
where $\CE_1\in \QCoh(S_1)^{\on{perf}}$ and $\CF_2\in \Coh(S_2)$.
Moreover, the functor $(\on{Id}_{\QCoh(S_1)}\otimes f^{\IndCoh,*})$ sends these
objects to compact objects in $\IndCoh(S_1)$.

\medskip

Hence, it is enough to show that for $\CE'_1,\CE''_1$ and $\CF'_2,\CF''_2$ as above, the map
\begin{multline}\label{e:ff tensor up 1}
\Maps_{\QCoh(S_1)\underset{\QCoh(S_2)}\otimes \IndCoh(S_2)}(\CE'_1\otimes \CF'_2,\CE''_1\otimes \CF''_2)\to \\
\to \Maps_{\IndCoh(S_1)}(\CE'_1\otimes f^{\IndCoh,*}(\CF'_2), \CE''_1\otimes f^{\IndCoh,*}(\CF''_2))
\end{multline}
is an isomorphism, where in the right-hand side $\otimes$ denotes the
action of $\QCoh$ on $\IndCoh$.

\medskip

We can rewrite the map in \eqref{e:ff tensor up 1} as
\begin{multline} \label{e:ff tensor up 2}
\Maps_{\QCoh(S_1)\underset{\QCoh(S_2)}\otimes \IndCoh(S_2)}(\CO_{S_1}\otimes \CF'_2,\CE_1\otimes \CF''_2)\to \\
\to \Maps_{\IndCoh(S_1)}(\CO_{S_1}\otimes f^{\IndCoh,*}(\CF'_2), \CE_1\otimes f^{\IndCoh,*}(\CF''_2)),
\end{multline}
where $\CE_1\simeq \CE''_1\otimes (\CE'_1)^\vee$.

\medskip

We note that the functor right adjoint to
$$\IndCoh(S_2)\simeq \QCoh(S_2)\underset{\QCoh(S_2)}\otimes \IndCoh(S_2)\overset{f^*\otimes \on{Id}_{\IndCoh(S_2)}}
\longrightarrow \QCoh(S_1)\underset{\QCoh(S_2)}\otimes \IndCoh(S_2)$$
is given by
$$\IndCoh(S_2)\simeq \QCoh(S_2)\underset{\QCoh(S_2)}\otimes \IndCoh(S_2)\overset{f_*\otimes \on{Id}_{\IndCoh(S_2)}}
\longleftarrow \QCoh(S_1)\underset{\QCoh(S_2)}\otimes \IndCoh(S_2).$$

Hence, we can rewrite the map in \eqref{e:ff tensor up 2} as the map
$$\Maps_{\IndCoh(S_2)}(\CF'_2,f_*(\CE_1)\otimes \CF''_2) \to 
\Maps_{\IndCoh(S_2)}(\CF'_2,f^\IndCoh_*(\CE_1\otimes f^{\IndCoh,*}(\CF''_2)))$$
coming from \eqref{e:proj formula map}.

\medskip

Hence, the required isomorphism follows from \propref{p:proj form}.

\end{proof}

As a formal corollary we obtain:

\begin{cor}   \label{c:tensoring up}
In the situation of \propref{p:tensoring up IndCoh}, the natural map of endo-functors of $\IndCoh(S_2)$
$$f_*(\CO_{S_1})\otimes - \to f^\IndCoh_*\circ f^{\IndCoh,*}$$
is an isomorphism.
\end{cor}

\sssec{}

The next corollary expresses the category $\IndCoh$ on an open subscheme:

\begin{cor}  \label{c:open tensoring up}
Ler $j:\oS\hookrightarrow S$ be an open embedding. Then the functor
$$\QCoh(\oS)\underset{\QCoh(S)}\otimes \IndCoh(S)\to \IndCoh(\oS)$$
of \eqref{e:induce IndCoh} is an equivalence.
\end{cor}

\begin{proof}
By \propref{p:tensoring up IndCoh}, the functor in question is fully faithful.
Thus, it remains to show that its essential image generates $\IndCoh(\oS)$,
for which it would suffice to show that the essential image of the
composition
$$\IndCoh(S)\to \QCoh(\oS)\underset{\QCoh(S)}\otimes \IndCoh(S)\to \IndCoh(\oS)$$
generates $\IndCoh(\oS)$. However, the latter follows from \lemref{l:localization}.
\end{proof}

\sssec{}

Let now  
$$
\CD
\oS_1 @>{j_1}>>   S_1  \\
@V{\of}VV  @VV{f}V \\
\oS_2 @>{j_2}>>   S_2
\endCD
$$
be a Cartesian diagram of DG schemes, where the maps $j_1$ and $j_2$ are open embeddings and
the map $f$ is proper. 

\medskip

The base change isomorphism
$$\of^\IndCoh_*\circ  j_2^{\IndCoh,*}\simeq j_1^{\IndCoh,*}\circ f^\IndCoh_*$$
of \lemref{l:usual base change}
gives rise to a natural transformation
\begin{equation} \label{e:open proper compatibilty}
j_1^{\IndCoh,*}\circ f^!\to \of^! \circ j_2^{\IndCoh,*}
\end{equation}
of functors $\IndCoh(S_2)\to \IndCoh(\oS_1)$.

\medskip

We claim:
\begin{cor}  \label{c:! pullback and open embed}
The natural transformation \eqref{e:open proper compatibilty} is an isomorphism.
\end{cor}

\begin{rem}
As we shall see in \propref{p:! pullback and event coconn}, the assertion of \corref{c:! pullback and open embed}
holds more generally when the maps $j_i:\oS_i\to S_i$ just have a bounded Tor dimension.
\end{rem}

\begin{rem}
There exists another natural transformation
\begin{equation} \label{e:open proper compatibilty bis}
j_1^{\IndCoh,*}\circ f^!\to \of^! \circ j_2^{\IndCoh,*},
\end{equation}
namely one coming by adjunction from the base change isomorphism
$$f^!\circ (j_2)^\IndCoh_*\simeq (j_1)^\IndCoh_*\circ  \of^!$$
of \propref{p:proper base change}. A diagram chase
that the natural transformations \eqref{e:open proper compatibilty} and
\eqref{e:open proper compatibilty bis} are canonically isomorphic.
\end{rem}

\begin{proof}
Consider the diagram
$$
\CD
\IndCoh(S_1)  @>{\on{Id}_{\IndCoh(S_1)}\otimes j_2^*}>>  \IndCoh(S_1)\underset{\QCoh(S_2)}\otimes \QCoh(\oS_2) 
@>>>  \IndCoh(\oS_1)  \\
@A{f^!}AA    @AA{f^!\otimes \on{Id}_{\QCoh(\oS_2)}}A    @AA{\of^!}A  \\
\IndCoh(S_2)  @>{\on{Id}_{\IndCoh(S_2)}\otimes j_2^*}>>  \IndCoh(S_2)\underset{\QCoh(S_2)}\otimes \QCoh(\oS_2) 
@>>>  \IndCoh(\oS_2),
\endCD
$$
where both right horizontal arrows are those of \eqref{e:induce IndCoh}, and the middle vertical
arrow is well-defined in view of \corref{c:upgrading !}.

\medskip

The left square commutes tautologically. Hence, it remains to show that the right square commutes as well.
As the horizontal arrows are equivalences (by \corref{c:open tensoring up}), it would be sufficient to show
that the corresponding square commutes, when we replace the vertical arrows by their respective left adjoints.

\medskip

The commutativity of
$$
\CD
\IndCoh(S_1)\underset{\QCoh(S_2)}\otimes \QCoh(\oS_2)  @>>>  \IndCoh(\oS_1)  \\
@V{f^\IndCoh_*\otimes \on{Id}_{\QCoh(\oS_2)}}VV    @VV{\of^\IndCoh_*}V   \\
\IndCoh(S_2)\underset{\QCoh(S_2)}\otimes \QCoh(\oS_2) @>>>  \IndCoh(\oS_2)
\endCD
$$
as $\QCoh(\oS_2)$-module categories is equivalent to the commutativity of 
$$
\CD
\IndCoh(S_1)  @>{j_1^{\IndCoh,*}}>>   \IndCoh(\oS_1)  \\
@V{f^\IndCoh_*}VV    @VV{\of^\IndCoh_*}V   \\
\IndCoh(S_2)  @>{j_2^{\IndCoh,*}}>>   \IndCoh(\oS_2)  
\endCD
$$
as $\QCoh(S_2)$-module categories. However, the latter is the original isomorphism
of functors that gives rise to \eqref{e:open proper compatibilty}.
\end{proof}

\ssec{Smooth maps}  \label{sss:smooth maps}

\sssec{}

Let us recall that a morphism $f:S_1\to S_2$ of DG schemes is called \emph{flat} (resp., \emph{smooth}) if:

\begin{itemize}

\item The fiber product DG scheme $^{cl}\!S_2\underset{S_2}\times S_1$ is
classical;

\item The resulting map of classical schemes $^{cl}\!S_2\underset{S_2}\times S_1\to {}^{cl}\!S_2$
is flat (resp., smooth). 

\end{itemize}

\medskip

It is easy to see that a flat morphism is of bounded Tor dimension, and in particular, eventually coconnective. 

\sssec{}

We shall now prove a generalization of \corref{c:open tensoring up}
when instead of an open embedding we have a smooth map between DG schemes.

\begin{prop}  \label{p:smooth map}
Assume that in the situation of \propref{p:tensoring up IndCoh}, the map $f$ is smooth.
Then the functor
$$\QCoh(S_1)\underset{\QCoh(S_2)}\otimes \IndCoh(S_2)\to \IndCoh(S_1)$$
of \eqref{e:induce IndCoh} is an equivalence.
\end{prop}

\begin{proof}

By \propref{p:tensoring up IndCoh}, the functor in question is fully faithful, so we only
need to show that it its essential image generates $\IndCoh(S_1)$.

\medskip

We have a canonical isomorphism
$$\QCoh(S_1)\underset{\QCoh(S_2)}\otimes \IndCoh(({}^{cl}\!S_2)_{red})\simeq
\QCoh(({}^{cl}\!S_1)_{red})\underset{\QCoh(({}^{cl}\!S_2)_{red})}\otimes \IndCoh(({}^{cl}\!S_2)_{red}),$$
and by \corref{c:reduced generation} this reduces the assertion to the case when $S_2$ is classical and
reduced. 

\medskip

By Noetherian induction, we can assume that the assertion holds for all proper
closed subschemes of $S_2$. By \lemref{l:away from generic}, this reduces the assertion to the case
when $S_2=\Spec(K)$, where $K$ is a field. 

\medskip

However, in the latter case, the assertion follows from 
\lemref{l:smoothness}.

\end{proof}

\ssec{Behavior with respect to products}

We shall now establish one more property of $\IndCoh$. In this 
we will be assuming that all our schemes are almost of finite type
over the ground field $k$, i.e., $\dgSch_{\on{aft}}\subset \dgSch_{\on{Noeth}}$. 

\sssec{}

Thus, let $S_1$ and $S_2$ be two DG schemes almost of finite type over $k$.  Consider their
product $S_1\times S_2$, which also has the same property.

\medskip

External tensor product defines a functor
$$\QCoh(S_1)\otimes \QCoh(S_2)\to \QCoh(S_1\times S_2),$$
such that if $\CF_i\in \Coh(S_i)$, we have $\CF_1\boxtimes \CF_2\in \Coh(S_1\times S_2)$. 

\medskip

Consider the resulting functor
$$\IndCoh(S_1)\otimes \IndCoh(S_2)\overset{\Psi_{S_1}\otimes \Psi_{S_2}}\longrightarrow
\QCoh(S_1)\otimes \QCoh(S_2)\to \QCoh(S_1\times S_2).$$

By the above, it sends compact objects in $\IndCoh(S_1)\otimes \IndCoh(S_2)$ to
$\Coh(S_1\times S_2)$. Hence, we obtain a functor
\begin{equation}  \label{product map}
\IndCoh(S_1)\otimes \IndCoh(S_2)\to \IndCoh(S_1\times S_2),
\end{equation}
which makes the diagram
\begin{equation} \label{e:Psi and external}
\CD
\IndCoh(S_1)\otimes \IndCoh(S_2)   @>>>  \IndCoh(S_1\times S_2) \\
@V{\Psi_{S_1}\otimes \Psi_{S_2}}VV     @VV{\Psi_{S_1\times S_2}}V  \\
\QCoh(S_1)\otimes \QCoh(S_2)   @>>>  \QCoh(S_1\times S_2)
\endCD
\end{equation}
commute.

\begin{prop}  \label{p:products}
The functor \eqref{product map} is an equivalence.
\end{prop}

\begin{proof}

Both categories are compactly generated, and the functor in question sends
compact objects to compacts, by construction.

\medskip

Hence, to prove that it is fully faithful, it
is sufficient to show that for
$$\CF'_1,\CF''_1\in \Coh(S_1),\,\, \CF'_2,\CF''_2\in \Coh(S_2),$$
the map
$$\CMaps_{\Coh(S_1)}(\CF'_1,\CF''_1)\otimes \CMaps_{\Coh(S_2)}(\CF'_2,\CF''_2)\to
\CMaps_{\Coh(S_1\times S_2)}(\CF'_1\boxtimes \CF'_2,\CF''_1\boxtimes \CF''_2)$$
is an isomorphism (see \secref{sss:DG categories} for the notation $\CMaps(-,-)$).

\medskip

I.e., we have to show that for $\CF'_i$, $\CF''_i$ as above, the map
\begin{equation} \label{e:ten prod 0}
\CMaps_{\QCoh(S_1)}(\CF'_1,\CF''_1)\otimes \CMaps_{\QCoh(S_2)}(\CF'_2,\CF''_2)\to
\CMaps_{\QCoh(S_1\times S_2)}(\CF'_1\boxtimes \CF'_2,\CF''_1\boxtimes \CF''_2)
\end{equation}
is an isomorphism. 

\medskip

Recall that if $\bC_1,\bC_2$ are DG categories and $\bc'_i,\bc''_i\in \bC_i$, $i=1,2$
are such that $\bc'_1$ and $\bc'_2$ are compact, then the natural map
\begin{equation} \label{e:ten prod abs}
\CMaps_{\bC_1}(\bc'_1,\bc''_1)\otimes  \CMaps_{\bC_2}(\bc'_2,\bc''_2)\to
\CMaps_{\bC_1\otimes \bC_2}(\bc'_1\otimes \bc'_2,\bc''_1\otimes \bc''_2)
\end{equation}
is an isomorphism.

\medskip

The isomorphism\eqref{e:ten prod 0} is not immediate since the objects $\CF'_i\in \QCoh(S_i)$
are not compact. To circumvent this, we proceed as follows.

\medskip

It is enough to show that
\begin{multline*}
\tau^{\leq -n}\left(\CMaps_{\QCoh(S_1)}(\CF'_1,\CF''_1)\otimes \CMaps_{\QCoh(S_2)}(\CF'_2,\CF''_2)\right)\to \\
\to \tau^{\leq -n}\left(\CMaps_{\QCoh(S_1\times S_2)}(\CF'_1\boxtimes \CF'_2,\CF''_1\boxtimes \CF''_2)\right)
\end{multline*}
is an isomorphism for any $n$.

\medskip

Choose $\alpha_1:\wt\CF'_1\to \CF'_1$ (resp., $\alpha_2:\wt\CF'_2\to \CF'_2$) with $\wt\CF'_1$ (resp., $\wt\CF'_2$)
in $\QCoh(S_1)^{\on{perf}}$ (resp., $\QCoh(S_2)^{\on{perf}}$), such that
$$\on{Cone}(\alpha_1)\in \QCoh(S_1)^{\leq -N} \text{ and } \on{Cone}(\alpha_2)\in \QCoh(S_2)^{\leq -N}$$
for $N\gg 0$. 

\medskip

We choosing $N$ large enough, we can ensure that 
$$\tau^{\leq -m}\left(\CMaps_{\QCoh(S_i)}(\CF'_i,\CF''_i)\right)\to 
\tau^{\leq -m}\left(\CMaps_{\QCoh(S_i)}(\wt\CF'_i,\CF''_i)\right)$$
is an isomorphism for $m$ large enough, which in turn implies that
\begin{multline*}
\tau^{\leq -n}\left(\CMaps_{\QCoh(S_1)}(\CF'_1,\CF''_1)\otimes \CMaps_{\QCoh(S_2)}(\CF'_2,\CF''_2)\right)\to \\
\to \tau^{\leq -n}\left(\CMaps_{\QCoh(S_1)}(\wt\CF'_1,\CF''_1)\otimes \CMaps_{\QCoh(S_2)}(\wt\CF'_2,\CF''_2)\right)
\end{multline*}
and
$$\tau^{\leq -n}\left(\CMaps_{\QCoh(S_1\times S_2)}(\CF'_1\boxtimes \CF'_2,\CF''_1\boxtimes \CF''_2)\right)\to
\tau^{\leq -n}\left(\CMaps_{\QCoh(S_1\times S_2)}(\wt\CF'_1\boxtimes \wt\CF'_2,\CF''_1\boxtimes \CF''_2)\right)$$
are isomorphisms.

\medskip

Hence, it is enough to show that 
$$\CMaps_{\QCoh(S_1)}(\wt\CF'_1,\CF''_1)\otimes \CMaps_{\QCoh(S_2)}(\wt\CF'_2,\CF''_2)\to
\CMaps_{\QCoh(S_1\times S_2)}(\wt\CF'_1\boxtimes \wt\CF'_2,\CF''_1\boxtimes \CF''_2)$$
is an isomorphism. But this follows from the isomorphism \eqref{e:ten prod abs} and the 
fact that the fuctor
$$\QCoh(S_1)\boxtimes \QCoh(S_2)\to \QCoh(S_1\times S_2)$$
is an equivalence, by \cite{QCoh}, Prop. 1.4.4.

\medskip

Thus, it remains to show that the essential image of \eqref{product map} generates
the target category. It is sufficient to show that the essential image of
\begin{multline*}
\IndCoh(({}^{cl}\!S_1)_{red})\otimes \IndCoh(({}^{cl}\!S_2)_{red})\to 
\IndCoh(({}^{cl}\!S_1)_{red}\times ({}^{cl}\!S_1)_{red})\to \\
\to \IndCoh(({}^{cl}\!(S_1\times S_2)_{red})\to
\IndCoh(S_1\times S_2)
\end{multline*}
generates $\IndCoh(S_1\times S_2)$. Hence, by \corref{c:reduced generation}, we can assume that
$S_1$ and $S_2$ are both classical and reduced. Moreover, by Noetherian induction, we can 
assume that the statement holds for all proper closed subschemes of $S_1$ and $S_2$. So,
it is sufficient to show that the functor 
$$\IndCoh(\oS_1)\otimes \IndCoh(\oS_2)\to \IndCoh(\oS_1\times \oS_2)$$
is an equivalence for some non-empty open subschemes $\oS_i\subset S_i$.

\medskip

Now, the $\on{char}(k)=0$ assumption implies that $S_1$ and $S_2$ are generically smooth
over $k$. I.e., we obtain that it is sufficient to show that the functor \eqref{product map} is
an equivalence when $S_i$ are smooth. In this case $S_1\times S_2$ is also smooth,
and in particular regular. Thus, all vertical maps in the diagram \eqref{e:Psi and external}
are equivalences. However, the bottom arrow in \eqref{e:Psi and external} is an equivalence
as well, which implies the required assertion. 

\end{proof}

\bigskip

\centerline{\bf Part II. Correspondences, !-pullback and duality.}

\bigskip

\section{The !-pullback functor for morphisms almost of finite type}  \label{s:!}

All DG schemes in this section will be Noetherian. 
The goal of this section is to extend the functor $f^!$ to maps between DG schemes
that are not necessarily proper. 



\ssec{The paradigm of correspondences}   \label{ss:corr}

\sssec{}  \label{sss:two classes}

Let $\bC$ be a $(\infty,1)$-category. Let $vert$ and $horiz$ be two classes of $1$-morphisms in $\bC$,
such that:

\begin{enumerate}

\item Both classes contain all isomorphisms.

\smallskip

\item If a given $1$-morphism belongs to a given class, then so do all isomorphic $1$-morphisms.

\smallskip

\item Both classes are stable under compositions.

\smallskip

\item Given a pair of morphisms $f:\bc_1\to\bc_2$ with $f\in vert$ and 
$g_2:\bc'_2\to \bc_2$ with $g_2\in horiz$, the Cartesian square square
\begin{equation} \label{e:all side Cart}
\CD
\bc'_1   @>{g_1}>>  \bc_1  \\
@V{f'}VV   @VV{f}V  \\
\bc'_2   @>{g_2}>>  \bc_2
\endCD
\end{equation}
exists, and $f'\in vert$ and $g_1\in horiz$.  

\end{enumerate}

\medskip

We let $\bC_{vert}\subset \bC$ and
$\bC_{horiz}\subset \bC$ denote the corresponding 1-full subcategories of $\bC$.

\sssec{}   

We let $\bC_{\on{corr}:vert;horiz}$ denote a new $(\infty,1)$-category whose objects are the same
as objects of $\bC$. For $\bc_1,\bc_2\in \bC_{\on{corr}}$, the groupoid of $1$-morphisms
$\Hom_{\bC_{\on{corr}:vert;horiz}}(\bc_1,\bc_2)$ is the $\infty$-groupoid of diagrams
\begin{equation} \label{e:basic correspondence}
\CD
\bc_{1,2}  @>{g}>> \bc_1  \\
@VV{f}V  \\
\bc_2  
\endCD
\end{equation}
where  $f\in \bC_{vert}$ and $g\in \bC_{horiz}$ and where compositions are
given by taking Cartesian squares. 

\medskip

The higher categorical structure on $\bC_{\on{corr}:vert;horiz}$ is described in terms of the
corresponding \emph{complete Segal space}, i.e., an object of $\inftygroup^{\bDelta^{\on{op}}}$, 
see \secref{sss:Segal review} for a brief review.

\medskip

Namely, the $\infty$-groupoid of $n$-fold compositions is that of diagrams

\begin{equation} \label{e:grids}
\CD
\bc_{0,n} @>>>  \bc_{0,n-1}  @>>>   ...   @>>>  \bc_{0,1}  @>>>  \bc_{0}  \\
@VVV   @VVV & & @VVV   \\
\bc_{1,n} @>>>  \bc_{2,n-1}  @>>>   ...   @>>>  \bc_{1}   \\
@VVV   @VVV  @VVV \\
... @>>> ... @>>> ...  \\
@VVV   @VVV \\
\bc_{n-1,n} @>>>  \bc_{n-1}   \\
@VVV    \\
\bc_{n},
\endCD
\end{equation}
where we require that all vertical (resp., horizontal) maps belong to $vert$ (resp., horiz),
and each square be Cartesian. 

\medskip

Note that $\bC_{\on{corr}:vert;horiz}$ contains $\bC_{vert}$
and $(\bC_{horiz})^{\on{op}}$ as 1-full subcategories, obtained by restricting $1$-morphisms
to those diagrams \eqref{e:basic correspondence}, for which $g$ (resp., $f$) is an isomorphism.

\sssec{}

Our input will be a functor
$$P_{\on{corr}:vert;horiz}:\bC_{\on{corr}:vert;horiz}\to \StinftyCat_{\on{cont}}.$$
We shall denote its value on objects simply by $P(\bc)$. 

\medskip

Consider the restriction of $P_{\on{corr}:vert;horiz}$ to $\bC_{vert}$, which we denote by 
$P_{vert}$. For $$(f:\bc_1\to \bc_2)\in \bC_{vert}$$ we let $P_{vert}(f)$ denote the corresponding
$1$-morphism $P(\bc_1)\to P(\bc_2)$ in $\StinftyCat_{\on{cont}}$. 

\medskip

Consider also the restriction of $P_{\on{corr}:vert;horiz}$ to $(\bC_{horiz})^{\on{op}}$, which we denote by 
$P^!_{horiz}$. For $$(g:\bc_1\to \bc_2)\in \bC_{horiz}$$
we let $P^!_{horiz}(g)$ denote the corresponding $1$-morphism $P(\bc_2)\to P(\bc_1)$ in $\StinftyCat_{\on{cont}}$.

\begin{rem}  \label{r:abstract base change}
Note that a datum of a functor $P_{\on{corr}:vert;horiz}$ can be thought of as assignment to every
Cartesian square as in \eqref{e:all side Cart}, with
$f,f'\in \bC_{vert}$ and $g_1,g_2\in \bC_{horiz}$
an isomorphism of functors $P(\bc_1)\rightrightarrows P(\bc'_2)$:
\begin{equation} \label{e:basic commutation}
P_{vert}(f')\circ P^!_{horiz}(g_1)\simeq P^!_{horiz}(g_2)\circ P_{vert}(f).
\end{equation}
Formulating it as a functor $P_{\on{corr}:vert;horiz}$ out of the category of correspondences
is a way to formulate the compatibilities
that the isomorphisms \eqref{e:basic commutation} need to satisfy. The author learned this idea from
J.~Lurie.
\end{rem}

\ssec{A provisional formulation}

\sssec{}

In the above framework we take $\bC:=\dgSch_{\on{Noeth}}$, and $vert$ to be all $1$-morphisms.
We take $horiz$ to be morphisms almost of finite type. 

\medskip

Let $(\dgSch_{\on{Noeth}})_{\on{corr:all;aft}}$ denote the corresponding 
category of correspondences. 

\begin{thm}   \label{t:upper shriek}
There exists a canonically defined functor
$$\IndCoh_{(\dgSch_{\on{Noeth}})_{\on{corr:all;aft}}}: (\dgSch_{\on{Noeth}})_{\on{corr:all;aft}}\to \StinftyCat_{\on{cont}}$$
with the following properties:

\smallskip

\noindent{\em(a)}
The restriction $\IndCoh_{(\dgSch_{\on{Noeth}})_{\on{corr:all;aft}}}|_{((\dgSch_{\on{Noeth}})_{\on{proper}})^{\on{op}}}$
is canonically isomorphic to the functor $\IndCoh^!_{(\dgSch_{\on{Noeth}})^{\on{op}}}$
of \corref{c:upper ! DG funct proper}. 

\smallskip

\noindent{\em(b)}
The restriction $\IndCoh_{(\dgSch_{\on{Noeth}})_{\on{corr:all;aft}}}|_{((\dgSch_{\on{Noeth}})_{\on{open}})^{\on{op}}}$
is canonically isomorphic to the functor 
$$\IndCoh^*_{(\dgSch_{\on{Noeth}})_{\on{open}}}:=
\IndCoh^*_{(\dgSch_{\on{Noeth}})_{\on{ev-coconn}}}|_{((\dgSch_{\on{Noeth}})_{\on{open}})^{\on{op}}}$$
given by \corref{c:upper * DG funct}.

\smallskip

\noindent{\em(c)}
The restriction $\IndCoh_{(\dgSch_{\on{Noeth}})_{\on{corr:all;aft}}}|_{\dgSch_{\on{Noeth}}}$
is canonically isomorphic to the functor $\IndCoh_{\dgSch_{\on{Noeth}}}$ of \propref{p:upgrading IndCoh to functor}. 

\end{thm}

\sssec{}

As one of the main corollaries of \thmref{t:upper shriek}, we obatin:

\begin{cor} \label{c:existence of !}
There exists a well-defined functor 
$$\IndCoh^!_{(\dgSch_{\on{Noeth}})_{\on{aft}}}:((\dgSch_{\on{Noeth}})_{\on{aft}})^{\on{op}}\to \StinftyCat_{\on{cont}},$$
such that

\smallskip

\noindent{\em(a)}
The restriction $\IndCoh_{(\dgSch_{\on{Noeth}})_{\on{aft}}}|_{((\dgSch_{\on{Noeth}})_{\on{proper}})^{\on{op}}}$
is canonically isomorphic to the functor $\IndCoh^!_{(\dgSch_{\on{Noeth}})_{\on{proper}}}$
of \corref{c:upper ! DG funct proper}. 

\smallskip

\noindent{\em(b)}
The restriction $\IndCoh_{(\dgSch_{\on{Noeth}})_{\on{aft}}}|_{((\dgSch_{\on{Noeth}})_{\on{open}})^{\on{op}}}$
is canonically isomorphic to the restriction of the functor $\IndCoh^*_{(\dgSch_{\on{Noeth}})_{\on{open}}}$
given by \corref{c:upper * DG funct}.

\end{cor}

\sssec{}

A crucial property of the functor $\IndCoh^!_{(\dgSch_{\on{Noeth}})_{\on{aft}}}$ is the following base change property,
see Remark \ref{r:abstract base change}. Let 
\begin{equation} \label{e:basic Cart diag}
\CD
S'_1  @>{g_1}>>  S_1  \\
@V{f'}VV    @VV{f}V   \\
S'_2  @>{g_2}>>  S_2
\endCD
\end{equation}
be a Cartesian diagram in $\dgSch_{\on{Noeth}}$, where the horizontal arrows are almost of finite type. 
Then there exists a canonical isomorphism of functors
\begin{equation} \label{e:basic base change}
g_2^!\circ f^\IndCoh_*\simeq (f')^\IndCoh_*\circ g_1^!,
\end{equation}
which we shall refer to as \emph{base change}. 

\medskip

Note that, unlike the case when the morphisms $g_i$ are proper or open embeddings, there is a priori
no map in \eqref{e:basic base change} in either direction. 

\medskip

One can say that the real content of \thmref{t:upper shriek} is the well-definedness of the functor of
the !-pullback such that the isomorphism \eqref{e:basic base change} holds.

\ssec{Compatibility with adjunction for open embeddings}

\sssec{}

Let $j:\oS\hookrightarrow S$ be an open embedding. Note that the square 
$$
\CD
\oS  @>{\on{id}}>>  \oS \\
@V{\on{id}}VV   @VV{j}V  \\
\oS  @>{j}>>  S
\endCD
$$
is Cartesian.

\medskip

Hence, the base change isomorphism \eqref{e:basic base change} defines an isomorphism of
functors
$$j^!\circ j_*^{\IndCoh}\simeq \on{Id}_{\IndCoh(\oS)},$$
and in particular, a map in one direction
\begin{equation} \label{e:base change for open from corr}
j^!\circ j_*^{\IndCoh}\to \on{Id}_{\IndCoh(\oS)}.
\end{equation}

\medskip

It will follow from the construction of the functor $\IndCoh_{(\dgSch_{\on{Noeth}})_{\on{corr:all;aft}}}$
(specifically, from \thmref{t:extension by adjoints}) that the map \eqref{e:base change for open from corr}
is canonically isomorphic to the counit of the adjunction for $(j^!,j^\IndCoh_*)$, where we are using the
identification
$$j^!\simeq j^{\IndCoh,*}$$
of \corref{c:existence of !}(b). 

\medskip

We can interpret this is saying that the datum of the functor $\IndCoh_{(\dgSch_{\on{Noeth}})_{\on{corr:all;aft}}}$,
encodes the datum for the $(j^!,j^\IndCoh_*)$-adjunction for an open embedding $j$. 

\sssec{}

From here we are going to deduce:

\begin{prop}  \label{p:open compat}  Consider a Cartesian diagram \eqref{e:basic Cart diag}
and the corresponding isomorphism of functors 
\begin{equation} \label{e:basic base change open}
g_2^!\circ f^\IndCoh_*\simeq (f')^\IndCoh_*\circ g_1^!
\end{equation}
of \eqref{e:basic base change}.

\smallskip

\noindent{\em(a)} Suppose that the morphism $f$ (and hence $f'$) is an open embedding. Then the 
natural transformation $\to$ in \eqref{e:basic base change open} comes via the adjunctions of $(f^!,f^\IndCoh_*)$
and $((f')^!,(f')^\IndCoh_*)$ from the isomorphism
$$(f')^!\circ g_2^!\simeq g_1^!\circ f^!.$$

\smallskip

\noindent{\em(b)} Suppose that the morphism $g_2$ (and hence $g_1$) is an open embedding. Then the 
natural transformation $\to$ in \eqref{e:basic base change open} comes via the adjunctions of $(g_1^!,(g_1)^\IndCoh_*)$
and $(g_2^!,(g_2)^\IndCoh_*)$ from the isomorphism
$$f^\IndCoh_*\circ (g_1)^\IndCoh_*\simeq (g_2)^\IndCoh_*\circ (f')^\IndCoh_*.$$

\end{prop}

\begin{proof}

We will prove point (a); point (b) is proved similarly. Applying the $(f^!,f^\IndCoh_*)$-adjunction, the assertion is
equivalent to the fact that the two natural transformations 
\begin{equation} \label{e:two nat trans open}
(f')^!\circ g_2^!\circ f^\IndCoh_*\rightrightarrows g_1^!
\end{equation}
are canonically isomorphic. The first natural transformation is 
$$(f')^!\circ g_2^!\circ f^\IndCoh_*\simeq g_1^!\circ f^!\circ f^\IndCoh_*\overset{\on{id}\circ \on{counit}}\longrightarrow g_1^!,$$
and the second natural transformation is 
$$(f')^!\circ g_2^!\circ f^\IndCoh_*\overset{\text{\eqref{e:basic base change}}}\simeq (f')^!\circ (f')^\IndCoh_*\circ g_1^!
\overset{\on{counit}\circ \on{id}}\longrightarrow g_1^!.$$

\medskip

Taking into account that the counit maps
$$f^!\circ f^\IndCoh_*\to \on{Id} \text{ and } (f')^!\circ (f')^\IndCoh_*\to \on{Id}$$
are given by base change for the squares
$$
\CD
S_1 @>{\on{id}}>>  S_1  \\
@V{\on{id}}VV   @VV{f}V   \\
S_1 @>{f}>>  S_2
\endCD
$$
and
$$
\CD
S'_1 @>{\on{id}}>>  S'_1  \\
@V{\on{id}}VV   @VV{f'}V   \\
S'_1 @>{f'}>>  S'_2,
\endCD
$$
respectively, we obtain that the two natural transformations in question come as base change from the Cartesian square
$$
\CD
S'_1   @>{g_1}>>      S_1 \\
@V{\on{id}}VV   @VVV  \\
S'_1 @>{g_2\circ f'=f\circ g_1}>>   S_2,
\endCD
$$
for the first one factoring it as
$$
\CD
S'_1 @>{g_1}>>   S_1   @>{\on{id}}>>  S_1  \\
@V{\on{id}}VV   @V{\on{id}}VV   @VV{f}V   \\
S'_1 @>{g_1}>>  S_1  @>{f}>> S_2
\endCD
$$
and for the second one factoring it as
$$
\CD
S'_1  @>{\on{id}}>>  S'_1  @>{g_1}>>  S_1  \\
@V{\on{id}}VV    @V{f'}VV  @VV{f}V \\
S'_1  @>{f'}>>  S'_2  @>{g_2}>>  S_2.
\endCD
$$

\end{proof}

\ssec{Compatibility with adjunction for proper maps}

\sssec{}

Let $f:S_1\to S_2$ be a proper map. By contrast with the case of open embeddings, the datum
of the functor $\IndCoh_{(\dgSch_{\on{Noeth}})_{\on{corr:all;aft}}}$ does not contain the datum 
for either the unit 
$$\on{Id}_{\IndCoh(S_1)}\to f^!\circ f^\IndCoh_*$$
or the counit
$$f^\IndCoh\circ f^!\to \on{Id}_{\IndCoh(S_2)}$$
for the $(f^\IndCoh_*,f^!)$-adjunction.  \footnote{We remark that it does encode this adjunction in the
world of D-modules, see \cite{FG}, Sect. 1.4.5.} The situation can be remedied by considering the
2-category of correspondences, see \secref{sss:two cat form} below. 

\medskip

The following property of the base change isomorphisms \eqref{e:basic base change} follows
from the construction of the functor $\IndCoh_{(\dgSch_{\on{Noeth}})_{\on{corr:all;aft}}}$:

\begin{prop}  \label{p:proper compat}  Consider a Cartesian diagram \eqref{e:basic Cart diag}
and the corresponding isomorphism of functors 
\begin{equation} \label{e:basic base change proper}
(f')^\IndCoh_*\circ g_1^!\simeq g_2^!\circ f^\IndCoh_*
\end{equation}
of \eqref{e:basic base change}.

\smallskip

\noindent{\em(a)} Suppose that the morphism $f$ (and hence also $f'$) is proper, and
$g_2$ (and hence $g_1$) almost of finite type. Then the map $\to$
in the isomorphism \eqref{e:basic base change proper} comes via the adjunctions of $(f^{\IndCoh}_*,f^!)$
and $((f')^{\IndCoh}_*,(f')^!)$ from the isomorphism
$$g_1^!\circ f^!\simeq (f')^!\circ g_2^!.$$

\smallskip

\noindent{\em(b)} Suppose that the morphism $g_2$ (and hence also $g_1$) is proper.
Then the map $\to$ in the isomorphism \eqref{e:basic base change proper} comes via the adjunctions of
$((g_1)^{\IndCoh}_*,g_1^!)$ and $((g_2)^{\IndCoh}_*,g_2^!)$ from the isomorphism
$$(g_2)^\IndCoh_*\circ (f')^\IndCoh_*\simeq f^\IndCoh_*\circ (g_1)^\IndCoh_*.$$

\end{prop}

\sssec{A 2-categorical formulation}  \label{sss:two cat form}

As was explained to us by J.~Lurie, the most natural formulation of the existence of the !-pullback
functors that would encode the adjunction for proper maps used the language of $(\infty,2)$-categories.  

\medskip

Namely, according to Lurie, one needs to consider an $(\infty,2)$-category 
$(\dgSch_{\on{Noeth}})_{\on{corr:all;aft}}^{\on{2-Cat}}$, whose objects and
$1$-morphisms are the same as in $(\dgSch_{\on{Noeth}})_{\on{corr:all;aft}}$,
however, we allow non-invertible $2$-morphisms which are commutative diagrams
\begin{gather} \label{e:2categ diag}
\xy
(-15,10)*+{S_2}="A";
(15,10)*+{S_1}="B";
(0,15)*+{S}="C";
(0,30)*+{S'}="D";
{\ar@{->}^{f} "C";"A"};
{\ar@{->}_{f'} "D";"A"};
{\ar@{->}_{g} "C";"B"};
{\ar@{->}^{g'} "D";"B"};
{\ar@{->}^{h} "D";"C"};
\endxy
\end{gather}
where the map $h$ is required to be a proper. 

\medskip

The higher categorical structure on $\IndCoh_{(\dgSch_{\on{Noeth}})_{\on{corr:all;aft}}}^{\on{2-Cat}}$
is described via the corresponding higher Segal space. Namely, the corresponding $(\infty,1)$-category
of $n$-fold compositions is has as objects diagrams as in \eqref{e:grids}, where $1$-morphisms
between such diagrams are maps that induce proper maps on veritices. 

\medskip

In this framework, one will consider $\IndCoh_{(\dgSch_{\on{Noeth}})_{\on{corr:all;aft}}}^{\on{2-Cat}}$ as a functor
$$(\dgSch_{\on{Noeth}})_{\on{corr:all;aft}}^{\on{2-Cat}}\to \StinftyCat^{\on{2-Cat}}_{\on{cont}},$$ where 
$\StinftyCat{\on{2-Cat}}_{\on{cont}}$ is the 2-category of cocomplete DG categories and continuous functors,
where we now allow non-invertible natural transformations as 2-morphisms. 

\medskip

The construction of this functor will be carried out in \cite{GR3}. 

\sssec{}

Let explain how the datum of the functor $\IndCoh_{(\dgSch_{\on{Noeth}})_{\on{corr:all;aft}}}^{\on{2-Cat}}$
encodes the adjunction for $(f_*^{\IndCoh},f^!)$ for a proper map $f:S_1\to S_2$.

\medskip

Consider the diagram 
\begin{gather}   \label{e:diagram for adj}
\xy
(-25,0)*+{S_1}="A";
(25,0)*+{S_1,}="B";
(0,15)*+{S_1\underset{S_2}\times S_1}="C";
(0,40)*+{S_1}="D";
{\ar@{->}^{\tilde{f}^l} "C";"A"};
{\ar@{->}_{\on{id}_{S_1}} "D";"A"};
{\ar@{->}_{\tilde{f}^r} "C";"B"};
{\ar@{->}^{\on{id}_{S_1}} "D";"B"};
{\ar@{->}_{h} "D";"C"};
\endxy
\end{gather}
where $h=\Delta_{S_1/S_2}$.

\medskip

Then the resulting $2$-morphism 
$$\on{Id}_{\IndCoh(S)}\simeq (\on{id}_{S_1})_*^{\IndCoh}\circ \on{id}_{S_1}^!\to 
(\tilde{f}^l)^{\IndCoh}_*\circ (\tilde{f}^r)^!\simeq f^!\circ f^{\IndCoh}_*$$
is the unit of the $(f_*^{\IndCoh},f^!)$-adjunction.

\sssec{}

As an illustration, assuming the above description of the unit of the adjunction for $(f_*^{\IndCoh},f^!)$, let us deduce
the assertion of \propref{p:proper compat}(b). 

\medskip

As in the proof of \propref{p:open compat}, it suffices to show that
the two natural transformations
\begin{equation} \label{e:two nat trans prop}
g_1^!\rightrightarrows (f')^!\circ g_2^!\circ f^\IndCoh_*
\end{equation}
are canonically isomorphic, the first one being
$$g_1^!\overset{\on{id}\circ\on{unit}} \longrightarrow g_1^!\circ f^!\circ f^\IndCoh_*\simeq (f')^!\circ g_2^!\circ f^\IndCoh_*,$$
and the second one
$$g_1^!\overset{\on{unit}\circ \on{id}} \longrightarrow (f')^!\circ (f')^\IndCoh_*\circ g_1^!\overset{\text{\eqref{e:basic base change}}}
\longrightarrow (f')^!\circ g_2^!\circ f^\IndCoh_*.$$

Recall that the above unit maps
$$\on{Id}\to f^!\circ f^\IndCoh_* \text{ and } \on{Id}\to (f')^!\circ (f')^\IndCoh_*$$
are given by the $2$-morphisms coming from the diagrams
\begin{gather}   
\xy
(-25,0)*+{S_1}="A";
(25,0)*+{S_1,}="B";
(0,15)*+{S_1\underset{S_2}\times S_1}="C";
(0,40)*+{S_1}="D";
{\ar@{->}^{\tilde{f}^l} "C";"A"};
{\ar@{->}_{\on{id}_{S_1}} "D";"A"};
{\ar@{->}_{\tilde{f}^r} "C";"B"};
{\ar@{->}^{\on{id}_{S_1}} "D";"B"};
{\ar@{->}_{h} "D";"C"};
\endxy
\end{gather}
and
\begin{gather}  
\xy
(-25,0)*+{S'_1}="A";
(25,0)*+{S'_1,}="B";
(0,15)*+{S'_1\underset{S'_2}\times S'_1}="C";
(0,40)*+{S'_1}="D";
{\ar@{->}^{\tilde{f}'{}^l} "C";"A"};
{\ar@{->}_{\on{id}_{S'_1}} "D";"A"};
{\ar@{->}_{\tilde{f}'{}^r} "C";"B"};
{\ar@{->}^{\on{id}_{S'_1}} "D";"B"};
{\ar@{->}_{h'} "D";"C"};
\endxy
\end{gather}
respectively.

\medskip

Hence, both maps in \eqref{e:two nat trans prop} are given by the diagram
\begin{gather}   
\xy
(-40,0)*+{S'_1}="A";
(40,0)*+{S_1,}="B";
(0,15)*+{S'_1\underset{S_2}\times S_1=S'_1\underset{S'_2}\times S'_1}="C";
(0,40)*+{S'_1}="D";
{\ar@{->} "C";"A"};
{\ar@{->} "D";"A"};
{\ar@{->} "C";"B"};
{\ar@{->} "D";"B"};
{\ar@{->} "D";"C"};
\endxy
\end{gather}
one time presenting it as a composition 
\begin{gather}   
\xy
(25,0)*+{S_1}="A";
(75,0)*+{S_1,}="B";
(50,15)*+{S_1\underset{S_2}\times S_1}="C";
(50,40)*+{S_1}="D";
(-25,0)*+{S'_1}="E";
(0,15)*+{S'_1}="F";
(0,40)*+{S'_1}="G";
{\ar@{->} "C";"A"};
{\ar@{->} "D";"A"};
{\ar@{->} "C";"B"};
{\ar@{->} "D";"B"};
{\ar@{->} "D";"C"};
{\ar@{->} "F";"E"};
{\ar@{->} "G";"E"};
{\ar@{->} "F";"A"};
{\ar@{->} "G";"A"};
{\ar@{->} "F";"E"};
{\ar@{->} "G";"F"};
\endxy
\end{gather}
and another time as a composition
\begin{gather}  
\xy
(-25,0)*+{S'_1}="A";
(25,0)*+{S'_1}="B";
(0,15)*+{S'_1\underset{S'_2}\times S'_1}="C";
(0,40)*+{S'_1}="D";
(75,0)*+{S_1,}="E";
(50,15)*+{S'_1}="F";
(50,40)*+{S'_1}="G";
{\ar@{->} "C";"A"};
{\ar@{->} "D";"A"};
{\ar@{->} "C";"B"};
{\ar@{->} "D";"B"};
{\ar@{->} "D";"C"};
{\ar@{->} "F";"B"};
{\ar@{->} "G";"B"};
{\ar@{->} "F";"E"};
{\ar@{->} "G";"E"};
{\ar@{->} "G";"F"};
\endxy
\end{gather}

\ssec{Compatibility with the action of $\QCoh$}  \label{ss:compat with action of QCoh}

\sssec{}

Recall (see \secref{ss:action}) that for an individual DG scheme $S$, the category $\IndCoh(S)$
is naturally a module category for $\QCoh(S)$.

\medskip

Recall also that for a map $f:S_1\to S_2$, which is proper (resp., open), the functor
$$f^!:\IndCoh(S_2)\to \IndCoh(S_1)$$
has a natural structure of 1-morphism of $\QCoh(S_2)$-module categories.

\medskip

Indeed, for $f$ proper, this structure is given by \corref{c:upgrading !}. For $f$
an open embedding, we have $f^!=f^{\IndCoh,*}$ and this structure is given by
\propref{p:* pullback}.

\sssec{}  \label{sss:SymMonMod}

Consider now the $\infty$-category $\SymMonModStinftyCat_{\on{cont}}$, whose objects
are pairs
$$(\bO,\bC),$$
where $\bO\in \SymMonStinftyCat$ and $\bC$ is an $\bO$-module category.

\medskip

Morphisms in $\SymMonModStinftyCat_{\on{cont}}$ between $(\bO_1,\bC_1)$
and $(\bO_2,\bC_2)$ are pairs $(F_\bO,F_\bC)$, where 
$F_\bO:\bO_1\to \bO_2$ is a morphism in $\SymMonStinftyCat_{\on{cont}}$, and 
$$F_\bC:\bC_1\to \bC_2$$
is a morhism of $\bO_1$-module categories.

\medskip

The higher categorical structure on $\SymMonModStinftyCat_{\on{cont}}$ is defined naturally.

\sssec{}

The assignments
$$S\rightsquigarrow (\QCoh(S),\IndCoh(S)),\,\, (f:S_1\to S_2)\rightsquigarrow (f^*,f^!)$$
naturally upgrade to functors
$$(\QCoh^*,\IndCoh^!)_{(\dgSch_{\on{Noeth}})_{\on{proper}}}:
((\dgSch_{\on{Noeth}})_{\on{proper}})^{\on{op}}\to \SymMonModStinftyCat_{\on{cont}},$$
and 
$$(\QCoh^*,\IndCoh^!)_{(\dgSch_{\on{Noeth}})_{\on{open}}}:
((\dgSch_{\on{Noeth}})_{\on{open}})^{\on{op}}\to \SymMonModStinftyCat_{\on{cont}},$$
respectively.

\sssec{}

The following ehnancement of \corref{c:existence of !} will be proved in \cite{GR3}:

\begin{thm} \label{t:action of QCoh}
There exists a canonical upgrading of the functor 
$$\IndCoh^!_{\dgSch_{\on{Noeth}}}:((\dgSch_{\on{Noeth}})_{\on{aft}})^{\on{op}}\to \StinftyCat_{\on{cont}}$$
to a functor 
$$(\QCoh^*,\IndCoh^!)_{(\dgSch_{\on{Noeth}})_{\on{aft}}}:((\dgSch_{\on{Noeth}})_{\on{aft}})^{\on{op}}\to
\SymMonModStinftyCat_{\on{cont}},$$
whose restrictions to 
$$((\dgSch_{\on{Noeth}})_{\on{proper}})^{\on{op}} \text{ and }
((\dgSch_{\on{Noeth}})_{\on{open}})^{\on{op}}$$
identify with
$$(\QCoh^*,\IndCoh^!)_{(\dgSch_{\on{Noeth}})_{\on{proper}}} \text{ and }
(\QCoh^*,\IndCoh^!)_{(\dgSch_{\on{Noeth}})_{\on{open}}},$$
respectively.
\end{thm}

A plausibility check for this theorem follows the outline of the proof of \thmref{t:upper shriek} given in the 
next section.

\sssec{}

Let us rewrite the statement of \thmref{t:action of QCoh} in concrete terms for an individual morphism
$f:S_1\to S_2$:

\medskip

It says that that for $\CE\in \QCoh(S_2)$ and $\CF\in \IndCoh(S_2)$, there exists a canonical
isomorphism
\begin{equation} \label{e:QCoh linearity for !}
f^!(\CE\otimes \CF)\simeq f^*(\CE)\otimes f^!(\CF).
\end{equation}

\ssec{The multiplicative structure}

In this subsection we will specialize to the full subcategory
$$\dgSch_{\on{aft}}\subset \dgSch_{\on{Noeth}}$$
of DG schemes almost of finite type over $k$. 

\sssec{}  \label{sss:IndCoh on aft}

We consider the full subcategory 
$$(\dgSch_{\on{aft}})_{\on{corr:all;all}}\subset (\dgSch_{\on{Noeth}})_{\on{corr:all;aft}},$$
obtained by taking as objects DG schemes that are almost of finite type. 

\medskip

Let us denote by $\IndCoh_{(\dgSch_{\on{aft}})_{\on{corr:all;all}}}$ the restriction of the functor 
$$\IndCoh_{(\dgSch_{\on{Noeth}})_{\on{corr:all;aft}}}:(\dgSch_{\on{Noeth}})_{\on{corr:all;aft}}\to \StinftyCat_{\on{cont}}$$ to 
the subcategory $(\dgSch_{\on{aft}})_{\on{corr:all;all}}$. 

\medskip

We let 
$$\IndCoh_{\dgSch_{\on{aft}}}:\dgSch_{\on{aft}}\to \StinftyCat_{\on{cont}}$$
denote the restriction of the functor $\IndCoh_{(\dgSch_{\on{aft}})_{\on{corr:all;all}}}$
to the 1-full subcategory 
$$\dgSch_{\on{aft}}\subset (\dgSch_{\on{aft}})_{\on{corr:all;all}}.$$

\medskip

Similarly, we let
$$\IndCoh^!_{\dgSch_{\on{aft}}}:(\dgSch_{\on{aft}})^{\on{op}}\to \StinftyCat_{\on{cont}}$$
denote the restriction of the functor $\IndCoh_{(\dgSch_{\on{aft}})_{\on{corr:all;all}}}$
to the 1-full subcategory 
$$(\dgSch_{\on{aft}})^{\on{op}}\subset (\dgSch_{\on{aft}})_{\on{corr:all;all}}.$$

\sssec{}

Let us observe that the categories $\dgSch_{\on{aft}}$ and $(\dgSch_{\on{aft}})_{\on{corr:all;all}}$ possess
natural symmetric monoidal structures given by Cartesian product over $\on{pt}:=\Spec(k)$.

\medskip

The category $\StinftyCat_{\on{cont}}$ also possesses a natural symmetric monoidal structure given by tensor
product. We state the following result without a proof, as it is obtained by retracing the argument
of \thmref{t:upper shriek} (the full proof will be given in \cite{GR3}):

\begin{thm} \label{t:upper shriek monoidal}
The functor 
$$\IndCoh_{(\dgSch_{\on{aft}})_{\on{corr:all;all}}}:(\dgSch_{\on{aft}})_{\on{corr:all;all}}\to \StinftyCat_{\on{cont}}$$ has a natural
right-lax symmetric monoidal structure.
\end{thm}

Combined with \propref{p:products}, we obtain:
\begin{cor} \label{c:upper shriek monoidal}
The right-lax symmetric monoidal structure on $\IndCoh_{(\dgSch_{\on{aft}})_{\on{corr:all;all}}}$ is strict,
i.e., is symmetric monoidal. 
\end{cor}

\sssec{}

Restricting the functor $\IndCoh_{(\dgSch_{\on{aft}})_{\on{corr:all;all}}}$ to
$$(\dgSch_{\on{aft}})^{\on{op}}\to \StinftyCat_{\on{cont}},$$
from \thmref{t:upper shriek monoidal} and \corref{c:upper shriek monoidal}, we obtain:

\begin{cor} \label{c:! monoidal}
The functor 
$$\IndCoh^!_{\dgSch_{\on{aft}}}:(\dgSch_{\on{aft}})^{\on{op}}\to \StinftyCat_{\on{cont}}$$
has a natural symmetric monoidal structure.
\end{cor}

Note that at the level of objects, the statement of \corref{c:! monoidal} coincides with that of
\propref{p:products}. 

\medskip

At the level of 1-morphisms, it says that for two pairs of objects of
$\dgSch_{\on{aft}}$:
$$(f_1:S_1\to S'_1) \text{ and } (f_2:S_2\to S'_2),$$
the diagram
$$
\CD
\IndCoh(S_1)\otimes \IndCoh(S_2)  @>{\boxtimes}>> \IndCoh(S_1\times S_2)  \\
@A{f_1^!\otimes f_2^!}AA   @AA{(f_1\times f_2)^!}A    \\
\IndCoh(S'_1)\otimes \IndCoh(S'_2)  @>{\boxtimes}>> \IndCoh(S'_1\times S'_2) 
\endCD
$$
canonically commutes.

\sssec{}   \label{sss:sotimes}

Note now that for any $S\in \dgSch_{\on{aft}}$, the diagonal morphism on $S$ defines on it a structure
of commutative coalgebra in $\dgSch_{\on{aft}}$. Hence, from \corref{c:! monoidal} we obtain:

\begin{cor}  \label{c:sotimes}
For $S\in \dgSch_{\on{aft}}$, the category $\IndCoh(S)$ has a natural symmetric monoidal structure.
Furthermore, the assignement
$$S\rightsquigarrow (\IndCoh(S),\sotimes)$$
naturally extends to a functor
$$(\dgSch_{\on{aft}})^{\on{op}}\to \SymMonStinftyCat_{\on{cont}}.$$
\end{cor}

Concretely, the monoidal operation on $\IndCoh(S)$ is the functor
$$\IndCoh(S)\otimes \IndCoh(S)\overset{\boxtimes}\longrightarrow
\IndCoh(S\times S)\overset{\Delta_S^!}\longrightarrow \IndCoh(S).$$

We shall use the notation
$$\CF_1,\CF_2\in \IndCoh(S)\, \mapsto \, \CF_1\sotimes \CF_2\in \IndCoh(S).$$

\medskip

The unit in this symmetric monoidal category is $\omega_S$, defined as 
$$\omega_S:=(p_S)^!(k)\in \IndCoh(S).$$

We shall refer to $\omega_S$ as the ``dualizing sheaf of $S$."

\ssec{Compatibility between the multiplicative structure and the action of $\QCoh$}

We are now going to discuss a common refinement of \corref{c:! monoidal} and \thmref{t:action of QCoh}:

\sssec{}  \label{sss:action and com}

Note that the category $\SymMonModStinftyCat_{\on{cont}}$ also has a natural symmetric monoidal
structure, given by
$$(\bO_1,\bC_1)\otimes (\bO_2,\bC_2):=(\bO_1\otimes \bO_2,\bC_1\otimes \bC_2).$$

We have:

\begin{thm} \label{t:action of QCoh and mult}
The symmetric monoidal structures on the functors
$$\IndCoh^!_{\dgSch_{\on{aft}}}:(\dgSch_{\on{aft}})^{\on{op}}\to \StinftyCat_{\on{cont}}$$
and
$$\QCoh^*_{\dgSch_{\on{aft}}}:(\dgSch_{\on{aft}})^{\on{op}}\to \SymMonStinftyCat_{\on{cont}}$$
naturally combine to a symmetric monoidal structure on the functor
$$(\QCoh^*,\IndCoh^!)_{\dgSch_{\on{aft}}}:(\dgSch_{\on{aft}})^{\on{op}}\to
\SymMonModStinftyCat_{\on{cont}}.$$
\end{thm}

\sssec{}

As a corollary, we obtain:

\begin{cor} \label{c:sotimes otimes}
For $S\in \dgSch_{\on{aft}}$, the symmetric monoidal structure on $\IndCoh(S)$ 
has a natural $\QCoh(S)$-linear structure. Furthermore, the assignment
$$S\rightsquigarrow (\QCoh(S),\otimes)\to (\IndCoh(S),\sotimes)$$
naturally extends to a functor
$$(\dgSch_{\on{aft}})^{\on{op}}\to \on{Funct}([1],\SymMonStinftyCat_{\on{cont}}),$$
where $[1]$ is the category $0\to 1$. 
\end{cor}

In the above corollary, the symmetric monoidal functor
\begin{equation} \label{e:QCoh to IndCoh}
\QCoh(S)\to \IndCoh(S),
\end{equation} 
is given by the action on the unit, i.e., 
$$\CE\mapsto \CE\otimes \omega_S,$$
where the action is understood in the sense of \secref{ss:action}. We shall denote the functor by the symbol $\Upsilon_S$.

\medskip

In concrete terms, the structure of $\QCoh(S)$-linearity on the symmetric monoidal category
$\IndCoh(S)$ means that for $\CF_1,\CF_2\in \IndCoh(S)$ and $\CE\in \QCoh(S)$, we
have canonical isomorphisms
\begin{equation} \label{e:QCoh linearity on sotimes}
\CE\otimes (\CF_1\sotimes \CF_2)\simeq (\CE\otimes \CF_1)\sotimes \CF_2\simeq 
\CF_1\sotimes (\CE\otimes \CF_2).
\end{equation}

\sssec{}  \label{sss:intr Upsilon}

Thus, we can consider a natural transformation

\begin{equation} \label{e:Upsilon}
\Upsilon_{\dgSch_{\on{aft}}}:\QCoh^*_{\dgSch_{\on{aft}}}\to \IndCoh^!_{\dgSch_{\on{aft}}},
\end{equation}
where $\QCoh^*_{\dgSch_{\on{aft}}}$ and $\IndCoh^!_{\dgSch_{\on{aft}}}$ are both
considered as functors
$$(\dgSch_{\on{aft}})^{\on{op}}\rightrightarrows \SymMonStinftyCat_{\on{cont}}.$$

\medskip

The functoriality statement of \corref{c:sotimes otimes} says that at the level of 1-morphisms,
for $f:S_1\to S_2$, we have a commutative diagram of symmetric monoidal categories. 
\begin{equation} \label{e:Upsilon funct}
\CD
\IndCoh(S_1)  @<{\Upsilon_{S_1}}<< \QCoh(S_1) \\
@A{f^!}AA  @AA{f^*}A  \\
\IndCoh(S_2)  @<{\Upsilon_{S_2}}<< \QCoh(S_2).
\endCD
\end{equation}

\sssec{}

Note that if we regard $\QCoh^*_{\dgSch_{\on{aft}}}$ and $\IndCoh^!_{\dgSch_{\on{aft}}}$ just as functors
$$(\dgSch_{\on{aft}})^{\on{op}}\rightrightarrows \StinftyCat_{\on{cont}},$$
then the structure of natural transfomration on $\Upsilon_{\dgSch_{\on{aft}}}$ is given by 
\thmref{t:action of QCoh}, i.e., we do not need to consider the finer structure given by
\thmref{t:action of QCoh and mult}.

\section{Proof of \thmref{t:upper shriek}}  \label{s:construction of !}

\ssec{Framework for the construction of the functor}

\sssec{}   \label{sss:by adjunction}

We first discuss the most basic example of a construction of a functor out of a category of correspondences. 

\medskip

Let $(\bC,vert,horiz)$ be as in \secref{sss:two classes}, and suppose that $horiz\subset vert$.

\medskip

Let $P_{vert}:\bC_{vert}\to \StinftyCat_{\on{cont}}$ be a functor. Assume that for every $(g:\bc_1\to \bc_2)\in \bC_{horiz}$,
the functor $P_{vert}(g):P(\bc_1)\to P(\bc_2)$ admits a continuous right (resp., left) adjoint; we denote
it $P^!_{horiz}(g)$. 

\medskip

The passage to adjoints defines a functor
$$P^!_{horiz}:(\bC_{horiz})^{\on{op}}\to \StinftyCat_{\on{cont}}.$$ 

\medskip

Consider a Cartesian square \eqref{e:all side Cart}. By adjunction, we obtain 
a map
\begin{equation} \label{e:abstract base change one}
P_{vert}(f')\circ P^!_{horiz}(g_1)\to P^!_{horiz}(g_2)\circ P_{vert}(f)
\end{equation}
(in the case of right adjoints), and
\begin{equation} \label{e:abstract base change two}
P^!_{horiz}(g_2)\circ P_{vert}(f) \to P_{vert}(f')\circ P^!_{horiz}(g_1)
\end{equation}
(in the case of left adjoints).

\medskip

We shall say that $P_{vert}$ satisfies the \emph{left base change} (resp., \emph{right base change})
condition with respect to $horiz$ if the corresponds adjoints $P^!_{horiz}(g)$, $g\in {horiz}$ exist,
and map \eqref{e:abstract base change one} (resp., \eqref{e:abstract base change two})
is an isomorphism for any Cartesian diagram as above.

\medskip

\begin{thm}  \label{t:extension by adjoints}  
Suppose that $P_{vert}$ satisfies the left (resp., right) base change condition with respect to
$horiz$. Then there exists
a canonically defined functor 
$$P_{\on{corr}:vert;horiz}:\bC_{\on{corr}:vert;horiz}\to \StinftyCat_{\on{cont}},$$
equipped with isomorphisms $$P_{vert}\simeq P_{\on{corr}:vert;horiz}|_{\bC_{vert}} \text{ and }
P^!_{horiz}\simeq P_{\on{corr}:vert;horiz}|_{(\bC_{horiz})^{\on{op}}}.$$ 
\end{thm}

The proof of this theorem when $\bC$ is an ordinary category is easy. The higher-categorical
version will appear in \cite{GR3}. 

\sssec{}  \label{sss:letter conditions}

Let $\bC,vert,horiz$ be as in \secref{sss:two classes}. Let ${adm}$ be a subclass of $vert$,
which satisfies the following conditions:

\medskip

\noindent(A) ${adm}$ satisfies conditions (1)-(3) of \secref{sss:two classes}.

\medskip

\noindent(B) The pairs of classes $(vert,{adm})$, $({adm},horiz)$ 
satisfy condition (4) of \secref{sss:two classes}. (Note that since 
${adm}\subset vert$, the pair $({adm},{adm})$ also satisfies 
condition (4).) 

\medskip

\noindent(C) If in a diagram \eqref{e:all side Cart}, we take $g_2=f$ and it belongs
to both $horiz$ and ${adm}$, then the arrows $f'$ and $g_1$ are isomorphisms. 

\medskip

\noindent(D) For $horiz$ and ${adm}$, if $h=h_1\circ h_2$ and $h$ and $h_1$ belong to
the given class, then so does $h_2$.

%

\sssec{}  \label{sss:roman numeral conditions}

Consider the category $\bC_{\on{corr}:vert;horiz}$, and let 
$P_{\on{corr}:vert;horiz}$ be a functor $$P_{\on{corr}:vert;horiz}:\bC_{\on{corr}:vert;horiz}\to \StinftyCat_{\on{cont}}.$$

\medskip

We will impose the following conditions:

\medskip

\noindent(I) The functor $P_{vert}:\bC_{vert}\to\StinftyCat_{\on{cont}}$ satisfies the left base change condition
with respect to adm. 

\medskip

\noindent(II) For a Cartesian square as in \eqref{e:all side Cart}, with
$f,f'\in {adm}$ and $g_1,g_2\in horiz$ the morphism between the resulting two functors
$P(\bc_2)\rightrightarrows P(\bc'_1)$
$$P^!_{horiz}(g_1)\circ P^!_{adm}(f)\to P^!_{adm}(f')\circ P^!_{horiz}(g_2)$$ that comes by the
$(P_{vert}(-),P^!_{adm}(-))$-adjunction from
the isomorphism \eqref{e:basic commutation}, is an isomorphism.

\sssec{}  \label{sss:former III}

For a morphism $f:\bc_1\to \bc_2$ with $f\in adm\cap horiz$ consider the diagram
$$
\CD
\bc_1  @>{\on{id}}>>  \bc_1 \\
@V{\on{id}}VV  @VV{f}V     \\
\bc_1  @>{f}>>  \bc_2,
\endCD
$$
which is Cartesian due to Condition (C). 

\medskip

Note that from Condition (II) we obtain a canonical isomorphism
$$P^!_{horiz}(f)\simeq P^!_{adm}(f).$$



\sssec{}  \label{sss:small numeral conditions}

We let ${horiz_{\on{new}}}$ be yet another class of $1$-morphisms in $\bC$. We impose the following two conditions on ${horiz_{\on{new}}}$: 

\medskip

\noindent(i) ${horiz_{\on{new}}}$ satisfies conditions (1)-(3) of \secref{sss:two classes}. 

\medskip

\noindent(ii) ${horiz_{\on{new}}}$ contains both $horiz$ and ${adm}$.

\medskip

\noindent(iii) Every morphism $h\in {horiz_{\on{new}}}$ can be factored
as $h=f\circ g$ with $g\in horiz$ and $f\in {adm}$. 

\begin{rem}
It follows formally that ${horiz_{\on{new}}}$ is precisely the class of $1$-morphisms that can be factored
as in condition (iii). Thus, the actual condition is that this class of $1$-morphisms is stable
under compositions.
\end{rem}

\sssec{}   \label{sss:star condition}

Finally, we impose the following crucial condition on the classes $horiz$ and ${adm}$.
Let us call it condition $(\star)$:

\medskip

For given $1$-morphism $(h:\bc_1\to \bc_2)\in {horiz_{\on{new}}}$ consider the category ${\mathsf{Factor}}(h)$, 
whose objects are factorizations
of $h$ into a composition
$$\bc_1\overset{g}\longrightarrow \bc_{3/2}\overset{f}\longrightarrow \bc_2$$
with $g\in horiz$ and $f\in {adm}$. Morphisms in this category are commutative diagrams
\begin{gather}   \label{e:new diamond}
\xy
(-10,0)*+{\bc_1}="A";
(10,10)*+{\bc'_{3/2}}="B";
(10,-10)*+{\bc''_{3/2}}="C";
(30,0)*+{\bc_2.}="D";
{\ar@{->}^{g'} "A";"B"};
{\ar@{->}_{g''} "A";"C"};
{\ar@{->}^{e} "B";"C"};
{\ar@{->}^{f'} "B";"D"};
{\ar@{->}_{f''} "C";"D"};
\endxy
\end{gather}
(Note that by condition (D), the arrow $e:\bc'_{3/2}\to \bc''_{3/2}$ is also in ${adm}$.)

\medskip

Condition $(\star)$ reads as follows: the above category ${\mathsf{Factor}}(h)$ is contractible.

\sssec{}

Consider the category $\bC_{\on{corr}:vert;{horiz_{\on{new}}}}$. 
Note that $\bC_{\on{corr}:vert;{horiz_{\on{new}}}}$ contains as 1-full subcategories both
$\bC_{\on{corr}:vert;{adm}}$ and $\bC_{\on{corr}:vert;horiz}$. 

\medskip

We have the following assertion, generalizing \thmref{t:extension by adjoints}

\begin{thm}  \label{t:construction}
There exists a canonically define functor
$$P_{\on{corr},vert,{horiz_{\on{new}}}}:\bC_{\on{corr}:vert;{horiz_{\on{new}}}}\to \StinftyCat_{\on{cont}},$$
equipped with identifications
$$P_{\on{corr},vert,{horiz_{\on{new}}}}|_{\bC_{\on{corr}:vert;horiz}}\simeq
P_{\on{corr},vert,{horiz}}$$
and
$$
P_{\on{corr},vert,{horiz_{\on{new}}}}|_{\bC_{\on{corr}:vert;{adm}}}\simeq
P_{\on{corr},vert,{adm}},$$
which is compatible with the further restriction to
$$\bC_{\on{corr}:vert;{adm}}\leftarrow \bC_{vert}\to \bC_{\on{corr}:vert;horiz}.$$
\end{thm}

\sssec{Sketch of proof of \thmref{t:construction}}

In this section we will indicate the proof of \thmref{t:construction}, modulo homotopy-theoretic issues
(i.e., a proof that works when $\bC$ is an ordinary category). The full proof will appear in \cite{GR3}. 

\medskip

First, we are going to construct the functor $P_{horiz_{\on{new}}}:(\bC_{horiz_{\on{new}}})^{\on{op}}\to \StinftyCat_{\on{cont}}$. 
Let $h:\bc_1\to \bc_2$ be a $1$-morphism in ${horiz_{\on{new}}}$, and let us factor it as a composition 
$f\circ g$ as in Condition (iii). First, we need to show that 
the functor $$P^!_{horiz}(g)\circ P^!_{adm}(f):P(\bc_2)\to P(\bc_1)$$ is canonically independent of the factorization.
Since by condition $(\star)$ the category of factorizations is contractible, it suffices to show that for
any $1$-morphism between factorizations give as in diagram \eqref{e:new diamond}, the resulting two
functors
$$P^!_{horiz}(g')\circ P^!_{adm}(f') \text{ and } P^!_{horiz}(g'')\circ P^!_{adm}(f'')$$
are canonically isomorphic.

\medskip

This easily reduces to the case of a diagram 
\begin{gather*} 
\xy
(-10,0)*+{\bc}="A";
(10,10)*+{\bc_1}="B";
(10,-10)*+{\bc_2}="C";
{\ar@{->}^{g_1} "A";"B"};
{\ar@{->}_{g_2} "A";"C"};
{\ar@{->}^{f} "B";"C"};
\endxy
\end{gather*}
with $g',g''\in horiz$ and $f\in {adm}$, and we need to establish an isomorphism 
$$P^!_{horiz}(g_2)\simeq P^!_{horiz}(g_1)\circ P^!_{adm}(f)$$
as functors $P(\bc_2)\rightrightarrows P(\bc)$.

\medskip

We will show more generally that given a commutative (but not necessarily Cartesian) square
$$
\CD
\bc'_1  @>{g_1}>>  \bc_1 \\
@V{\tilde{f}}VV    @VV{f}V    \\
\bc'_2  @>{g_2}>>  \bc_2
\endCD
$$
with $f,\tilde{f}\in {adm}$ and $g_1,g_2\in horiz$, we have a canonical isomorphism
\begin{equation} \label{e:for compos}
P^!_{horiz}(g_1)\circ P^!_{adm}(f)\simeq P^!_{adm}(\tilde{f})\circ P^!_{horiz}(g_2)
\end{equation}
as functors $P(\bc_2)\rightrightarrows P(\bc'_1)$.

\medskip

Consider the diagram
$$
\CD
\bc'_1   @>{h}>>  \bc'_2\underset{\bc_2}\times \bc_1   @>{g'_2}>>   \bc_1  \\
& &  @V{f'}VV    @VV{f}V   \\ 
& & \bc'_2 @>{g_2}>>  \bc_2
\endCD
$$
where $g'_2\circ h\simeq g_1$ and $f'\circ h\simeq \tilde{f}$. 

\medskip

By Condition (D)
we obtain that 
$$h\in {adm}\cap horiz.$$ 
Therefore, we have:
\begin{multline*} 
P^!_{horiz}(g_1)\circ P^!_{adm}(f)\simeq P^!_{horiz}(h)\circ P^!_{horiz}(g'_2)\circ P^!_{adm}(f)\overset{\text{Condition II}}\simeq \\
\simeq P^!_{horiz}(h)\circ P^!_{adm}(f')\circ P^!_{horiz}(g_2) \overset{\text{\secref{sss:former III}}}\simeq \\
\simeq  P^!_{adm}(h)\circ P^!_{adm}(f')\circ P^!_{horiz}(g_2)\simeq P^!_{adm}(\tilde{f})\circ P^!_{horiz}(g_2).
\end{multline*}

\medskip

Isomorphism \eqref{e:for compos} allows one to define the functor $P_{horiz_{\on{new}}}$
on compositions of morphisms. 

\medskip

The base change data, needed to extend the functors
$P_{horiz_{\on{new}}}$ and $P_{vert}$ to a functor 
$$P_{\on{corr},vert,{horiz_{\on{new}}}}:\bC_{\on{corr}:vert;{horiz_{\on{new}}}}\to \StinftyCat_{\on{cont}},$$
follows from the corresponding data of $P_{\on{corr}:vert;horiz}$ and $P_{\on{corr}:vert;{adm}}$
by construction.

\ssec{Construction of the functor $\IndCoh_{(\dgSch_{\on{Noeth}})_{\on{corr:all;aft}}}$}

\sssec{Step 1}

We start with $\bC:=\dgSch_{\on{Noeth}}$, and we take $vert$ to be the class
of all morphisms, and $horiz$ to be the class of open embeddings.

\medskip

Consider the functor $P_{vert}:=\IndCoh_{\dgSch_{\on{Noeth}}}$ of \propref{p:upgrading IndCoh to functor}.
It satisfies the right base change condition with respect to the class of open embeddings by 
\propref{p:upgrading IndCoh to functor}.

\medskip

Applying \thmref{t:extension by adjoints}, we obtain a functor 
$$\IndCoh_{(\dgSch_{\on{Noeth}})_{\on{corr}:\on{all};\on{open}}}:(\dgSch_{\on{Noeth}})_{\on{corr}:\on{all};\on{open}}\to 
\StinftyCat_{\on{cont}}.$$

\sssec{Step 2}

We take $\bC:=\dgSch_{\on{Noeth}}$, $vert$ to be the class of all morphisms,
$horiz$ to be the class of open embeddings, and ${adm}$ to be the class
of proper morphisms. It is easy to see that conditions (A)-(D) of \secref{sss:letter conditions}
are satisfied.

\medskip

We consider the functor 
$$\IndCoh_{(\dgSch_{\on{Noeth}})_{\on{corr}:\on{all};\on{open}}}:
(\dgSch_{\on{Noeth}})_{\on{corr}:\on{all};\on{open}}\to \StinftyCat_{\on{cont}}$$
constructed in Step 1. 

\medskip

We claim that it satisfies Conditions (I) and (II) of \secref{sss:roman numeral conditions}.
Indeed, Condition (I) is given by \propref{p:proper base change}, and Condition (II)
is given by \corref{c:! pullback and open embed}.  

\medskip

We take ${horiz_{\on{new}}}$ to be the class of separated morphisms almost of finite type. We claim
that it satisfies Conditions (i), (ii) and (iii) of \secref{sss:small numeral conditions} and condition $(\star)$
of \secref{sss:star condition}.

\medskip

Conditions (i) and (ii) are evident. Condition $(\star)$, which contains Condition (iii) as a particular
case, will be verified in \secref{ss:check factorization}. 

\medskip

Applying \thmref{t:construction}, we obtain a functor
$$\IndCoh_{(\dgSch_{\on{Noeth}})_{\on{corr}:\on{all};\on{aft-sep}}}:
(\dgSch_{\on{Noeth}})_{\on{corr}:\on{all};\on{aft-sep}}\to \StinftyCat_{\on{cont}},$$
where
$$\on{aft-sep}\subset \on{aft}$$
denotes the class of separated morphisms almost of finite type. 

\medskip

It is clear from the construction that the functor $\IndCoh_{(\dgSch_{\on{Noeth}})_{\on{corr}:\on{all};\on{aft-sep}}}$
satisfies the Conditions (a), (b) and (c) of \thmref{t:upper shriek}.

\sssec{Interlude}  \label{sss:RKE corr}

In order to extend the functor 
$$\IndCoh_{(\dgSch_{\on{Noeth}})_{\on{corr}:\on{all};\on{aft-sep}}}:
(\dgSch_{\on{Noeth}})_{\on{corr}:\on{all};\on{aft-sep}}\to \StinftyCat_{\on{cont}},$$
to a functor
$$\IndCoh_{(\dgSch_{\on{Noeth}})_{\on{corr}:\on{all};\on{aft}}}:
(\dgSch_{\on{Noeth}})_{\on{corr}:\on{all};\on{aft}}\to \StinftyCat_{\on{cont}}$$
we will use the following construction.

\medskip

Let 
$$(\bC^1,vert^1,horiz^1) \text{ and } (\bC^2,vert^2,horiz^2)$$
be a pair of categories and classes of morphisms as in \secref{sss:two classes}.

\medskip

Let $\Phi:\bC^1\to \bC^2$ be a functor. Assume that $\Phi$ sends morphisms from $vert^1$ (resp., $horiz^1$) to
morphisms from $vert^2$ (resp., $horiz^2$). In particular, $\Phi$ induces functors
$$\Phi_{vert}:\bC^1_{vert}\to \bC^2_{vert}  \text{ and } \Phi_{horiz}:\bC^1_{horiz}\to \bC^2_{horiz},$$
and a functor
$$\Phi_{\on{corr}:vert;horiz}:\bC^1_{\on{corr}:vert;horiz}\to \bC^2_{\on{corr}:vert;horiz}.$$

\medskip

Let now $$P_{\on{corr}:vert;horiz}:\bC^1_{\on{corr}:vert;horiz}\to \bD$$
be a functor, where $\bD$ is an $(\infty,1)$-category that contains limits. 
Consider the functors
$$\on{RKE}_{\Phi_{\on{corr}:vert;horiz}}(P_{\on{corr}:vert;horiz}):\bC^2_{\on{corr}:vert;horiz}\to \bD$$
and
$$\on{RKE}_{(\Phi_{horiz})^{\on{op}}}(P^!_{horiz}):
(\bC^2_{horiz})^{\on{op}}\to \bD.$$

\medskip

We claim:

\begin{prop} \label{p:RKE corr}
Assume that for any $\bc^1\in \bC^1$, the functor $\Phi$ induces an equivalence
$$(\bC^1_{vert})_{/\bc^1}\to (\bC^2_{vert})_{/\Phi(\bc^1)}.$$
Then the natural map
$$\on{RKE}_{\Phi_{\on{corr}:vert;horiz}}(P_{\on{corr}:vert;horiz})|_{(\bC^2_{horiz})^{\on{op}}}\to
\on{RKE}_{(\Phi_{horiz})^{\on{op}}}(P^!_{horiz})$$
is an equivalence.
\end{prop}

\begin{proof}

It is enough to show that the map in question induces an isomorphism at the level
of objects. 

\medskip

The value of $\on{RKE}_{\Phi_{\on{corr}:vert;horiz}}(P_{\on{corr}:vert;horiz})$ on $\bc^2\in \bC^2$ is 
$$lim\, P(\bc^1),$$
where the limit is taken over the category of diagrams
$$
\CD
\bc'{}^2  @>{g}>>  \bc_2  \\
@V{f^2}VV  \\
\Phi(\bc^1),
\endCD
$$
with $f^2\in vert^2$ and $g\in horiz^2$.

\medskip

The condition of the proposition implies that this category is equivalent to that of diagrams
$$
\CD
\Phi(\bc'{}^1)  @>{g}>>  \bc_2  \\
@V{\Phi(f^1)}VV  \\
\Phi(\bc^1),
\endCD
$$
with $f^1\in vert^1$ and $g\in horiz^2$. 

\medskip

However, it is clear that cofinal in the (opposite of the above) category is the
full subcategory consisting of diagrams with $f^1$ being an isomorphism. The latter category 
is the same as
$$(\bC^1\underset{\bC^2}\times (\bC^2_{horiz})_{/\bc^2})^{\on{op}}.$$

I.e., the value of $\on{RKE}_{\Phi_{\on{corr}:vert;horiz}}(P_{\on{corr}:vert;horiz})$ on $\bc^2$ maps isomorphically
to 
$$\underset{\bc^1\in (\bC^1\underset{\bC^2}\times (\bC^2_{horiz})_{/\bc^2})^{\on{op}}}{lim}\, P^!(\bc^1),$$
while the latter limit computes the value of $\on{RKE}_{(\Phi_{horiz})^{\on{op}}}(P^!_{horiz})$ on $\bc^2$.

\end{proof}

\sssec{Step 3}   

In the set-up of \secref{sss:RKE corr} we take $\bC^1=\bC^2=\dgSch_{\on{Noeth}}$,
$vert^1=vert^2=\on{all}$ and
$$horiz^1=\on{aft-sep} \text{ and } horiz^2=\on{aft}.$$

\medskip

We define the functor 
$$\IndCoh_{(\dgSch_{\on{Noeth}})_{\on{corr}:\on{all};\on{aft}}}:
(\dgSch_{\on{Noeth}})_{\on{corr}:\on{all};\on{aft}}\to \StinftyCat_{\on{cont}}$$
as the right Kan extension of 
$$\IndCoh_{(\dgSch_{\on{Noeth}})_{\on{corr}:\on{all};\on{aft-sep}}}:
(\dgSch_{\on{Noeth}})_{\on{corr}:\on{all};\on{aft-sep}}\to \StinftyCat_{\on{cont}}$$
along the tautological functor
$$(\dgSch_{\on{Noeth}})_{\on{corr}:\on{all};\on{aft-sep}}\to (\dgSch_{\on{Noeth}})_{\on{corr}:\on{all};\on{aft}}.$$

\medskip

We claim that the resulting functor $\IndCoh_{(\dgSch_{\on{Noeth}})_{\on{corr}:\on{all};\on{aft}}}$
satisfies the conditions of \thmref{t:upper shriek}. By Step 2, it remains to show that for
$S\in \dgSch_{\on{Noeth}}$, the natural map
$$\IndCoh(S)\to \IndCoh_{(\dgSch_{\on{Noeth}})_{\on{corr}:\on{all};\on{aft}}}(S)$$
is an isomorphism in $\StinftyCat_{\on{cont}}$.

\medskip

Denote 
$$\IndCoh^!_{(\dgSch_{\on{Noeth}})_{\on{aft-sep}}}:=\IndCoh_{(\dgSch_{\on{Noeth}})_{\on{corr}:\on{all};\on{aft-sep}}}
|_{((\dgSch_{\on{Noeth}})_{\on{corr}:\on{all};\on{aft-sep}})^{\on{op}}}$$
and
$$\IndCoh^!_{(\dgSch_{\on{Noeth}})_{\on{aft}}}:=\IndCoh_{(\dgSch_{\on{Noeth}})_{\on{corr}:\on{all};\on{aft}}}
|_{((\dgSch_{\on{Noeth}})_{\on{corr}:\on{all};\on{aft}})^{\on{op}}}.$$

\medskip

Note that by \propref{p:RKE corr}, we have
$$\IndCoh^!_{(\dgSch_{\on{Noeth}})_{\on{aft}}}\simeq 
\on{RKE}_{((\dgSch_{\on{Noeth}})_{\on{aft-sep}})^{\on{op}}\to ((\dgSch_{\on{Noeth}})_{\on{aft}})^{\on{op}}}
(\IndCoh^!_{(\dgSch_{\on{Noeth}})_{\on{aft-sep}}}).$$

\medskip

Consider the full subcategory
$$\dgSch_{\on{Noeth,sep}}\subset \dgSch_{\on{Noeth}}$$
that consists of separated DG schemes. 
Denote
$$\IndCoh^!_{(\dgSch_{\on{Noeth,sep}})_{\on{aft}}}:=\IndCoh^!_{(\dgSch_{\on{Noeth}})_{\on{aft-sep}}}|_{((\dgSch_{\on{Noeth,sep}})_{\on{aft}})^{\on{op}}}.$$

We claim:

\begin{lem}  \label{l:from sep to sep}
The map 
\begin{multline} \label{e:sep}
\IndCoh^!_{(\dgSch_{\on{Noeth}})_{\on{aft-sep}}}\to \\
\on{RKE}_{((\dgSch_{\on{Noeth,sep}})_{\on{aft}})^{\on{op}}\to ((\dgSch_{\on{Noeth}})_{\on{aft-sep}})^{\on{op}}}
(\IndCoh^!_{(\dgSch_{\on{Noeth,sep}})_{\on{aft}}})
\end{multline}
is an isomorphism.
\end{lem}

The proof of this lemma will be given in \secref{sss:proof of sep to sep}.

\medskip

From the above lemma, we obtain:
$$\IndCoh^!_{(\dgSch_{\on{Noeth}})_{\on{aft}}}\simeq 
\on{RKE}_{((\dgSch_{\on{Noeth,sep}})_{\on{aft}})^{\on{op}}\to ((\dgSch_{\on{Noeth}})_{\on{aft}})^{\on{op}}}
(\IndCoh^!_{(\dgSch_{\on{Noeth,sep}})_{\on{aft}}}).$$

\medskip

Hence, it remains to show that for $S\in \dgSch_{\on{Noeth}}$, the map 
$$\IndCoh(S)\to \underset{S'}{lim}\, \IndCoh(S')$$
is an isomorphism in $\StinftyCat_{\on{cont}}$, where the limit is taken over the category (opposite to)
$$(\dgSch_{\on{Noeth,sep}})_{\on{aft}}\underset{(\dgSch_{\on{Noeth}})_{\on{aft}}}\times ((\dgSch_{\on{Noeth}})_{\on{aft}})_{/S}.$$

However, the above category is the same as
\begin{equation} \label{e:separated slice}
(\dgSch_{\on{Noeth,sep}})_{\on{aft}}\underset{(\dgSch_{\on{Noeth}})_{\on{aft}}}\times ((\dgSch_{\on{Noeth}})_{\on{aft-sep}})_{/S}
\end{equation}
(indeed, a map from a separated DG scheme to any DG scheme is separated). Note, however, that the limit of
$\IndCoh(S')$ over the category (opposite to one) in \eqref{e:separated slice}
is the value of
$$\on{RKE}_{((\dgSch_{\on{Noeth,sep}})_{\on{aft}})^{\on{op}}\to ((\dgSch_{\on{Noeth}})_{\on{aft-sep}})^{\on{op}}}
(\IndCoh^!_{(\dgSch_{\on{Noeth,sep}})_{\on{aft}}})$$
on $S$. In particular, the map to it from $\IndCoh(S)$ is an isomorphism because \eqref{e:sep} is an isomorphism.

\qed

\ssec{Factorization of separated morphisms}  \label{ss:check factorization}

Let $f:S_1\to S_2$ be a separated map between Noetherian DG schemes.
In this subsection we will prove that the 
category ${\mathsf{Factor}}(f)$ of factorizations of $f$ as 
\begin{equation} \label{e:factor}
S_1 \overset{j}\longrightarrow S_{3/2}  \overset{g}\longrightarrow S_2,
\end{equation}
where $j$ is an open embedding and $g$ proper, is contractible (in particular, non-empty).

\sssec{Step 1}
First we show that ${\mathsf{Factor}}(f)$ is non-empty. By Nagata's theorem, we can factor the morphism
$$({}^{cl}\!S_1)_{red}\to ({}^{cl}S_2)_{red}$$
as 
$$({}^{cl}\!S_1)_{red}\to S'_{3/2}\to ({}^{cl}S_2)_{red},$$
where $S'_{3/2}$ is a reduced classical scheme, with the morphism
$({}^{cl}\!S_1)_{red}\to S'_{3/2}$ being an open embedding and $S'_{3/2}\to ({}^{cl}S_2)_{red}$
proper.

\medskip

We define an object of ${\mathsf{Factor}}(f)$ by setting
$$S_{3/2}:=S_1\underset{({}^{cl}\!S_1)_{red}}\sqcup S'_{3/2}.$$
(we refer the reader to \cite[Sect. 3.3]{GR2}, where the existence and properties 
of push-out for DG schemes are reviewed). 

\sssec{Digression}

For a map of Noetherian DG schemes $h:T_1\to T_2$ we let
$$(\dgSch_{\on{Noeth}})_{T_1/,\,\on{closed}\,\on{in}\,T_2} \subset 
(\dgSch_{\on{Noeth}})_{T_1/\,/T_2}$$
be the full subcategory spanned by those objects
$$T_1\to T_{3/2}\to T_2,$$
where the map $T_{3/2}\to T_2$ is a closed embedding (see Definition \ref{defn:closed embeddings}).

\medskip

The following is established in \cite[Sect. 3.1]{GR2}:

\begin{prop}  \label{p:closed factorizations}  \hfill

\smallskip

\noindent{\em(a)} The category $(\dgSch_{\on{Noeth}})_{T_1/,\,\on{closed}\,\on{in}\,T_2}$
contains finite colimits (in particular, an initial object). 

\smallskip

\noindent{\em(b)} The formation of finite colimits in $(\dgSch_{\on{Noeth}})_{T_1/,\,\on{closed}\,\on{in}\,T_2}$
commutes with Zariski localization with respect to $T_2$.

\end{prop}

The initial object in $(\dgSch_{\on{Noeth}})_{T_1/,\,\on{closed}\,\on{in}\,T_2}$ will be denoted
$\ol{f(T_1)}$ and referred to as \emph{the closure of $T_1$ in $T_2$}.

\medskip

It is easy to see that if $h$ is a closed embedding that the canonical map
$$T_1\to \ol{f(T_1)}$$
is an isomorphism.

\medskip

We will need the following transitivity property of the operation of taking the closure. Let 
$$T_1\overset{h_{1,2}}\longrightarrow T_2 \overset{h_{2,3}}\longrightarrow T_3$$
be a pair of morphisms between Noetherian DG schemes. Set $h_{1,3}=h_{2,3}\circ h_{1,2}$ and
$T'_2:=\ol{h_{1,2}(T_1)}$.

\medskip

By the universal property of closure, we have a canonically defined map
\begin{equation} \label{e:trans closure}
\ol{h_{1,3}(T_1)}\to \ol{h_{1,2}(T'_2)}
\end{equation}
in $(\dgSch_{\on{Noeth}})_{T_1/,\,\on{closed}\,\on{in}\,T_3}$. We have:

\begin{lem} \label{l:trans closure}
The map \eqref{e:trans closure} is an isomorphism.
\end{lem}

\sssec{Step 2}

Let  ${\mathsf{Factor}}_{\on{dense}}(f)\subset  {\mathsf{Factor}}(f)$ be the full subcategory consisting of those objects
$$S_1\overset{j}\to S_{3/2}\overset{g}\to S_2,$$
for which the map
$$\ol{j(S_1)}\to S_{3/2}$$
is an isomorphism.

\medskip

We claim that the tautological embedding 
$${\mathsf{Factor}}_{\on{dense}}(f)\hookrightarrow  {\mathsf{Factor}}(f)$$
admits a right adjoint that sends a given object \eqref{e:factor} to 
$$S_1\to \ol{j(S_1)}\to S_2.$$

Indeed, the fact that the map $S_1\to \ol{j(S_1)}$ is an open embedding follows from \propref{p:closed factorizations}(b).
The fact that the above operation indeed produces a right adoint follows from \lemref{l:trans closure}.

\medskip

Hence, it suffices to show that the category ${\mathsf{Factor}}_{\on{dense}}(f)$ is contractible.

\sssec{Step 3} We will show that the category ${\mathsf{Factor}}_{\on{dense}}(f)$ contains
products. 

\medskip

Given two objects
$$(S_1\to S'_{3/2}\to S_2) \text{ and } (S_1\to S''_{3/2}\to S_2)$$
of ${\mathsf{Factor}}_{\on{dense}}(f)$ consider
$$T:=S'_{3/2}\underset{S_2}\times S''_{3/2},$$
and let $h$ denote the resulting map $S_1\to T$. 

\medskip

Set $S_{3/2}:=\ol{h(S_1)}$. We claim that the map $S_1\to S_{3/2}$ is an open embedding. Indeed, consider the open
subscheme of $\oT\subset T$ equal to $S_1\underset{S_2}\times S_1$. By \propref{p:closed factorizations}(b),
$$\oS_{3/2}:=S_{3/2}\cap \oT$$
is the closure of the map
$$\Delta_{S_1/S_2}:S_1\to S_1\underset{S_2}\times S_1.$$
However, $S_1\to \ol{\Delta_{S_1/S_2}(S_1)}$ is an isomorphism since $\Delta_{S_1/S_2}$ is a closed embedding. 

\sssec{Step 4} 

Finally, we claim that the resulting object
$$S_1\to S_{3/2}\to S_2$$
is the product of $S_1\to S'_{3/2}\to S_2$ and $S_1\to S''_{3/2}\to S_2$ in ${\mathsf{Factor}}_{\on{dense}}(f)$.

\medskip

Indeed, let 
$$S_1\to \wt{S}_{3/2}\to S_2$$ be another object of ${\mathsf{Factor}}_{\on{dense}}(f)$, endowed with maps
to $$S_1\to S'_{3/2}\to S_2 \text{ and } S_1\to S''_{3/2}\to S_2.$$ Let $i$ denote the resulting morphism 
$$\wt{S}_{3/2}\to S'_{3/2}\underset{S_2}\times S''_{3/2}=T.$$

\medskip

We have a canonical map in ${\mathsf{Factor}}(f)$
$$(S_1\to \wt{S}_{3/2}\to S_2) \to (S_1\to \ol{i(\wt{S}_{3/2})}\to S_2).$$

However, from \lemref{l:trans closure} we obtain that the natural map
$$S_{3/2}\to \ol{i(\wt{S}_{3/2})}$$
is an isomorphism. This gives rise to the desired map
$$(S_1\to \wt{S}_{3/2}\to S_2) \to (S_1\to S_{3/2}\to S_2).$$

\qed

\ssec{The notion of density}  \label{ss:density}

In this subsection we shall discuss the notion of density of a full subcategory 
in an $\infty$-category.

\sssec{}  \label{sss:setting for density}

Let $\bC$ be an $\infty$-category with fiber products, and equipped with 
a Grothendieck topology. Let
$$i:\bC'\hookrightarrow \bC$$ 
be a full subcategory. 

\medskip

We define a Grothendieck topology on $\bC'$ be declaring that a 1-morphism
is a covering if its image in $\bC$ is.

\medskip

We shall say that $\bC'$ is dense in $\bC$ if:

\begin{itemize}

\item Every object in $\bC$ admits a covering $\underset{\alpha}\sqcup\, \bc'_\alpha\to \bc$, $\bc'_\alpha\in \bC'$. 

\item If $\underset{\alpha}\sqcup\, \bc'_{1,\alpha}\to \bc$, and $\underset{\alpha}\sqcup\, \bc'_{2,\beta}\to \bc$,
are coverings and $\bc'_{1,\alpha},\bc'_{2,\beta}\in \bC'$, then each $\bc'_{1,\alpha}\underset{\bc}\times \bc'_{2,\beta}$
belongs to $\bC'$.

\end{itemize}

\sssec{}

Let $\bD$ be an $\infty$-category that contains limits. We let
$$\on{Funct}(\bC^{op},\bD)_{\on{descent}}\subset \on{Funct}(\bC^{op},\bD)$$
be the full subcategory of functors that satisfy descent with respect to the
given Grothendieck topology. I.e., these are $\bD$-valued sheaves as a subcategory
of $\bD$-valued presheaves.

\medskip

We will prove:

\begin{prop} \label{p:density}
Suppose $\bC'\subset \bC$ is dense. Then for any $\bD$ as above, the adjoint functors
$$\on{Res}_i:\on{Funct}(\bC^{\on{op}},\bD)\rightleftarrows \on{Funct}(\bC'{}^{\on{op}},\bD):\on{RKE}_i$$
define mutually inverse equivalences
$$\on{Funct}(\bC^{\on{op}},\bD)_{\on{descent}}\rightleftarrows \on{Funct}(\bC'{}^{\on{op}},\bD)_{\on{descent}}.$$
\end{prop}

The proposition is apparently well-known. For the sake of completeness, we will give a proof 
in \secref{sss:proof of density}.

\sssec{}  \label{sss:proof of sep to sep}

Let us show how \propref{p:density} implies \lemref{l:from sep to sep}. Indeed, we take
$$\bC:=(\dgSch_{\on{Noeth}})_{\on{aft-sep}},\,\, \bC':=(\dgSch_{\on{Noeth,sep}})_{\on{aft}},$$
and we consider the functor
$$((\dgSch_{\on{Noeth}})_{\on{aft-sep}})^{\on{op}}\to \StinftyCat_{\on{cont}}$$
equal to $\IndCoh^!_{(\dgSch_{\on{Noeth}})_{\on{aft-sep}}}$.

\medskip

We equip $\bC:=(\dgSch_{\on{Noeth}})_{\on{aft-sep}}$ with the Zariski topology, i.e., coverings are surjective maps
$\underset{\alpha}\sqcup\ S'_\alpha\to S$, where each $S'_\alpha$ is an open DG subscheme of $S$. 

\medskip

The fact that the functor $\IndCoh^!_{(\dgSch_{\on{Noeth}})_{\on{aft-sep}}}$ satisfies Zariski descent 
follows from \propref{p:Zariski descent} and the fact that the functor $\IndCoh_{(\dgSch_{\on{Noeth}})_{\on{corr}:\on{all};\on{aft-sep}}}$
satisfies condition (b) of \thmref{t:upper shriek}.

\qed 

\sssec{Proof of \propref{p:density}} \label{sss:proof of density}

First, it easy to see that the functors $\on{Res}_i$ and $\on{RKE}_i$ indeed send the subcategories
in question to one another. It is equally easy to see that the functor $\on{Res}_i$ is conservative.

\medskip

Hence, it remains to show that for $F:\bC^{\on{op}}\to \bD$ that satisfies descent, the natural map
$$F\to \on{RKE}_i(F|_{\bC'{}^{\on{op}}})$$
is an isomorphism.

\medskip

For $\bc\in \bC$ we have
$$\on{RKE}_i(F|_{\bC'{}^{\on{op}}})\simeq \underset{\bc'\in (\bC'_{/\bc})^{\on{op}}}{lim}\, F(\bc').$$

\medskip

Choose a covering $\bc'_A:=\underset{\alpha}\sqcup\, \bc'_\alpha\to \bc$, with $\bc'_\alpha\in \bC'$.
Let $\bc'_A{}^\bullet/\bc$ be its \v{C}ech nerve. Restriction defines
a map
$$\underset{\bc'\in (\bC'_{/\bc})^{\on{op}}}{lim}\, F(\bc') \to \on{Tot}(F( \bc'_A{}^\bullet/\bc)),$$
and the composition
$$F(\bc)\to \underset{\bc'\in (\bC'_{/\bc})^{\on{op}}}{lim}\, F(\bc') \to \on{Tot}(F( \bc'_A{}^\bullet/\bc))$$
is the natural map
$$F(\bc)\to \on{Tot}(F( \bc'_A{}^\bullet/\bc)),$$
and hence is an isomorphism, since $F$ satisfies descent.

\medskip

Consider now the object
$$\on{Tot}\left(\underset{\bc'\in (\bC'_{/\bc})^{\on{op}}}{lim}\, F(\bc'\underset{\bc}\times ( \bc'_A{}^\bullet/\bc))\right).$$

We shall complete the above maps to a commutative diagram
\begin{equation} \label{sss:retr diag}
\xy
(-25,15)*+{F(\bc)}="A";
(25,15)*+{\underset{\bc'\in (\bC'_{/\bc})^{\on{op}}}{lim}\, F(\bc').}="B";
(-25,-15)*+{\on{Tot}(F( \bc'_A{}^\bullet/\bc))}="C";
(25,-15)*+{\on{Tot}\left(\underset{\bc'\in (\bC'_{/\bc})^{\on{op}}}{lim}\, F(\bc'\underset{\bc}\times ( \bc'_A{}^\bullet/\bc))\right).}="D";
{\ar@{->} "A";"B"};
{\ar@{->} "A";"C"};
{\ar@{->} "B";"D"};
{\ar@{->} "C";"D"};
{\ar@{->} "B";"C"};
\endxy
\end{equation}
in which the right vertical arrow is an isomorphism. This will prove that the top horizontal arrow is also an isomorphism,
as required. 

\medskip

The map
$$\on{Tot}(F( \bc'_A{}^\bullet/\bc))\to \on{Tot}\left(\underset{\bc'\in (\bC'_{/\bc})^{\on{op}}}{lim}\, F(\bc'\underset{\bc}\times ( \bc'_A{}^\bullet/\bc))\right)$$
is the simplex-wise map 
$$F( \bc'_A{}^\bullet/\bc)\to \underset{\bc'\in (\bC'_{/\bc})^{\on{op}}}{lim}\, F(\bc'\underset{\bc}\times ( \bc'_A{}^\bullet/\bc)),$$
given by restriction.

\medskip

We shall now consider three maps 
\begin{multline} \label{e:3 maps}
\underset{\bc'\in (\bC'_{/\bc})^{\on{op}}}{lim}\, F(\bc')\to  
\on{Tot}\left(\underset{\bc'\in (\bC'_{/\bc})^{\on{op}}}{lim}\, F(\bc'\underset{\bc}\times ( \bc'_A{}^\bullet/\bc))\right)\simeq \\
\simeq \underset{([n]\times \bc')\in \bDelta\times (\bC'_{/\bc})^{\on{op}}}{lim}\, F(\bc'\underset{\bc}\times ( \bc'_A{}^n/\bc))\simeq
\underset{\bc'\in (\bC'_{/\bc})^{\on{op}}}{lim}\, \on{Tot}(F(\bc'\underset{\bc}\times ( \bc'_A{}^\bullet/\bc))).
\end{multline}

\medskip

The first map in \eqref{e:3 maps} corresponds to the map of index categories that sends
$$([n]\times \bc')\in \bDelta\times (\bC'_{/\bc})^{\on{op}}\, \mapsto  \bc'_A{}^n/\bc\in (\bC'_{/\bc})^{\on{op}}$$
and the map
$$F( \bc'_A{}^n/\bc)\to F(\bc'\underset{\bc}\times ( \bc'_A{}^n/\bc))$$ is given by the projection
$$\bc'\underset{\bc}\times ( \bc'_A{}^n/\bc)\to  \bc'_A{}^n/\bc.$$

\medskip

It is clear that for this map the lower traingle in \eqref{sss:retr diag} commutes. 

\medskip

The second map in \eqref{e:3 maps} corresponds to the map of index categories that sends
$$([n]\times \bc')\in \bDelta\times (\bC'_{/\bc})^{\on{op}}\, \mapsto \bc'\underset{\bc}\times ( \bc'_A{}^n/\bc)\in (\bC'_{/\bc})^{\on{op}},$$
and the identity map
$$F(\bc'\underset{\bc}\times ( \bc'_A{}^n/\bc))\to  F(\bc'\underset{\bc}\times ( \bc'_A{}^n/\bc)).$$

We note, however, that the first and the second maps are canonically homotopic.

\medskip

The third map in \eqref{e:3 maps} corresponds to the map of index categories that sends
$$([n]\times \bc')\in \bDelta\times (\bC'_{/\bc})^{\on{op}}\, \mapsto \bc'\in (\bC'_{/\bc})^{\on{op}},$$
and the map
$$F(\bc')\to F(\bc'\underset{\bc}\times ( \bc'_A{}^n/\bc))$$ is given by the projection
$$\bc'\underset{\bc}\times ( \bc'_A{}^n/\bc)\to \bc'.$$

Again, we note that the third and the second map are canonically homotopic.

\medskip

Finally, it remains to see that the third map is an isomorphism. This follows from the fact that for
each $\bc'\in (\bC'_{/\bc})^{\on{op}}$, the map
$$F(\bc')\to \on{Tot}(F(\bc'\underset{\bc}\times ( \bc'_A{}^\bullet/\bc)))$$
is an isomorphism, since $F$ satisfies descent.

\qed 

\sssec{}

In what follows we will use the following variant of \propref{p:density}. Let $\bC$
be a category with fiber products and equipped with a Grothendick topology.

\medskip

Let $\bC_0$ be a 1-full subcategory of $\bC$. We will assume that the following
is satisfied: 

\medskip

\noindent Whenever $\underset{\alpha}\sqcup\, \bc_\alpha\to \bc$ in $\bC$
is a covering, then each arrow $\bc_\alpha\to \bc$ belongs to $\bC_0$.

\medskip

Note that in this case we can talk about descent for presheaves on $\bC_0$:
the fiber products used in the formulation of the are taken inside the
ambient category $\bC$. 

\sssec{}

Let $\bC'\subset \bC$ be a full subactegory satisfying the assumptions of
\secref{sss:setting for density}. Let $\bC'_0$ be the corresponding 1-full subcategory of $\bC'$. 

\medskip

We shall assume that the following additional condition on $\bC_0\to \bC$ is satisfied:

\begin{itemize}

\item If $\underset{\alpha}\sqcup\, \bc_\alpha\to \bc$ is a covering and $\wt\bc\to \bc$ is a $1$-morphism
in $\bC_0$, then each of the maps $\bc_\alpha\underset{\bc}\times \wt\bc\to \bc_\alpha$
belongs to $\bC_0$. 

\end{itemize}

In this case, the proof of \propref{p:density} applies and gives the following:

\begin{prop} \label{p:density bis}
The adjoint functors
$$\on{Res}_i:\on{Funct}((\bC_0)^{\on{op}},\bD)\rightleftarrows \on{Funct}((\bC'_0)^{\on{op}},\bD):\on{RKE}_i$$
define mutually inverse equivalences
$$\on{Funct}((\bC_0)^{\on{op}},\bD)_{\on{descent}}\rightleftarrows \on{Funct}((\bC'_0)^{\on{op}},\bD)_{\on{descent}}.$$
\end{prop}

\section{Eventually coconnective, Gorenstein and smooth morphisms}  \label{s:Gorenstein}

In this section we continue to assume that all DG schemes are Noetherian. 

\medskip

The results in this section were obtained in collaboration with D.~Arinkin and some
of them appear also in \cite[Appendix E]{AG}.

\ssec{The !-pullback functor for eventually coconnective morphisms}

\sssec{}

First, we notice:

\begin{lem}  \label{l:! vs * finite}
Let $S_1\to S_2$ be a morphism almost of finite type. Then the following conditions are equivalent:

\begin{itemize}

\item The morphism $f$ is eventually coconnective.

\item The functor $f^!:\IndCoh(S_2)\to \IndCoh(S_1)$ sends $\Coh(S_2)$ to $\Coh(S_1)$.

\end{itemize}

In this case, the functor $f^!$ has a bounded cohomological amplitude. 

\end{lem}

\begin{proof}

As in the proof of \propref{p:Xi and *}, the assertion reduces to the case when $f$ is a closed embedding.
Denote $\CE:=f_*(\CO_{S_1})\in \Coh(S_2)$. Then the same argument as in \lemref{l:finite Tor}
shows that both conditions in the lemma are equivalent to $\CE$ being perfect.

\end{proof}
 
\sssec{}

Let now 
$$
\CD
S'_1  @>{g_1}>>  S_1  \\
@V{f'}VV   @VV{f}V   \\
S'_2  @>{g_2}>>  S_2
\endCD
$$
be a Cartesian diagram where horizontal arrows are almost of finite type, and vertical arrows are of bounded Tor dimension
(in particular, eventually coconnective). 

\medskip

Starting from the base change isomorphism
$$g_2^!\circ f^{\IndCoh}_*\simeq (f')^{\IndCoh}_*\circ g_1^!$$
by the $(f^{\IndCoh,*},f^\IndCoh_*)$- and $((f')^{\IndCoh,*},(f')^\IndCoh_*)$-adjunctions, we obtain a
natural transformation
\begin{equation} \label{e:! and * IndCoh}
(f')^{\IndCoh,*}\circ g_2^! \to g_1^! \circ f^{\IndCoh,*}.
\end{equation}

\begin{rem}
If the map $g_2$ (and hence $g_1$) is proper, then diagram chase shows that the map
in \eqref{e:! and * IndCoh} is canonically isomorphic to one obtained from the isomorphism
$$(g_1)^\IndCoh_*\circ (f')^{\IndCoh,*}\simeq f^{\IndCoh,*}\circ (g_2)^\IndCoh_*$$
of \lemref{l:usual base change}.
\end{rem}

\sssec{}

We claim:

\begin{prop}  \label{p:! pullback and event coconn}
The map \eqref{e:! and * IndCoh} is an isomorphism.
\end{prop}

\begin{proof}

First, it is easy to see that the assertion is Zariski-local in $S_2$ and $S'_2$. By the
construction of the !-pullback functor, we can consider separately the cases of $g_2$
(and hence $g_1$) proper and an open embedding, respectively. 

\medskip

The case of an open embedding is immediate. So, in what follows we shall assume that
the maps $g_1$ and $g_2$ are proper.

\medskip 

Note that the statement is Zariski-local also in $S_1$. Hence, we can assume
that the morphism $f$ (and hence $f'$) is affine. 

\medskip

Since all the functors involved are continuous, it is enough to show that the map 
\eqref{e:! and * IndCoh} is an isomorphism when evaluated on $\Coh(S_2)$. We will
show more generally that it is an isomorphism, when evaluated on objects of $\IndCoh(S_2)^+$.

\medskip

The assumption of $f$ implies that the functor $f^{\IndCoh,*}$ (resp., $(f')^{\IndCoh,*}$)
sends the category $\IndCoh(S_2)^+$ (resp., $\IndCoh(S'_2)^+$) to $\IndCoh(S_1)^+$ (resp., $\IndCoh(S'_1)^+$).

\medskip

We claim:

\begin{lem}  \label{l:! left exact}
For any morphism $g:T_1\to T_2$ amost of finite type, the functor $g^!$ is left t-exact, up to a finite
shift.
\end{lem}
The proof of the lemma is given below.

\medskip

Assuminh the lemma, we obtain that both functors
in \eqref{e:! and * IndCoh} send the category $\IndCoh(S_2)^+$ to $\IndCoh(S'_1)^+$. Since $f$ (and hence $f'$)
is affine, the functor $f'_*$ is conservative. It follows from \propref{p:equiv on D plus}
that the functor $(f')^{\IndCoh}_*$ is conservative when restricted to $\IndCoh(S'_1)^+$.

\medskip

Hence, it remains to show that the map
\begin{equation}  \label{e:! and * IndCoh composed}
(f')^{\IndCoh}_*\circ (f')^{\IndCoh,*}\circ g_2^! (\CF)\to (f')^{\IndCoh}_*\circ g_1^! \circ f^{\IndCoh,*}(\CF)
\end{equation}
is an isomorphism for any $\CF\in \IndCoh(S_2)$. 

\medskip

By \corref{c:tensoring up}, we have canonical isomorphisms
$$(f')^{\IndCoh}_*\circ (f')^{\IndCoh,*}\simeq f'_*(\CO_{S'_1})\otimes - \text{ and }
f^{\IndCoh}_*\circ f^{\IndCoh,*}\simeq f'_*(\CO_{S_1})\otimes -.$$ 

Note also that
$$f'_*(\CO_{S'_1})\simeq g_2^*(f_*(\CO_{S_1})).$$

Hence, we can rewrite the left-hand side in \eqref{e:! and * IndCoh composed} as
$$g_2^*(f_*(\CO_{S_1})) \otimes g_2^! (\CF)\overset{\text{Equation \eqref{e:upgrading !}}}\simeq g_2^!(f_*(\CO_{S_1})\otimes \CF).$$

We rewrite the right-hand side in \eqref{e:! and * IndCoh composed} as
$$g_2^!\circ f^\IndCoh_*\circ f^{\IndCoh,*}(\CF)\simeq g_2^!(f_*(\CO_{S_1})\otimes \CF).$$

It follows by unwinding the constructions that the map in \eqref{e:! and * IndCoh composed} identifies
with the identity map on $g_2^!(f_*(\CO_{S_1})\otimes \CF)$.

\end{proof}

\begin{proof}[Proof of \lemref{l:! left exact}]

By the construction of $g^!$, it suffices to consider separately the cases of $g$ being an open
embedding and a proper map. The case of an open embedding is evident; in this case the functor
$g^!=g^{\IndCoh,*}$ is t-exact.

\medskip

For a proper map, the functor $g^!$ is by definition the right adjoint of $g^\IndCoh_*$. The required
assertion follows from the fact that the functor $g^\IndCoh_*$ is right t-exact, up to a finite shift,
which in turn follows from the corresponding property of $g_*$.

\end{proof}

\ssec{The !-pullback functor on $\QCoh$}

In this subsection we let $f:S_1\to S_2$ be an eventually coconnective morphism almost of finite type. 

\sssec{}

We claim:

\begin{prop} \label{p:QCoh !}
There exists a uniquely defined continuous functor
$$f^{\QCoh,!}:\QCoh(S_2)\to \QCoh(S_1),$$
which is of bounded cohomological amplitude, and makes the following diagram 
$$
\CD
\IndCoh(S_1)  @>{\Psi_{S_1}}>>  \QCoh(S_1) \\
@A{f^!}AA   @AA{f^{\QCoh,!}}A   \\ 
\IndCoh(S_2)  @>{\Psi_{S_2}}>>  \QCoh(S_2) 
\endCD
$$
commute. Furthermore, the functor $f^{\QCoh,!}$ has a unique structure of $1$-morphism of
$\QCoh(S_2)$-module categories, for which the above diagram commutes as
a diagram of $\QCoh(S_2)$-module categories.
\end{prop}

\begin{proof}

This follows using \propref{p:QCoh as left compl} from the following general observation. 

\medskip

Let $F:\bC_1\to \bC_2$ be a continuous functor
between cocomplete DG categories. Suppose that $\bC_1$ and $\bC_2$ are endowed with t-structures
compatible with filtered colimits. Suppose that $F$ is right t-exact up to a finite shift. 

\medskip

Let $\bC'_1$ and $\bC'_2$ be the left completions of $\bC_1$ and $\bC_2$, respectively, in their
t-structures. 

\begin{lem}  \label{l:extension to left completions}
Under the above circumstances there exists a uniquely defined continuous functor $F':\bC'_1\to \bC'_2$ which
makes the diagram
$$
\CD
\bC_1 @>>>  \bC'_1  \\
@V{F}VV  @VV{F'}V   \\
\bC_2 @>>>  \bC'_2
\endCD
$$
commute.
\end{lem}

\end{proof}

\sssec{}

Note that the $\QCoh(S_2)$-linearity of $f^{\QCoh,!}$ tautologically implies:

\begin{cor}  \label{c:exp for QCoh !}
There is a canonical isomorphism
$$f^*(\CE)\otimes f^{\QCoh,!}(\CO_{S_2})\simeq f^{\QCoh,!}(\CE),\quad \CE\in \QCoh(S_2).$$
\end{cor}

We now claim:

\begin{prop} \label{p:! and * pullback under eventually coconn}
There exists a uniquely defined natural transformation
\begin{equation} \label{e:! and * pullback under eventually coconn}
 f^{\QCoh,!}(\CO_{S_2})\otimes f^{\IndCoh,*}(\CF)\to f^!(\CF),\quad \CF\in \IndCoh(S_2)
\end{equation}
that makes the diagram
$$
\CD
\Psi_{S_1}(f^!(\CF))   @>{\sim}>>  f^{\QCoh,!}(\Psi_{S_2}(\CF))  \\
@A\Psi_{S_1}({\text{\eqref{e:! and * pullback under eventually coconn}})}AA   @A{\sim}A{\text{\corref{c:exp for QCoh !}}}A   \\
\Psi_{S_1}(f^{\QCoh,!}(\CO_{S_2})\otimes f^{\IndCoh,*}(\CF))  & &  f^*(\Psi_{S_2}(\CF))\otimes f^{\QCoh,!}(\CO_{S_2}) \\
@A{\sim}AA   @AA{\sim}A  \\ 
f^{\QCoh,!}(\CO_{S_2})\otimes \Psi_{S_1}(f^{\IndCoh,*}(\CF)) @>{\sim}>{\text{\propref{p:* pullback}}}>  
f^{\QCoh,!}(\CO_{S_2})\otimes f^*(\Psi_{S_2}(\CF))
\endCD
$$
commute. 
\end{prop}

\begin{proof}

It suffuces to construct (and prove the uniquness of) the isomorphism in question on the compact generators of
$\IndCoh(S_2)$, i.e., for $\CF\in \Coh(S_2)$. For such $\CF$, we have $f^!(\CF)\in \IndCoh(S_1)^+$. 
Hence, it suffices to construct (and prove the uniquness of) a map
$$\tau^{\geq -n}\left( f^{\QCoh,!}(\CO_{S_2})\otimes f^{\IndCoh,*}(\CF)\right)\to f^!(\CF)$$
for $n\gg 0$ that makes the corresponding diagram commute.

\medskip

By \propref{p:equiv on D plus}, the latter map is equivalent to a map
\begin{multline}  \label{e:! and * pullback under eventually coconn to construct}
\tau^{\geq -n}\left(\Psi_{S_1}\left(f^{\QCoh,!}(\CO_{S_2})\otimes f^{\IndCoh,*}(\CF)\right)\right)\simeq \\
\simeq \Psi_{S_1}\left(\tau^{\geq -n}\left(f^{\QCoh,!}(\CO_{S_2})\otimes f^{\IndCoh,*}(\CF)\right)\right)\to
\Psi_{S_1}\left(f^!(\CF)\right).
\end{multline}

We have:
\begin{multline*}
\Psi_{S_1}\left(f^{\QCoh,!}(\CO_{S_2})\otimes f^{\IndCoh,*}(\CF)\right)\simeq 
f^{\QCoh,!}(\CO_{S_2})\otimes \left(\Psi_{S_1}\circ f^{\IndCoh,*}(\CF)\right)\simeq \\
\simeq f^{\QCoh,!}(\CO_{S_2})\otimes \left(f^*\circ \Psi_{S_2}(\CF)\right),
\end{multline*}
which, by \corref{c:exp for QCoh !}, identifies with 
$$f^{\QCoh,!}(\Psi_{S_2}(\CF))\simeq \Psi_{S_1}(f^!(\CF)).$$

In particular, 
$$\tau^{\geq -n}\left(\Psi_{S_1}\left(f^{\QCoh,!}(\CO_{S_2})\otimes f^{\IndCoh,*}(\CF)\right)\right)\simeq
\Psi_{S_1}(f^!(\CF))$$
for all $n\gg 0$, and the required map in \eqref{e:! and * pullback under eventually coconn to construct}
is the identity.

\end{proof}

\begin{rem}
As we shall see in Propositions \ref{p:pullback under Gorenstein} and \ref{p:pullback under Gorenstein converse},
the map in \eqref{e:! and * pullback under eventually coconn} is an isomorphism \emph{if and only if}
the map $f$ is Gorenstein.
\end{rem}

\sssec{}

Here are some properties of the functor of the functor $f^{\QCoh,!}$ introduced in \propref{p:QCoh !}.

\begin{prop} \label{p:properties QCoh !} \hfill
\smallskip

\noindent{\em(a)}
If $f$ is proper, the functor $f^{\QCoh,!}$ is the right adjoint of $f_*$.

\smallskip

\noindent{\em(b)} Let 
$$
\CD
S'_1  @>{g_1}>>  S_1  \\
@V{f'}VV   @VV{f}V   \\
S'_2  @>{g_2}>>  S_2
\endCD
$$ 
be a Cartesian diagram. Then there exists a canonical isomorphism of functors
\begin{equation} \label{e:base change for QCoh !}
(g_1)_* \circ (f')^{\QCoh,!}\simeq  f^{\QCoh,!}\circ (g_2)_*
\end{equation}
Moreover, if $f$ is proper, the map $\to$ in \eqref{e:base change for QCoh !}
comes by the $(f_*,f^{\QCoh,!})$- and $(f'_*,(f')^{\QCoh,!})$-adjunctions from the isomorphism
$$f_*\circ (g_1)_*\simeq (g_2)_*\circ f'_*.$$

\smallskip

\noindent{\em(c)} 
The map
\begin{equation} \label{e:! and * QCoh}
g_1^*\circ f^{\QCoh,!}\to (f')^{\QCoh,!} \circ g_2^*
\end{equation}
obtained from \eqref{e:base change for QCoh !} by the $(g_1^*,(g_1)_*)$- and
$(g_2^*,(g_2)_*)$-adjunctions, is an isomorphism.

\end{prop}

\begin{rem}
A diagram chase shows that if $f$ (and hence $f'$) is proper, the map in \eqref{e:! and * QCoh} is canonically
isomorphic to one obtained by the $(f_*,f^{\QCoh,!})$- and $(f'_*,(f')^{\QCoh,!})$-adjunctions
from the usual base change isomorphism for $\QCoh$:
$$f'_*\circ g_1^*\simeq g_2^*\circ f_*.$$
\end{rem}

\begin{proof}

We first prove point (a). Since $\QCoh(S_1)$ and $\QCoh(S_2)$ are left-complete in their respective t-structures
and since $f^{\QCoh,!}$ is right t-exact up to a finite shift, it suffices
construct a canonical isomorphism
\begin{equation} \label{e:adj for QCoh}
\Maps_{\QCoh(S_2)}(f_*(\CE_1),\CE_2)\to \Maps_{\QCoh(S_2)}(\CE_1,f^{\QCoh,!}(\CE_2))
\end{equation} 
for $\CE_2\in \QCoh(S_2)^+$.  

\medskip

Since $f^{\QCoh,!}$ is also left t-exact up to a finite shift, and since 
$f_*$ is left t-exact, we can also take $\CE_1\in \QCoh(S_1)^+$. In this case, 
\propref{p:equiv on D plus} reduces the isomorphism \eqref{e:adj for QCoh} for
one for $\IndCoh$. 

\medskip

Point (b) follows from \lemref{l:extension to left completions} and the base change isomorphism
for $\IndCoh$.

\medskip

Let us prove point (c). As in the proof of \propref{p:! pullback and event coconn}, we can
assume that $f$ (and hence $f'$) is proper, and $g_2$ (and hence $g_1$) is affine. 

\medskip

In this case, the functor
$(g_1)_*$ is conservative. Hence, it is sufficient to show that the natural transformation 
\begin{equation} \label{e:! and * QCoh composed}
(g_1)_*\circ g_1^*\circ f^{\QCoh,!}\to (g_1)_*\circ (f')^{\QCoh,!} \circ g_2^*
\end{equation}
is an isomorphism.

\medskip

The value of the left-hand side of \eqref{e:! and * QCoh composed} on $\CE\in \QCoh(S_2)$
is canonically isomorphic to 
$$(g_1)_*(\CO_{S'_1})\otimes f^{\QCoh,!}(\CE)\simeq
f^*((g_2)_*(\CO_{S'_2}))\otimes f^{\QCoh,!}(\CE)\simeq
f^{\QCoh,!}((g_2)_*(\CO_{S'_2})\otimes \CE).$$

The value of the right-hand side of \eqref{e:! and * QCoh composed} on $\CE\in \QCoh(S_2)$
is, by point (b), canonically isomorphic to
$$f^{\QCoh,!}\circ (g_2)_*\circ  g_2^*(\CE)\simeq f^{\QCoh,!}((g_2)_*(\CO_{S'_2})\otimes \CE).$$

By unwinding the constructions, we obtain that the map in \eqref{e:! and * QCoh composed} is
the identity map on $f^{\QCoh,!}((g_2)_*(\CO_{S'_2})\otimes \CE)$.

\end{proof}

\ssec{Gorenstein morphisms}

\sssec{}

Let $f:S_1\to S_2$ be a morphism between Noetherian DG schemes. 

\begin{defn}
We shall say that $f$ is Gorenstein if it is eventually coconnective, locally almost of finite type,
and the object $f^{\QCoh,!}(\CO_{S_2})\in \QCoh(S_1)$ is a graded line bundle. 
\end{defn}

In what follows, for a Gorenstein morphism morphism $f:S_1\to S_2$, we shall denote 
by $\CK_{S_1/S_2}$ the graded line bundle appearing in the above definition, and refer
to it as the ``relative dualizing line bundle."

\sssec{}

We shall say that $S\in \dgSch_{\on{Noeth}}$ is Gorenstein if the map
$$p_S:S\to \on{pt}:=\Spec(k)$$
is Gorenstein. I.e., if $S$ is almost of finite type, eventually coconnective and the object
$$\omega_S:=(p_S)^!(k)\in \Coh(S),$$
thought of as an object of $\QCoh(S)$, is a graded line bundle.

\sssec{}

It follows from \propref{p:properties QCoh !} that the class of Gorenstein morphisms is stable under base change and that 
the formation of $\CK_{S_1/S_2}$ is compatible with base change, i.e., for
a Cartesian square
$$
\CD
S'_1  @>{g_1}>>  S_1  \\
@V{f'}VV   @VV{f}V   \\
S'_2  @>{g_2}>>  S_2
\endCD
$$ 
we have a canonical isomorphism
\begin{equation} \label{e:relative dualizing}
g_1^*(\CK_{S_1/S_2})\simeq \CK_{S'_1/S'_2}.
\end{equation}

In addition, we claim:

\begin{cor} \label{c:Goren via fibers} 
Let $f:S_1\to S_2$ be an eventually coconnective morphism almost
of finite type. Then $f$ is Gorenstein if only if its geometric fibers
are Gorenstein.
\end{cor}

\begin{proof}

Assume that the geometric fibers of $f$ are Gorenstein. We need to show that
$f^{\QCoh,!}(\CO_{S_2})$ is a graded line bundle. As in \lemref{l:finite Tor},
it suffices to show that the *-restrictions of $f^{\QCoh,!}(\CO_{S_2})$ to the
geometric fibers of $f$ are graded line bundles. The latter follows from 
\propref{p:properties QCoh !}(c). 

\end{proof}

\begin{cor} \label{c:Goren via class} 
Let $f:S_1\to S_2$ be an eventually coconnective morphism almost
of finite type. If the base change of $f$ by $i:{}^{cl}\!S_2\to S_2$
is Gorenstein, then $f$ is Gorenstein.
\end{cor}

\sssec{}

We are going to prove:

\begin{prop}  \label{p:pullback under Gorenstein}
Let $f:S_1\to S_2$ be Gorenstein. Then the map 
$$\CK_{S_1/S_2}\otimes f^{\IndCoh,*}(\CF)\to f^!(\CF),\quad \CF\in \IndCoh(S_2)$$
of \eqref{e:! and * pullback under eventually coconn} is an isomorphism.
\end{prop}

\begin{proof}

Going back to the proof of \propref{p:! and * pullback under eventually coconn},
it suffices to show that for $\CF\in \Coh(S)$, the object
$$\CK_{S_1/S_2}\otimes f^{\IndCoh,*}(\CF)\in \IndCoh(S_1)$$
belongs to $\IndCoh(S_1)^+$. We have: $f^{\IndCoh,*}(\CF)\in \IndCoh(S_1)^+$,
and the required assertion follows from the fact that $\CK_{S_1/S_2}\otimes-$
shifts degrees by a finite amount.

\end{proof}

\ssec{Characterizations of Gorenstein morphisms}

The material of this subsection is included for the sake of completeness and will not be needed elsewhere in the paper.

\sssec{}

We have:

\begin{prop} \label{p:Drinfeld}
Let $f:S_1\to S_2$ be an eventually coconnective morphism almost
of finite type. Suppose that the object $f^{\QCoh,!}(\CO_{S_2})$ belongs to 
$\QCoh(S_1)^{\on{perf}}$. Then $f$ is Gorenstein.
\end{prop}

As a particular case, we have the following characterization of Gorenstein
DG schemes, suggested by Drinfeld:

\begin{cor} \label{c:Drinfeld}
Let $S\in \dgSch_{\on{aft}}$ be eventually coconnective. Suppose that the object
$$\omega_S\in \Coh(S)$$ belongs to $\QCoh(S)^{\on{perf}}\subset \Coh(S)$. 
Then $S$ is Gorenstein.
\end{cor}

\begin{proof}

It is easy to see that an object of $\QCoh(S_1)^{\on{perf}}$ is a line bundle if and only
if such is the retsriction to all of its geometric fibers over $S_2$. Using \propref{p:properties QCoh !}(c), 
this reduces the statement of the proposition to that of \corref{c:Drinfeld}. The latter will
be proved in \secref{sss:proof of Drinf}.

\end{proof}

\sssec{}

We are now going to prove a converse to \propref{p:pullback under Gorenstein}.

\begin{prop} \label{p:pullback under Gorenstein converse}
Let $f:S_1\to S_2$ be an eventually coconnective morphism almost
of finite type, such that the natural transformation \eqref{e:! and * pullback under eventually coconn} is an isomorphism.
Then $f$ is Gorenstein.
\end{prop}

\begin{proof}

By \propref{p:! pullback and event coconn}, the assumption of the proposition is stable under base change.
Using \corref{c:Goren via class}, we reduce the assertion to the case when $S_2$ is classical, in particular,
eventually coconnective as a DG scheme. 

\medskip

By assumption, 
$$f^{\QCoh,!}(\CO_{S_2})\otimes f^{\IndCoh,*}(\Xi_{S_2}(\CO_{S_2}))\simeq f^!(\Xi_{S_2}(\CO_{S_2}))\in \Coh(S_1).$$

However,
$$f^{\QCoh,!}(\CO_{S_2})\otimes f^{\IndCoh,*}(\Xi_{S_2}(\CO_{S_2}))\simeq
f^{\QCoh,!}(\CO_{S_2})\otimes \Xi_{S_1}(\CO_{S_1})\simeq \Xi_{S_1}(f^{\QCoh,!}(\CO_{S_2})).$$

From \lemref{l:Xi in Coh}, we obtain that $f^{\QCoh,!}(\CO_{S_2})\in \QCoh(S_1)^{\on{perf}}$.
Hence, the assertion follows from \propref{p:Drinfeld}.

\end{proof}

\sssec{}

We end this subsection with the following observation. Let $f:S_1\to S_2$ be an eventually coconnective morphism almost
of finite type, where $S_2$ (and hence $S_1$) is itself eventually coconnective. Then from the isomorphism
$$f^{\QCoh,!}\circ \Psi_{S_2}\simeq \Psi_{S_1}\circ f^!$$
we obtain a natural transformation:

\begin{equation}  \label{e:Xi and QCoh !}
\Xi_{S_1}\circ f^{\QCoh,!}\to f^!\circ \Xi_{S_2}.
\end{equation}

We claim:

\begin{prop}
The map $f$ is Gorenstein if and only if the natural transformation \eqref{e:Xi and QCoh !}
is an isomorphism.
\end{prop}

\begin{proof}

For $\CE\in \QCoh(S_2)$ we have 
$$\Xi_{S_1}\circ f^{\QCoh,!}(\CE)\simeq f^*(\CE)\otimes (\Xi_{S_1}\circ f^{\QCoh,!}(\CO_{S_2}))
\simeq f^*(\CE)\otimes f^{\QCoh,!}(\CO_{S_2})\otimes \Xi_{S_1}(\CO_{S_1}).$$

If $f$ is Gorenstein, then by \propref{p:pullback under Gorenstein}, we have:
\begin{multline*}
f^!\circ \Xi_{S_2}(\CE)\simeq \CK_{S_1/S_2}\otimes (f^{\IndCoh,*}\circ \Xi_{S_2}(\CE))\simeq
\CK_{S_1/S_2}\otimes f^*(\CE)\otimes (f^{\IndCoh,*}\circ \Xi_{S_2}(\CO_{S_2}))\simeq \\
\simeq \CK_{S_1/S_2}\otimes f^*(\CE)\otimes (\Xi_{S_1}\circ f^*(\CO_{S_2}))\simeq 
\CK_{S_1/S_2}\otimes f^*(\CE)\otimes \Xi_{S_1}(\CO_{S_1}),
\end{multline*}
and the isomorphism is manifest. 

\medskip

Vice versa, assume that \eqref{e:Xi and QCoh !} holds. We obtain that
$$\Xi_{S_1}\circ f^{\QCoh,!}(\CO_{S_2})\in \Coh(S_1).$$

Applying \lemref{l:Xi in Coh}, we deduce that $f^{\QCoh,!}(\CO_{S_2})\in \QCoh(S_1)^{\on{perf}}$.
Applying \propref{p:Drinfeld}, we deduce that $f$ is Gorenstein.

\end{proof}

\ssec{The functor of !-pullback for smooth maps}

\sssec{}

Note that from \corref{c:Goren via class} we obtain:

\begin{cor}
A smooth map is Gorenstein. 
\end{cor}

\sssec{}

As a corollary of \propref{p:pullback under Gorenstein}, we obtain:

\begin{cor}  \label{c:! and * pullback under smooth}
Let $f:S_1\to S_2$ be smooth. There exists a unique isomorphism
of functors
\begin{equation} \label{e:! and * pullback under smooth}
\CK_{S_1/S_2}\otimes f^{\IndCoh,*}(\CF)\simeq f^!(\CF),\quad \CF\in \IndCoh(S_2)
\end{equation}
that makes the diagram
$$
\CD
\Psi_{S_1}(f^!(\CF))   @>{\sim}>>  f^{\QCoh,!}(\Psi_{S_2}(\CF))  \\
@A{\text{\eqref{e:! and * pullback under smooth}}}AA   @A{\sim}A{\text{\corref{c:exp for QCoh !}}}A   \\
\Psi_{S_1}(\CK_{S_1/S_2}\otimes f^{\IndCoh,*}(\CF))  & &  f^*(\Psi_{S_2}(\CF))\otimes \CK_{S_1/S_2}) \\
@A{\sim}AA   @AA{\sim}A  \\ 
\CK_{S_1/S_2}\otimes \Psi_{S_1}(f^{\IndCoh,*}(\CF)) @>{\sim}>{\text{\propref{p:* pullback}}}>  
\CK_{S_1/S_2} \otimes f^*(\Psi_{S_2}(\CF))
\endCD
$$
commute. 
\end{cor}

\begin{cor}  \label{c:differ by line}
For a smooth map $f$, the functors $f^!$ and $f^{\IndCoh,*}$ differ by an $\QCoh(S_1)$-linear
automorphism of $\IndCoh(S_1)$.
\end{cor}

\sssec{}

Combining this with \propref{p:smooth map}, we obtain:

\begin{cor} \label{c:! tensor up smooth}
For a smooth map $f:S_1\to S_2$, the functor
$f^!:\IndCoh(S_2)\to \IndCoh(S_1)$ defines an equivalence
$$\QCoh(S_1)\underset{\QCoh(S_2)}\otimes \IndCoh(S_2)\to \IndCoh(S_1).$$
\end{cor}

\sssec{}

We are now going to use \corref{c:! tensor up smooth} to deduce:

\begin{prop} \label{p:! tensor up general}
Let $f:S_1\to S_2$ be eventually coconnective and almost of finite type. Then the functor
$$\QCoh(S_1)\underset{\QCoh(S_2)}\otimes \IndCoh(S_2)\to \IndCoh(S_1),$$
induced by $f^!$, is fully faithful.
\end{prop}

\begin{proof}[Proof of \propref{p:! tensor up general}]
By \corref{c:open tensoring up}, the assertion of the proposition is Zariski-local with respect
to both $S_1$ and $S_2$. 

\medskip

As in the proof of \lemref{l:finite Tor}, we can assume that $f$
can be factored as a composition $f'\circ f''$, where $f'$ is smooth, and $f''$
is an eventually coconnective closed embedding. By transitivity, it suffices to prove
the assertion for each of these two cases separately. 

\medskip

For smooth maps, the required assertion is given by \corref{c:! tensor up smooth}. 
For an eventually coconnective closed embedding, the argument repeats that given in
\propref{p:tensoring up IndCoh}, where we use the 
$(f^\IndCoh_*,f^!)$-adjunction instead of the
$(f^{\IndCoh,*},f^\IndCoh_*)$-adjunction.

\end{proof}

\ssec{A higher-categorical compatibility of $!$ and $*$-pullbacks}

The material of this subsection will be needed for developing $\IndCoh$
on Artin stacks, and may be skipped on the first pass.  

\medskip

We will introduce
a setting in which we can view $\IndCoh$ as a functor out of the category
of Noetherian DG schemes, where there are two kinds of pullback functors:
one with $!$ that can be applied to morphisms locally almost of finite type, and another
with $*$ that can be applied to morphisms of bounded Tor dimension. 

\medskip

Specifically,
we will explain an $\infty$-categorical framework that encodes the compatibiltiy
of the two pullbacks, which at the level of 1-morphisms is given by 
\propref{p:! pullback and event coconn}.

\sssec{}  \label{sss:Segal review}

For an $\infty$-category $\bC$ we let 
$$\on{Seg}^\bullet(\bC):=\on{Funct}([\bullet],\bC)\in \inftygroup^{\bDelta^{\on{op}}}$$
denote the Segal construction applied to $\bC$. 

\medskip

I.e., $\on{Seg}^\bullet(\bC)$ is a simplicial $\infty$-groupoid, whose $n$-simplices is the 
$\infty$-groupoid of functors 
$$[n]\to \bC,$$
where $[n]$ is the category corresponding to the ordered set $$0\to 1\to...\to n.$$

\medskip

The assignment $$\bC\rightsquigarrow \on{Seg}^\bullet(\bC)$$
is a functor 
$$\on{Seg}^\bullet(-):\inftyCat\to \inftygroup^{\bDelta^{\on{op}}}.$$

\medskip

The following is well-known:

\begin{thm} \label{t:Segal}
The functor $\on{Seg}^\bullet(-)$ is fully faithful.
\end{thm}

The objects of $\inftygroup^{\bDelta^{\on{op}}}$ that lie in the essential image of the functor
$\on{Seg}^\bullet(-)$ are called ``complete Segal spaces." Thus, \thmref{t:Segal} implies that
the category of complete Segal spaces is equivalent to $\inftyCat$.

\sssec{}

We let 
$$\on{Seg}^{\bullet,\bullet}(\bC):=\on{Seg}^\bullet(\on{Seg}^\bullet(\bC))\in \inftygroup^{(\bDelta\times \bDelta)^{\on{op}}}$$
be the iteration of the above construction. I.e.,
$$\on{Seg}^{m,n}(\bC):=\on{Funct}([m]\times [n],\bC)\in \inftygroup.$$

\medskip

Let $\pi_h,\pi_v$ denote the two maps $\bDelta\times \bDelta\rightrightarrows \bDelta$. For 
a simplicial object $\be^\bullet$ of some $\infty$-category $\bE$ we let
$$\pi_h(\be^\bullet) \text{ and } \pi_v(\be^\bullet)$$ denote the corresponding
bi-simplicial objects of $\bE$.

\medskip

We have canonically defined maps
$$\pi_h(\on{Seg}^\bullet(\bC))\to \on{Seg}^{\bullet,\bullet}(\bC)\leftarrow \pi_h(\on{Seg}^\bullet(\bC))$$
corresponding to taking the functors
$$[m]\times [n]\to \bC$$
that constant along the second (for $h$) and first (for $v$) coordinate, respectively.

\sssec{}   \label{sss:bivariant}

Let $\bC$ be an $\infty$-category, with two distinguished classes of 1-morphisms $vert$ and $horiz$, satisfying
the assumptions of \secref{sss:two classes}.

\medskip

We let
$$\on{Cart}^{\bullet,\bullet}_{vert;horiz}(\bC)$$
denote the bi-simplicial groupoid, whose $(m,n)$-simplices is the full subgroupoid of
$$\on{Funct}([m]\times [n],\bC),$$
that consists of commutative diagrams
\begin{equation} \label{e:grid}
\CD
\bc_{0,0} @>>>  \bc_{1,0}  @>>>   ...   @>>>  \bc_{m-1,0}  @>>>  \bc_{m,0}  \\
@VVV   @VVV @VVV @VVV   @VVV \\
\bc_{0,1} @>>>  \bc_{1,1}  @>>>   ...   @>>>  \bc_{m-1,1}  @>>>  \bc_{m,1}  \\
@VVV   @VVV @VVV @VVV   @VVV \\
... @>>> ... @>>> ... @>>> ... @>>> ... \\
@VVV   @VVV @VVV @VVV   @VVV \\
\bc_{0,n-1} @>>>  \bc_{1,n-1}  @>>>   ...   @>>>  \bc_{m-1,n-1}  @>>>  \bc_{m,n-1}  \\
@VVV   @VVV @VVV @VVV   @VVV \\
\bc_{0,n} @>>>  \bc_{1,n}  @>>>   ...   @>>>  \bc_{m-1,n}  @>>>  \bc_{m,n}
\endCD
\end{equation}
in which every square is Cartesian, and in which all vertical arrows belong to $vert$ and all
horizontal arrows belong to $horiz$. 

\medskip

Note that we have a canonically defined maps in $\inftygroup^{(\bDelta\times \bDelta)^{\on{op}}}$
$$\pi_h(\on{Seg}^\bullet(\bC_{horiz}))\to \on{Cart}^{\bullet,\bullet}_{vert;horiz}(\bC)\leftarrow \pi_v(\on{Seg}^\bullet(\bC_{vert})).$$

At the level of $(m,n)$-simplices, they correspond to diagrams \eqref{e:grid}, in which (for $h$) 
the vertical maps are identity maps, and (for $v$) the horizontal maps are identity maps.

\sssec{}

Consider the functors
$$\IndCoh^!_{(\dgSch_{\on{Noeth}})_{\on{aft}}}:((\dgSch_{\on{Noeth}})_{\on{aft}})^{\on{op}}\to \StinftyCat_{\on{cont}}$$
(see \propref{p:upgrading IndCoh to functor}), and 
$$\IndCoh^*_{(\dgSch_{\on{Noeth}})_{\on{bdd-Tor}}}:((\dgSch_{\on{Noeth}})_{\on{bdd-Tor}})^{\on{op}}\to \StinftyCat_{\on{cont}}$$
(see \corref{c:upper * DG funct}),
where
$$(\dgSch_{\on{Noeth}})_{\on{bdd-Tor}}\subset (\dgSch_{\on{Noeth}})_{\on{ev-conconn}}$$
is the 1-full subcategory, where we retsrict 1-morphisms to those of bounded Tor dimension.

\medskip

Consider the corresponding maps in $\inftygroup^{\bDelta^{\on{op}}}$
$$\on{Seg}^\bullet(\IndCoh^!_{(\dgSch_{\on{Noeth}})_{\on{aft}}}):
\on{Seg}^\bullet(((\dgSch_{\on{Noeth}})_{\on{aft}})^{\on{op}}) \to \on{Seg}^\bullet((\StinftyCat_{\on{cont}})^{\on{op}}),$$
\text{ and }
$$\on{Seg}^\bullet(\IndCoh^*_{(\dgSch_{\on{Noeth}})_{\on{bdd-Tor}}}):
\on{Seg}^\bullet(((\dgSch_{\on{Noeth}})_{\on{bdd-Tor}})^{\on{op}})\to
\on{Seg}^\bullet((\StinftyCat_{\on{cont}})^{\on{op}}),$$
respectively. 

\sssec{}

We consider the category $\dgSch_{\on{Noeth}}$ equipped with the following classes of 1-morphisms:
$$vert=\on{bdd-Tor} \text{ and } horiz=\on{aft}.$$

\medskip

We will prove:

\begin{prop} \label{p:!* IndCoh}
There exists a canonically defined map in $\inftygroup^{(\bDelta\times \bDelta)^{\on{op}}}$
$$\IndCoh^{*!}_{(\dgSch_{\on{Noeth}})_{\on{bdd-Tor;aft}}}: 
\on{Cart}^{\bullet,\bullet}_{\on{bdd-Tor;aft}}(\dgSch_{\on{Noeth}})\to 
\on{Seg}^{\bullet,\bullet}((\StinftyCat_{\on{cont}})^{\on{op}})$$
that makes the following diagrams commute
$$
\CD
\on{Cart}^{\bullet,\bullet}_{\on{bdd-Tor;aft}}(\dgSch_{\on{Noeth}})  @>{\IndCoh^{*!}_{(\dgSch_{\on{Noeth}})_{\on{bdd-Tor;aft}}}}>>  
\on{Seg}^{\bullet,\bullet}((\StinftyCat_{\on{cont}})^{\on{op}}) \\
@AAA        @AAA  \\
\pi_h(\on{Seg}^\bullet((\dgSch_{\on{Noeth}})_{\on{aft}}))  @>>{\on{Seg}^\bullet(\IndCoh^!_{(\dgSch_{\on{Noeth}})_{\on{aft}}})}>
\pi_h(\on{Seg}^\bullet((\StinftyCat_{\on{cont}})^{\on{op}}))  
\endCD
$$
and
$$
\CD
\on{Cart}^{\bullet,\bullet}_{\on{bdd-Tor;aft}}(\dgSch_{\on{Noeth}})  @>{\IndCoh^{*!}_{(\dgSch_{\on{Noeth}})_{\on{bdd-Tor;aft}}}}>>  
\on{Seg}^{\bullet,\bullet}((\StinftyCat_{\on{cont}})^{\on{op}}) \\
@AAA        @AAA  \\
\pi_v(\on{Seg}^\bullet((\dgSch_{\on{Noeth}})_{\on{bdd-Tor}}))  @>>{\on{Seg}^\bullet(\IndCoh^*_{(\dgSch_{\on{Noeth}})_{\on{bdd-Tor}}})}>
\pi_v(\on{Seg}^\bullet((\StinftyCat_{\on{cont}})^{\on{op}})).
\endCD
$$
\end{prop}

\begin{proof}

Given a Cartesian diagram
\begin{equation} \label{e:grid Sch}
\CD
S_{0,0} @>{f_?}>>    ...   @>{f_?}>>  S_{m,0}  \\
@V{g_?}VV    @VVV    @VV{g_?}V \\
...  @>>>   ...  @>>> ...   \\
@V{g_?}VV    @VVV   @VV{g_?}V \\
S_{0,n}  @>>{f_?}>   ...   @>>{f_?}>   S_{m,n}
\endCD
\end{equation}
in $(\dgSch_{\on{Noeth}})$ with horizontal arrows being locally almost of finite type, and vertical arrows
of bounded Tor dimension, we need to construct a commutative diagram in $(\StinftyCat_{\on{cont}})^{\on{op}}$:

\begin{equation} \label{e:diag up}
\CD
\IndCoh(S_{0,0}) @<{(f_?)^!}<<    ...   @<{(f_?)^!}<<  \IndCoh(S_{m,0})  \\
@A{(g_?)^{\IndCoh,*}}AA    @AAA    @AA{(g_?)^{\IndCoh,*}}A \\
...  @>>>   ...  @>>> ...   \\
@A{(g_?)^{\IndCoh,*}}AA    @AAA   @AA{(g_?)^{\IndCoh,*}}A \\
\IndCoh(S_{0,n})  @<<{(f_?)^!}<    ...   @<<{(f_?)^!}<    \IndCoh(S_{m,n}).
\endCD
\end{equation}

Moreover, the assignment
$$\text{\eqref{e:grid Sch}}\, \rightsquigarrow \, \text{\eqref{e:diag up}}$$
must be compatible with maps
$$[m]\times [n]\to [m']\times [n']$$
in $\bDelta\times \bDelta$.

\medskip

We start with the commutative diagram
\begin{equation} \label{e:diag left}
\CD
\IndCoh(S_{0,0}) @<{(f_?)^!}<<    ...   @<{(f_?)^!}<<  \IndCoh(S_{m,0})  \\
@V{(g_?)^\IndCoh_*}VV    @VVV    @VV{(g_?)^\IndCoh_*}V \\
...  @>>>   ...  @>>> ...   \\
@V{(g_?)^\IndCoh_*}VV    @VVV   @VV{(g_?)^\IndCoh_*}V \\
\IndCoh(S_{0,n})  @<<{(f_?)^!}<    ...   @<<{(f_?)^!}<   \IndCoh(S_{m,n})
\endCD
\end{equation}
whose datum is is given by the functor
$$\IndCoh_{(\dgSch_{\on{Noeth}})_{\on{corr:bdd-Tor;aft}}}:=
\IndCoh_{(\dgSch_{\on{Noeth}})_{\on{corr:all;aft}}}|_{(\dgSch_{\on{Noeth}})_{\on{corr:bdd-Tor;aft}}}.$$

The diagram \eqref{e:diag up} is obtained from \eqref{e:diag left} by passage to the left adjoints
along the vertical arrows. The fact that the arising natural transformations are isomorphisms is given by 
\propref{p:! and * pullback under eventually coconn}.

\end{proof}

\section{Descent}  \label{s:descent}

In this section we continue to assume that all DG schemes are Noetherian.

\ssec{A result on conservativeness}

In this subsection we will establish a technical result that will be useful in the sequel. 

\sssec{}

Let $f:S_1\to S_2$ be a map almost of finite type. 

\begin{prop}   \label{p:cons}
Assume that $f$ is surjective at the level of geometric points. Then 
the functor $f^!:\IndCoh(S_2)\to \IndCoh(S_1)$ is conservative.
\end{prop}

\begin{proof}

By \corref{c:reduced generation} we can assume that both $S_1$ are $S_2$ are classical and reduced. 
By Noetherian induction, we can assume that the statement is true for all proper
closed subschemes of $S_2$. 

\medskip

Hence, we can replace $S_1$ and $S_2$ be open subschemes and thus assume that $f$ is smooth.
In this case, by \corref{c:differ by line}, $f^!$ differs from $f^{\IndCoh,*}$ by 
an automorphism of $\IndCoh(S_1)$. So, it is enough to show that $f^{\IndCoh,*}$ is fully faithful.

\medskip

By \lemref{l:away from generic}, we can assume that $S_2$ is the spectrum of a field. In this case
$S_1$ is a smooth scheme over this field, and the assertion becomes manifest.

\end{proof}

\begin{cor} \label{c:proper generation}
Let $f:S_1\to S_2$ be a proper map, surjective at the level of geometric points.
Then the essential image of the functor $f_*:\Coh(S_1)\to \Coh(S_2)$
generates $\IndCoh(S_2)$.
\end{cor}

\ssec{h-descent} \label{ss:cosimp general}

\sssec{}

Let $f:S_1\to S_2$ be a map locally almost of finite type. We can form the cosimplicial category 
$\IndCoh^!(S_1^\bullet/S_2)$ \emph{using the !-pullback functors},
and !-pullback defines an augmentation
\begin{equation} \label{e:map to total general}
\IndCoh(S_2)\to \on{Tot}\left(\IndCoh^!(S_1^\bullet/S_2)\right).
\end{equation}

We shall say that $\IndCoh$ satisfies !-descent for $f$ if \eqref{e:map to total general}
is an equivalence.

\medskip

The main theorem of this section is:

\begin{thm}  \label{t:h descent}
$\IndCoh$ satisfies !-descent for maps that are covers for the h-topology. 
\end{thm}

Taking into account \propref{p:Zariski descent} and \cite[Theorem 4.1]{GoLi}, to prove
\thmref{t:h descent}, it suffices to prove the following:

\begin{prop} \label{p:proper descent}
$\IndCoh$ satisfies !-descent for maps that are proper and surjective at the level of
geometric points.
\end{prop}

\begin{rem}
In \secref{ss:ff descent} we will give a direct proof of the fact that 
$\IndCoh$ satisfies !-descent for fppf maps. 
\end{rem}

\begin{proof}[Proof of \propref{p:proper descent}]

Consider the forgetful functor 
$$\on{ev}_0:\on{Tot}\left(\IndCoh^!(S_1^\bullet/S_2)\right)\to \IndCoh(S_1)$$
given by evaluation on the 0-th term. Its composition with the functor
\eqref{e:map to total general} is the functor $f^!$. 

\medskip

We will show that $\on{ev}_0$ admits a left adjoint, to be denoted $\on{ev}_0^L$, such that the natural map
\begin{equation} \label{e:map between monads}
\on{ev}_0\circ \on{ev}_0^L\to f^!\circ f^{\IndCoh}_{*}
\end{equation}
is an isomorphism. We will also show that both pairs
$$(\on{ev}_0^L,\on{ev}_0) \text{ and } (f^{\IndCoh}_{*},f^!)$$
satisfy the conditions of the Barr-Beck-Lurie theorem (see \cite{DG}, Sect. 3.1). This would imply
the assertion of the proposition, as the isomorphism \eqref{e:map between monads}
would imply that the resulting two monads acting on $\IndCoh(S_1)$ are isomorphic.

\medskip

Consider the augmented simplicial scheme $S_1\underset{S_2}\times (S_1^\bullet/S_2)$, which
is equipped with a map of simplicial schemes
$$f^\bullet:S_1\underset{S_2}\times (S_1^\bullet/S_2)\to S_1^\bullet/S_2.$$
The operation of !-pullback defines a functor between the corresponding cosimplicial categories
$$(f^\bullet)^!:\IndCoh(S_1^\bullet/S_2)\to \IndCoh(S_1\underset{S_2}\times (S_1^\bullet/S_2)).$$
By \propref{p:proper base change}, the term-wise adjoint $(f^\bullet)^{\IndCoh}_*$
is also a functor of cosimplicial categories. 

\medskip

In particular, we obtain a pair of adjoint functors
$$\on{Tot}((f^\bullet)^{\IndCoh}_*):\on{Tot}\left(\IndCoh^!(S_1\underset{S_2}\times (S_1^\bullet/S_2))\right)\rightleftarrows
\on{Tot}\left(\IndCoh^!(S_1^\bullet/S_2)\right):\on{Tot}((f^\bullet)^!).$$

\medskip

However, the augmented simplicial scheme $S_1\underset{S_2}\times (S_1^\bullet/S_2)$
is split by $S_1$, and therefore, the functors of !-pullback from the augmentation and evaluation on the splitting
define mutual inverse equivalences
$$\on{Tot}\left(\IndCoh^!(S_1\underset{S_2}\times (S_1^\bullet/S_2))\right)\simeq \IndCoh(S_1).$$
The resulting functor
$$\on{Tot}\left(\IndCoh^!(S_1^\bullet/S_2)\right)\overset{\on{Tot}((f^\bullet)^!)}\longrightarrow
\on{Tot}\left(\IndCoh^!(S_1\underset{S_2}\times (S_1^\bullet/S_2))\right)\simeq \IndCoh(S_1)$$
is easily seen to be canonically isomorphic to the functor $\on{ev}_0$. Thus,
the functor
$$\IndCoh(S_1)\simeq \on{Tot}\left(\IndCoh^!(S_1\underset{S_2}\times (S_1^\bullet/S_2))\right)
\overset{\on{Tot}((f^\bullet)^{\IndCoh}_*)}\longrightarrow 
\on{Tot}\left(\IndCoh^!(S_1^\bullet/S_2)\right)$$ provides
the desired functor $\on{ev}_0^L$.  

\medskip

The composition $\on{ev}_0\circ \on{ev}_0^L$ identifies with the functor 
$$\on{pr}_1^!\circ \on{pr}_2{}_*^{\IndCoh},$$
where 
\begin{gather*}
\xy
(-15,0)*+{S_1}="A";
(0,15)*+{S_1\underset{S_2}\times S_1}="B";
(15,0)*+{S_1,}="C";
{\ar@{->}_{\on{pr}_1} "B";"A"};
{\ar@{->}^{\on{pr}_2} "B";"C"};
\endxy
\end{gather*}
and the fact that the map \eqref{e:map between monads} is an isomorphism follows from
\propref{p:proper base change}.

\medskip

Finally, let us verify the conditions of the Barr-Beck-Lurie theorem. The functors $f^!$ and $\on{ev}_0$,
being continuous, commute with all colimits. The functor $\on{ev}_0$ is conservative by definition, and the functor
$f^!$ is conservative by \propref{p:cons}.

\end{proof}

\ssec{Faithfully flat descent}  \label{ss:ff descent}

\sssec{}

Let is recall that a map $f:S_1\to S_2$ is said to be fppf if:

\begin{itemize}

\item It is almost of finite type;

\item It is flat;

\item It is surjective on geometric points.

\end{itemize}

The following is a corollary of \thmref{t:h descent}: 

\begin{thm} {\em(Lurie)} \label{t:fppf descent}
$\IndCoh$ satisfies !-descent with respect to fppf morphisms.
\end{thm}

Below we will give an alternative proof that does not rely on \cite{GoLi}. This will
result from \propref{p:proper descent}, and the combination of the following two assertions:

\begin{prop}  \label{p:Nisnevich descent}
$\IndCoh$ satisfies !-descent with respect to Nisnevich covers.
\end{prop}

\begin{prop} \label{p:Nisnevich+finite flat}
Any presheaf on $\dgSch_{\on{Noeth}}$ that satisfies descent 
with respect to proper surjective maps and Nisnevich covers
satisfies fppf descent.
\end{prop}

\sssec{}

To prove \propref{p:Nisnevich descent}, we will prove a more general statement,
which is itself a particular case of \thmref{t:fppf descent}:

\begin{prop}  \label{p:!-descent smooth}
$\IndCoh$ satisfies !-descent with respect to smooth surjective maps.
\end{prop}

\begin{proof} 

By \corref{c:! tensor up smooth}, the co-simplicial category $\IndCoh^!(S_1^\bullet/S_2)$
identifies with
$$\QCoh^*(S_1^\bullet/S_2)\underset{\QCoh(S_2)}\otimes \IndCoh(S_2),$$
where $\QCoh^*(S_1^\bullet/S_2)$ is the usual co-simplicial category attached to
the simplicial DG schemes $S_1^\bullet/S_2$ and the functor
$$\QCoh^*_{\dgSch}:(\dgSch)^{\on{op}}\to \StinftyCat_{\on{cont}}.$$

\medskip

Since $\QCoh(S_2)$ is rigid and $\IndCoh(S_2)$ is dualizable, by \cite{DG}, Corollaries 6.4.2 and 4.3.2,
the operation
$$\IndCoh(S_2)\underset{\QCoh(S_2)}\otimes -:\QCoh(S_2)\mmod\to \StinftyCat_{\on{cont}}$$
commutes with limits.

\medskip

Hence, from we obtain that the natural map
$$\on{Tot}\left(\QCoh^*(S_1^\bullet/S_2)\right)\underset{\QCoh(S_2)}\otimes \IndCoh(S_2)\to 
\on{Tot}\left(\QCoh^*(S_1^\bullet/S_2)\underset{\QCoh(S_2)}\otimes \IndCoh(S_2)\right)$$
is an equivalence. Thus, we obtain an equivalence
\begin{equation} \label{e:tensoring with Tot}
\on{Tot}\left(\QCoh^*(S_1^\bullet/S_2)\right)\underset{\QCoh(S_2)}\otimes \IndCoh(S_2)\simeq 
\on{Tot}\left(\IndCoh^!(S_1^\bullet/S_2)\right).
\end{equation}

\medskip

By faithfully flat descent for $\QCoh$, we obtain that the natural map
$$\QCoh(S_2)\to \on{Tot}\left(\QCoh^*(S_1^\bullet/S_2)\right)$$
is an equivalence. So, the assertion of the proposition follows from
\eqref{e:tensoring with Tot}. 

\end{proof}

\sssec{}

Note that the same proof (using \propref{p:smooth map} instead of \corref{c:! tensor up smooth}) implies
the following statement:

\medskip

Let $f:S_1\to S_2$ be an eventually coconnective map. We can form the cosimplicial category 
$\IndCoh^*(S_1^\bullet/S_2)$ \emph{using the *-pullback functors},
and *-pullback defines an augmentation
\begin{equation} \label{e:map to total general *}
\IndCoh(S_2)\to \on{Tot}\left(\IndCoh^*(S_1^\bullet/S_2)\right).
\end{equation}

We shall say that $\IndCoh$ satisfies *-descent for $f$ if \eqref{e:map to total general *}
is an equivalence. We have:

\begin{prop}  \label{p:*-descent smooth}
$\IndCoh$ satisfies *-descent with respect to smooth surjective maps.
\end{prop}

\ssec{Proof of \propref{p:Nisnevich+finite flat}}

The assertion of the proposition, as well as the proof given below, are apparently well-known.
The author has learned it from J.~Lurie.

\sssec{}

We have the following two general lemmas, valid for any $(\infty,1)$-category $\bC$ and
a presheaf of $(\infty,1)$-categories $P$ on it: 

\begin{lem}
If a morphism $f$ admits a section, then any presheaf satisfies descent with respect to $f$.
\end{lem}

\begin{lem}
Let $S''\overset{g''}\to S'\overset{g'}\to S$ be maps,
such that $P$ satisfies descent with respect to $g''$. Then $P$
satisfies descent with respect to 
$g'$ if and only if satisfies descent with respect to $g'\circ g''$. 
\end{lem}

Consider now a Cartesian square
$$
\CD
S'_1 @>{g_1}>>  S_1 \\
@V{f'}VV   @VV{f}V   \\
S'_2 @>{g_2}>>  S_2.
\endCD
$$

\begin{cor}  \label{c:strategy}
If a presheaf $P$ satisfies descent with respect to $g_1$, $g_2$ and $f'$, then it also
satisfies descent with respect to $f$. 
\end{cor}

\sssec{}

First, applying \corref{c:strategy} and \propref{p:proper descent}, we can replace
$S_2$ by ${}^{cl}\!S_2$, and $S_1$ by $S_1\underset{S_2}\times {}^{cl}\!S_2\simeq {}^{cl}\!S_1$. 
Thus, we can assume that the DG schemes involved are classical.

\sssec{}

We are going to show that for any faithfully flat map $f$ of finite type, there exists a diagram
$$S_2''\overset{g''}\to S_2'\overset{g'}\to S_2,$$
such that 

\begin{itemize}

\item The composition $g'\circ g'':S''_2\to S_2$ lifts to a map $S''_2\to S_1$.

\item $\IndCoh$ satisfies !-descent with respect to any map obtained as 
base change of either $g'$ or $g''$.

\end{itemize}

In view of \corref{c:strategy}, this will prove \thmref{t:fppf descent}. 

\medskip

Thus, it remains to show:

\begin{prop}  \label{p:factorization}
For any faithfully flat map of finite type between classical Noetherian schemes $f:S_1\to S_2$, there exist maps
$S''_2\overset{g''}\to S'_2\overset{g'}\to S_2$ such that $g'\circ g''$ lifts to a map $S''_2\to S_1$,
and such that (a) $g'$ is a Nisnevich cover, and (b) $g''$ is finite, flat and surjective. 
\end{prop}

\ssec{Proof of \propref{p:factorization}}

\sssec{Step 1}

First, we claim that we can assume that $S_2$ is the spectrum of a local Henselian Noetherian ring.

\medskip

Indeed, let $x_2\in S_2$ be a point. Let $A_2$ denote the Henselisation of $S_2$ at $x_2$. 
Let $$\Spec(B_2)\to \Spec(A_2)$$ be a finite, flat and surjective map such that the composition
$$\Spec(B_2)\to \Spec(A_2)\to S_2$$
lifts to a map $\Spec(B_2)\to S_1$.

\medskip

Write $A_2=\underset{i\in I}{colim}\, A^i_2$, where $I$ is a filtered category, and where each 
$\Spec(A^i_2)$ is endowed with an \'etale map to $S_2$, and contains a point $x^i_2$ that
maps to $x_2$ inducing anisomorphism on residue fields. With no restriction of generality,
we can assume that each $\Spec(A^i_2)$ is connected.

\medskip

Since $\Spec(B_2)\to \Spec(A_2)$ is finite and flat, there exists an index $i_0\in I$, and 
a finite and flat (and automatically surjective) map 
$$\Spec(B^{i_0}_2)\to \Spec(A^{i_0}_2)$$ equipped
with an isomorphism
$$\Spec(B_2)\simeq \Spec(B^{i_0}_2)\underset{\Spec(A^{i_0}_2)}\times \Spec(A_2).$$

For $i\in I_{i_0}$ set 
$$\Spec(B^i_2)\simeq \Spec(B^{i_0}_2)\underset{\Spec(A^{i_0}_2)}\times \Spec(A^i_2).$$

Replacing $I$ be $I_{/i_0}$, we obtain an isomorphism 
$$B_2 \simeq \underset{i\in I}{colim}\, B^i_2,$$
where $I$ is filtered.

\medskip

Since $S_1$ is of finite type over $S_2$, the morphism
$$\Spec(B_2)\to S_1$$
of schemes over $S_2$ factors as
$$\Spec(B_2)\to \Spec(B^{i'}_2)\to S_1$$
for some index $i'\in I$.

\medskip

We let $S'_2$ be the (finite) disjoint union of the above schemes 
$\Spec(A^{i'}_2)$ that covers $S_2$, and we let $S''_2$ to be 
the corresponding union of the schemes $\Spec(B^{i'}_2)$.

\sssec{Step 2}

Next, we are going to show that we can assume that the map $f:S_1\to S_2$ is quasi-finite.

\medskip

With no restriction of generality, we can assume that $S_1$ is affine. 
Let $x_2$ be the closed point of $S_2$, and let $S'_1$ be the fiber of $S_1$ over it. 
By assumption, $S'_1$ is a scheme of finite type over $k'$,  the residue field
of $x_2$. By Noether normalization, there exists a map $h':S'_1\to \BA^n_{k'}$,
which is finite and flat \emph{at some} closed point $x_1\in S'_1$. Replacing
$S_1$ by an open subset containing $x_1$, we can assume that $h'$ comes 
from a map $h:X\to \BA^n\times S_2$.  By construction, the map $h$ is quasi-finite 
and flat at $x_1$. We can replace $S_1$ by an even smaller open subset 
containing $x_1$, and thus assume that $h$ is quasi-finite and flat on all of $S_1$.

\medskip

The sought-for quasi-finite fppf map is
$$(0\times S_2) \underset{\BA^n\times S_2}\times S_1 \to S_2.$$

\sssec{Step 3} 

Thus, we have reduced the situation to the case when $S_2$ is the spectrum of a local Henselian
Noetherian ring, and $S_1$ is affine and its map to $S_2$ is quasi-finite. We claim that in this
case $S_1$ contains a connected component over which $f$ is finite, see \lemref{l:qf} below. 
This proves \propref{p:factorization}. 

\qed

\begin{lem} \label{l:qf}
Let $\phi:Z\to Y=\Spec(R)$ be a quasi-finite and separated map, where $R$ is a local
Henselian Noetherian ring. Suppose, moreover, that the closed point of $Y$ is in the image
of $\phi$. Then $Z$ contains a connected component $Z'$, such that
$\phi|_{Z'}$ is finite.
\end{lem}

\begin{proof}

By Zariski's Main Theorem, we can factor $\phi$ as 
\begin{gather*}
\xy
(-10,0)*+{Z}="A";
(10,0)*+{\ol{Z}}="B";
(0,-10)*+{Y,}="C";
{\ar@{->}^{j} "A";"B"};
{\ar@{->}_{\phi} "A";"C"};
{\ar@{->}^{\ol\phi} "B";"C"};
\endxy
\end{gather*}
where $j$ is an open embedding and $\ol\phi$ is finite. 

\medskip

Since $\ol{Z}$ is finite over $Y$, and $Y$ is Henselian, $\ol{Z}$ is a union
of connected components, each of which is local. Let $z$ be a closed point of
$Z$ that maps to the closed point $y$ of $Y$. Let $\ol{Z}{}'$ be the
connected component of $\ol{Z}$ that contains $z$. Set $Z':=\ol{Z}{}'\cap Z$.
It remains to show that the open embedding $Z'\hookrightarrow \ol{Z}{}'$
is an equality.

\medskip

However, since $\ol{Z}{}'$ is finite over $Y$, we obtain that $z$ is closed in 
$\ol{Z}{}'$. We obtain that $Z'$ contains the \emph{unique} closed point
of $\ol{Z}{}'$, which implies the assertion.

\end{proof}

\section{Serre duality}  \label{s:Serre}

In this section we restrict our attention to the category $\dgSch_{\on{aft}}$ of DG schemes almost
of finite type over $k$. 

\ssec{Self-duality of $\IndCoh$}

\sssec{}   \label{sss:intr self-duality}

Consider the category $(\dgSch_{\on{aft}})_{\on{corr:all;all}}$ and the functor
$$\IndCoh_{(\dgSch_{\on{aft}})_{\on{corr:all;all}}}:(\dgSch_{\on{aft}})_{\on{corr:all;all}}\to
 \StinftyCat_{\on{cont}}.$$

\medskip

Note that by construction, there exists a canonical involutive equivalence
$$\varpi:(\dgSch_{\on{aft}})_{\on{corr:all;all}}^{\on{op}}\simeq (\dgSch_{\on{aft}})_{\on{corr:all;all}},$$
obtained by interchanging the roles of vertical and horizontal arrows. 

\sssec{}

Let $\StinftyCat^{\on{dualizable}}_{\on{cont}}$ denote the full subcategory of $\StinftyCat_{\on{cont}}$
formed by dualizable categories. Recall also (see., e.g., \cite{DG}, Sect. 2.3) that the category
$\StinftyCat^{\on{dualizable}}_{\on{cont}}$ carries a canonical involutive equivalence
$$\on{dualization}:(\StinftyCat^{\on{dualizable}}_{\on{cont}})^{\on{op}}\simeq \StinftyCat^{\on{dualizable}}_{\on{cont}}.$$
given by taking the dual category: $\bC\mapsto \bC^\vee$.

\medskip

By construction, the functor $\IndCoh_{(\dgSch_{\on{aft}})_{\on{corr:all;all}}}$ takes values in the subcategory
$$\StinftyCat^{\on{comp.gener.}}_{\on{cont}}\subset 
\StinftyCat^{\on{dualizable}}_{\on{cont}}\subset \StinftyCat_{\on{cont}}.$$

\sssec{}

We will encode the Serre duality structure of the functor $\IndCoh$ by the following theorem:

\begin{thm} \label{t:Serre duality}
The following diagram of functors 
\begin{equation} \label{e:Serre duality diag}
\CD
(\dgSch_{\on{aft}})_{\on{corr:all;all}}^{\on{op}}   @>{\IndCoh_{(\dgSch_{\on{aft}})_{\on{corr:all;all}}}^{\on{op}}}>>  
(\StinftyCat^{\on{dualizable}}_{\on{cont}})^{\on{op}}  \\
@V{\varpi}VV     @VV{\on{dualization}}V   \\
(\dgSch_{\on{aft}})_{\on{corr:all;all}}   @>{\IndCoh_{(\dgSch_{\on{aft}})_{\on{corr:all;all}}}}>> \StinftyCat^{\on{dualizable}}_{\on{cont}}.
\endCD
\end{equation}
is canonically commutative. Moreover, this structure is canonically compatible
with the involutivity structure on the equivalences $\varpi$ and $\on{dualization}$.
\end{thm}

The rest of this subsection is devoted to the proof of this theorem.

\sssec{}  \label{sss:abstract duality}

Let us consider the following general paradigm. Let $\bC$ be a symmetric monoidal category,
which is rigid. Then the assignment $\bc\mapsto \bc^\vee$ defines a canonical involutive
self-equivalence
$$\bC^{\on{op}}\overset{\on{dualization}}\longrightarrow \bC.$$

\medskip

If $\bC_1$ is another rigid symmetric monoidal category, and $F:\bC\to \bC_1$ 
is a symmetric monoidal functor, then the diagram
$$
\CD
\bC^{\on{op}}   @>{F^{\on{op}}}>>  (\bC_1)^{\on{op}}  \\
@V{\on{dualization}}VV     @VV{\on{dualization}}V   \\
\bC  @>{F}>> \bC_1
\endCD
$$
naturally commutes in a way compatible with the involutive structure on the vertical arrows. 

\sssec{}  \label{sss:abs dual corr}

Let $\bC^0$ be a category with fiber products and a final object, 
and let $\bC^0_{\on{corr:all;all}}$ be the resulting category of correspondences.
We make $\bC^0_{\on{corr:all;all}}$ into a symmetric monoidal category by means of the product 
operation.

\medskip

The category $\bC^0_{\on{corr:all;all}}$ is automatically rigid: every $\bc\in \bC^0_{\on{corr:all;all}}$ is self-dual,
where the counit of the duality datum is given by the correspodence
$$
\CD
\bc  @>{\Delta_{\bc}}>>  \bc\times \bc \\
@VVV  \\
\on{pt},
\endCD
$$
(here $\on{pt}$ is the final object in $\bC^0$) and the unit is given by the correspondence
$$
\CD
\bc  @>>>  \on{pt} \\
@V{\Delta_{\bc}}VV  \\
\bc\times \bc.
\endCD
$$

\medskip

The following results from the definitions (a full proof will be given in \cite{GR3}):

\begin{lem}
The involutive self-equivalence 
$$\bC^0_{\on{corr:all;all}} \overset{\on{dualization}}\longrightarrow (\bC^0_{\on{corr:all;all}})^{\on{op}},$$
is canonically isomorphic to $\varpi$.
\end{lem}

\sssec{}

To prove \thmref{t:Serre duality} we apply the discussion of Sections
\ref{sss:abstract duality} and \ref{sss:abs dual corr} to $$\bC^0=\dgSch_{\on{aft}},\,\,
\bC:=(\dgSch_{\on{aft}})_{\on{corr:all;all}},\,\, \bC_1:=\StinftyCat^{\on{dualizable}}_{\on{cont}},$$ and 
$F:=\IndCoh_{(\dgSch_{\on{aft}})_{\on{corr:all;all}}}$, where the 
symmetric monoidal structure on $F$ is given by \thmref{t:upper shriek monoidal} and
\corref{c:upper shriek monoidal}. 

\qed

\ssec{Properties of self-duality}

In this subsection shall explain what \thmref{t:Serre duality} says in concrete terms. 

\sssec{}

The assertion of \thmref{t:Serre duality} at the level of objects means that 
for every $S\in \dgSch_{\on{aft}}$ we have a canonical equivalence
\begin{equation} \label{e:Serre duality}
\bD_S^{\on{\on{Serre}}}:(\IndCoh(S))^\vee\simeq \IndCoh(S),
\end{equation}
which in the sequel we will refer to as Serre duality on a given scheme (see \secref{ss:classical Serre}
for the relation to a more familiar formulation of Serre duality). 

\medskip

Moreover, the equivalence obtained by iterating \eqref{e:Serre duality}
$$((\IndCoh(S))^\vee)^\vee\simeq (\IndCoh(S))^\vee\simeq \IndCoh(S)$$
is canonically isomorphic to the tautological equivalence $((\IndCoh(S))^\vee)^\vee\simeq \IndCoh(S)$.

\sssec{} \label{sss:explicit duality data}

The construction of the commutative diagram \eqref{e:Serre duality diag} at the level of objects amounts
to the following description of the duality data 
$$\epsilon_S:\IndCoh(S)\otimes \IndCoh(S)\to \Vect  \text{ and } \mu_S:\Vect\to \IndCoh(S)\otimes \IndCoh(S).$$

\medskip

The functor $\epsilon_S$ is the composition
$$\IndCoh(S)\otimes \IndCoh(S) \overset{\boxtimes}\longrightarrow \IndCoh(S\times S)
\overset{\Delta_S^!}\longrightarrow \IndCoh(S)\overset{(p_S)^{\IndCoh}_*}\to \Vect,$$
where $p_S$ is the projection $S\to \on{pt}$.

\medskip

The functor $\mu_S$ is the composition
$$\Vect\overset{p_S^!}\longrightarrow \IndCoh(S) \overset{\Delta^{\IndCoh}_*}\longrightarrow  \IndCoh(S\times S)\simeq
\IndCoh(S)\otimes \IndCoh(S).$$

\medskip

The fact that the functors $(\epsilon_S,\mu_S)$ specified above indeed define a duality data for $\IndCoh(S)$
is an easy diagram chase. So, the content of \thmref{t:Serre duality} is the higher categorical
functoriality of this construction.

\sssec{}

At the level of $1$-morphisms, the construction in \thmref{t:Serre duality} implies that 
for a morphism $f:S_1\to S_2$ we have the following commutative diagrams of functors
\begin{equation} \label{e:dual of f_*}
\CD
(\IndCoh(S_1))^\vee  @>{\bD_{S_1}^{\on{\on{Serre}}}}>> \IndCoh(S_1)  \\
@A{(f^{\IndCoh}_*)^\vee}AA   @AA{f^!}A  \\
(\IndCoh(S_2))^\vee  @>{\bD_{S_2}^{\on{\on{Serre}}}}>> \IndCoh(S_2)
\endCD
\end{equation}
and
\begin{equation} \label{e:dual of f^!}
\CD
(\IndCoh(S_1))^\vee  @>{\bD_{S_1}^{\on{\on{Serre}}}}>> \IndCoh(S_1)  \\
@V{(f^!)^\vee}VV   @VV{f^{\IndCoh}_*}V  \\
(\IndCoh(S_2))^\vee  @>{\bD_{S_2}^{\on{\on{Serre}}}}>> \IndCoh(S_2)
\endCD
\end{equation}

Moreover, the isomorphism obtained by iteration 
$$((f^{\IndCoh}_*)^\vee)^\vee\simeq (f^!)^\vee\simeq f^{\IndCoh}_*$$
is canonically isomorphic to the tautological isomorphism 
$((f^{\IndCoh}_*)^\vee)^\vee\simeq f^{\IndCoh}_*$.

\medskip

In particular, the dual of the functor $(p_S)^{\IndCoh}_*:\IndCoh(S)\to \Vect$ is $(p_S)^!$,
and vice versa. 

\medskip

Again, we note that the explicit description of the duality data given in \secref{sss:explicit duality data}
makes the commutativity of the above diagrams evident. 

\ssec{Compatibility with duality on $\QCoh$}

\sssec{}

Recall that the category $\QCoh(S)$, being a rigid monoidal category,
is also self-dual (see \cite{DG}, Sect. 6.1.).  We denote the corresponding functor by
\begin{equation} \label{e:naive duality}
\bD_S^{\on{naive}}:(\QCoh(S))^\vee\to \QCoh(S).
\end{equation}


\medskip

The corresponding pairing
$$\QCoh(S)\otimes \QCoh(S)\to \Vect$$
is by definition given by
$$\QCoh(S)\otimes \QCoh(S)\overset{\boxtimes}\longrightarrow \QCoh(S\times S)\overset{\Delta_S^*}\longrightarrow
\QCoh(S)\overset{\Gamma(S,-)}\longrightarrow k.$$

\sssec{}

Recall the functor 
$$\Psi_S:\IndCoh(S)\to \QCoh(S).$$

Passing to dual functors, and using the identifications $\bD_S^{\on{naive}}$ and $\bD_S^{\on{Serre}}$,
we obtain a functor
\begin{equation} \label{e:Psi check}
\Psi^\vee_S:\QCoh(S)\to \IndCoh(S).
\end{equation}

We claim:

\begin{prop}  \label{p:Psi check}
The functor $\Psi^\vee_S$ identifies canonically with the functor $\Upsilon_S$ of 
\eqref{e:QCoh to IndCoh}, i.e., 
$$\CE\mapsto \CE\otimes \omega_S,$$
where the latter is understood in the sense of the action of $\QCoh(S)$ on $\IndCoh(S)$.  
\end{prop}

\begin{proof}
Let 
$$\langle-,-\rangle_{\IndCoh(S)} \text{ and } \langle-,-\rangle_{\QCoh(S)}$$
denote the functors
$$\IndCoh(S)\times \IndCoh(S)\to \Vect \text{ and } \QCoh(S)\times \QCoh(S)\to \Vect$$
given by the counits of the duality isomorphisms $\bD_S^{\on{Serre}}$ and $\bD_S^{\on{naive}}$,
respectively.

\medskip

We need to show that for $\CF\in \IndCoh(S)$ and $\CE\in \QCoh(S)$ we have a canonical
isomorphism
\begin{equation} \label{e:two dualities}
\langle\CF,\CE\otimes \omega_S\rangle_{\IndCoh(S)}\simeq 
\langle\Psi_S(\CF),\CE\rangle_{\QCoh(S)}.
\end{equation}

By the construction of $\bD^{\on{Serre}}_S$, we have:
$$\langle\CF,\CE\otimes \omega_S\rangle_{\IndCoh(S)}\simeq
\Gamma\left(S,\Psi_S(\CF\sotimes (\CE\otimes \omega_S))\right).$$

Using \eqref{e:QCoh linearity on sotimes}, we have:
$$\CF\sotimes (\CE\otimes \omega_S)\simeq \CE\otimes (\CF\sotimes \omega_S)\simeq \CE\otimes \CF.$$

Hence, the left-hand side in \eqref{e:two dualities} identifies with
$$\Gamma\left(S,\Psi_S(\CE\otimes \CF)\right)\simeq \Gamma(S,\CE\otimes \Psi_S(\CF)),$$
while the latter identifies with the left-hand side in \eqref{e:two dualities}.

\end{proof}

\ssec{A higher-categorical compatibility of the dualities}

The material in this subsection will not be used elsewhere in the paper,
and is included for the sake of completeness.

\sssec{}

According to \thmref{t:Serre duality}, the functor obtained from $\IndCoh_{\dgSch_{\on{aft}}}$
by passing to dual categories, identifies canonically with $\IndCoh^!_{\dgSch_{\on{aft}}}$.

\medskip

Similarly, the functor obtained from $\QCoh_{\dgSch_{\on{aft}}}$
by passing to dual categories, identifies canonically with $\QCoh^*_{\dgSch_{\on{aft}}}$.

\medskip

Hence, the natural transformation
$$\Psi_{\dgSch_{\on{aft}}}:\IndCoh_{\dgSch_{\on{aft}}}\to \QCoh_{\dgSch_{\on{aft}}}$$
gives rise to the natural transformation
$$\Psi^\vee_{\dgSch_{\on{aft}}}:\QCoh^*_{\dgSch_{\on{aft}}}\to \IndCoh^!_{\dgSch_{\on{aft}}}.$$

\medskip

\propref{p:Psi check} says that at the level of objects, we have a canonical isomorphism:
\begin{equation}  \label{e:Psi and Ups}
\Psi^\vee_S\simeq \Upsilon_S,\quad S\in \dgSch_{\on{aft}}.
\end{equation}

Moreover, the following assertion follows from the construction of the isomorphism in 
\propref{p:Psi check}:

\begin{lem}
For a morphism $f:S_1\to S_2$, the isomorphisms
$$\Psi^\vee_{S_1}\simeq \Upsilon_{S_1} \text{ and } \Psi^\vee_{S_2}\simeq \Upsilon_{S_2}$$
are compatible with the data of commutativity for the diagrams
$$
\CD
\IndCoh(S_1) @<{\Psi^\vee_{S_1}}<<  \QCoh(S_1)  \\
@A{f^!}AA   @AA{f^*}A \\
\IndCoh(S_2) @<{\Psi^\vee_{S_2}}<<  \QCoh(S_2)
\endCD
$$
and 
$$
\CD
\IndCoh(S_1) @<{\Upsilon_{S_1}}<<  \QCoh(S_1)  \\
@A{f^!}AA   @AA{f^*}A \\
\IndCoh(S_2) @<{\Upsilon_{S_2}}<<  \QCoh(S_2)
\endCD
$$
\end{lem}

The following theorem, strengthening the above lemma will be proved in \cite{GR3}:

\begin{thm} \label{t:Psi check and Upsilon}
There exists a canonical isomorphism 
$$\Psi^\vee_{\dgSch_{\on{aft}}}\simeq \Upsilon_{\dgSch_{\on{aft}}}$$
as natural transformations
$\QCoh^*_{\dgSch_{\on{aft}}}\rightrightarrows \IndCoh^!_{\dgSch_{\on{aft}}}$
between the functors
$$(\dgSch_{\on{aft}})^{\on{op}}\to \StinftyCat_{\on{cont}}.$$ 
\end{thm}

\sssec{}

\thmref{t:Psi check and Upsilon} is a consequence of a more general statement, described below.
Recall the category $\SymMonModStinftyCat$ introduced in \secref{sss:SymMonMod}.

\medskip

We introduce another category $\SymMonModopStinftyCat$, whose objects are the same
as for $\SymMonModStinftyCat$, but 1-morphisms 
$$(\bO_1,\bC_1)\to (\bO_2,\bC_2)$$
are now pairs $(F_\bO,F_\bC)$, where $F_\bO$ is a symmetric monoidal functor
$F_\bO:\bO_2\to \bC_1$ (not the direction of the arrow!), and $F_\bC:\bC_1\to \bC_2$,
which is a morphism of $\bO_2$-module categories.

\medskip

We let
$$(\SymMonModStinftyCat)^{\on{dualizable}}\subset \SymMonModStinftyCat$$
$$(\SymMonModopStinftyCat)^{\on{dualizable}}\subset \SymMonModopStinftyCat$$
denote full subcategories, spanned by those $(\bO,\bC)$, for which $\bC$
is dualizable as an $\bO$-module category, see \cite[Sect. 4]{DG}. (Note that
if $\bO$ is rigid, this condition is equivalent to $\bC$ being dualizable as a plain
DG category, see \cite[Corollary, 6.4.2]{DG}.)

\medskip

Passing to duals $$(\bO,\bC)\mapsto (\bO,\bC^\vee)$$ defines an equivalence
$$\on{dualization}:((\SymMonModStinftyCat)^{\on{dualizable}})^{\on{op}}\to
(\SymMonModopStinftyCat)^{\on{dualizable}}.$$

\sssec{}

Recall (see \thmref{t:action of QCoh}) that we have a functor
$$(\QCoh^*,\IndCoh^!)_{\dgSch_{\on{aft}}}:(\dgSch_{\on{aft}})^{\on{op}}\to
\SymMonModStinftyCat_{\on{cont}}.$$

\medskip

Note that the constructions in \secref{ss:upgrading direct image to a functor} and
\propref{p:dir image as modules} combine to a functor
$$(\QCoh^*,\IndCoh)_{\dgSch_{\on{aft}}}:(\dgSch_{\on{aft}})^{\on{op}}\to
\SymMonModopStinftyCat_{\on{cont}}.$$

\medskip

We have the following assertion which will be proved in \cite{GR3}:

\begin{thm}   \label{t:QCoh-linear duality}
There is a commutative diagram of functors
$$
\CD
\dgSch_{\on{aft}}   @>{((\QCoh^*,\IndCoh^!)_{\dgSch_{\on{aft}}})^{\on{op}}}>>  ((\SymMonModStinftyCat)^{\on{dualizable}})^{\on{op}} \\
@V{\on{Id}}VV  @VV{\on{dualization}}V  \\
\dgSch_{\on{aft}}   @>{(\QCoh^*,\IndCoh)_{\dgSch_{\on{aft}}}}>>  (\SymMonModopStinftyCat)^{\on{dualizable}}.
\endCD
$$
Moreover, the above isomorphism is compatible with the symmetric monoidal structure on both functors.
\end{thm}

\sssec{}

In concrete terms, \thmref{t:QCoh-linear duality} says that for $S\in \dgSch_{\on{aft}}$, 
the duality isomorphism
$$\bD_S^{\on{Serre}}:(\IndCoh(S))^\vee\simeq \IndCoh(S)$$
is compatible with the action of $\QCoh(S)$ on both sides. 

\medskip

Furthermore, for a map $f:S_1\to S_2$, the isomorphism of functors
$$(f^!)^\vee\simeq f^{\IndCoh}_*$$
is compatible with the $\QCoh(S)$-linear structures on both sides. 

\ssec{Relation to the classical Serre duality}  \label{ss:classical Serre}

In this subsection we will explain the relation between the self-duality functor
$$\bD_S^{\on{\on{Serre}}}:\IndCoh(S)^\vee\simeq \IndCoh(S)$$
and the classical Serre duality functor
$$\BD_S^{\on{Serre}}:\Coh(S)^{\on{op}}\to \Coh(S).$$

\sssec{}

Recall that if $\bC_1,\bC_2$ are compactly generated categories, then the datum of an equivalence
$\bC_1^{\on{op}}\simeq \bC_2$ is equivalent to the datum of an equivalence $(\bC_1^c)^{\on{op}}\simeq \bC^c_2$.

\medskip

For example, the self-duality functor
$$\bD_S^{\on{naive}}:(\QCoh(S))^\vee\to \QCoh(S)$$
of \eqref{e:naive duality} is induced by
the ``naive" duality on the category $\QCoh(S)^{\on{perf}}\simeq \QCoh(S)^c$:
$$\BD^{\on{naive}}_S:(\QCoh(S)^{\on{perf}})^{\on{op}}\to \QCoh(S)^{\on{perf}},\quad \CE\mapsto \CE^\vee,$$
the latter being the passage to the dual object in $\QCoh(S)^{\on{perf}}$ as a symmetric monoidal category.

\medskip

In particular, we obtain that the equivalence $\bD_S^{\on{\on{Serre}}}$ of \eqref{e:Serre duality} induces a certain involutive equivalence
\begin{equation} \label{e:abstract Serre}
'\BD_S^{\on{Serre}}:\Coh(S)^{\on{op}}\to \Coh(S).
\end{equation}

\sssec{}   \label{sss:internal Hom}

Recall the action of the monoidal category $\QCoh(S)$ on $\IndCoh(S)$, see \secref{ss:action}.
For $\CF\in \IndCoh(S)$ we can consider the relative internal Hom functor
$$\underline\Hom_{\QCoh(S)}(\CF,-):\IndCoh(S)\to \QCoh(S)$$
as defined in \cite{DG}, Sect. 5.1. Explicitly, for $\CE\in \QCoh(S)$ and $\CF'\in \IndCoh(S)$ we have
$$\Maps_{\QCoh(S)}(\CE,\underline\Hom_{\QCoh(S)}(\CF,\CF')):=
\Maps_{\IndCoh(S)}(\CE\otimes \CF,\CF'),$$
where $-\otimes-$ denotes the action of $\QCoh(S)$ on $\IndCoh(S)$ of \secref{ss:action}. 

\medskip

\begin{lem}  \label{l:inner Hom continuous}
If $\CF\in \IndCoh(S)^c=\Coh(S)$, the functor $\underline\Hom_{\QCoh(S)}(\CF,-)$ is continuous.
\end{lem}

\begin{proof}
It is enough to show that for a set of compact generators $\CE\in \QCoh(S)$, the functor 
$$\CF'\mapsto \Maps_{\QCoh(S)}(\CE,\underline\Hom_{\QCoh(S)}(\CF,\CF')),\quad 
\IndCoh(S)\to \inftygroup$$
commutes with filtered colimits. 

\medskip

We take $\CE\in \QCoh(S)^{\on{perf}}$. In this case $\CE\otimes \CF\in \IndCoh(S)^c$,
and the assertion follows.

\end{proof}

\sssec{}

We now claim:

\begin{lem}
For $\CF\in \Coh(S)$, 
the functor $\CF\mapsto \underline\Hom_{\QCoh(S)}(\CF,\omega_S)$ sends to 
$\Coh(S)$ to $\Coh(S)\subset \QCoh(S)$. 
\end{lem}

\begin{proof}
Let $i$ denote the canonical map $^{cl}\!S\to S$. It is enough to prove the lemma 
for $\CF$ of the form $i^\IndCoh_*(\CF')$ for $\CF'\in \Coh({}^{cl}\!S)$.

\medskip

It is easy to see from \eqref{e:upgrading !} that for a proper map $f:S_1\to S_2$
$\CF_i\in \IndCoh(S_i)$, we have
$$\underline\Hom_{\QCoh(S_2)}(f^\IndCoh_*(\CF_1),\CF_2)\simeq
f_*\left(\underline\Hom_{\QCoh(S_1)}(\CF_1,f^!(\CF_2)\right).$$

Hence, 
$$\underline\Hom_{\QCoh(S)}(\CF,\omega_S)\simeq i_*\left( 
\underline\Hom_{\QCoh({}^{cl}\!S)}(\CF',\omega_{^{cl}\!S})\right).$$

This reduces the assertion to the case of classical schemes, where it
is well-known (proved by the same manipulation as above by locally
embedding into $\BA^n$).

\end{proof}

\sssec{}

From the above lemma we obtain a well-defined functor $\Coh(S)^{\on{op}}\to \Coh(S)$ that we denote
$\BD_S^{\on{Serre}}$.

\begin{prop} \label{p:comparing Serre dualities}
The functors $\BD_S^{\on{Serre}}$ and $'\BD_S^{\on{Serre}}$ of \eqref{e:abstract Serre} are
canonically isomorphic.
\end{prop}

\begin{proof}

Let $\CF_1,\CF_2$ be two objects of $\Coh(S)$. By definition,
\begin{multline} \label{e:Hom into Serre dual one}
\Hom_{\IndCoh(S)}(\CF_1,{}'\BD_S^{\on{Serre}}(\CF_2))\simeq
\Hom_{\IndCoh(S)\otimes \IndCoh(S)}(\CF_1\boxtimes \CF_2,\mu_{\IndCoh(S)}(k))\simeq \\
\simeq \Hom_{\IndCoh(S\times S)}(\CF_1\boxtimes \CF_2,\Delta_S{}_*^{\IndCoh}(\omega_S)).
\end{multline}

Note that both $\CF_1\boxtimes \CF_2$ and $\Delta_S{}_*^{\IndCoh}(\omega_S)$ belong to 
$\IndCoh(S)^+$, so by \propref{p:equiv on D plus}, we can rewrite the right-hand side
of \eqref{e:Hom into Serre dual one} as
\begin{multline} \label{e:Hom into Serre dual two}
\Hom_{\QCoh(S\times S)}(\Psi_{S}(\CF_1)\boxtimes \Psi_S(\CF_2),\Delta_S{}_*(\Psi_S(\omega_S)))\simeq \\
\simeq \Hom_{\QCoh(S)}(\Psi_{S}(\CF_1)\underset{\CO_S}\otimes \Psi_S(\CF_2),\Psi_S(\omega_S)).
\end{multline}

\medskip

By definition, $\Hom_{\IndCoh(S)}(\CF_1,\BD_S^{\on{Serre}}(\CF_2))$ is isomorphic to
\begin{multline} \label{e:Hom into Serre dual three} 
\Hom_{\QCoh(S)}(\Psi_S(\CF_1),\Psi_S(\BD_S^{\on{Serre}}(\CF_2))):=\\
\Hom_{\QCoh(S)}(\Psi_S(\CF_1),\underline\Hom_{\QCoh(S)}(\CF_2,\omega_S)):= \\
\Hom_{\IndCoh(S)}(\Psi_S(\CF_1)\underset{\CO_S}\otimes \CF_2,\omega_S),
\end{multline}
where $\Psi_S(\CF_1)\underset{\CO_S}\otimes \CF_2$ is understood in the sense
of the action of $\QCoh(S)$ on $\IndCoh(S)$.

\medskip

However, since $\omega_S\in \IndCoh(S)^+$, by \propref{p:equiv on D plus}, the right-hand side
of \eqref{e:Hom into Serre dual three} can be rewritten as 
$$\Hom_{\QCoh(S)}(\Psi_S(\Psi_S(\CF_1)\underset{\CO_S}\otimes \CF_2),\Psi(\omega_S)),$$
which by \lemref{l:Psi compat with action} is isomorphic to the right-hand side of
\eqref{e:Hom into Serre dual two}.

\end{proof}

\sssec{}

Combining \propref{p:comparing Serre dualities} with \cite[Lemma 2.3.3]{DG}, we obtain:

\begin{cor} \label{c:D conjugates two} 
Let $S_1\to S_2$ be a morphism in $\dgSch_{\on{aft}}$.

\smallskip

\noindent{\em(a)} Suppose that $f$ is eventually coconnective. Then 
we have canonical isomorphisms of functors $\Coh(S_2)^{\on{op}}\to \Coh(S_1)$:
$$\BD_{S_1}^{\on{Serre}}\circ (f^{\IndCoh,*})^{\on{op}}\simeq f^!\circ \BD^{\on{Serre}}_{S_2};$$
and 
$$\BD_{S_1}^{\on{Serre}}\circ (f^!)^{\on{op}}\simeq f^{\IndCoh,*}\circ \BD^{\on{Serre}}_{S_2}.$$

\smallskip

\noindent{\em(b)} Suppose that $f$ is proper. Then we have a canonical isomorphism
of functors $\Coh(S_1)^{\on{op}}\to \Coh(S_2)$
$$\BD_{S_2}^{\on{Serre}}\circ (f^\IndCoh_*)^{\on{op}}\simeq 
f^\IndCoh_* \circ \BD_{S_1}^{\on{Serre}}.$$

\end{cor}

\ssec{Serre duality in the eventually coconnective case}  \label{other functor event coconn}

In this subsection we will assume that $S\in \dgSch_{\on{aft}}$ is eventually coconnective. 

\sssec{}

From \propref{p:Xi} and \cite{DG}, Sect. 2.3.2, we obtain:

\begin{cor}  \label{c:dual of Psi}
The functor $\Psi_S^\vee$ sends compact objects
to compact ones, and is fully faithful. The functor $\Psi^\vee_S$ realizes $\QCoh(S)$ as a colocalization
of $\IndCoh(S)$. 
\end{cor}

Taking into account \propref{p:Psi check}, we obtain: 

\begin{cor}  \label{c:ppties of Upsilon}
The functor $\Upsilon_S\simeq  -\underset{\CO_S}\otimes \omega_S$ 
sends compact objects to compact ones and is fully faithful. 
\end{cor}

In particular:

\begin{cor}
For $S$ eventually coconnective, the object $\omega_S\in \IndCoh(S)$ is compact.
\end{cor}

\begin{rem}
We have an isomorphism 
\begin{equation}
'\BD_S^{\on{Serre}}\circ (\Xi_S)^{\on{op}}\simeq \Psi^\vee_S\circ \BD_S^{\on{naive}}
\end{equation}
as functors $(\QCoh(S)^{\on{perf}})^{\on{op}}\rightrightarrows \Coh(S)$, see \cite[Lemma 2.3.3]{DG}. 

\medskip

Applying \propref{p:comparing Serre dualities} and \propref{p:Psi check} we obtain an isomorphism
\begin{equation}  \label{e:conjugate Ups}
\BD_S^{\on{Serre}}\circ (\Xi_S)^{\on{op}}\simeq \Upsilon_S\circ \BD_S^{\on{naive}}.
\end{equation}

It is easy to see that the isomorphism in \eqref{e:conjugate Ups} is the tautoloigical isomorphism
$$\underline\Hom_{\QCoh(S)}(\CE,\omega_S)\simeq \CE^\vee\otimes \omega_S,\quad \CE\in \QCoh(S)^{\on{perf}}.$$

\end{rem}

\sssec{}

By \cite[Sect. 2.3.2 and Lemma 2.3.3]{DG}, the functor $\Psi^\vee_S$ admits a continuous right adjoint, which 
identifies with the functor $\Xi^\vee_S$, dual of $\Xi_S$. 

\medskip

Since $\omega_S\in \IndCoh(S)$ is compact, by \secref{sss:internal Hom}, we have a
continuous functor
\begin{equation} \label{e:internal Hom from dual}
\CF\mapsto \underline\Hom_{\QCoh(S)}(\omega_S,\CF):\IndCoh(S)\to \QCoh(S).
\end{equation}

\medskip

From the definition and  the isomorphism $\Psi^\vee_S\simeq \Upsilon_S$ of \propref{p:Psi check}, we obtain:

\begin{lem}  \label{l:expression for Xi dual}
The functor $\IndCoh(S)\to \QCoh(S)$, right adjoint of $\Psi^\vee_S$
is given by 
$$\CF\mapsto \underline\Hom_{\QCoh(S)}(\omega_S,\CF).$$
\end{lem}

\begin{rem}  
The functor right adjoint to $\Psi^\vee_S\simeq \Upsilon_S$ is defined for any $S$ (i.e., not necessarily eventually
coconnective), but in general it will fail to be continuous. 
\end{rem}

Hence, we obtain:

\begin{cor}  \label{c:expression for Xi dual}
The functor $$\Xi^\vee_S:\IndCoh(S)\to \QCoh(S),$$ dual to $\Xi_S:\QCoh(S)\to \IndCoh(S)$, is given by
$\underline\Hom_{\QCoh(S)}(\omega_S,-)$,
\end{cor}

We also note:

\begin{lem}
Let $f:S_1\to S_2$ be an eventually coconnective map in $\dgSch_{\on{aft}}$, where
$S_1$ and $S_2$ are themselves eventually coconnective. 
Then there exists a canonical isomorphism:
$$\Xi_{S_1}^\vee\circ f^!\simeq f^*\circ \Xi^\vee_{S_2}:\IndCoh(S_2)\to \QCoh(S_1).$$
\end{lem}

\begin{proof}
Obtained by passing to dual functors in the 
isomorphism 
$$\Xi_{S_2}\circ f_*\simeq f^\IndCoh_*\circ \Xi_{S_1}$$
of \propref{p:Xi and *}.
\end{proof}

\sssec{}

We note that the logic of the previous discussion can be inverted, and we obtain:

\begin{prop}  \label{p:other functor event coconn}
The following conditions on $S$ are equivalent:

\smallskip

\noindent(a) The functor $\Psi_S$ admits a left adjoint.

\smallskip

\noindent(b) The functor $\Psi^\vee_S\simeq\Upsilon_S$ sends compact
objects to compact ones.

\smallskip

\noindent(c) The object $\omega_S\in \IndCoh(S)$ is compact.

\smallskip

\noindent(d) $S$ is eventually coconnective. 

\end{prop}

\begin{proof}

The equivalence (b) $\Leftrightarrow$ (c) is tautological. The equivalence
(a) $\Leftrightarrow$ (d) has been established in \propref{p:Xi and event coconn}. 
The equivalence (a) $\Leftrightarrow$ (b) follows from \cite{DG}, Sect. 2.3.2.

\end{proof}

\sssec{}

Let $S\in \dgSch_{\on{aft}}$ be affine. In this case, it is easy to see that for every $\CF\in \Coh(S)$ and $k\in \BN$
there exists an object $\CF'\in \QCoh(S)^{\on{perf}}$ and a map $\CF'\to \CF$, such that
$$\on{Cone}(\Xi_S(\CF')\to \CF)[-1]\in \Coh(S)^{\leq -k}.$$

We claim that the functor $\Psi_S^\vee$ plays a dual role:

\begin{lem}
For $\CF\in \Coh(S)$ and $k\in \BN$ there exists an object $\CF'\in \QCoh(S)^{\on{perf}}$
and a map $\CF\to \Upsilon_S(\CF')$ so that 
$$\on{Cone}(\CF\to \Upsilon_S(\CF'))\in \Coh(S)^{\geq k}.$$
\end{lem}

\begin{proof}
Let $m$ be an integer such that $\BD^{\on{Serre}}_S$ sends 
$$\Coh(S)^{\leq 0}\to \Coh(S)^{\geq -m}$$
(in fact, $m$ can be taken to be the dimension of $^{cl}\!S$.)

\medskip

For $\CF$ as in the lemma, consider $\BD^{\on{Serre}}_S(\CF)\in \Coh(S)$, and let $\CF''\in \QCoh(S)^{\on{perf}}$
and $\Xi_S(\CF'')\to \BD^{\on{Serre}}_S(\CF)$ be such that
$$\on{Cone}(\Xi_S(\CF'')\to \BD^{\on{Serre}}_S(\CF))[-1]\in \Coh(S)^{\leq -(k+m)}.$$

Set $\CF':=\BD^{\on{naive}}_S(\CF'')$. Applying Serre duality to $\Xi_S(\CF'')\to  \BD^{\on{Serre}}_S(\CF)$
we obtain a map
$$\CF\to \BD_S^{\on{Serre}}(\Xi_S(\CF''))\simeq \Upsilon_S(\CF')$$
with the desired properties.

\end{proof}

\sssec{Proof of \corref{c:Drinfeld}} \label{sss:proof of Drinf}

It is enough to show that $\Psi_S(\omega_S)\in \QCoh(S)^{\on{perf}}$ is invertible. 

\medskip

For $\CF_1,\CF_2\in \IndCoh(S)$ and $\CE\in \QCoh(S)$ consider
the canonical map
$$\CE\otimes \underline\Hom_{\QCoh(S)}(\CF_1,\CF_2)\to
\underline\Hom_{\QCoh(S)}(\CF_1,\CE\otimes \CF_2).$$
It is easy to see that this map is an isomorphism if $\CE\in \QCoh(S)^{\on{perf}}$.
In particular, we obtain an isomorphism
$$\Psi_S(\omega_S)\otimes \underline\Hom_{\QCoh(S)}(\omega_S,\Xi_S(\CO_S))\simeq
\underline\Hom_{\QCoh(S)}(\omega_S,\Psi_S(\omega_S)\otimes \Xi_S(\CO_S)).$$

\medskip

By \corref{c:upgrading Xi}, we have
$$\Psi_S(\omega_S)\otimes \Xi_S(\CO_S)\simeq \Xi_S\circ \Psi_S(\omega_S),$$
and since $\Psi_S(\omega_S)\in \QCoh(S)^{\on{perf}}$ we have $\Xi_S\circ \Psi_S(\omega_S)\simeq \omega_S$.
Thus, we obtain:
$$\Psi_S(\omega_S)\otimes \underline\Hom_{\QCoh(S)}(\omega_S,\Xi_S(\CO_S))\simeq
\underline\Hom_{\QCoh(S)}(\omega_S,\omega_S)\simeq \Xi^\vee_S\circ \Psi^\vee_S(\CO_S)\simeq \CO_S.$$
I.e., $\Psi_S(\omega_S)\in \QCoh(S)$ is invertible, as required.

\qed

\bigskip

\bigskip

\centerline{\bf Part III. Extending to stacks}

\bigskip

\section{$\IndCoh$ on prestacks}  \label{s:stacks}

As was mentioned in the introduction, our conventions regarding prestacks and stacks follow \cite{Stacks}. 
In this section we shall mostly be interested in the full subcategory
$$\inftydgprestack_{\on{laft}}\subset \inftydgprestack$$
consisting of prestacks, \emph{locally almost of finite type}, see \cite[Sect. 1.3.9]{Stacks}
for the definition. 

\ssec{Definition of $\IndCoh$}

\sssec{}

By definition, the category $\inftydgprestack_{\on{laft}}$ is the category of functors
$$({}^{<\infty}\!\affdgSch_{\on{aft}})^{\on{op}}\to \inftygroup.$$

We have the fully faithful embeddings
$$^{<\infty}\!\affdgSch_{\on{aft}}\hookrightarrow {}^{<\infty}\!\dgSch_{\on{aft}}\hookrightarrow 
\dgSch_{\on{aft}}\hookrightarrow \on{PreStk}_{\on{laft}},$$
where the composed arrow is the Yoneda embedding. 

\sssec{}     \label{sss:defn of IndCoh prestacks}

We let $\IndCoh^!_{^{<\infty}\!\affdgSch_{\on{aft}}}$ be the functor
$$({}^{<\infty}\!\affdgSch_{\on{aft}})^{\on{op}}\to \StinftyCat_{\on{cont}}$$
obtained by retsricting the functor $\IndCoh^!_{\dgSch_{\on{aft}}}$ 
under
$$({}^{<\infty}\!\affdgSch_{\on{aft}})^{\on{op}}\hookrightarrow (\dgSch_{\on{aft}})^{\on{op}}.$$

\medskip 

We define the functor
$$\IndCoh^!_{\on{PreStk}_{\on{laft}}}:(\on{PreStk}_{\on{laft}})^{\on{op}}\to \StinftyCat_{\on{cont}}$$
as the right Kan extension of $\IndCoh^!_{^{<\infty}\!\affdgSch_{\on{aft}}}$ under the Yoneda embedding
$$({}^{<\infty}\!\affdgSch_{\on{aft}})^{\on{op}}\hookrightarrow (\on{PreStk}_{\on{laft}})^{\on{op}}.$$

The following is immediate from the definition:

\begin{lem} \label{l:IndCoh and colimits}
The functor $\IndCoh^!_{\on{PreStk}_{\on{laft}}}$ introduced above takes colimits in $\on{PreStk}_{\on{laft}}$ to limits
in $\StinftyCat_{\on{cont}}$.
\end{lem}

\sssec{}    \label{sss:! for stacks}

For $\CY\in \on{PreStk}_{\on{laft}}$ we let $\IndCoh(\CY)$ denote the value of $\IndCoh^!_{\on{PreStk}_{\on{laft}}}$
on $\CY$.

\medskip

Tautologically, we have:

\begin{equation} \label{e:indcoh as a limit}
\IndCoh(\CY)\simeq \underset{(S,y)\in (({}^{<\infty}\!\affdgSch_{\on{aft}})_{/\CY})^{\on{op}}}{lim}\, \IndCoh(S),
\end{equation}
where the functors for $f:S_1\to S_2$, $y_1=y_2\circ f$ are 
$$f^!:\IndCoh(S_2)\to \IndCoh(S_1).$$

For $(S,y)\in ({}^{<\infty}\!\affdgSch_{\on{aft}})_{/\CY}$ we let $y^!$ denote the corresponding evaluation functor
$$\IndCoh(\CY)\to \IndCoh(S).$$

\medskip

For a morphism $f:\CY_1\to \CY_2$ in $\on{PreStk}_{\on{laft}}$ we let $f^!$ denote the corresponding
functor
$$\IndCoh(\CY_2)\to \IndCoh(\CY_1).$$

\medskip

In particular, we let $\omega_\CY\in \IndCoh(\CY)$ denote the dualizing complex of $\CY$,
defined as $p_\CY^!(k)$, where $p_\CY:\CY\to \on{pt}$.

\sssec{The $n$-coconnective case}  \label{sss:n-coconn}

Suppose that $\CY\in \on{PreStk}_{\on{laft}}$ is $n$-coconnective, i.e., when viewed as a functor
$$({}^{<\infty}\!\affdgSch_{\on{aft}})^{\on{op}}\hookrightarrow \inftygroup,$$
it belongs to the essential image of the fully faithful embedding
$$\on{Funct}(({}^{<n}\!\affdgSch_{\on{ft}})^{\on{op}},\inftygroup)\hookrightarrow 
\on{Funct}(({}^{<\infty}\!\affdgSch_{\on{aft}})^{\on{op}},\inftygroup),$$
given by left Kan extension along 
$$({}^{<n}\!\affdgSch_{\on{ft}})^{\on{op}}\hookrightarrow ({}^{<\infty}\!\affdgSch_{\on{aft}})^{\on{op}}.$$

This is equivalent to the condition that the fully faithful embedding
$$({}^{<n}\!\affdgSch_{\on{ft}})_{/\CY}\hookrightarrow ({}^{<\infty}\!\affdgSch_{\on{aft}})_{/\CY}$$
be cofinal.

\medskip

In particular, we obtain that the restriction functor
\begin{equation} \label{e:IndCoh on n-coconn}
\IndCoh(\CY)\to  \underset{(S,y)\in (({}^{<n}\!\affdgSch_{\on{ft}})_{/\CY})^{\on{op}}}{lim}\, \IndCoh(S)
\end{equation}
is an equivalence.

\medskip

Hence, to calculate the value of $\IndCoh(\CY)$ on an $n$-coconnective prestack, it suffices
to consider only $n$-coconnective affine DG schemes. In particular, if $\CY$ is $0$-coconnective,
i.e., is classical, it suffices to consider only classical affine schemes. 

\ssec{Convergence}

In this subsection we will discuss several applications of \propref{p:convergence} to the study of
properties of the category $\IndCoh(\CY)$ for $\CY\in \inftydgprestack_{\on{laft}}$.

\medskip

First, we note that the assertion of \propref{p:convergence} implies::

\begin{cor} \label{c:convergence as RKE}
The functor 
$$\IndCoh^!_{\affdgSch_{\on{aft}}}:(\affdgSch_{\on{aft}})^{\on{op}}\to \StinftyCat_{\on{cont}}$$
is the right Kan extension from the full subcategory 
$$({}^{< \infty}\!\affdgSch_{\on{aft}})^{\on{op}}\hookrightarrow (\affdgSch_{\on{aft}})^{\on{op}}.$$
\end{cor}

\begin{proof}
We wish to show that for $S\in \affdgSch_{\on{aft}}$, the functor
\begin{equation} \label{e:via coconn}
\IndCoh(S)\to \underset{S'\in {}^{< \infty}\!\affdgSch_{\on{aft}},S'\to S}{lim}\, \IndCoh(S')
\end{equation}
is an equivalence. The right-hand side of \eqref{e:via coconn} can be rewritten as
$$\underset{n}{lim}\, \underset{S'\in {}^{< n}\!\affdgSch_{\on{ft}},S'\to S}{lim}\in \IndCoh(S').$$
However, since the functor $^{< n}\!\affdgSch_{\on{ft}}\hookrightarrow \affdgSch_{\on{ft}}$
admits a right adjoint, given by $S\mapsto \tau^{\leq n}(S)$, we have:
$$\underset{S'\in {}^{< n}\!\affdgSch_{\on{ft}},S'\to S}{lim}\, \IndCoh(S')\simeq \IndCoh(\tau^{\leq n}(S)).$$
Thus, we need to show that the functor
$$\IndCoh(S)\to \underset{n}{lim}\, \IndCoh(\tau^{\leq n}(S))$$
is an equivalence, but the latter is the content of \propref{p:convergence}.
\end{proof}

Tautologically, \corref{c:convergence as RKE} implies that 
in the description of $\IndCoh(\CY)$ given by \eqref{e:indcoh as a limit}, instead
of eventually coconnective affine schemes of finite type, we can use all affine schemes almost of finite type:

\begin{cor}  \label{c:using laft}
The restriction functor
$$\underset{(S,y)\in ((\affdgSch_{\on{aft}})_{/\CY})^{\on{op}}}{lim}\, \IndCoh(S) \to
\underset{(S,y)\in (({}^{<\infty}\!\affdgSch_{\on{aft}})_{/\CY})^{\on{op}}}{lim}\, \IndCoh(S)=:\IndCoh(\CY)$$
is an equivalence.
\end{cor}

\sssec{}

Equivalently, we can formulate the above corollary as saying that the functor
$$\IndCoh^!_{\inftydgprestack_{\on{laft}}}:(\on{PreStk}_{\on{aft}})^{\on{op}}\to \StinftyCat_{\on{cont}}$$
is the right Kan extension of 
$$\IndCoh^!_{\affdgSch_{\on{aft}}}:(\affdgSch_{\on{aft}})^{\on{op}}\to \StinftyCat_{\on{cont}}$$
along the tautological embedding
$$(\affdgSch_{\on{aft}})^{\on{op}}\hookrightarrow (\on{PreStk}_{\on{aft}})^{\on{op}}.$$

\sssec{}

Finally, let us note that the value of the functor $\IndCoh^!_{\inftydgprestack_{\on{laft}}}$
on a given $\CY\in \inftydgprestack_{\on{laft}}$ can be recovered from the $n$-coconnective
truncations $\tau^{\leq n}(\CY)$ of $\CY$, where we recall that
$$\tau^{\leq n}(\CY):=\on{LKE}_{({}^{<n}\!\affdgSch_{\on{ft}})^{\on{op}}\hookrightarrow ({}^{<\infty}\!\affdgSch_{\on{aft}})^{\on{op}}}
(\CY|_{{}^{<n}\!\affdgSch_{\on{ft}})^{\on{op}}}).$$

Namely, we claim:

\begin{lem}  \label{l:limit over n stacks}
For $\CY\in \inftydgprestack_{\on{laft}}$, the natural map
$$\IndCoh(\CY)\to \underset{n}{lim}\, 
\IndCoh(\tau^{\leq n}(\CY))$$
is an isomorphism.
\end{lem}

\begin{proof}
The assertion follows from the fact that the natural map 
$$\underset{n}{colim}\, 
\on{LKE}_{^{\leq n}\!\affdgSch_{\on{ft}}\hookrightarrow {}^{<\infty}\!\affdgSch_{\on{aft}}}(\tau^{\leq n}(\CY))\to \CY$$
is an isomorphism in $\inftydgprestack_{\on{laft}}$, combined with \lemref{l:IndCoh and colimits}.
\end{proof}

\ssec{Relation to $\QCoh$ and the multiplicative structure}

\sssec{}

Following \secref{sss:intr Upsilon}, we view the assignment
$$S\in \dgSch_{\on{aft}}\rightsquigarrow \Upsilon_S:\QCoh(S)\to \IndCoh(S)$$
as a natural transformation
$$\Upsilon_{\dgSch_{\on{aft}}}:\QCoh^*_{\dgSch_{\on{aft}}}\to \IndCoh^!_{\dgSch_{\on{aft}}},$$
when we view $\QCoh^*_{\dgSch_{\on{aft}}}$ and $\IndCoh^!_{\dgSch_{\on{aft}}}$ as functors
$$(\dgSch_{\on{aft}})^{\on{op}}\to \SymMonStinftyCat_{\on{cont}}.$$

\medskip

Restricting to $^{<\infty}\!\affdgSch_{\on{aft}}\subset \dgSch_{\on{aft}}$, we obtain the corresponding
functors and the natural transformation
$$\Upsilon_{^{<\infty}\!\affdgSch_{\on{aft}}}:\QCoh^*_{^{<\infty}\!\affdgSch_{\on{aft}}}\to \IndCoh^!_{^{<\infty}\!\affdgSch_{\on{aft}}}.$$

\sssec{}

Since the forgetful functor $\SymMonStinftyCat_{\on{cont}}\to \StinftyCat_{\on{cont}}$ commutes
with limits, we obtain that the functor
$$\IndCoh^!_{\on{PreStk}_{\on{laft}}}:(\on{PreStk}_{\on{aft}})^{\on{op}}\to \StinftyCat_{\on{cont}}$$
naturally upgrades to a functor
$$\IndCoh^!_{\on{PreStk}_{\on{laft}}}:(\on{PreStk}_{\on{aft}})^{\on{op}}\to  \SymMonStinftyCat_{\on{cont}}.$$

\medskip

In particular, for every $\CY\in \on{PreStk}_{\on{aft}}$, the category $\IndCoh(\CY)$ acquires a natural
symmetric monoidal structure; we denote the corresponding monoidal operation by $\sotimes$.
The unit in this category is $\omega_\CY$. 

\medskip

For a map $f:\CY_1\to \CY_2$, the functor $f^!$ has a natural symmetric monoidal structure.

\sssec{}

Applying the functor
$$\on{RKE}_{({}^{<\infty}\!\affdgSch_{\on{aft}})^{\on{op}}\hookrightarrow (\on{PreStk}_{\on{laft}})^{\on{op}}}$$
to $\Upsilon_{^{<\infty}\!\affdgSch_{\on{aft}}}$, we obtain a natural transformation
$$\on{RKE}_{({}^{<\infty}\!\affdgSch_{\on{aft}})^{\on{op}}\hookrightarrow (\on{PreStk}_{\on{laft}})^{\on{op}}}
(\QCoh^*_{^{<\infty}\!\affdgSch_{\on{aft}}})\to \IndCoh^!_{\on{PreStk}_{\on{laft}}}.$$

\medskip

Recall that the functor
$$\QCoh^*_{\on{PreStk}}:(\on{PreStk})^{\on{op}}\to \StinftyCat_{\on{cont}}$$ is defined as
$$\on{RKE}_{(\affdgSch)^{\on{op}}\hookrightarrow (\on{PreStk})^{\on{op}}}(\QCoh^*_{\affdgSch}).$$

\medskip

Hence, we have a natural transformation
\begin{multline}  \label{e:mult on QCoh stacks}
\QCoh^*_{\on{PreStk}}:(\on{PreStk})^{\on{op}}|_{(\on{PreStk}_{\on{laft}})^{\on{op}}}=:
\QCoh^*_{\on{PreStk}_{\on{laft}}}\to \\
\to \on{RKE}_{({}^{<\infty}\!\affdgSch_{\on{aft}})^{\on{op}}\hookrightarrow (\on{PreStk}_{\on{laft}})^{\on{op}}}
(\QCoh^*_{^{<\infty}\!\affdgSch_{\on{aft}}}).
\end{multline}

Composing, we obatin a functor
$$\QCoh^*_{\on{PreStk}_{\on{laft}}}\to \IndCoh^!_{\on{PreStk}_{\on{laft}}}$$
that we shall denote by $\Upsilon_{\on{PreStk}_{\on{laft}}}$. 

\medskip

For an individual $\CY\in \on{PreStk}_{\on{laft}}$ we shall denote by
$$\Upsilon_{\CY}:\QCoh(\CY)\to \IndCoh(\CY)$$
the resulting functor.

\medskip

Furthermore, the functors and the natural transformation in \eqref{e:mult on QCoh stacks} naturally
upgrade to take values in $\SymMonStinftyCat_{\on{cont}}$, and so does the natural transformation
$\Upsilon_{\on{PreStk}_{\on{laft}}}$. 

\medskip

\begin{lem}  \label{l:dual of Psi for stacks}
Assume that $\CY$ belongs to $^{\leq n}\!\on{PreStk}_{\on{laft}}$. Then the functor
functor $\Upsilon_\CY$ is fully faithful.
\end{lem}

\begin{proof}
Follows from the isomorphism \eqref{e:IndCoh on n-coconn} (and the corresponding assertion for $\QCoh$),
and \corref{c:ppties of Upsilon}.
\end{proof}

\sssec{Behavior with respect to products}

Let $\CY_1$ and $\CY_2$ be prestacks. Pulling back along the two projections
$$\CY_1\times \CY_2\to \CY_i$$
and applying the monoidal operation, we obtain a functor
\begin{equation} \label{e:product map, prestacks}
\IndCoh(\CY_1)\otimes \IndCoh(\CY_2)\to \IndCoh(\CY_1\times \CY_2).
\end{equation}

Repeating the argument of \cite{QCoh}, Prop. 1.4.4, from \propref{p:products} we deduce:

\begin{cor}
Assume that $\IndCoh(\CY_1)$ is dualizable as a category. Then for any $\CY_2$, the functor
\eqref{e:product map, prestacks} is an equivalence.
\end{cor}

\ssec{Descent properties}

\sssec{}

As in \cite{QCoh}, Sect. 1.3.1, we can consider the notion of descent for presheaves on 
the category $^{<\infty}\!\affdgSch_{\on{aft}}$ with values in an arbitrary $(\infty,1)$-category $\bC$.

\sssec{}

We observe that \thmref{t:fppf descent} implies that $\IndCoh^!_{^{<\infty}\!\affdgSch_{\on{aft}}}$, regarded as 
a presheaf on $^{<\infty}\!\affdgSch_{\on{aft}}$ with values in $\StinftyCat_{\on{cont}}$,
is a sheaf in the fppf topology. 

\medskip

By \cite{Lu0}, Sect. 6.2.1, we obtain that whenever $\CY_1\to \CY_2$ is a map in $\inftydgprestack_{\on{laft}}$
such that $L_{\on{laft}}(\CY_1)\to L_{\on{laft}}(\CY_2)$ is an isomorphism, the map
$$\IndCoh(\CY_2)\to \IndCoh(\CY_1)$$ is
an isomorphism. 

\medskip

Here $L_{\on{laft}}$ denotes the localization functor on 
$$\on{PreStk}_{\on{laft}}=\on{Funct}({}^{<\infty}\!\affdgSch_{\on{aft}},\inftygroup)$$
in the fppf topology. 

\medskip

From here we obtain:

\begin{cor}
For $\CY\in \inftydgprestack_{\on{laft}}$, the 
natural map $\CY\to L_{\on{laft}}(\CY)$
induces an isomorphism 
$$\IndCoh(L_{\on{laft}}(\CY))\to \IndCoh(\CY).$$ 
\end{cor}

\sssec{}

For a map  $f:\CY_1\to \CY_2$ consider the cosimplicial category 
$\IndCoh^!(\CY_1^\bullet/\CY_2)$, formed using the !-pullback functors.

\medskip

Assume now that $f$ is a surjection for the fppf topology on $^{<\infty}\!\affdgSch_{\on{aft}}$
(see \cite{Stacks}, Sect. 2.4.8). From \cite{Lu0}, Cor. 6.2.3.5 we obtain:
 
\begin{cor} \label{c:Cech !}
Under the above circumstances, we obtain that the functor 
$$\IndCoh(\CY_2)\to \on{Tot}\left(\IndCoh^!(\CY_1^\bullet/\CY_2)\right)$$
given by !-pullback, is an equivalence.
\end{cor}

\ssec{Two definitions of $\IndCoh$ for DG schemes}

\sssec{}

Let $X$ be an object of $\dgSch_{\on{aft}}$. 

\medskip

Note that
we have two \emph{a priori} different definitions of $\IndCoh(X)$: one given in \secref{ss:defn}
(which we temporarily denote $\IndCoh(X)'$), and another as in \secref{sss:defn of IndCoh prestacks} 
(which we temporarily denote $\IndCoh(X)''$), when we regard $X$ as an object of 
$\inftydgprestack_{\on{laft}}$. 

\sssec{}

Note that !-pullback defines a natural functor:

\begin{equation} \label{e:two definitions}
\IndCoh(X)'\to \IndCoh(X)''.
\end{equation}

\begin{prop} \label{p:two defns of IndCoh}
The functor \eqref{e:two definitions} is an equivalence. 
\end{prop}

\begin{proof}
First, assume that $X$ is affine (but not necessarily eventually coconnective). Then the assertion follows
from \corref{c:convergence as RKE}.

\medskip

Next, assume that $X$ is separated. Let $f:X'\to X$ be an affine Zariski cover, and let $X'{}^\bullet/X$
be its \v{C}ech nerve (whose terms are affine, since $X$ was assumed
separated). Then the validity of the assertion in the affine case, combined with 
\corref{c:Cech !} and \propref{p:Zariski descent}, imply that \eqref{e:two definitions} is an 
equivalence.

\medskip

Let now $X$ be arbitrary. We choose a Zariski cover $f:X'\to X$, where $X'$ is separated, and repeat
the same argument, using the fact that the terms of $X'{}^\bullet/X$ are now separated. 

\end{proof}

\sssec{}

The above proposition implies that in the definition of $\IndCoh(\CY)$ for $\CY\in \inftydgprestack_{\on{laft}}$
we can use all DG schemes almost of finite type, instead of the affine ones:

\begin{cor}  \label{c:IndCoh via all schemes}
For $\CY\in \inftydgprestack_{\on{laft}}$, the restriction functor
$$\underset{(S,y)\in ((\dgSch_{\on{aft}})_{/\CY})^{\on{op}}}{lim}\, \IndCoh(S)\to 
\underset{(S,y)\in (({}^{<\infty}\!\affdgSch_{\on{aft}})_{/\CY})^{\on{op}}}{lim}\, \IndCoh(S):=\IndCoh(\CY)$$
is an equivalence.
\end{cor}

\begin{proof}

The left-hand side in the corollary calculates the value on $\CY$ of the right Kan extension of
$\IndCoh^!_{\dgSch_{\on{aft}}}$ along
$$(\dgSch_{\on{aft}})^{\on{op}}\hookrightarrow (\on{PreStk}_{\on{laft}})^{\on{op}}.$$

Hence, it is enough to show that the map
$$\IndCoh^!_{\dgSch_{\on{aft}}}\to
\on{RKE}_{({}^{<\infty}\!\affdgSch_{\on{aft}})^{\on{op}}\hookrightarrow (\dgSch_{\on{aft}})^{\on{op}}}(\IndCoh^!_{^{<\infty}\!\affdgSch_{\on{aft}}})$$
is an isomorphism.

\medskip

However, the latter is equivalent to the statement of \propref{p:two defns of IndCoh}.

\end{proof}

\ssec{Functoriality for direct image under schematic morphisms}  \label{ss:gen dir im stacks}

Let $f:\CY_1\to \CY_2$ be a map in $\inftydgprestack_{\on{laft}}$. We will not
be able to define the functor
$$f_*^{\IndCoh}:\IndCoh(\CY_1)\to \IndCoh(\CY_2)$$ in general.

\medskip

However, we will be able to do this in the case when $f$ is schematic and quasi-compact 
morphism between arbitrary prestacks, which is goal of this subsection.





\sssec{}

We shall say that a map $f:\CY_1\to \CY_2$ in $\on{PreStk}$
is schematic if for any
$(S_2,y_2)\in \affdgSch_{/\CY_2}$, the fiber product 
$$S_1:=S_2\underset{\CY_2}\times \CY_1$$ 
is a DG scheme. 

\medskip

We shall say that $f$ is schematic and quasi-separated and quasi-compact if for all $(S_2,y_2)\in \affdgSch_{/\CY_2}$,
the above DG scheme $S_1$ is quasi-separated and quasi-compact.

\medskip

We shall say that $f$ is schematic and proper if for all $(S_2,y_2)\in \affdgSch_{/\CY_2}$, the resulting map $S_1\to S_2$ of
DG schemes is proper. 

\medskip

It is easy to see that if $f$ is schematic, then for any $(S_2,y_2)\in \dgSch_{\CY_2}$ (i.e., $S_2$
is not necessarily affine), the fiber product $S_1:=S_2\underset{\CY_2}\times \CY_1$ is a DG scheme.

\medskip

The next assertion results easily from the definitions:

\begin{lem}  \label{l:crit sch}  
Let $f:\CY_1\to \CY_2$ be a map in $\on{PreStk}$.

\smallskip

\noindent{\em(a)} Suppose that $\CY_1$ and $\CY_2$ are convergent (see \cite[Sect. 1.2]{Stacks} where the notion 
is introduced).  Then the condition for $f$ to be schematic (resp., schematic and quasi-separated and quasi-compact,
schematic and proper)
is enough to test on $(S_2,y_2)\in {}^{<\infty}\!\affdgSch_{/\CY_2}$. 

\smallskip

\noindent{\em(b)} Suppose that $\CY_1,\CY_2\in \on{PreStk}_{\on{laft}}$. Then 
the condition for $f$ to be schematic (resp., schematic and quasi-separated and quasi-compact, schematic and proper)
is enough to test on $(S_2,y_2)\in ({}^{<\infty}\!\affdgSch_{\on{aft}})_{/\CY_2}$. 

\end{lem}

\sssec{}  \label{sss:prestack corr}

We are going to make $\IndCoh$ into a functor on the category $\inftydgprestack_{\on{laft}}$
with $1$-morphisms being correspondences, where we allow to take direct images
along morphisms that are schematic and quasi-compact. 

\medskip

Namely, in the framework of \secref{sss:two classes}, we take $\bC:=\inftydgprestack_{\on{laft}}$,
${horiz}$ to be the class of all $1$-morphisms, and ${vert}$ to be the class of $1$-morphisms
that are schematic and quasi-compact (the quasi-separetedness condition comes for free because
of the finite type assumption).

\medskip

Let $(\inftydgprestack_{\on{laft}})_{\on{corr:sch-qc;all}}$
denote the resulting category of correspondences. 

\medskip

We shall now extend the assignment $\CY\mapsto \IndCoh(\CY)$ to a functor
$$\IndCoh_{(\inftydgprestack_{\on{laft}})_{\on{corr:sch-qc;all}}}:
(\inftydgprestack_{\on{laft}})_{\on{corr:sch-qc;all}}\to \StinftyCat_{\on{cont}},$$
such that for $g:\CY_1\to \CY_2$ the corresponding functor
$$\IndCoh(\CY_2)\to \IndCoh(\CY_1)$$ is
$g^!$ defined in \secref{sss:! for stacks}.

\sssec{} 

Consider the tautological functor
\begin{equation}  \label{e:schemes to stacks, corr}
(\dgSch_{\on{aft}})_{\on{corr:all;all}}\to (\inftydgprestack_{\on{laft}})_{\on{corr:sch-qc;all}},
\end{equation}
and define the functor
\begin{equation} \label{e:IndCoh corr on prestacks}
\IndCoh_{(\inftydgprestack_{\on{laft}})_{\on{corr:sch-qc;all}}}:
(\inftydgprestack_{\on{laft}})_{\on{corr:sch-qc;all}}\to \StinftyCat_{\on{cont}}
\end{equation}
as the right Kan extension of $\IndCoh_{(\dgSch_{\on{aft}})_{\on{corr:all;all}}}$ along the functor \eqref{e:schemes to stacks, corr}.

\begin{prop}   \label{p:correspondences for stacks}
The diagram of functors
$$
\CD
(\inftydgprestack_{\on{laft}})_{\on{corr:sch-qc;all}} @>{\IndCoh_{(\inftydgprestack_{\on{laft}})_{\on{corr:sch-qc;all}}}}>>   
\StinftyCat_{\on{cont}}  \\
@AAA    @AA{\on{Id}}A    \\
(\on{PreStk}_{\on{laft}})^{\on{op}}   @>{\IndCoh^!_{\inftydgprestack_{\on{laft}}}}>>   \StinftyCat_{\on{cont}} 
\endCD
$$
is canonically commutative
\end{prop}

\begin{proof}

This follows from \propref{p:RKE corr}.

\end{proof}

\ssec{Adjunction for proper maps}

Let $(\inftydgprestack_{\on{laft}})_{\on{corr:sch-proper;all}}$ be the 1-full subcategory of
$(\inftydgprestack_{\on{laft}})_{\on{corr:sch-qc;all}}$, where we restrict vertical morphisms to be
proper. 

\medskip

In this subsection we shall study the restriction of the functor $\IndCoh_{(\inftydgprestack_{\on{laft}})_{\on{corr:sch-qc;all}}}$
to this subcategory.

\sssec{}

Let us return to the setting of \secref{sss:RKE corr}. Assume that $vert^1\subset horiz^1$ and 
$vert^2\subset horiz^2$.

\medskip

Let us start with a functor
$$P^!_{horiz}:(\bC^1_{horiz})^{\on{op}}\to \StinftyCat_{\on{cont}}.$$
Assume that $P^!_{horiz}$ satisfies the right (resp., left)
base change condition with respect to $vert^1$, see \secref{sss:by adjunction}.
I.e., for every $1$-morphism $f:\wt\bc^1\to \bc^1$ in $\bC^1$, with $f\in vert^1$, the functor
$$P^!_{horiz}:P(\bc^1)\to P(\wt\bc^1)$$
admits a left (right) adjoint, denoted $P_{vert}(f)$, and that for a Cartesian square 
$$
\CD
\wt\bc'{}^1 @>{\wt{g}}>>  \wt\bc^1 \\
@V{f'}VV  @VV{f}V  \\
\bc'{}^1 @>{g}>>  \bc^1
\endCD
$$
the resulting natural transformation 
$$P_{vert}(f')\circ P^!_{horiz}(\wt{g})\to P^!_{horiz}(g)\circ P_{vert}(f)$$ 
(in the case of left adjoints) and 
$$P^!_{horiz}(g)\circ P_{vert}(f)\to P_{vert}(f')\circ P^!_{horiz}(\wt{g})$$
(in the case of right adjoints) 
is an isomorphism. 

\medskip

We claim:

\begin{prop}  \label{p:RKE of base change}
Under the assumptions of \propref{p:RKE corr}, the functor 
$$Q^!_{horiz}:=\on{RKE}_{(\Phi_{horiz})^{\on{op}}}(P^!_{horiz}):(\bC^2_{horiz})^{\on{op}}\to \StinftyCat_{\on{cont}}$$
satisfies the right (resp., left) base change condition with respect to $vert^2\subset horiz^2$.
\end{prop}

\begin{proof}

Let $f:\wt\bc^2\to \bc^2$ be a 1-morphism in $vert^2$. The condition of the proposition
implies that the assignment
$$\bc^1\in \bC^1\underset{\bC^2}\times (\bC^2_{horiz})_{/\bc^2}\, \rightsquigarrow \,
\Phi_{horiz}(\bc^1)\underset{\bc^2}\times \wt\bc_2\simeq \Phi_{horiz}(\wt\bc^1)$$
defines a functor
$$ \bC^1\underset{\bC^2}\times (\bC^2_{horiz})_{/\bc^2}\to 
\bC^1\underset{\bC^2}\times (\bC^2_{horiz})_{/\wt\bc^2},$$
which, moreover, is cofinal.
 
\medskip

Hence, we can calculate the value of $Q^!_{horiz}$ on $\wt\bc^2$
as
$$\underset{\bc^1\in (\bC^1\underset{\bC^2}\times (\bC^2_{horiz})_{/\bc^2})^{\on{op}}}{lim}\,
P^!_{horiz}(\wt\bc^1).$$

The existsence of the left (resp., right) adjoint $Q_{vert}(f)$ of $Q^!_{horiz}(f)$ follows now from the
following general paradigm:

\medskip

Let $I$ be an index category and 
$$i\mapsto \bC_i \text{ and } i\mapsto \wt\bC_i$$
be two functors $I\to \StinftyCat_{\on{cont}}$, and let 
$$i\mapsto F_i\in \on{Funct}_{\on{cont}}(\wt\bC_i,\bC_i)$$
be a natural transformations between them.

\medskip

Denote
$$\bC:=\underset{i\in I}{lim}\, \bC_i \text{ and } \wt\bC:=\underset{i\in I}{lim}\, \wt\bC_i,$$
and let $F$ be the corresponding functor $\wt\bC\to \bC$. 

\medskip

Assume that for everi $i$, the functor $F_i$ admits a left (resp., continuous right)
adjoint $G_i$, and that for every arrow $i\to i'$ in $I$, the square
$$
\CD
\wt\bC_i  @>>>  \wt\bC_{i'} \\
@A{G_i}AA   @AA{G_{i'}}A   \\
\bC_i  @>>>  \bC_{i'},
\endCD
$$
obtained by adjunction from the commutative square
$$
 \CD
\wt\bC_i  @>>>  \wt\bC_{i'} \\
@V{F_i}VV   @VV{F_{i'}}V   \\
\bC_i  @>>>  \bC_{i'},
\endCD
$$
which a priori commutes up to a natural transformation, actually commutes.

\begin{lem}  \label{l:adjunction in limit}
Under the above circumstances, the functor $F:\wt\bC\to \bC$ admits a left (resp., continuous right), denoted $G$,
and for every $i\in I$ the square,
$$
\CD
\wt\bC  @>>>  \wt\bC_{i} \\
@A{G}AA   @AA{G_{i}}A   \\
\bC @>>>  \bC_{i},
\endCD
$$
obtained by adjunction from the commutative square
$$
\CD
\wt\bC  @>>>  \wt\bC_{i} \\
@V{F}VV   @VV{F_{i'}}V   \\
\bC @>>>  \bC_{i},
\endCD
$$
which a priori commutes up to a natural transformation, actually commutes.
\end{lem}

The fact that $Q^!_{horiz}$ satisfies the right (resp., left) base change condition with respect to $vert^2\subset horiz^2$
is a formal consequence of the second statament in \lemref{l:adjunction in limit}.

\end{proof}

We shall need also the following statement which will be proved in \cite{GR3} along with 
\thmref{t:extension by adjoints}:

\begin{lem} \label{l:RKE of base change}
Let 
$$P_{\on{corr}:vert;horiz}:\bC^1_{\on{corr}:vert;horiz}\to \StinftyCat_{\on{cont}}$$
be the functor obtained from $P^!_{horiz}$ by \thmref{t:extension by adjoints} applied to
$vert^1\subset horiz^1$. 
Let 
$$Q_{\on{corr}:vert;horiz}:\bC^1_{\on{corr}:vert;horiz}\to \StinftyCat_{\on{cont}}$$
be the right Kan extension of $P_{\on{corr}:vert;horiz}$ along the functor
$$\Phi_{\on{corr}:vert;horiz}:\bC^1_{\on{corr}:vert;horiz}\to \bC^2_{\on{corr}:vert;horiz}.$$
Then in terms of the isomorphism
$$Q_{\on{corr}:vert;horiz}|_{(\bC^2_{horiz})^{\on{op}}}\simeq Q^!_{horiz}$$
of \propref{p:RKE corr}, the functor $Q_{\on{corr}:vert;horiz}$ identifies with one 
obtained from $Q^!_{horiz}$ by \thmref{t:extension by adjoints} applied to
$vert^2\subset horiz^2$. 
\end{lem}

\sssec{}

We apply \propref{p:RKE of base change} to the functor
$$\Phi:\dgSch_{\on{aft}}\to \on{PreStk}_{\on{laft}}$$
with $horiz^1=\on{all}$, $horiz^2=\on{all}$ and $vert^1=\on{proper}$,  $vert^1=\on{sch-proper}$, and
$$P^!_{horiz}:=\IndCoh^!_{\dgSch_{\on{aft}}}.$$

\medskip

By \corref{c:IndCoh via all schemes}, the resulting functor
$$(\on{PreStk}_{\on{laft}})^{\on{op}}\to \StinftyCat_{\on{cont}}$$
identifies with $\IndCoh_{\on{PreStk}_{\on{laft}}}$.

\medskip

Thus, we obtain that the functor $\IndCoh_{\on{PreStk}_{\on{laft}}}$ satisfies the right base change
condition with respect the class of schematic and proper maps. Applying \thmref{t:extension by adjoints},
we obtain a functor
$$(\inftydgprestack_{\on{laft}})_{\on{corr:sch-proper;all}}\to \StinftyCat_{\on{cont}}$$
that we shall denote by $\IndCoh_{(\inftydgprestack_{\on{laft}})_{\on{corr:sch-proper;all}}}$.

\medskip

We now claim:

\begin{prop}
There exists a canonical isomorphism
$$\IndCoh_{(\inftydgprestack_{\on{laft}})_{\on{corr:sch-qc;all}}}|_{(\inftydgprestack_{\on{laft}})_{\on{corr:sch-proper;all}}}\simeq
\IndCoh_{(\inftydgprestack_{\on{laft}})_{\on{corr:sch-proper;all}}},$$
compatible with the further restriction under
$$(\on{PreStk}_{\on{laft}})^{\on{op}}\hookrightarrow (\inftydgprestack_{\on{laft}})_{\on{corr:sch-proper;all}}.$$
\end{prop}

This follows by applying \lemref{l:RKE of base change} using the following statement implicit in the proof
of \thmref{t:upper shriek}, and which contains \propref{p:proper compat}(a) as a particular case:

\begin{prop}
The restriction of the functor $\IndCoh_{(\dgSch_{\on{aft}})_{\on{corr:all;all}}}$ under
$$(\dgSch_{\on{aft}})_{\on{corr:proper;all}}\to (\dgSch_{\on{aft}})_{\on{corr:all;all}}$$
identifies canonically with the functor obtained from 
$$\IndCoh^!_{\dgSch_{\on{aft}}}:(\dgSch_{\on{aft}})^{\on{op}}\to  \StinftyCat_{\on{cont}}$$
by \thmref{t:extension by adjoints} applied to $\on{proper}\subset \on{all}$.
\end{prop}

\section{$\IndCoh$ on Artin stacks}  \label{s:Artin}

\ssec{Recap: Artin stacks}

In this subsection we will recall some facts concerning Artin stacks. 
We refer the reader to \cite{Stacks}, Sect. 4.9 for a more detailed discussion.

\sssec{}

We let $\inftydgstack_{\on{Artin}}$ denote the full subcategory of
$\inftydgprestack$ consisting of Artin stacks. For $k\in \BN$, we let 
$\inftydgstack_{k\on{-Artin}}$ denote the full subcategory of $k$-Artin stacks.

\medskip

We set by definition:
$$\inftydgstack_{(-1)\on{-Artin}}=\dgSch,$$
(i.e., all DG schemes).
$$\inftydgstack_{(-2)\on{-Artin}}=\dgSch_{\on{sep}},$$
(i.e., separated DG schemes)
$$\inftydgstack_{(-3)\on{-Artin}}=\affdgSch.$$

\sssec{}

We shall say that a morphism in $\on{PreStk}$ is ``$k$-representable"
if its base change by any affine DG scheme yields an object of $\inftydgstack_{k\on{-Artin}}$.
E.g., ``$(-1)$-representable" is the same as ``schematic." 

\medskip

We shall say that a morphism in $\on{PreStk}$ is ``eventually representable"
if is $k$-representable for some $k$.

\medskip

We shall say that an eventually representable morphism $\CY_1\to \CY_2$ is smooth/flat/of bounded Tor dimension/
eventually coconnective if
for every $S_2\in \affdgSch$ equipped with a map to $\CY_2$, and $S_1\in \affdgSch$,
equipped with a smooth map to $S_2\underset{\CY_2}\times \CY_1$, the resulting map 
$S_1\to S_2$ is smooth/flat/of bounded Tor dimension/eventually coconnective.

\medskip

If $\CY_2$ is itself an Artin stack, it is enough to test the above condition for those maps $S_2\to \CY_2$
that are smooth (or flat), and in fact for just one smooth (or flat) covering of $\CY_2$. 

\medskip

We note that the diagonal morphism of a $k$-Artin stack is $(k-1)$-representable.

\sssec{}

We let $\inftydgstack_{\on{laft},\on{Artin}}$ denote the full subcategory of $\on{PreStk}_{\on{laft}}$ equal to
$$\on{PreStk}_{\on{laft}}\cap \inftydgstack_{\on{Artin}},$$
and similarly,
$$\inftydgstack_{\on{laft},k\on{-Artin}}=\on{PreStk}_{\on{laft}}\cap \inftydgstack_{k\on{-Artin}}.$$

\medskip

We also note that in order to check that in order to check that a morphism $f:\CY_1\to \CY_2$ in 
$\on{PreStk}_{\on{laft}}$ is $k$-representable (resp., $k$-representable and smooth/flat/of bounded Tor dimension/eventually coconnective)
it is enough to do so for $S_2\to \CY_2$ with $S_2\in {}^{<\infty}\!\dgSch_{\on{aft}}$. 

\sssec{}

The following is a basic fact concerning the subcategory 
$$\inftydgstack_{\on{laft},\on{Artin}}\subset  \inftydgstack_{\on{Artin}},$$
see \cite[Proposition 4.9.4]{Stacks}:

\begin{lem} 
For a $\CY\in \inftydgstack_{\on{laft},\on{Artin}}$ and a smooth map $S\to \CY$, where $S\in \affdgSch$,
the DG scheme $S$ is almost of finite type.
\end{lem}

In particular:
\begin{cor} \label{c:existence of laft covering}
Every object of $\inftydgstack_{\on{laft},\on{Artin}}$ admits a smooth surjective map
from $S\in \dgSch_{\on{laft}}$.
\end{cor}

\sssec{}

Let 
$$\IndCoh^!_{\inftydgstack_{\on{laft},\on{Artin}}}:(\inftydgstack_{\on{laft},\on{Artin}})^{\on{op}}\to \StinftyCat_{\on{cont}}$$
(resp., $\IndCoh^!_{\inftydgstack_{\on{laft},k\on{-Artin}}}$)
denote the restriction of the functor 
$$\IndCoh^!_{\on{PreStk}_{\on{laft}}}:(\on{PreStk}_{\on{laft}})^{\on{op}}\to \StinftyCat_{\on{cont}}$$
to the corresponding subcategory.

\ssec{Recovering from smooth/flat/eventually coconnective maps}

In this subsection we will show that if $\CY$ is an Artin stack, the category $\IndCoh(\CY)$
can be recovered from just looking at affine DG schemes equipped with a \emph{smooth}
map to $\CY$. 

\sssec{}

Let $\fc$ be a class of morphisms between Artin stacks belonging to the following set
$$\on{smooth} \subset \on{flat} \subset  \on{bdd-Tor}\subset  \on{ev-coconn}\subset \on{all}.$$

\medskip

Consider the corresponding fully faithful embedding
$$(\affdgSch)_{\fc}\hookrightarrow (\inftydgstack_{\on{Artin}})_{\fc}.$$

\medskip

We claim:

\begin{prop}  \label{p:RKE from smooth general}
Let $P$ be a presheaf on $(\inftydgstack_{\on{Artin}})_{\fc}$ with values 
in an arbitrary $\infty$-category. Assume that $P$ satisfies descent 
with respect to smooth surjective maps. Then the map
$$P\to \on{RKE}_{((\affdgSch)_{\fc})^{\on{op}}\hookrightarrow ((\inftydgstack_{\on{Artin}})_{\fc})^{\on{op}}}(P)$$
is an isomorphism.
\end{prop}

\begin{proof}

It is enough to prove the claim after the restriction to
$$(\inftydgstack_{k\on{-Artin}})_{\fc}\subset (\inftydgstack_{\on{Artin}})_{\fc}$$
for every $k$.

\medskip

We will argue by induction on $k$. For $k=-3$, the assertion is tautological. The induction step
follows from \propref{p:density bis} and \corref{c:existence of laft covering}, applied to:

$$\bC:=\inftydgstack_{k\on{-Artin}},\,\, \bC':=\inftydgstack_{(k-1)\on{-Artin}},\,\,
\bC_0:=(\inftydgstack_{k\on{-Artin}})_{\fc},$$
and the smooth topology. 

\end{proof}

The above proposition can be reformulated as follows:

\begin{cor} \label{c:RKE from smooth general}
For $P$ as in \propref{p:RKE from smooth general} and $\CY\in \inftydgstack_{\on{Artin}}$, the natural map
$$P(\CY)\to \underset{S\in 
((\affdgSch)_{\fc}\underset{(\inftydgstack_{\on{Artin}})_{\fc}}\times ((\inftydgstack_{\on{Artin}})_{\fc})_{/\CY})^{\on{op}}}{lim}\, P(S)$$
is an isomorphism.
\end{cor} 

We empasize that 
$$(\affdgSch)_{\fc}\underset{(\inftydgstack_{\on{Artin}})_{\fc}}\times ((\inftydgstack_{\on{Artin}})_{\fc})_{/\CY}$$
is the 1-full subcategory of $(\affdgSch)_{/\CY}$ spanned by those 
$S\to \CY$ whose map to $\CY$ belongs to $\fc$, and where we restrict 1-morphisms to those maps $S_1\to S_2$
that themselves belong to $\fc$. 

\medskip

As a corollary, we obtain:

\begin{cor} \label{c:RKE from smooth general non qc}
Under the assumptions of \corref{c:RKE from smooth general}, the maps
\begin{multline*}
P(\CY)\to 
\underset{S\in ((\dgSch)_{\fc}\underset{(\inftydgstack_{\on{Artin}})_{\fc}}\times 
((\inftydgstack_{\on{Artin}})_{\fc})_{/\CY})^{\on{op}}}{lim}\, P(S)\to \\
\to P(\CY)\to \underset{S\in (((\dgSch)_{\on{qsep-qc}})_{\fc}
\underset{(\inftydgstack_{\on{Artin}})_{\fc}}\times ((\inftydgstack_{\on{Artin}})_{\fc})_{/\CY})^{\on{op}}}{lim}\, P(S)
\end{multline*}
are also isomorphisms.
\end{cor}

\sssec{}

Let us fix $\CY\in \inftydgstack_{\on{Artin}}$, and let
$$(\inftydgstack_{\on{Artin}})_{\fc\,\on{over}\,\CY}\subset (\inftydgstack_{\on{Artin}})_{/\CY}$$
be the full subcategory spanned by those $f:\CY'\to \CY$, where $f$ belongs to $\fc$. 

\medskip

For another class $\fc'$ from the collection
$$\on{smooth} \subset \on{flat} \subset  \on{bdd-Tor}\subset  \on{ev-coconn}\subset \on{all},$$
consider the 1-full subcategory
$$((\inftydgstack_{\on{Artin}})_{\fc\,\on{over}\,\CY})_{\fc'}$$
and its full subcategory $((\affdgSch)_{\fc\,\on{over}\,\CY})_{\fc'}$.

\medskip

To decipher this, $((\dgSch)_{\fc\,\on{over}\,\CY})_{\fc'}$ is the 1-full subcategory of
$(\dgSch_{\on{Artin}})_{/\CY}$, spanned by those $f:S\to \CY$, where $f$ belongs to $\fc$,
and where we restrict 1-morphisms to those maps $S_1\to S_2$ that belong to $\fc'$.

\sssec{}

As in \propref{p:RKE from smooth general}, we have:

\begin{prop}  \label{p:RKE from smooth general bis}
Let $P$ be a presheaf on $((\inftydgstack_{\on{Artin}})_{\fc\,\on{over}\,\CY})_{\fc'}$ with values 
in an arbitrary $\infty$-category. Assume that $P$ satisfies descent 
with respect to smooth surjective maps. Then the map
$$P\to \on{RKE}_{(((\affdgSch)_{\fc\,\on{over}\,\CY})_{\fc'})^{\on{op}}\hookrightarrow 
(((\inftydgstack_{\on{Artin}})_{\fc\,\on{over}\,\CY})_{\fc'})^{\on{op}}}(P)$$
is an isomorphism.
\end{prop}

\begin{cor}
Under the assumptions of \corref{c:RKE from smooth general}, the maps 
\begin{multline*}
P(\CY)\to \underset{S\in (((\dgSch)_{\fc\,\on{over}\,\CY})_{\fc'})^{\on{op}}}{lim}\, P(S)\to \\
\to P(\CY)\to \underset{S\in ((((\dgSch)_{\on{qsep-qc}})_{\fc\,\on{over}\,\CY})_{\fc'})^{\on{op}}}{lim}\, P(S)\to \\
\to \underset{S\in (((\affdgSch)_{\fc\,\on{over}\,\CY})_{\fc'})^{\on{op}}}{lim}\, P(S)
\end{multline*}
are isomorphisms. 
\end{cor}

\sssec{}

Applying the above discussion to $\IndCoh^!_{\inftydgstack_{\on{laft},\on{Artin}}}$, we obtain:

\begin{prop} \label{p:recovering from smooth} 
Let $\CY$ be an object of $\inftydgstack_{\on{laft},\on{Artin}}$. Let $\fc$ be
one of the classes 
$$\on{smooth} \subset \on{flat} \subset \on{ev-coconn}\subset \on{all}.$$
Then the restriction maps
\begin{multline*} 
\IndCoh(\CY)\to \underset{S\in (((\dgSch)_{\on{aft}})_{\fc}\underset{(\inftydgstack_{\on{Artin}})_{\fc}}\times 
((\inftydgstack_{\on{Artin}})_{\fc})_{/\CY})^{\on{op}}}{lim}\, \IndCoh(S)\to \\
\to \underset{S\in (((\affdgSch)_{\on{aft}})_{\fc}\underset{(\inftydgstack_{\on{Artin}})_{\fc}}\times 
((\inftydgstack_{\on{Artin}})_{\fc})_{/\CY})^{\on{op}}}{lim}\, \IndCoh(S)
\end{multline*}
are isomorphisms.
\end{prop}

Similarly, we have: 

\begin{prop} \hfill \label{p:recovering from smooth bis} 
Let $\CY$ be an object of $\inftydgstack_{\on{laft},\on{Artin}}$. Let $\fc$and $\fc'$ be any two 
of the classes 
$$\on{smooth} \subset \on{flat} \subset \on{ev-coconn}\subset \on{all}.$$
Then the restriction maps
\begin{multline*} 
\IndCoh(\CY)\to \underset{S\in ((((\dgSch)_{\on{aft}})_{\fc\,\on{over}\,\CY})_{\fc'})^{\on{op}}}{lim}\, \IndCoh(S)\to \\
\to \underset{S\in ((((\affdgSch)_{\on{aft}})_{\fc\,\on{over}\,\CY})_{\fc'})^{\on{op}}}{lim}\, \IndCoh(S)
\end{multline*}
are isomorphisms.
\end{prop}

\sssec{}

The upshot of the above two propositions is that the category $\IndCoh(\CY)$ is recovered
from the knowledge of $\IndCoh(S)$ where $S$ belongs to $\affdgSch_{\on{aft}}$ (resp., $\dgSch_{\on{aft}}$, 
$\dgSch_{\on{laft}}$), endowed with a smooth/flat/eventually coconnective/arbitrary map to $\CY$.

\medskip

Furthermore, we can either take all maps between the schemes $S$, or restrict them to be 
smooth/flat/eventually coconnective.

\ssec{The $*$-version}

We are going to introduce another functor $$((\inftydgstack_{\on{laft},\on{Artin}})_{\on{ev-coconn}})^{\on{op}}\to \StinftyCat_{\on{cont}},$$
denoted $\IndCoh^*_{(\inftydgstack_{\on{laft},\on{Artin}})_{\on{ev-coconn}}}$.

\sssec{}

We set
\begin{multline*}\IndCoh^*_{(\inftydgstack_{\on{laft},\on{Artin}})_{\on{ev-coconn}}}:=\\
=\on{RKE}_{(((\affdgSch_{\on{aft}}))_{\on{ev-coconn}})^{\on{op}}\hookrightarrow 
((\inftydgstack_{\on{laft},\on{Artin}})_{\on{ev-coconn}})^{\on{op}}}
(\IndCoh^*_{(\affdgSch_{\on{aft}})_{\on{ev-coconn}}}),
\end{multline*}
where 
$$\IndCoh^*_{(\affdgSch_{\on{aft}})_{\on{ev-coconn}}}:((\affdgSch_{\on{aft}})_{\on{ev-coconn}})^{\on{op}}\to
\StinftyCat_{\on{cont}}$$
is the functor obtained from the functor $\IndCoh^*_{(\affdgSch_{\on{Noeth}})_{\on{ev-coconn}}}$ 
of \corref{c:upper * DG funct} by restriction along
$$((\affdgSch_{\on{aft}})_{\on{ev-coconn}})^{\on{op}}\hookrightarrow
((\affdgSch_{\on{Noeth}})_{\on{ev-coconn}})^{\on{op}}.$$

\sssec{}

In other words, 
$$\IndCoh^*(\CY)=\underset{S\to \CY}{lim}\, \IndCoh^*(S),$$
where the limit is taken over the category opposite to
$$(\affdgSch_{\on{aft}})_{\on{ev-coconn}}\underset{(\inftydgstack_{\on{laft,Artin}})_{\on{ev-coconn}}}\times 
((\inftydgstack_{\on{laft,Artin}})_{\on{ev-coconn}})_{/\CY},$$
which is a 1-full subcategory of $(\affdgSch_{\on{aft}})_{/\CY}$, spanned by those $S\to \CY$,
which are eventually coconnective, and where 1-morphisms $f:S_1\to S_2$ are restricted to also
be eventually coconnective. 

\medskip

In the above formula, for $S\in \affdgSch_{\on{aft}}$, the category $\IndCoh^*(S)$ is the usual
$\IndCoh(S)$, and for a 1-morphism $f:S_1\to S_2$, the functor $\IndCoh(S_2)\to \IndCoh(S_1)$ is $f^{\IndCoh,*}$.

\sssec{}

We claim:

\begin{lem}  \label{l:* descent Artin}
The functor $\IndCoh^*_{(\inftydgstack_{\on{laft},\on{Artin}})_{\on{ev-coconn}}}$ satisfies descent with respect
to smooth surjective morphisms.
\end{lem}

\begin{proof}

Follows from the fact that the category $\inftydgstack_{\on{laft},\on{Artin}}$ embeds fully fiathfully into
$$\on{Funct}((\affdgSch_{\on{aft}})^{\on{op}},\inftygroup),$$
combined with \cite{Lu0}, Cor. 6.2.3.5 and \propref{p:*-descent smooth}.

\end{proof}

From Corollaries \ref{c:RKE from smooth general} and \ref{c:RKE from smooth general non qc}, we obtain:

\begin{cor}
Let $\fc$ be one of the classes 
$$\on{smooth} \subset \on{flat} \subset \on{ev-coconn}.$$
Then for $\CY\in \inftydgstack_{\on{laft},\on{Artin}}$
the restriction functors 
\begin{multline*}
\IndCoh^*(\CY)
\to \underset{S\in (((\dgSch)_{\on{aft}})_{\fc}\underset{(\inftydgstack_{\on{laft,Artin}})_{\fc}}\times 
((\inftydgstack_{\on{laft,Artin}})_{\fc})_{/\CY})^{\on{op}}}{lim}\, \IndCoh^*(S)\to \\
\to \underset{S\in (((\affdgSch)_{\on{aft}})_{\fc}\underset{(\inftydgstack_{\on{laft,Artin}})_{\fc}}\times 
((\inftydgstack_{\on{laft,Artin}})_{\fc})_{/\CY})^{\on{op}}}{lim}\, \IndCoh^*(S)
\end{multline*}
are isomorphisms.
\end{cor}

\ssec{Comparing the two versions of $\IndCoh$: the case of algebraic stacks} \label{ss:!* comparison}

In this subsection we will show that for $\CY\in \inftydgstack_{\on{laft,Artin}}$, which is an algebraic stack, 
the categories $\IndCoh^*(\CY)$ and $\IndCoh^!(\CY):=\IndCoh(\CY)$ are canonically
equivalent.

\sssec{}   \label{sss:alg stack}

Our conventions regarding algebraic stacks follow those of \cite[Sect. 1.1.3]{DrGa1}.
Namely, an algebraic stack is an object of $\inftydgstack_{1\on{-Artin}}$, for which 
the diagonal morphism 
$$\CY\to \CY\times \CY$$
is schematic, quasi-separated and quasi-compact.

\sssec{}

We are going to prove:

\begin{prop} \label{p:!* comparison alg}
For an algebraic stack $\CY$ locally almost of finite type, there exists a canonical equivalence 
$$\IndCoh^!(\CY)\simeq \IndCoh^*(\CY),$$
such that for a Cartesian square
$$
\CD
S @>{g''}>>  S'' \\
@V{f'}VV   @VV{f}V  \\
S' @>{g}>>  \CY,
\endCD
$$
with $S',S''\in \dgSch_{\on{aft}}$, the morphism $g$ being arbitrary and $f$ being eventually coconnective,
the functors
$$(f')^{\IndCoh,*}\circ g^!:\IndCoh^!(\CY)\to \IndCoh(S) \text{ and }
(g'')^!\circ f^{\IndCoh,*}:\IndCoh^*(\CY)\to \IndCoh(S)$$
are canonically identified.
\end{prop}

The rest of this subsection is devoted to the proof of this proposition.

\sssec{}  \label{sss:bootstrapping from Segal}

By restricting to DG schemes almost of finite type, \propref{p:!* IndCoh} defines a map
in $\inftygroup^{(\bDelta\times \bDelta)^{\on{op}}}$:
$$\IndCoh^{*!}_{(\dgSch_{\on{aft}})_{\on{ev-coconn;all}}}: 
\on{Cart}^{\bullet,\bullet}_{\on{ev-coconn;all}}(\dgSch_{\on{aft}})\to 
\on{Seg}^{\bullet,\bullet}((\StinftyCat_{\on{cont}})^{\on{op}}).$$

\medskip

Using \thmref{t:Segal}, the map $\IndCoh^{*!}_{(\dgSch_{\on{aft}})_{\on{ev-coconn;all}}}$ 
gives rise to the following construction:

\medskip

Let $\bI'$ and $\bI''$ be a pair of $\infty$-categories, and let $F$ be a functor
$$\bI'\times \bI''\to \dgSch_{\on{aft}},$$ with the property that for every
$\bi'\in \bI'$ and a 1-morphism $(\bi''_0\to \bi''_1)\in \bI''$, the map
$$F(\bi'\times \bi''_0)\to F(\bi'\times \bi''_1)$$ is
eventually coconnective, and for any pair
of 1-morphisms 
$$(\bi'_0\to \bi'_1)\in \bI' \text{ and } (\bi''_0\to \bi''_1)\in \bI'',$$
the square
$$
\CD
F(\bi'_0\times \bi''_0)  @>>>  F(\bi'_1\times \bi''_0)  \\
@VVV    @VVV   \\
F(\bi'_0\times \bi''_1)  @>>>  F(\bi'_1\times \bi''_1)  
\endCD
$$
is Cartesian. 

\medskip

Then the map $\IndCoh^{*!}_{(\dgSch_{\on{aft}})_{\on{ev-coconn;all}}}$ canonically attaches to a functor $F$ as above,
a datum of a functor
$$\IndCoh^{*!}_{\dgSch_{\on{aft}}}\circ F:(\bI'\times \bI'')^{\on{op}}\to \StinftyCat_{\on{cont}}.$$

\medskip

Moreover, for every fixed object $\bi'\in \bI'$ (resp., $\bi''\in \bI''$) the resulting functors
$$\IndCoh^{*!}_{\dgSch_{\on{aft}}}\circ F|_{\bi'\times \bI''}:(\bI'')^{\on{op}}\to \StinftyCat_{\on{cont}}$$
and
$$\IndCoh^{*!}_{\dgSch_{\on{aft}}}\circ F|{\bI'\times \bi''}:(\bI')^{\on{op}}\to \StinftyCat_{\on{cont}}$$
identify with 
$$\IndCoh^*_{\dgSch_{\on{aft}}}\circ F|_{\bi'\times \bI''} \text{ and }
\IndCoh^!_{\dgSch_{\on{aft}}}\circ F|{\bI'\times \bi''},$$
respectively.

\sssec{}

We let
$$\bI':=(\dgSch_{\on{aft}})_{/\CY} \text{ and } \bI'':=(\dgSch_{\on{aft}})_{\on{ev-coconn}}
\underset{(\inftydgstack_{\on{laft,Artin}})_{\on{ev-coconn}}}\times ((\inftydgstack_{\on{laft,Artin}})_{\on{ev-coconn}})_{/\CY}.$$

We let $F$ be the functor that sends
$$(g:S'\to \CY),(f:S''\to \CY)\mapsto S'\underset{\CY}\times S''.$$

\medskip

Applying the construction from \secref{sss:bootstrapping from Segal}, we obtain
a functor
$$\IndCoh^{*!}_{\dgSch_{\on{aft}}}\circ F.$$

\medskip

We claim that there are canonical equivalences
$$\IndCoh^!(\CY)\simeq \underset{(\bI'\times \bI'')^{\on{op}}}{lim}\, \IndCoh^{*!}_{\dgSch_{\on{aft}}}\circ F\simeq \IndCoh^*(\CY),$$
which would imply the assertion of \propref{p:!* comparison alg}.

\sssec{}

Let us construct the isomorphism 
$$\IndCoh^!(\CY)\simeq \underset{(\bI'\times \bI'')^{\on{op}}}{lim}\, \IndCoh^{*!}_{\dgSch_{\on{aft}}}\circ F.$$

We calculate
$$\underset{(\bI'\times \bI'')^{\on{op}}}{lim}\, \IndCoh^{*!}_{\dgSch_{\on{aft}}}\circ F$$
as the limit over $(\bI')^{\on{op}}$ of the functor that sends
$(g:S'\to \CY)$ to 
$$\underset{(f:S''\to \CY)\in (\bI'')^{\on{op}}}{lim}\, \IndCoh^*(S'\underset{\CY}\times S'').$$

\medskip

However, we claim:

\begin{lem}  \label{l:limit over representable}
The natural map
$$\IndCoh(S')\to \underset{(f:S''\to \CY)\in (\bI'')^{\on{op}}}{lim}\, \IndCoh^*(S'\underset{\CY}\times S'')$$
is an isomorphism. 
\end{lem}

Hence, we obtain that the above functor on $(\bI')^{\on{op}}$ identifies canonically with 
$$(\IndCoh^!_{\dgSch_{\on{aft}}})|_{(\dgSch_{\on{aft}})_{/\CY}}.$$

The limit of the latter is $\IndCoh^!(\CY)$, by definition.

\sssec{}

The isomorphism 
$$\underset{(\bI'\times \bI'')^{\on{op}}}{lim}\, \IndCoh^{*!}_{\dgSch_{\on{aft}}}\circ F\simeq \IndCoh^*(\CY)$$
is constructed similarly.

\sssec{Proof of \lemref{l:limit over representable}}

The functor
$$(f:S''\to \CY)\mapsto \IndCoh^*(S'\underset{\CY}\times S'')$$
is a presheaf on $\bI''$ that satisfies descent with respect to smooth
surjective maps. 

\medskip

Hence, the limit $$\underset{(f:S''\to \CY)\in (\bI'')^{\on{op}}}{lim}\, \IndCoh^*(S'\underset{\CY}\times S'')$$
can be calculated as
$$\on{Tot}(\IndCoh^*(S'\underset{\CY}\times (S''_0{}^\bullet/\CY))),$$
where $S''_0\to \CY$ is some smooth cover, and $S''_0{}^\bullet/\CY$ is its
\v{C}ech nerve. 

\medskip

However, $S'\underset{\CY}\times (S''_0{}^\bullet/\CY)$ is the \v{C}ech nerve
of the map 
$$S'\underset{\CY}\times S''_0\to S',$$
and the required isomorphism follows from smooth descent for $\IndCoh^*_{\dgSch_{\on{aft}}}$.

\ssec{Comparing the two versions of $\IndCoh$: general Artin stacks}

In this subsection we will generalize the construction of \secref{ss:!* comparison} to arbitrary Artin stacks. 

\sssec{}

We consider the category $\inftydgstack_{\on{laft},\on{Artin}}$
endowed with the classes of 1-morphisms
$$(\on{ev-coconn};\on{all}),$$
and we consider the corresponding object
$$\on{Cart}^{\bullet,\bullet}_{\on{ev-coconn;all}}(\inftydgstack_{\on{laft},\on{Artin}})\in 
\inftygroup^{(\bDelta\times \bDelta)^{\on{op}}}.$$

We claim: 

\begin{prop} \label{p:!* IndCoh Artin}
There exists a uniquely defined map in $\inftygroup^{(\bDelta\times \bDelta)^{\on{op}}}$
$$\IndCoh^{*!}_{(\inftydgstack_{\on{laft},\on{Artin}})_{\on{ev-coconn;all}}}: 
\on{Cart}^{\bullet,\bullet}_{\on{ev-coconn;all}}(\inftydgstack_{\on{laft},\on{Artin}})\to 
\on{Seg}^{\bullet,\bullet}((\StinftyCat_{\on{cont}})^{\on{op}})$$
that makes the following diagrams commute
$$
\CD
\on{Cart}^{\bullet,\bullet}_{\on{ev-coconn;all}}(\inftydgstack_{\on{laft},\on{Artin}})  
@>{\IndCoh^{*!}_{(\inftydgstack_{\on{laft},\on{Artin}})_{\on{ev-coconn;all}}}}>>  
\on{Seg}^{\bullet,\bullet}((\StinftyCat_{\on{cont}})^{\on{op}}) \\
@AAA        @AAA  \\
\pi_h(\on{Seg}^\bullet(\inftydgstack_{\on{laft},\on{Artin}})  @>>{\on{Seg}^\bullet(\IndCoh^!_{\inftydgstack_{\on{laft},\on{Artin}}})}>
\pi_h(\on{Seg}^\bullet((\StinftyCat_{\on{cont}})^{\on{op}}))  
\endCD
$$
and
$$
\CD
\on{Cart}^{\bullet,\bullet}_{\on{ev-coconn;all}}(\inftydgstack_{\on{laft},\on{Artin}}) 
@>{\IndCoh^{*!}_{(\inftydgstack_{\on{laft},\on{Artin}})_{\on{ev-coconn;all}}}}>>  
\on{Seg}^{\bullet,\bullet}((\StinftyCat_{\on{cont}})^{\on{op}}) \\
@AAA        @AAA  \\
\pi_v(\on{Seg}^\bullet((\inftydgstack_{\on{laft},\on{Artin}})_{\on{ev-coconn}}))  
@>>{\on{Seg}^\bullet(\IndCoh^*_{(\inftydgstack_{\on{laft},\on{Artin}})_{\on{ev-coconn}}})}>
\pi_v(\on{Seg}^\bullet((\StinftyCat_{\on{cont}})^{\on{op}})),
\endCD
$$
and which extends the map 
$$\IndCoh^{*!}_{(\dgSch_{\on{aft}})_{\on{ev-coconn;all}}}: 
\on{Cart}^{\bullet,\bullet}_{\on{ev-coconn;all}}(\dgSch_{\on{aft}})\to 
\on{Seg}^{\bullet,\bullet}((\StinftyCat_{\on{cont}})^{\on{op}}).$$
\end{prop}

\begin{cor}
For $\CY\in \inftydgstack_{\on{laft},\on{Artin}}$ there exists a canonical equivalence 
$$\IndCoh^!(\CY)\simeq \IndCoh^*(\CY),$$
such that for a Cartesian square of Artin stacks
\begin{equation} \label{e:two pullbacks Artin}
\CD
\CZ @>{g''}>>  \CZ'' \\
@V{f'}VV   @VV{f}V  \\
\CZ' @>{g}>>  \CY,
\endCD
\end{equation}
where the morphism $g$ is arbitrary and $f$ is eventually coconnective,
the functors
$$(f')^{\IndCoh,*}\circ g^!:\IndCoh^!(\CY)\to \IndCoh(\CZ) \text{ and }
(g'')^!\circ f^{\IndCoh,*}:\IndCoh^*(\CY)\to \IndCoh(\CZ)$$
are canonically identified.
\end{cor}

The rest of this subsection is devoted to a sketch of the proof of \propref{p:!* IndCoh Artin};
details will be supplied elsewhere.

\sssec{}

We proceed by induction on $k$. Thus, we assume that the existence and uniqueness of the map
$$\IndCoh^{*!}_{(\inftydgstack_{\on{laft},(k-1)\on{-Artin}})_{\on{ev-coconn;all}}}: 
\on{Cart}^{\bullet,\bullet}_{\on{ev-coconn;all}}(\inftydgstack_{\on{laft},(k-1)\on{-Artin}})\to 
\on{Seg}^{\bullet,\bullet}((\StinftyCat_{\on{cont}})^{\on{op}})$$
with the required properties. 

\medskip

The base of the induction is $k=-3$, in which case the assertion follows by restriction from
\propref{p:!* IndCoh} inder $\affdgSch_{\on{aft}}\hookrightarrow \dgSch_{\on{aft}}$.

\medskip

We shall now perform the induction step. 

\begin{rem}
Note that by repeating the construction
in the proof of \propref{p:!* comparison alg}, we obtain that for an individual $\CY\in \inftydgstack_{\on{laft},k\on{-Artin}}$
there exists a canonical equivalence
$$\IndCoh^!(\CY)\simeq \IndCoh^*(\CY),$$
such that for every diagram \eqref{e:two pullbacks Artin} with $\CZ',\CZ''\in \inftydgstack_{\on{laft},(k-1)\on{-Artin}}$,
the diagram
$$
\CD
\IndCoh^*(\CZ)  @<{\sim}<< \IndCoh^!(\CZ)  @<{(g'')^!}<<  \IndCoh^!(\CZ'') @<{\sim}<< \IndCoh^*(\CZ'') \\
@A{(f')^{\IndCoh,*}}AA  & & & &  @AA{f^{\IndCoh,*}}A  \\
\IndCoh^*(\CZ') @<{\sim}<<  \IndCoh^!(\CZ')  @<{g^!}<< \IndCoh^!(\CY)  @<{\sim}<< \IndCoh^*(\CY)
\endCD
$$
is commutative.

\medskip

The above amounts to calculating the value of the map $\IndCoh^{*!}_{(\inftydgstack_{\on{laft},k\on{-Artin}})_{\on{ev-coconn;all}}}$
on $[0]\times [0]\in \bDelta\times \bDelta$. To make this construction functorial in $\CY$ amounts to constructing
the map $\IndCoh^{*!}_{(\inftydgstack_{\on{laft},k\on{-Artin}})_{\on{ev-coconn;all}}}$. The construction of the latter
follows the same idea: approximating Cartesian diagrams of $k$-Artin stacks by Cartesian diagrams of
$(k-1)$-Artin stacks.

\end{rem}

\sssec{}

As in \secref{sss:bootstrapping from Segal}, the datum of $\IndCoh^{*!}_{(\inftydgstack_{\on{laft},(k-1)\on{-Artin}})_{\on{ev-coconn;all}}}$
gives rise to the following construction:

\medskip

For a pair of $\infty$-categories
$\bI'$ and $\bI''$ and a functor 
$$F:\bI'\times \bI''\to \inftydgstack_{\on{laft},(k-1)\on{-Artin}}$$
with the corresponding properties, we obtain a functor
$$\IndCoh^{*!}_{(\inftydgstack_{\on{laft},(k-1)\on{-Artin}})_{\on{ev-coconn;all}}}\circ F:
(\bI'\times \bI'')^{\on{op}}\to \StinftyCat_{\on{cont}}.$$

\sssec{}

To carry out the induction step we need to construct, for every
$$[m]\times [n]\in \bDelta\times \bDelta,$$
a map of $\infty$-groupoids
$$\on{Cart}^{m,n}_{\on{ev-coconn;all}}(\inftydgstack_{\on{laft},k\on{-Artin}})\to 
\on{Seg}^{m,n}((\StinftyCat_{\on{cont}})^{\on{op}}),$$
in a way functorial in $[m]\times [n]$.

\medskip

We fix an object $\CY^{m,n}\in \on{Cart}^{m,n}_{\on{ev-coconn;all}}(\inftydgstack_{\on{laft},k\on{-Artin}})$
and will construct the corresponding functor
$$\IndCoh^{*!}(\CY^{m,n}):([m]\times [n])^{\on{op}}\to \StinftyCat_{\on{cont}}.$$

\sssec{}

We introduce an $\infty$-category 
$$_{\on{all}}\on{Cart}^{m,n}_{\on{ev-coconn;all}}(\inftydgstack_{\on{laft},k\on{-Artin}}),$$
which has the same objects as $\on{Cart}^{m,n}_{\on{ev-coconn;all}}(\inftydgstack_{\on{laft},k\on{-Artin}})$,
but where we now allow arbitrary maps between such diagrams. 

\medskip

Similarly, we introduce an $\infty$-category 
$$_{\on{ev-coconn}}\on{Cart}^{m,n}_{\on{ev-coconn;all}}(\inftydgstack_{\on{laft},k\on{-Artin}}),$$
which has the same objects as $\on{Cart}^{m,n}_{\on{ev-coconn;all}}(\inftydgstack_{\on{laft},k\on{-Artin}})$,
but where we allow arbitrary maps between diagrams that are eventually coconnective in each coordinate.

\medskip

We let $\bJ'$ be the full subcategory in
$$\left({}_{\on{all}}\on{Cart}^{m,n}_{\on{ev-coconn;all}}(\inftydgstack_{\on{laft},k\on{-Artin}})\right)_{/\CY^{m,n}},$$
whose objects are diagrams with entries from $\inftydgstack_{\on{laft},(k-1)\on{-Artin}}$.

\medskip

Similarly, we let $\bJ''$ be the full subcategory in 
$$\left({}_{\on{ev-coconn}}\on{Cart}^{m,n}_{\on{ev-coconn;all}}(\inftydgstack_{\on{laft},k\on{-Artin}})\right)_{/\CY^{m,n}}$$
whose objects are diagrams with entries from $\inftydgstack_{\on{laft},(k-1)\on{-Artin}}$.

\medskip

We set 
$$\bI':=\bJ'\times [m] \text{ and } \bI'':=\bJ''\times [n].$$

We define the functor
$$\bI'\times \bI''\to \inftydgstack_{\on{laft},(k-1)\on{-Artin}}$$
by sending
$$(g:\CZ'{}^{m,n}\to \CY^{m,n})\in \bJ',\,\, (f:\CZ''{}^{m,n}\to \CY^{m,n})\in \bJ'',\,\, a\in [0,m],\,\, b\in [0,n]$$
to 
$$\CZ'{}^{m,n}_{a,b}\underset{\CY^{m,n}_{a,b}}\times \CZ''{}^{m,n}_{a,b},$$
where the subscript ``$a,b$'' stands for the corresponding entry of $\CZ'{}^{m,n}$, $\CZ''{}^{m,n}$ and $\CY^{m,n}$,
respectively.

\medskip

Consider the corresponding functor
$$\IndCoh^{*!}_{(\inftydgstack_{\on{laft},(k-1)\on{-Artin}})_{\on{ev-coconn;all}}}\circ F:
(\bI'\times \bI'')^{\on{op}}\to \StinftyCat_{\on{cont}}.$$

\medskip

Finally, we define the sought-for functor $\IndCoh^{*!}(\CY^{m,n})$ as the right Kan extension of
$\IndCoh^{*!}_{(\inftydgstack_{\on{laft},(k-1)\on{-Artin}})_{\on{ev-coconn;all}}}\circ F$ along the projection
$$(\bI'\times \bI'')^{\on{op}}\to ([m]\times [n])^{\on{op}}.$$

\sssec{}

To extend the assignment
$$\CY^{m,n}\rightsquigarrow \IndCoh^{*!}(\CY^{m,n})$$
to a functor
$$\IndCoh^{*!}_{(\inftydgstack_{\on{laft},k\on{-Artin}})_{\on{ev-coconn;all}}}: 
\on{Cart}^{\bullet,\bullet}_{\on{ev-coconn;all}}(\inftydgstack_{\on{laft},(k-1)\on{-Artin}})\to 
\on{Seg}^{\bullet,\bullet}((\StinftyCat_{\on{cont}})^{\on{op}}),$$
we need to define the following data.

\medskip

Let $$\phi:([m']\times [n'])\to ([m]\times [n])$$
be a map in $\bDelta\times \bDelta$. 

\medskip

For $\CY^{m,n}\in \on{Cart}^{m,n}_{\on{ev-coconn;all}}(\inftydgstack_{\on{laft},k\on{-Artin}})$ as above, 
we denote by $\CY^{m',n'}$ the corresponding object of
$\on{Cart}^{m',n'}_{\on{ev-coconn;all}}(\inftydgstack_{\on{laft},k\on{-Artin}})$
obtained by restriction.

\medskip

We need to construct a canonical isomorphism between 
$$\IndCoh^{*!}(\CY^{m,n})\circ (\phi)^{\on{op}} \text{ and } \IndCoh^{*!}(\CY^{m',n'})$$
as functors $([m']\times [n'])^{\on{op}}\rightrightarrows \StinftyCat_{\on{cont}}$.

Furthermore, we need the above isomorphism to be homotopy coherent with respect to compositions
of morphisms in $\bDelta\times \bDelta$. 

\medskip

The map 
$$\IndCoh^{*!}(\CY^{m,n})\circ (\phi)^{\on{op}} \to \IndCoh^{*!}(\CY^{m',n'})$$
follows from the universal property of right Kan extensions. The fact that it
is an isomorphism is an easy cofinality argument.

\ssec{The *-pullback for eventually representable morphisms}

In this section we will further extend the assignment
$$\CY\to \IndCoh(\CY),\quad \CY\in \on{PreStk}_{\on{laft}}$$
where we can $!$-pullback along any morphisms, and $*$-pullback along 
eventually representable eventually coconnective morphisms of prestacks.

\sssec{}

Let
$$\on{ev-repr-coconn}\subset \on{all}$$
denote the class of eventually representable eventually coconnective morphisms
in $\on{PreStk}_{\on{laft}}$.

\medskip

Consider the corresponding object
$$\on{Cart}^{\bullet,\bullet}_{\on{ev-repr-coconn;all}}(\on{PreStk}_{\on{laft}})\in 
\inftygroup^{(\bDelta\times \bDelta)^{\on{op}}}.$$

We claim:

\begin{prop}
There exists a uniquely defined map in $\inftygroup^{(\bDelta\times \bDelta)^{\on{op}}}$
$$\IndCoh^{*!}_{\on{Cart}^{\bullet,\bullet}_{\on{ev-repr-coconn;all}}(\on{PreStk}_{\on{laft}})}: 
\on{Cart}^{\bullet,\bullet}_{\on{ev-repr-coconn;all}}(\on{PreStk}_{\on{laft}})\to 
\on{Seg}^{\bullet,\bullet}((\StinftyCat_{\on{cont}})^{\on{op}})$$
that makes the following diagram commute
$$
\CD
\on{Cart}^{\bullet,\bullet}_{\on{ev-repr-coconn;all}}(\on{PreStk}_{\on{laft}})
@>{\IndCoh^{*!}_{\on{Cart}^{\bullet,\bullet}_{\on{ev-repr-coconn;all}}(\on{PreStk}_{\on{laft}})}}>>  
\on{Seg}^{\bullet,\bullet}((\StinftyCat_{\on{cont}})^{\on{op}}) \\
@AAA        @AAA  \\
\pi_h(\on{Seg}^\bullet(\on{PreStk}_{\on{laft}})  @>>{\on{Seg}^\bullet(\IndCoh^!_{\on{PreStk}_{\on{laft}}})}>
\pi_h(\on{Seg}^\bullet((\StinftyCat_{\on{cont}})^{\on{op}}))  
\endCD
$$
and which extends the map 
$$\IndCoh^{*!}_{(\inftydgstack_{\on{laft},\on{Artin}})_{\on{ev-coconn;all}}}: 
\on{Cart}^{\bullet,\bullet}_{\on{ev-coconn;all}}(\inftydgstack_{\on{laft},\on{Artin}})\to 
\on{Seg}^{\bullet,\bullet}((\StinftyCat_{\on{cont}})^{\on{op}}).$$
\end{prop}

The rest of this subsection is devoted to a sketch of the proof of this proposition. 

\sssec{}

The idea of the proof is similar to that of \propref{p:!* IndCoh Artin}. 

\medskip

First, from the map
$$\IndCoh^{*!}_{(\inftydgstack_{\on{laft},\on{Artin}})_{\on{ev-coconn;all}}}: 
\on{Cart}^{\bullet,\bullet}_{\on{ev-coconn;all}}(\inftydgstack_{\on{laft},\on{Artin}})\to 
\on{Seg}^{\bullet,\bullet}((\StinftyCat_{\on{cont}})^{\on{op}})$$
we obtain that for a pair of $\infty$-categories
$\bI'$ and $\bI''$ and a functor 
$$F:\bI'\times \bI''\to \inftydgstack_{\on{laft},\on{Artin}}$$
with the corresponding properties, we can produce a functor
$$\IndCoh^{*!}_{(\inftydgstack_{\on{laft},\on{Artin}})_{\on{ev-coconn;all}}}\circ F:
(\bI'\times \bI'')^{\on{op}}\to \StinftyCat_{\on{cont}}.$$

\sssec{}

To define the map $\IndCoh^{*!}_{\on{Cart}^{\bullet,\bullet}_{\on{ev-repr-coconn;all}}(\on{PreStk}_{\on{laft}})}$,
we need to construct, for every
$$[m]\times [n]\in \bDelta\times \bDelta,$$
a map of $\infty$-groupoids
$$\on{Cart}^{m,n}_{\on{ev-repr-coconn;all}}(\on{PreStk}_{\on{laft}}) \to 
\on{Seg}^{m,n}((\StinftyCat_{\on{cont}})^{\on{op}}),$$
in a way functorial in $[m]\times [n]$.

\medskip

We fix an object $\CY^{m,n}\in \on{Cart}^{\bullet,\bullet}_{\on{ev-repr-coconn;all}}(\on{PreStk}_{\on{laft}})$
and will construct the corresponding functor
$$\IndCoh^{*!}(\CY^{m,n}):([m]\times [n])^{\on{op}}\to \StinftyCat_{\on{cont}}.$$

\sssec{}

We introduce an $\infty$-category 
$$_{\on{all}}\on{Cart}^{m,n}_{\on{ev-repr-coconn;all}}(\on{PreStk}_{\on{laft}}),$$
which has the same objects as $\on{Cart}^{m,n}_{\on{ev-repr-coconn;all}}(\on{PreStk}_{\on{laft}})$, 
but where we now allow arbitrary maps between such diagrams. 

\medskip

We let $\bJ$ denote the full subcategory in
$$\left({}_{\on{all}}\on{Cart}^{m,n}_{\on{ev-repr-coconn;all}}(\on{PreStk}_{\on{laft}})\right)_{/\CY^{m,n}}$$
whose objects are diagrams with entries from $\inftydgstack_{\on{laft},\on{Artin}}$.

\medskip

We set
$$\bI':=\bJ\times [m] \text{ and } \bI'':=[n].$$

We define the functor
$$F:(\bI'\times \bI'')^{\on{op}}\to \inftydgstack_{\on{laft},\on{Artin}}$$
by sending 
$$(g:\CZ^{m,n}\to \CY^{m,n})\in \bJ,\,\, a\in [0,m],\,\, b\in [0,n]$$
to $\CZ{}^{m,n}_{a,b}$.

\medskip

Consider the resulting functor
$$\IndCoh^{*!}_{(\inftydgstack_{\on{laft},\on{Artin}})_{\on{ev-coconn;all}}}\circ F:
(\bI'\times \bI'')^{\on{op}}\to \StinftyCat_{\on{cont}}.$$

We define the sought-for functor $\IndCoh^{*!}(\CY^{m,n})$ as the right Kan extension of
the functor $\IndCoh^{*!}_{(\inftydgstack_{\on{laft},\on{Artin}})_{\on{ev-coconn;all}}}\circ F$
along the projection
$$(\bI'\times \bI'')^{\on{op}}\to ([m]\times [n])^{\on{op}}.$$

\medskip

The functoriality of this construction in
$$([m]\times [n])\in (\bDelta\times \bDelta)^{\on{op}}$$
is established in the same way as in \propref{p:!* IndCoh Artin}. 

\ssec{Properties of $\IndCoh$ of Artin stacks}

According to \propref{p:!* IndCoh Artin}, if $\CY$ is an Artin stack locally almost of finite type,
there exists a well-defined category $\IndCoh(\CY)$, and pullback functors
$$f^!:\IndCoh(\CY)\to \IndCoh(S),$$
for $(f:S\to \CY)\in (\dgSch_{\on{aft}})_{\CY}$, and also 
$$f^{\IndCoh,*}:\IndCoh(\CY)\to \IndCoh(S),$$
if $f$ is eventually coconnective.

\medskip

This will allow to define certain additional pieces of structure on $\IndCoh(\CY)$,
which are less obvious to see when we view $\IndCoh(\CY)$ just as $\IndCoh^!(\CY)$.

\sssec{The functor $\Psi$}  \label{sss:Psi for stacks}

We claim that the functor $\IndCoh^*_{(\inftydgstack_{\on{laft},\on{Artin}})_{\on{ev-coconn}}}$
comes equipped with a natural transformation
$$\Psi^*_{(\inftydgstack_{\on{laft},\on{Artin}})_{\on{ev-coconn}}}:
\IndCoh^*_{(\inftydgstack_{\on{laft},\on{Artin}})_{\on{ev-coconn}}}\to 
\QCoh^*_{(\inftydgstack_{\on{laft},\on{Artin}})_{\on{ev-coconn}}},$$
where $\QCoh^*_{(\inftydgstack_{\on{laft},\on{Artin}})_{\on{ev-coconn}}}$ denotes the restriction
of $\QCoh^*_{\on{PreStk}_{\on{laft}}}$ to
$$((\inftydgstack_{\on{laft},\on{Artin}})_{\on{ev-coconn}})^{\on{op}}\to (\on{PreStk}_{\on{laft}})^{\on{op}}.$$

\medskip

Indeed, this natural transformation is obtained as the right Kan extension of
$$\Psi^*_{(\dgSch_{\on{aft}})_{\on{ev-coconn}}}:
\IndCoh^*_{(\affdgSch_{\on{aft}})_{\on{ev-coconn}}}\to 
\QCoh^*_{(\affdgSch_{\on{aft}})_{\on{ev-coconn}}},$$
given by \corref{c:upper * DG funct}.

\medskip

Above we have used the fact that the natural map
\begin{multline*}
\QCoh^*_{(\inftydgstack_{\on{laft},\on{Artin}})_{\on{ev-coconn}}}\to \\
\to \on{RKE}_{(((\affdgSch_{\on{aft}}))_{\on{ev-coconn}})^{\on{op}}\hookrightarrow 
((\inftydgstack_{\on{laft},\on{Artin}})_{\on{ev-coconn}})^{\on{op}}}(\QCoh^*_{(\dgSch_{\on{aft}})_{\on{ev-coconn}}})
\end{multline*}
is an isomorphism, which holds due to \propref{p:RKE from smooth general}.

\sssec{}

In plain terms, the existence of the above natural transformation means that 
for $\CY\in \inftydgstack_{\on{laft},\on{Artin}}$, there exists a
canonically defined functor 
$$\Psi_{\CY}:\IndCoh(\CY)\to \QCoh(\CY),$$
such that for a smooth map $f:\CY_1\to \CY_2$ between Artin stacks
we have a commutative diagram:
$$
\CD
\IndCoh(\CY_1)   @>{\Psi_{\CY_1}}>>  \QCoh(\CY_1) \\
@A{f^{\IndCoh,*}}AA    @AA{f^*}A   \\
\IndCoh(\CY_2)   @>{\Psi_{\CY_2}}>>  \QCoh(\CY_2). 
\endCD
$$

\sssec{}  \label{sss:Xi for stacks}

By a similar token, we claim that if $\CY\in \inftydgstack_{\on{laft},\on{Artin}}$ is eventually coconnective,
i.e., $n$-coconnective as a stack for some $n$ (see \cite{Stacks}, Sect. 4.6.3 for the terminology), 
the functor $\Psi_\CY$ admits a left adjoint, denoted $\Xi_\CY$, such that for $S\in \affdgSch_{\on{aft}}$ endowed with
an eventually coconnective map $f:S\to \CY$, the diagram
$$
\CD
\IndCoh(S)   @<{\Xi_{S}}<<  \QCoh(S) \\
@A{f^{\IndCoh,*}}AA   @AA{f^*}A   \\
\IndCoh(\CY)   @<{\Xi_{\CY}}<<  \QCoh(\CY)
\endCD
$$
commutes. Moreover, the functor $\Xi_{\CY}$ is fully faithful. 

\medskip

Indeed, we construct the corresponding natural transformation 
$$\Xi^*_{({}^{\leq n}\!\inftydgstack_{\on{laft},\on{Artin}})_{\on{ev-coconn}}}:
\QCoh^*_{({}^{\leq n}\!\inftydgstack_{\on{laft},\on{Artin}})_{\on{ev-coconn}}}\to 
\IndCoh^*_{({}^{\leq n}\!\inftydgstack_{\on{laft},\on{Artin}})_{\on{ev-coconn}}},$$
as the right Kan extension of the corresponding natural transformation
$$\Xi^*_{({}^{\leq n}\!\affdgSch_{\on{aft}})_{\on{ev-coconn}}}:
\QCoh^*_{({}^{\leq n}\!\affdgSch_{\on{aft}})_{\on{ev-coconn}}}\to 
\IndCoh^*_{({}^{\leq n}\!\affdgSch_{\on{aft}})_{\on{ev-coconn}}},$$
the latter being given by \secref{sss:Xi and pullback}.

\medskip

The adjunction for $(\Xi_\CY,\Psi_\CY)$ follows from \lemref{l:adjunction in limit}.

\medskip

In particular, we obtain that for $\CY$ eventually coconnective, the functor 
$\Psi_\CY$ realizes $\QCoh(\CY)$
as a co-localization of $\IndCoh(\CY)$.

\sssec{t-structure}  \label{sss:t structure on stacks}

We now claim:

\begin{prop}  \label{p:t structure on Artin}
Let $\CY$ be an object of $\inftydgstack_{\on{laft},\on{Artin}}$. Then the category $\IndCoh(\CY)$
admits a unique t-structure, characterized by the property that for
$S\in \dgSch_{\on{aft}}$ with a flat map $f:S\to \CY$, the functor 
$$f^{\IndCoh,*}:\IndCoh(\CY)\to \IndCoh(S)$$ is t-exact. Moreover:

\smallskip

\noindent{\em(a)} The above t-structure is compatible
with filtered colimits. 

\smallskip

\noindent{\em(b)} The functor $\Psi_\CY$ is t-exact, induces an equivalence 
$$\IndCoh(\CY)^+\to \QCoh(\CY)^+,$$ and identifies $\QCoh(\CY)$ with the left completion
of $\IndCoh(\CY)$ in its t-structure. 

\end{prop}

\begin{proof}

The proof follows from \corref{c:RKE from smooth general} combined with  \propref{p:equiv on D plus},
using the following general observation:

\medskip

Let 
$$A\to \StinftyCat_{\on{cont}},\,\, \alpha\mapsto \bC_\alpha$$
be a functor, and set
$$\bC:=\underset{\alpha\in A}{lim}\, \bC_\alpha.$$

Assume now that each $\bC_\alpha$ is endowed with a t-structure, and that for each arrow
$\alpha_1\to\alpha_2$ in $A$, the functor
\begin{equation} \label{e:trans functors}
\bC_{\alpha_1}\to \bC_{\alpha_2}
\end{equation}
is t-exact.

\medskip

Then the category $\bC$ has a unique t-structure characterized by the property that the evaluation 
functors $\bC\to \bC_\alpha$ are t-exact.

\medskip

Moreover, if the t-structure on each $\bC_\alpha$ is compatible
with filtered colimits, then the t-structure on $\bC$ has the same property.

\medskip

Point (b) of the proposition follows similarly. 

\end{proof}

\end{document}